\documentclass{acta_2024_us}
\usepackage{amssymb}
\usepackage{amsmath}
\usepackage{mathtools}

\usepackage[longnamesfirst]{natbib} \usepackage{har2nat}
\setcitestyle{comma,aysep={}} 
\shortcites{} 

\usepackage{newtxtext} 
\usepackage{newtxmath} 
\DeclareSymbolFont{CMlargesymbols}{OMX}{cmex}{m}{n} 
\DeclareMathDelimiter{(}{\mathopen} {operators}{"28}{CMlargesymbols}{"00}
\DeclareMathDelimiter{)}{\mathclose}{operators}{"29}{CMlargesymbols}{"01}
\DeclareMathAlphabet\mathcal{OMS}{cmsy}{m}{n} 
\SetMathAlphabet\mathcal{bold}{OMS}{cmsy}{b}{n} 
\pagerange{\pageref{firstpage}--\pageref{lastpage}}
\doi{XXXXXXXX}  
\usepackage[twoside,
	paperheight=247mm,
	paperwidth=174mm,
	marginratio={3:5,2:3},
	left=18mm,
	top=20mm]{geometry}
\usepackage{amssymb,amsmath}
\numberwithin{figure}{section}
\numberwithin{table}{section}
\usepackage{graphics,graphicx}
\usepackage[hang]{subfigure}
\usepackage[UKenglish]{babel}
\usepackage[shortlabels]{enumitem} 
\newcommand{\ignore}[1]{}
\usepackage[section]{placeins}
\allowdisplaybreaks
\usepackage{etoolbox}

\newtheorem{theorem}{Theorem}[section]
\newtheorem{informaltheorem}[theorem]{Informal Theorem}
\newtheorem{informallemma}[theorem]{Informal Lemma}

\newtheorem{lemma}[theorem]{Lemma}

\newtheorem{proposition}[theorem]{Proposition}
\newtheorem{corollary}[theorem]{Corollary}
\newtheorem{assumption}[theorem]{Assumption}
\newtheorem{example}[theorem]{Example}
\newtheorem{definition}[theorem]{Definition}
\newtheorem{remark}[theorem]{Remark}

\AtEndEnvironment{theorem}{\par\vspace{.5em}}
\AtEndEnvironment{informaltheorem}{\par\vspace{.5em}}
\AtEndEnvironment{informallemma}{\par\vspace{.5em}}
\AtEndEnvironment{lemma}{\par\vspace{.5em}}
\AtEndEnvironment{proposition}{\par\vspace{.5em}}
\AtEndEnvironment{corollary}{\par\vspace{.5em}}
\AtEndEnvironment{example}{\par\vspace{.5em}}
\AtEndEnvironment{definition}{\par\vspace{.5em}}
\AtEndEnvironment{assumption}{\par\vspace{.5em}}
\AtEndEnvironment{remark}{\par\vspace{.5em}}

\newcommand{\beq}{\begin{equation}}
\newcommand{\eeq}{\end{equation}}
\newcommand{\beqs}{\begin{equation*}}
\newcommand{\eeqs}{\end{equation*}}
\newcommand{\bit}{\begin{itemize}}
\newcommand{\eit}{\end{itemize}}
\newcommand{\ben}{\begin{enumerate}}
\newcommand{\een}{\end{enumerate}}
\newcommand{\bal}{\begin{align}}
\newcommand{\eal}{\end{align}}
\newcommand{\bals}{\begin{align*}}
\newcommand{\eals}{\end{align*}}
\newcommand{\bse}{\begin{subequations}}
\newcommand{\ese}{\end{subequations}}
\newcommand{\bpr}{\begin{proposition}}
\newcommand{\epr}{\end{proposition}}
\newcommand{\bre}{\begin{remark}}
\newcommand{\ere}{\end{remark}}
\newcommand{\bpf}{\begin{proof}}
\newcommand{\epf}{\end{proof}}
\newcommand{\ble}{\begin{lemma}}
\newcommand{\ele}{\end{lemma}}

\usepackage{tikz}
\usepackage{pgf}
\usetikzlibrary{arrows}
\usetikzlibrary{calc}
\usetikzlibrary{decorations.pathmorphing}
\usetikzlibrary{arrows.meta}
\tikzset{>=Stealth} 

\newcommand{\supp}{\operatorname{supp}}
\newcommand{\Op}{\operatorname{Op}_{\hbar}}
\newcommand{\Opk}{\operatorname{Op}_{k}}
\newcommand{\loc}{\operatorname{loc}}
\newcommand{\comp}{\operatorname{comp}}

\newcommand{\tr}{\rm tr}

\usepackage{soul}

\definecolor{escol}{rgb}{0,0,0.8}
\definecolor{estcol}{rgb}{0,0.5,0}
\definecolor{esnewcol}{rgb}{0,0.5,0}

\newcommand{\cR}{\mathcal{R}}
\newcommand{\cP}{\mathcal{P}}

\newcommand{\cV}{\mathcal{V}}

\newcommand{\cJ}{\mathcal{J}}
\newcommand{\cH}{\mathcal{H}}
\newcommand{\cZ}{\mathcal{Z}}
\newcommand{\cY}{\mathcal{Y}}
\newcommand{\cT}{\mathcal{T}}

\newcommand{\cF}{{\cal F}}

\newcommand{\tand}{ \text{ and } }

\newcommand*{\N}[1]{\left\|#1\right\|}

\newcommand{\pdiff}[2]{\frac{\partial #1}{\partial #2}}
\newcommand{\diff}[2]{\frac{d #1}{d #2}}

\newcommand{\fdspace}{\mathcal{V}_h}

\newcommand{\Rscat}{R_{\rm scat}}

\newcommand{\RPMLo}{R_{\rm PML}}

\newcommand{\uout}{u_{\rm out}}
\newcommand{\uin}{u_{\rm in}}

\newcommand{\tfa}{\text{ for all }}
\newcommand{\tfor}{\text{ for }}

\newcommand{\tin}{\text{ in }}
\newcommand{\ton}{\text{ on }}
\newcommand{\tas}{\text{ as }}

\newcommand{\re}{e}
\newcommand{\ri}{i}
\newcommand{\Rea}{\mathbb{R}}
\newcommand{\Com}{\mathbb{C}}
\newcommand{\PMLs}{g}
\newcommand{\e}{\epsilon}

\newcommand{\newm}{n}

\newcommand{\mythmname}[1]{\textbf{{(#1)}}}

 \newcommand{\cavity}{\mathcal{K}}
 \newcommand{\visible}{\mathcal{V}}
 \newcommand{\invisible}{\mathcal{I}}
 \newcommand{\pml}{\mathcal{P}}
 \newcommand{\mc}[1]{\mathcal{#1}}
 \newcommand{\bestt}[2]{\pi_{_{\!#1\to #2}}}
 \newcommand{\domainnumber}{M}
 \newcommand{\oldT}{W}
 \definecolor{light-blue}{rgb}{0.53,.8,98}
 \definecolor{verde}{RGB}{50,180,50}
 \newcommand{\diam}{\operatorname{diam}}
 \usepackage{mathrsfs}
 \newcommand{\transferIntro}{\mathscr{T}}
 \newcommand{\Oursubset}[1]{\Omega_{#1}'}
 \newcommand{\newell}{m}
 \newcommand{\operator}{\mathcal{A}}
 \newcommand{\pert}{\mathcal{L}}
 \newcommand{\Breg}{B_{\rm reg}}
\newcommand{\DL}{K}
\newcommand{\Reg}{V_{\ri k}}
\newcommand{\best}{\pi}
\newcommand{\lot}{{\rm l.o.t.}}

\newcommand{\rhoB}{\rho_{\!_\mathcal{A}}}
\newcommand{\mapF}{\mathscr{F}}

\newcommand{\vertiii}[1]{{\left\vert\kern-0.25ex\left\vert\kern-0.25ex\left\vert #1
    \right\vert\kern-0.25ex\right\vert\kern-0.25ex\right\vert}} 

\newcommand{\width}{\kappa}
\newcommand{\fPML}{f_{\newtheta}}
\newcommand{\DeltaPML}{\Delta_{\newtheta}}
\newcommand{\DeltaPMLj}{\Delta_{\newtheta, j}}

\newcommand{\newtheta}{{\rm s}}

\newcommand{\parallele}{e}
\newcommand{\bparallelepsilon}{\boldsymbol \epsilon}
\newcommand{\parallelepsilon}{\epsilon}
\newcommand{\parallelejn}{e_{j}^n}

\newcommand{\parallelujn}{u_{j}^n}

\newcommand{\parallelepsilonjn}{\epsilon_{j}^n}
\newcommand{\parallelu}{u}

\newcommand{\Tchi}{\chi^{>}}
\newcommand{\TTchi}{\Tchi}

\newcommand{\divergence}{\operatorname{div}}
\newcommand{\WF}{\operatorname{WF}_\hbar}

\newcommand{\LtG}{L^2(\Gamma_-)}
\newcommand{\tempIT}{}

\title[Numerical analysis of the Helmholtz equation using semiclassical analysis]{Numerical analysis of the high-frequency Helmholtz equation using semiclassical analysis 
}

\author[J. Galkowski and E. A. Spence]{Jeffrey Galkowski\\
Department of Mathematics, University College London,\\ London, WC1H 0AY, United Kingdom\\
E-mail: {j.galkowski@ucl.ac.uk}\\
\and
Euan A.~Spence\\
Department of Mathematics, University of Bath,\\ Bath, BA2 7AY, United Kingdom\\
E-mail: {e.a.spence@bath.ac.uk}}

\begin{document}

\maketitle
\label{firstpage}
\vspace{-1cm}
\begin{abstract}
We consider the numerical solution of high-frequency scattering problems modeled by the Helmholtz equation with a bounded obstacle.  
Although the analysis of this problem dates back at least 50 years,
over the past decade or so, tools and techniques from \emph{semiclassical analysis} have provided a new perspective and been used to settle several long-standing open problems in this area. 
Semiclassical analysis works in phase space (i.e., position and frequency) and describes rigorously the extent to which solutions of high-frequency PDEs are dictated by the properties of the corresponding geometric-optic rays.

The goals of the article are to (i) give a introduction to semiclassical analysis aimed at non-experts and (ii) showcase some of the numerical-analysis results about 
finite-element methods, boundary-element methods, and domain-decomposition methods obtained using semiclassical techniques.




\end{abstract}

\tableofcontents 



\section{Introduction:~numerical solution of the Helmholtz equation}

The Helmholtz equation 
\begin{equation}\label{e:Helmholtz}
    (-\Delta -k^2) u= 0
\end{equation}
with wavenumber $k>0$ is the simplest model of time-harmonic wave propagation. 
Indeed, if $u(x)$ satisfies \eqref{e:Helmholtz}, then 
\begin{equation}\label{e:wave1}
U(x,t) =u(x) e^{-i k t}
\end{equation}
is a time-harmonic solution of the wave equation (with speed one),
\beq\label{e:wave}
\pdiff{^2 U}{t^2} - \Delta U =0.
\eeq

This article is focused on the Helmholtz equation posed outside some compact scatterer (e.g., a penetrable or impenetrable obstacle, or a combination of the two) with data corresponding to an incident wave. 
(For precise definitions of these boundary-value problems, see \S\ref{s:scat}.)
We are specifically interested in the case when the wavenumber $k\gg1$; because $k$ dictates the frequency of the waves in the time domain via \eqref{e:wave1}, this limit is known as the \emph{high-frequency limit}.

\

\noindent\textbf{The main computational goal considered in this article:}
Given a specified accuracy and a fixed $k\gg1$, compute solutions of the Helmholtz equation \eqref{e:Helmholtz} (and its variable-coefficient generalisation) to that specified accuracy.

\

We highlight immediately that the $k\to \infty$ limit of a Helmholtz solution is governed by the behaviour of the geometric-optic rays (which are straight lines for constant wave speed). Furthermore, there is a large body of research into methods that directly compute the high-frequency limit of the Helmholtz equation -- methods based on ray-tracing and/or solving the Eikonal and transport equations; see, e.g., the Acta Numerica article \cite{EnRu:03}. The accuracy of these methods -- by design -- increases with $k$, but for any fixed $k$ the error is bounded below; i.e., it cannot be decreased arbitrarily. 
 In addition, because asymptotic methods rely on a well-chosen asymptotic ansatz for the solution, 
 these methods are intractable in complicated geometries.
 
 

Therefore, while for Helmholtz problems with extremely large $k$ and simple geometries  asymptotic methods give solutions accurate enough for some practical applications, there are a large class of Helmholtz problems arising from applications where achieving the goal above 
is necessary. 

To compute an approximation to the solution of the Helmholtz equation, broadly speaking one has to make the following three choices.
\begin{itemize}
\item[Choice 1:]~\emph{Choice of formulation}.

When the wave speed is (piecewise) constant, one can reduce solving the PDE \eqref{e:Helmholtz} 
to solving an integral equation on the boundary of the scatterer. When the wave speed is variable, one usually truncates the domain exterior to the scatterer (with some form of local \emph{absorbing boundary condition} \cite{EnMa:77,EnMa:77a} or \emph{perfectly matched layer (PML)} \cite{Be:94}) to minimise reflections of the waves from this artificial boundary) and considers some variational formulation of \eqref{e:Helmholtz}  posed in this truncated domain.

\item[Choice 2:]~\emph{Choice of finite-dimensional approximation space.} 

Both the finite-element method (FEM) and the boundary-element method (BEM) (i.e., solving boundary integral equations using the FEM) usually seek the solution as a  piecewise polynomial -- a mesh is imposed on the domain, and the solution is represented on each mesh element as a polynomial. 
Alternatively, one can try to tailor the space to the wave nature of the solution -- a classic choice here is to use plane waves on each element, i.e., $e^{i k x \cdot a_j}$ for some directions $a_j$, with such methods that use problem-adapted basis functions called \emph{Trefftz methods}.
We highlight that Choices 1 and 2 cannot always be made independently; e.g., since continuity across element interfaces cannot be imposed directly for many Trefftz basis functions (such as piecewise plane waves), they cannot be used in the standard Helmholtz variational formulation (see the discussion of variational formulations suitable for Trefftz methods in  \cite[\S2]{HiMoPe:16})

\item[Choice 3:]~\emph{Choice of how to solve the linear system.}

Choices 1 and 2 result in a linear system $A {\bf x} = {\bf b}$, where 
${\bf x}\in \mathbb{C}^N$ contains the coefficients of the approximation to the solution in the space of Choice 2. 
This linear system is large, indefinite, non-self-adjoint, and generally non-normal; solving such a system is a very challenging problem in numerical linear algebra.
The linear system is large because 
accurately approximating an arbitrary function oscillating with frequency $\lesssim k$ 
in $d$ dimensions requires $\sim k^d$ degrees of freedom (DoF)
(in other words, the space of functions oscillating at frequency $\lesssim k$ in $d$ dimensions has dimension $\sim k^d$).
Furthermore, the \emph{pollution effect} means that, when the polynomial degree is fixed, having $k^d$ DoF in the FEM is \emph{not} enough to maintain accuracy as $k\to \infty$ -- we discuss this in much more detail below. 
For the BEM, the linear systems require at least $\sim k^{d-1}$ DoF (due to the dimension reduction of working on the boundary), but are dense, as opposed to the sparse FEM matrices. 
The linear system is non-self-adjoint because of the condition that the scattered wave propagates away from the scatterer (i.e., the radiation condition). This non-self-adjointness means that the conjugate gradient method cannot be used to solve the system, and the method of choice is the generalised minimum residual method (GMRES) \cite{SaSc:86}.
The linear system of the standard variational formulation 
is indefinite (roughly speaking, has eigenvalues with both positive and negative real parts),
and iterative solvers famously struggle with this feature (see, e.g., \cite{ErGa:12}).
On top of all these difficulties, any ill-conditioning at the continuous level (e.g., due to trapped waves) is inherited by the discrete linear system 
(see, e.g., \cite{MaGaSpSp:22} for an investigation of this in the context of BEM).
Because the linear system is difficult to solve, one seeks a \emph{preconditioner}, i.e., a matrix $B^{-1}$ 
whose action is relatively cheap to compute such that $B^{-1} A {\bf x} = B^{-1}{\bf b}$ is easier to solve than $A {\bf x} = {\bf b}$.
Popular choices of preconditioners are those based on \emph{domain decomposition (DD)}, where an approximate inverse for $A$ is obtained by dividing the computational domain into subdomains and solving Helmholtz problems on these subdomains. 
\end{itemize}




\section{Summary of the contents of this article}

\subsection{The results}

This paper contains five main results. 

\begin{enumerate}[label=\textbf{R\arabic*}, ref=\textbf{R\arabic*}]
    \item \label{R1} 
    Upper bounds on the error for the $h$-version of the 
    FEM (i.e., with piecewise polynomials of fixed degree $p$)
    showing how the $k$-dependence of the meshwidth $h$ 
    dictates the accuracy as $k\to\infty$
    (extending the results in  \cite{GS3}). 
    (See Informal Theorem~\ref{t:informalR1}.)
    These bounds are observed empirically to be optimal when the mesh is uniform. 
    \item \label{R2} 
    Upper bounds on the error for the $hp$-version of the FEM showing how the $k$-dependence of $h$ and $p$ dictate the accuracy as $k\to\infty$ 
    (extending the results in \cite{LSW3, LSW4, GLSW1}). In particular, these estimates show that $k$-uniform accuracy can be maintained with $\sim k^d$ degrees of freedom; i.e., the $hp$-FEM \emph{does not suffer from the pollution effect} 
     (see Informal Theorem~\ref{t:informalR2}.)
    As a component of this analysis, Appendix~\ref{a:poly} gives a self-contained treatment of $p$-and $h$-explicit polynomial approximation 
    %
with novel aspects including precise Sobolev-space estimates for polynomial approximation of high-regularity functions.
    \item \label{R3} Upper and lower bounds on the error for the $h$-version of the BEM showing how $h$ must depend on $k$ to maintain accuracy as $k\to\infty$ \cite{GS2, GaRaSp:25}. 
    In particular, these prove that, in many situations, $\sim k^{d-1}$ degrees of freedom is \emph{not} enough to maintain accuracy as $k\to\infty$; i.e., the $h$-BEM \emph{suffers from the pollution effect}. (See Informal Theorems~\ref{t:informalR3a} and~\ref{t:informalR3b}.)
    \item \label{R4} Non-uniform $h$-FEM meshes for scattering problems with PML truncation, where the meshes are defined by ray dynamics \cite{AGS2}. 
    These meshes violate the conditions on $h$ given by \ref{R1} and involve coarser meshes away from trapping and in the PML while still maintaining accuracy as $k\to \infty$. 
    (See Table \ref{tab:regimes}.)
    The meshes are designed using bounds on the Galerkin error showing how varying the meshwidth in one region affects errors both in that region and elsewhere.    
    \item \label{R5} Analysis of overlapping Schwarz methods for the Helmholtz problem, truncated by a PML, where the subdomains have an absorbing layer at their boundaries, such as a PML, and the number of iterations in the convergence results is specified in terms of the interaction between the ray dynamics and the geometry of the subdomains \cite{GGGLS2}. 
    (See Informal Theorem~\ref{t:informalR5}.)
\end{enumerate}
 These results are stated informally (i.e., with minimum technicalities) in \S\ref{s:informal}, along with a discussion of the history and context of each result. Later in the article, we give the precise statements, along with complete proofs of \ref{R1}, \ref{R2}, and the upper bounds in \ref{R3}, and sketch proofs of the lower bounds in \ref{R3} and the results \ref{R4} and \ref{R5}.
 
Therefore, in terms of the three choices in the previous subsection, the results of this paper are mainly concerned with the standard variational formulation posed in a truncated domain exterior to the obstacle, with piecewise-polynomial approximation spaces, and (one-level) domain decomposition used as a solver.






\subsection{How these results were obtained: semiclassical analysis}\label{s:how}

The unifying feature of the above five results is that they were obtained using \emph{semiclassical analysis (SCA)}. SCA is a branch of PDE theory that, in particular, 
%
%
describes rigorously the extent to which Helmholtz solutions in the $k\to \infty$ limit 
are governed by 
the rays. 
In our context, the adjective ``semiclassical'' 
essentially means ``high frequency'',
 but 
comes from the fact that 
these techniques can be used to show how particle motion in classical mechanics 
arises from wave functions in quantum mechanics; see \cite[\S1.2]{Zw:12}. 
SCA is part of the wider field of \emph{microlocal analysis}, a generalisation of Fourier analysis:
Fourier analysis splits functions up into components oscillating at different frequencies, whereas 
 microlocal analysis localises in both frequency and space (to the extent allowed by 
 the uncertainty principle).
 The 
combination of physical and Fourier space is then known as \emph{phase space}.

Two key concepts from SCA used to obtain results \ref{R1}-\ref{R5} above are \emph{ellipticity} and \emph{propagation}. We now illustrate these in the constant-coefficient case using the Fourier transform. Given $u:\mathbb{R}^d\to \mathbb{C}$, let
\beq\label{e:Fourier}
\mathcal{F}_k(u)(\xi) := \int_{\mathbb{R}^d} e^{-i k\langle\xi, x\rangle} u(x) d x;
\eeq
i.e., $\mathcal{F}_k(u)$ is the Fourier transform of $u$, with frequency variable scaled by $k$ (so that frequency $k$ corresponds to $|\xi|=1$). If $u$ solves the Helmholtz equation \eqref{e:Helmholtz} then 
\beq\label{e:FourierMultHelmholtz}
(|\xi|^2 -1)\mathcal{F}_k(u)(\xi)=0.
\eeq
When $||\xi| -1| \geq c>0$ (i.e., at unscaled frequencies far from $k$) the symbol $|\xi|^2 -1$ is invertible; thus the inverse of the Helmholtz operator acts like the inverse of an elliptic operator (e.g., the inverse of the Laplacian) on such frequencies. 
However, when $|\xi|$ is close to $1$ (i.e., the unscaled frequency is close to $k$), SCA shows that the inverse of the Helmholtz operator with $k\gg1$  encodes propagation along geometric-optic rays.


We highlight that the results \ref{R1}, \ref{R2}  -- the analyses of the $h$-FEM and $hp$-FEM -- \emph{only use ellipticity at high frequencies} (i.e., no information about propagation is required). 
The results \ref{R3}, \ref{R4}, and \ref{R5} then use information about \emph{both} ellipticity \emph{and} propagation.

Results not included in this article, but also using SCA ideas, are the following four results, 
with the first one using propagation, and the last three using ellipticity.
\bit
\item Sharp upper and lower bounds on the error incurred by truncating the unbounded domain exterior to an obstacle by a local absorbing boundary condition \cite{GLS1}, with these bounds showing that the relative error is bounded away from zero, independent of the frequency, and regardless of the geometry of the artificial boundary.
\item Bounds showing that, for general Helmholtz scattering problems (specifically, those in the so-called ``black-box scattering framework" \cite{SjZw:91}) the error incurred by truncation by a PML is exponentially small in frequency, the width of the PML, and the strength of the scaling in the PML region 
\cite{GLS2} (these results are summarised in \S\ref{s:PMLaccuracy} below).
\item New finite-dimensional approximation spaces for the high-frequency Helmholtz problems with smooth variable coefficients based on coherent states using $O(k^{d-\frac{1}{2}})$ degrees of freedom
\cite{ChDoIn:24}.
\item Convergence theory for two-level hybrid Schwarz domain-decomposition preconditioners for high-frequency Helmholtz problems \cite{GS4}.
\eit

\bre[Hybrid numerical-asymptotic methods]
When one knows the asymptotic form of the solution as $k\to \infty$, up to amplitudes that have a $k$-independent rate of oscillation, one can build this information into the approximation space. These spaces allow for methods that require only $O(1)$ degrees of freedom as $k\to \infty$.

There is a long history of such methods, going back to the 1970s -- see the Acta Numerica review article \cite{ChGrLaSp:12} (and, in particular, the literature review in Section 1) -- and in recent years these methods have come to be known as 
\emph{hybrid numerical-asymptotic methods}. 

These approximation spaces have mostly been developed for boundary integral methods (because the dimension reduction makes the asymptotic form of solutions more tractable). Two situations where these methods can be rigorously analysed are scattering of highly-structured incident fields
 (e.g. plane waves or point sources) by
\begin{enumerate}
\item smooth, strictly convex obstacles, where the asymptotic form of the solution was obtained in \cite{MeTa:85} using microlocal methods, and 
\item certain nontrapping polygons
where asymptotic form of the solution was obtained by the method of images together with 
representations of the solution near corners using separation of variables \cite[Theorems 3.2 and 3.3]{ChLa:07}, \cite[Theorem 3.6]{ChHeLaTw:15}. (The form of the diffracted wave can also be obtained from a microlocal approach; see~\cite{CheTay82,CheTay82a}.)
\end{enumerate}
While hybrid numerical-asymptotic methods therefore fall into the category of numerical methods/analyses that (at least morally) use semiclassical analysis, they require detailed a priori knowledge of both the data and the scatterer. In this article, we focus on methods/analyses using semiclassical ideas that are applicable without such detailed a priori knowledge (with the resulting methods then requiring more degrees of freedom than hybrid numerical-asymptotic methods).
\ere


\subsection{From Fourier multipliers to pseudodifferential operators}\label{s:Fouriertopseudo}

In the previous section, we saw that the constant-coefficient Helmholtz operator is a Fourier multiplier; 
i.e., 
$$
(-k^{-2}\Delta -1)u(x):=\cF_k^{-1}\big(
(|\bullet|^2-1)
\cF_k(u)(\bullet) \big)(x).
$$
In motivating the use of pseudodifferential operators to study more-general Helmholtz operators, 
it is instructive to first consider the PDE
\beq\label{e:modH1}
(-k^{-2}\Delta + 1)v =f \quad\ton \Rea^d
\eeq
(sometimes called the \emph{modified Helmholtz equation})
which becomes 
\beqs
(|\xi|^2 +1) \cF_k(v)(\xi) =\cF_k(f)(\xi)
\eeqs
after taking the ($k$-weighted) Fourier transform \eqref{e:Fourier}. Since $|\xi|^2+1 \geq 1>0$ for all $\xi$, both the partial differential operator $(-k^{-2}\Delta + 1)$ and its inverse $(-k^{-2}\Delta + 1)^{-1}$ are Fourier multipliers and the solution $v$ to \eqref{e:modH1} can be written in terms of the multiplier via
\beq\label{e:modH2}
v(x) = \cF_k^{-1}\big((|\bullet|^2+1)^{-1} \cF_k(f)(\bullet) \big)(x).
\eeq
One can understand how the Fourier transform arrives at \eqref{e:modH2} in two ways:
\begin{enumerate}
\item Since the partial differential operator $(-k^{-2}\Delta + 1)$ is translation invariant (and hence is a convolution operator), the Fourier transform diagonalises $(-k^{-2}\Delta + 1)$ and thus gives the full description of how $(-k^{-2}\Delta + 1)$ and its inverse act in a (generalised) basis of $L^2$. 
\item The Fourier transform decomposes functions into its different frequencies, and
$(-k^{-2}\Delta + 1)$ acts by multiplication by  $\sigma(\xi):=|\xi|^2+1$ on each frequency. Thus, since $|\sigma(\xi)|\geq 1>0$, multiplication by $1/\sigma(\xi)$ gives the inverse operator.
\end{enumerate}
For constant coefficient PDEs on $\mathbb{R}^d$, Points 1 and 2 are actually the same phenomenon. However, their generalisations to more complicated operators, e.g., the variable-coefficient generalisation of \eqref{e:modH1}, namely
\beq\label{e:modH3}
-k^{-2}\divergence(A\nabla v) + n v =f \quad\ton \Rea^d,
\eeq
are different.
\begin{enumerate}
\item The analogue of the first approach 
is to expand in eigenfunctions of a self-adjoint operator and use the functional calculus/spectral theorem to decompose the operator. For this to be effective, one needs an appropriate self-adjoint operator, and information about its eigenfunctions.

\item The analogue of the second approach is to use microlocal analysis; i.e., to decompose in both space and frequency
and observe that the PDE \eqref{e:modH3} approximately acts by multiplication by $\sigma(x,\xi):= \langle A(x)\xi, \xi\rangle + n(x)$. Multiplication by $(\sigma(x,\xi))^{-1}$ is then a good approximation to the action of the inverse of the PDE. At a technical level, \emph{semiclassical pseudodifferential operators} enable us to make this second approach work. 
\end{enumerate}

The result \ref{R1} is, in fact, obtained using only Approach 1 (i.e., expanding in eigenfunctions of a self-adjoint operator, whilst still ``thinking semiclassically").  \ref{R2} and \ref{R4} intertwine Approaches 1 and 2, using the both expansions in eigenfunctions and the tools of semiclassical analysis,
while \ref{R3} and \ref{R5} use only semiclassical tools.




The standard (i.e., non-semiclassical) calculus of pseudodifferential operators was introduced in the 1960's to obtain approximate inverses, modulo smooth errors, for elliptic variable coefficient PDEs. 
The semiclassical pseudodfferential calculus then builds in the large parameter $k$ to produce approximate inverses to operators such as \eqref{e:modH1} and \eqref{e:modH3} (i.e., 
\emph{semiclassically} elliptic operators), modulo an error that is both smoothing \emph{and} smaller than any algebraic power of $k$. 
Semiclassical pseudodifferential operators take the following form:~given a function $a(x,\xi)$ on phase space, 
\beq\label{e:quantIntro}
(\widetilde{\Opk}(a)u)(x):=\frac{k^d}{(2\pi)^d}\int_{\Rea^d}\int_{\Rea^d} e^{ik\langle x-y,\xi\rangle}a(x,\xi)u(y)dyd\xi
\eeq
is a semiclassical pseudodifferential operator with symbol $a$ (the reason for our use of the tilde is explained in \S\ref{s:symbol}). Observe that if $a(x,\xi)=a(\xi)$, then $(\widetilde{\Opk}(a)v)(x)= \mathcal{F}_k^{-1} \big(a(\bullet)(\mathcal{F}_k v)(\bullet)\big)(x)$, i.e., $\widetilde{\Opk}(a)$ is a Fourier multiplier.

As discussed in \S\ref{s:how} (after \eqref{e:FourierMultHelmholtz}), 
the Helmholtz operator is semiclassically elliptic only at high frequencies (i.e., at frequencies $>k$), and thus its inverse is \emph{not} a semiclassical pseudodifferential operator. 
In this article, we obtain information about the inverse of the Helmholtz operator via propagation estimates (see \S\ref{s:propagate}), which describe how the phase-space structure of the Helmholtz solution is governed by the ray dynamics. 


The discussion above described how results relying on semiclassical ellipticity (such as \ref{R1} and \ref{R2}) can sometimes be obtained without explicit use of SCA.
Furthermore, for some special geometries and coefficients (e.g., star-shaped obstacles and coefficients satisfying certain monotonicity properties) bounds on the solution operator (which implicitly contain information about propagation) can also be obtained without explicit use of SCA -- see the discussion in \S\ref{s:energy} and Remark \ref{r:nt}. However, 
we know of no other way of obtaining the results 
relying on propagation (and depending on ray dynamics) in \ref{R3}, \ref{R4}, and \ref{R5} 
without using the semiclassical machinery.

\subsection{The technical overhead required to use SCA, and how we deal with this in this article}\label{s:SCAintro}

One of the main obstacles in using semiclassical/microlocal techniques is the technical overhead in stating and proving results about pseudodifferential operators. 
In this article we deal with this in the following ways. 

Section \ref{s:SCA} contains the results from SCA used in the article. While these results are stated rigorously, their proofs are sketched using informal statements of the fundamental properties of semiclassical pseudodifferential operators; see Informal Theorem \ref{t:rules} below. 

The precise statement of Informal Theorem \ref{t:rules} is then given in an appendix (see Theorem \ref{t:rules2}), referencing the complete treatments of SCA in~\cite{DiSj:99}, \cite{Zw:12}, and \cite[Appendix E]{DyZw:19}.
The sketch proofs can then be made rigorous by repeating the arguments with Informal Theorem \ref{t:rules} replaced by Theorem \ref{t:rules2} (although we do not do this here). 




Our hope is that this approach provides an accessible 
introduction to the ideas behind semiclassical analysis and the 
``phase-space point of view". 
We emphasise that it is this phase-space point of view that gives 
the conceptual/heuristic ideas underpinning the results in this article, with the technical details of the pseudodifferential calculus then used in the rigorous implementation of these ideas. 
Furthermore, even in rigorous proofs, it is often possible to use
Theorem \ref{t:rules2} 
as a ``black box" without needing the details of how these fundamental properties of pseudodifferential operators  are obtained.

\section{Informal statement of the main results}\label{s:informal}

The scattering problem is stated precisely in \S\ref{s:scat}; we state it informally here as:~given a bounded open set $\Omega_-\subset\Rea^d$, $d\geq 1$, with connected open complement $\Omega_+$, variable coefficients $A$ and $n$ such that $A\equiv I$ and $n\equiv 1$ outside a compact set, and a compactly-supported source $f$, find $u$ satisfying the Sommerfeld radiation condition and
such that 
\beq\label{e:informalPDE}
Pu:= -k^{-2}{\operatorname{div}}
( A\nabla u) -n u=f\text{ in }\Omega_+, \quad Bu|_{\partial\Omega_+}=0,
\eeq
where $B=I$ or $B=\partial_{\nu,A}$ (the conormal derivative).

We work in norms where each derivative is scaled by $k^{-1}$; i.e., for an open set $\Omega\subset\Rea^d$, we define for $m\in\mathbb{N}$
\beq\label{e:norm}
\| u\|_{H^m_k(\Omega)}^2:= 
\sum_{\substack{|\alpha|\leq m\\\alpha\in\mathbb{N}^d}}
\frac{1}{|\alpha|!\alpha!}\|(k^{-1}\partial)^{\alpha}u\|_{L^2(\Omega)}^2.
\eeq
(In this section, the $\alpha$ dependence of the constants in \eqref{e:norm} is not important, but it \emph{is} important  for the results on $p$-explicit polynomial approximation in \S\ref{a:poly} and hence also for obtaining the $hp$-FEM result \ref{R2}.)
For $m\in\mathbb{N}$ and $\Omega\subset \mathbb{R}^d$ open, we also define 
$$
H_0^m(\Omega):=\overline{C_c^\infty(\Omega)}^{\|\cdot \|_{H_k^m(\Omega)}},
$$
and $H_k^{-m}(\Omega):= (H_k^m(\Omega)\cap H_0^m(\Omega))^*$ (where $^*$ denotes the dual space, and the 
norm is defined as a linear functional on $H_k^m(\Omega)$).

Let $\mathcal{R}$ be the solution operator for the problem \eqref{e:informalPDE} (i.e., the map $f\mapsto u$, aka the resolvent).
Given $\chi \in C_c^\infty(\mathbb{R}^d)$ with $\chi \equiv 1$ on 
$\supp(A-I) \cup \supp(n-1) \cup \Omega_-$, 
let
\begin{equation}\label{e:rho}
    \rho(k):= \|\chi \mathcal{R}(k)\chi\|_{L^2\to L^2};
\end{equation}
i.e., $\rho(k)$ is the $L^2\to L^2$ norm of the cutoff resolvent. (One can show that, up to a $k$-independent constant, $\rho(k)$ does not depend on the choice of $\chi$.)
With this definition, $\rho(k)\sim k$ for nontrapping problems (i.e., when all geometric-optic rays escape a neighbourhood of the scatterer)
and $\rho(k)\gg k $ for trapping problems (see the references in \S\ref{s:rho} below). Furthermore, $\rho(k)$ is polynomially bounded in $k$ for ``most" frequencies (see \S\ref{s:rho}).

The advantage of scaling each derivative with $k^{-1}$ -- both in the Helmholtz equation \eqref{e:informalPDE} and in the norm \eqref{e:norm} -- is that the norm of $\mathcal{R}(k)$ has the same $k$-dependence between any spaces on which it is defined; e.g., given $\widetilde{\chi}\in C^\infty_c(\Rea^d)$ with $\supp \widetilde{\chi} \cap \supp(1-\chi)=\emptyset$, $\|\widetilde{\chi} \cR(k)\chi\|_{L^2 \to H^1_k}\leq C (1+\rho(k))$ 
-- this can be proved by Green's identity; see, e.g., \cite[Lemma 2.2]{Sp:14}.

\subsection{\ref{R1}: $k$-explicit error analysis of the $h$-FEM}\label{s:informalR1}



\begin{informaltheorem}
\mythmname{Preasymptotic error estimates for the $h$-FEM}
\label{t:informalR1} 
Suppose that the Helmholtz problem \eqref{e:informalPDE} -- with the Sommerfeld radiation condition approximated by \emph{either} a radial perfectly matched layer \emph{or} the exact Dirichlet-to-Neumann map on the boundary of a ball, \emph{or} an impedance boundary condition -- 
is solved using the Galerkin FEM  with degree $p$ polynomials (with approximation space denoted by $\cV_h$). 
If the coefficients are 
piecewise $C^{p-1,1}$, the boundary of the domain and the boundaries across which the coefficients jump are all $C^{p,1}$, and 
\beqs
\rho(k)(hk)^{2p}
\text{ is sufficiently small,} 
\eeqs
then the Galerkin solution $u_h$ exists, is unique, and satisfies
\begin{align}\label{e:highGalerkinintro}
\N{u-u_h}_{H^{1}_k(\Omega)}&\leq C \Big(1 + \rho(k)(hk)^p \Big) \min_{v_h \in \fdspace} \N{u-v_h}_{H^1_k(\Omega)}
\end{align}
and, for $\newm\in\{1,\ldots,p\}$,
\beq\label{e:lowGalerkinintro}
\N{u-u_h}_{H^{1-\newm}_k(\Omega)}\leq C \Big((hk)^\newm + \rho(k)(hk)^p 
\Big) \min_{v_h \in \fdspace} \N{u-v_h}_{H^1_k(\Omega)}.
\eeq
Furthermore, if the data $f$ has frequency $\lesssim k$
then 
\beq\label{e:relError}
\frac{\N{u-u_h}_{H^1_k(\Omega)}}
{
\N{u}_{H^1_k(\Omega)}
}
\leq C \Big(1  +  \rho(k)(hk)^p\Big)(hk)^p.
\eeq
\end{informaltheorem}

The rigorous statement of Informal Theorem \ref{t:informalR1} (in the specific case of PML truncation) is Theorem \ref{t:R1} below.

\paragraph{The context of Informal Theorem \ref{t:informalR1}.}

Informal Theorem \ref{t:informalR1} contains the following two results:
\begin{enumerate}[label=\arabic*, ref=\arabic*]
\item
\label{i:asymptotic}
if $\rho(k)(hk)^p$ is sufficiently small then the Galerkin solution is \emph{quasioptimal} (i.e., the Galerkin error is bounded by a multiple of the best approximation error), uniformly in $k$, via \eqref{e:highGalerkinintro}, and 
\item\label{i:preasymptotic} 
for physically-relevant data, the relative $H^1_k$ error can be made controllably small by making $\rho(k)(hk)^{2p}$ sufficiently small 
via \eqref{e:relError}.
\end{enumerate}
Regarding Result~\ref{i:asymptotic} (quasioptimality):~the Schatz duality argument \cite{Sc:74} (recapped in Lemma \ref{l:Schatz} below) combined with $H^2$ regularity of the solution implies quasioptimality for any $p\geq 1$ when $\rho(k)hk$ is sufficiently small; 
this was first recognised for 1-d problems in \cite{AzKeSt:88}, with a $k$-explicit bound on the solution operator obtained using Rellich/Morawetz identities (discussed in \S\ref{s:energy} and Remark \ref{r:nt} below). 
Other uses of this argument on 1-d problems included \cite[Lemma 2.6]{DoSaShBe:93}, \cite[Theorem 3]{IhBa:95}, \cite[Theorem 3.2]{MaIhBa:96} and on the 2-d interior impedance problem in \cite[Prop.~8.2.7]{Me:95}. 

Result \ref{i:asymptotic} was then proved for certain constant-coefficient Helmholtz problems 
\cite{MeSa:10}, \cite{MeSa:11} as part of their  $hp$-FEM analysis (discussed in \S\ref{s:informalR2} below).
One notable feature of this work is that (as far as the authors are aware) it was the first 
use in a numerical-analysis context of the fact that the Helmholtz solution operator is 
well-behaved on high frequencies (which 
we understood in \S\ref{s:how} in terms of semiclassical ellipticity ).

Result \ref{i:asymptotic} for general Helmholtz problems satisfying essentially the same assumptions as the rigorous version of Informal Theorem \ref{t:informalR1}
(see Corollary \ref{c:preasymptotic}/Theorem \ref{t:R1} below) was then proved in \cite{ChNi:20}. These assumptions are essentially that the solution operator satisfies the natural elliptic-regularity shift:~i.e., if the data $f\in H^{m-1}$ then $u\in H^{m+1}$ for $0\leq m\leq p$  (see Assumptions \ref{a:ellipticRegularity} 
and \ref{a:adjointEllipticRegularity}
below). 


Regarding Result \ref{i:preasymptotic} (controllably-small relative error):~the 
threshold ``$k(hk)^{2p}$ sufficiently small'' for controllably-small relative error was first identified for $p=1$ in \cite{IhBa:95} and for general $p\in \mathbb{Z}^+$ in \cite{IhBa:97} for constant-coefficient 1-d problems (where the solution can be written down explicitly). These papers also introduced the terminology that the solution is quasioptimal in the \emph{asymptotic regime}, the complement of this is the \emph{preasymptotic regime}, and the fact that the asymptotic regime is not ``$hk$ sufficiently small" (and thus $\gg k^d$ degrees of freedom is required to maintain accuracy) is the \emph{pollution effect} (see also \cite{BaSa:97} and Definition~\ref{d:pollution} below).
The condition ``$k(hk)^{2p}$ sufficiently small'' also appears from so-called \emph{dispersion analysis}, namely analysis of the FEM linear system for the constant coefficient Helmholtz equation 
\cite{HaHu:91, IhBa:95a, IhBa:95, Ih:98,
DeBaBo:99, Ai:04}.

Result \ref{i:preasymptotic} when $p=1$ was proved 
\bit
\item for the interior impedance problem for discontinous Galerkin methods in \cite{FeWu:09, FeWu:11}
and for continuous interior penalty methods and the standard FEM in \cite{Wu:14}, and
\item for truncation via the exact Dirichlet-to-Neumann map \cite[Theorem 4.1]{LSW2}, 
truncation via a radial, $k$-independent PML \cite[Theorem 4.4]{LiWu:19}, and truncation via a radial, $k$-dependent PML \cite[Theorem 7.2]{ChGaNiTo:22}.
\eit
All these results used the \emph{elliptic projection} argument introduced in \cite{FeWu:09} -- we discuss this argument in detail in \S\ref{s:ellipticProjection} below.

Result \ref{i:preasymptotic} for general $p\in \mathbb{Z}^+$ was then proved for the interior impedance problem in 
\cite[Theorem 5.1]{DuWu:15} (for constant coefficients) and \cite[Theorem 2.39]{Pe:20} (for variable coefficients), and then for general Helmholtz problems satisfying an elliptic-regularity shift in \cite{GS3} (as described in Informal Theorem \ref{t:informalR1}).

\subsection{\ref{R2}: $k$-explicit error analysis of the $hp$-FEM}\label{s:informalR2}

\begin{informaltheorem}[The $hp$-FEM does not suffer from the pollution effect]
\label{t:informalR2}
Suppose that the Helmholtz problem \eqref{e:informalPDE} with a Dirichlet boundary condition and the Sommerfeld radiation condition  approximated by a radial PML is solved using the Galerkin FEM with degree $p$ polynomials. 

Suppose that the coefficients are analytic near $\Omega_-$ and smooth elsewhere, $\Omega_-$ has analytic boundary, the PML truncation boundary is $C^{1,1}$ and $\rho(k)$ is polynomially bounded in $k$. Then there exists $C>0$ such that given $k_0, \epsilon>0$ there exists $c>0$ such that 
if $k\geq k_0$, 
\beqs
\frac{hk}{p} \leq c, \quad\tand\quad
p\geq 1+\epsilon \log k,
\eeqs
then the Galerkin solution $u_h$ exists, is unique, and satisfies
\begin{align}\label{e:H1bound2}
\N{u-u_h}_{H^1_k(\Omega)}&\leq C \min_{v_h \in \fdspace} \N{u-v_h}_{H^1_k(\Omega)}.
\end{align}
\end{informaltheorem}

The rigorous statement of Informal Theorem \ref{t:informalR2} is Theorem \ref{t:concreteHPFEM} below.

\paragraph{The context of Informal Theorem \ref{t:informalR2}.}

With a uniform mesh, the number of degrees of freedom of the piecewise polynomial approximation space $\sim (p/h)^d$; thus 
Informal Theorem \ref{t:informalR2} implies that, with $hk/p=c$, $k$-uniform quasioptimality is achieved with the number of degrees of freedom $\sim k^d$; i.e., the $hp$-FEM does not suffer from the pollution effect.

The first results of this form were proved 
for constant-coefficient Helmholtz problems 
in \cite{MeSa:10, MeSa:11, EsMe:12, MePaSa:13}. As mentioned in the discussion in \S\ref{s:informalR1}, these results used the fact that Helmholtz operator is well behaved on high frequencies, with the splitting achieved via the Fourier transform in $\Rea^d$ and then combined with extension operators to define splittings on Lipschitz domains (see \cite{MeSa:11}).
 
Results about quasioptimality of the $hp$-FEM applied to variable coefficient problems were then obtained via two approaches:
\ben[label=\arabic*, ref=\arabic*]
\item\label{i:GLSW} \cite{LSW3, LSW4, GLSW1}, and
\item\label{i:BCM} \cite{BeChMe:25},
\een 
with the rigorous version of Informal Theorem \ref{t:informalR2} proved in \cite{GLSW1}.

Both approaches \ref{i:GLSW} and \ref{i:BCM} use the fact that the Helmholtz solution operator is well behaved on high frequencies. In \cite{LSW3, LSW4, GLSW1} this is understood and implemented via semiclassical ellipticity (with the expository paper \cite{Sp:23} reproducing the result in \cite{MeSa:10} using this approach). \cite{BeChMe:25} uses semiclassical ellipticity implicitly via 
the fact that $(-k^{-2}\Delta-1)^{-1}$ and 
$(-k^{-2}\Delta+1)^{-1}$ 
act similarly on high-frequency data (see the discussion in \cite[\S2.5]{BeChMe:25}.
The approach of \cite{LSW3, LSW4, GLSW1} 
allows for smooth coefficients (with no obstacles) and the addition of an analytic Dirichlet obstacle, provided that the coefficients are analytic in a neighbourhood of the obstacle. The approach of \cite{BeChMe:25} considers piecewise-analytic coefficients and analytic boundaries/interfaces.

In this article, Informal Theorem \ref{t:informalR2} is proved 
via establishing, first, a new abstract result about quasioptimality of the $hp$-FEM under abstract regularity-type assumptions on the natural ``low frequencies" associated with the operator (see Theorem \ref{t:hpFEMAbstract} below); this result is conceptually similar to -- and simpler than --  the analysis in \cite{LSW4, GLSW1} that reduces proving quasioptimality of the $hp$-FEM applied to ``black box" scattering problems (in the sense of \cite{SjZw:91}) to proving certain elliptic-regularity-type statements. 

The regularity assumptions in this abstract result are then shown to be satisfied in the set up of Informal Theorem \ref{t:informalR2} (see \S\ref{s:analyticestimates} below).
In fact, we expect these regularity-type assumptions to also hold for coefficients that are piecewise Gevrey in a neighbourhood of a Gevrey interface (which would then link -- and extend -- the approaches \ref{i:GLSW} and \ref{i:BCM} above), but, to the best of the authors' knowledge, the elliptic-regularity results necessary are not immediately available in the literature.






\subsection{\ref{R3}:~Optimal error estimates for the $h$-BEM}\label{s:informalR3}

When $A\equiv I$ and $n\equiv1$ the scattering problem \eqref{e:informalPDE} can be reformulated as a boundary integral equation (BIE) on $\Gamma_-:=\partial\Omega_-$ (see Appendix~\ref{app:A}). We consider the standard second-kind BIE formulations for the Helmholtz scattering problems, which for the Dirichlet problem involve the operators 
\begin{equation}\label{e:DBIEs}
A_k':= \tfrac{1}{2} I + \DL'_k - \ri k V_k
    \quad\tand\quad
A_k:= \tfrac{1}{2} I + \DL_k - \ri k V_k
       \end{equation}
and for the Neumann problem 
\begin{equation}\label{e:NBIEs}
    \Breg := \ri \left(\tfrac{1}{2}I-\DL_k\right) +  \Reg H_k
    \quad\tand\quad
    \Breg' := \ri\left(\tfrac{1}{2}I-\DL_k'\right) +  H_k\Reg,
       \end{equation}
where the operators $V_k, \DL_k, \DL'_k$, and $H_k$ are the single-, double-, adjoint-double-layer and hypersingular operators defined by \eqref{e:SD'} and \eqref{e:DH} below. 
(For the history of these particular second-kind formulations; see Remark \ref{r:historyBIEs} below.)

The BIEs involving these operators can all be written in the form:~given $f\in L^2(\Gamma_-)$, find $v\in L^2(\Gamma_-)$ such that
\begin{equation}
\label{e:basicForm}
\operator v=f,\qquad \operator := c_0 (I+\pert),\qquad c_0\in\mathbb{C}\setminus \{0\},
\end{equation}
where the operator $\pert:L^2(\Gamma_-)\to L^2(\Gamma_-)$ is compact; see Theorem~\ref{thm:BIEs} below.
Let 
$$
\rhoB(k):=\|\operator^{-1}\|_{L^2(\Gamma_-)\to L^2(\Gamma_-)},
$$
and observe that the compactness of $\pert$ implies that $\rhoB(k) \geq c_0^{-1}$ (see Lemma \ref{lem:inversebound} below).

For the Dirichlet problem,
\beq\label{e:DirichletBIOsolution}
\rho_{\!_{A_k}}(k)=\rho_{\!_{A'_k}}(k)\leq Ck^{-1}\rho(k), 
\eeq
where $\rho(k)$ is defined by \eqref{e:rho} (see \cite{ChMo:08, BaSpWu:16} and \cite[Lemma 6.2]{ChSpGiSm:20}).  
For the Neumann problem, $\rho_{\!_{\Breg}}(k)=\rho_{\!_{\Breg'}}(k)$, but the relationship with the Helmholtz solution operator is more complicated; see \cite[Equation 7.8]{GaMaSp:21N}.

We study the Galerkin approximation to $v$ with $a:L^2(\Gamma_-)\times L^2(\Gamma_-)\to \mathbb{C}$ defined by
$$
a(v,w):= \langle \operator v,w\rangle_{L^2(\Gamma_-)}.
$$
\begin{informaltheorem}\mythmname{Preasymptotic error estimates for the $h$-BEM}
\label{t:informalR3a}
Suppose that $\Gamma_-$ is smooth and $\rhoB(k)$ is polynomially bounded in $k$. Suppose that \eqref{e:basicForm} is solved using the Galerkin BEM with degree $p$ polynomials (with approximation space $\cV_h$). If 
\begin{equation}
\label{e:approxRequirementsConcreteIntro}
(hk)^{2p+2}\rhoB(k) \text{ is sufficiently small,}
\end{equation}
then for all $v\in L^2(\Gamma_-)$, the Galerkin approximation, $v_h$, to $v$ in $\cV_h$ exists, is unique, and satisfies
\begin{equation}
\label{e:qoBEM}
\begin{aligned}
&\|v-v_h\|_{L^2(\Gamma_-)}\leq \Big(1+
C(hk)^{p+1}\rhoB(k)+Chk+
Ck^{-1}\Big)\min_{w_h\in \cV_h}\|v-w_h\|_{L^2(\Gamma_-)}.
\end{aligned}
\end{equation}
Furthermore, if the right-hand side $f$ has frequency $\lesssim k$ (e.g. in plane-wave scattering), then
\beq\label{e:PPMR2}
\frac{\N{v-v_h}_{L^2(\Gamma_-)}}{\N{v}_{L^2(\Gamma_-)}} \leq \Big(1+C(hk)^{p+1}\rhoB(k)+Chk+Ck^{-1}\Big)C (hk)^{p+1}.
\eeq
\end{informaltheorem}

Informal Theorem~\ref{t:informalR3a} and its rigorous version, Theorem~\ref{t:concretehBEM} below,
show that the Galerkin solution is quasioptimal, uniformly in $k$, if $(hk)^{p+1}\rhoB(k)$ is sufficiently small, and that the relative error can be made  controllably small if $(hk)^{2p+2}\rhoB(k)$ is sufficiently small (compare to the $h$-FEM result of Informal Theorem \ref{t:informalR1}). 
By \eqref{e:DirichletBIOsolution}, 
$\rho_{_{A_k}}(k)\leq C$ when 
the problem is nontrapping, and thus these results show that the $h$-BEM for the Dirichlet BIEs does not suffer from the pollution effect for nontrapping problems (with this first proved in \cite{GS2}). However, since $\rho_{_{A_k}}(k)\gg 1$ for trapping problems and $\rho_{_{\Breg}}(k)\gg 1$ even for nontrapping problems,
Informal Theorem~\ref{t:informalR3a}/Theorem~\ref{t:concretehBEM}
leave open the possibility that the $h$-BEM suffers from the pollution effect. 

To define pollution carefully, we recall that for a scattering problem, the solution $v$ to the BIE \eqref{e:basicForm} 
oscillates at frequency $\lesssim k$ with few additional properties.  The Weyl law (see e.g.~\cite[Chapter 14]{Zw:12}) or the Nyquist--Shannon--Whittaker sampling theorem 
\cite{Wh:15, Sh:49} (see, e.g., \cite[Theorem 5.21.1]{BaNaBe:00}) 
then implies that the dimension of this space is $\sim k^{d-1}$. Thus, it is certainly not possible to achieve accuracy for general scattering data with a space of dimension $\ll k^{d-1}$. Motivated by this, we say that a numerical BIE method \emph{suffers from the pollution effect} if $N\gg k^{d-1}$ is required to maintain accuracy of the computed solutions as $k\to \infty.$ We make this precise via the following definition.

\begin{definition}[The pollution effect in $L^2(\Gamma_-)$]\label{d:pollution}
Let $\mathcal{V}$ be a collection of subspaces of $L^2(\Gamma_-)$. For $V\in \mathcal{V}$ and $k>0$, let $\operator^{-1}_V(k) : L^2(\Gamma_-)\to V$ be an approximation of $\operator^{-1}(k)$. 
The pair $(\mathcal{V}, \{\operator^{-1}_V\}_{V\in \mathcal{V}})$ \emph{suffers from the pollution effect in $L^2(\Gamma_-)$} if 
\begin{equation}
\label{e:qo_pollution}
\inf_{\Lambda>0}\limsup_{k\to \infty}\sup_{\substack{V\in \mathcal{V}\\  \dim V \geq \Lambda k^{d-1}}} 
\sup_{f\in V} C_{\rm qo} (f,V)=\infty, 
\eeq
where
\beqs
C_{\rm qo}(f,V):=\inf \bigg\{ M>0 : \big\|\operator^{-1} f-\operator^{-1}_V f\big\|_{L^2(\Gamma_-)} 
\leq M \min_{w \in V} \N{ \operator^{-1} f-w}_{L^2(\Gamma_-)}\bigg\}.
\end{equation*}
\end{definition}

To unpack this definition, 
observe that if the right-hand side of \eqref{e:qo_pollution} is finite, then there exists $k_0, \Lambda,$ and $M<\infty$ such that for all $k\geq k_0$, $V\in \mathcal{V}$ with $\dim V\geq \Lambda k^{d-1},$ and $f\in \LtG$, 
\beq\label{e:Cqo}
 \N{\operator^{-1} f-\operator^{-1}_V f}_{\LtG} \leq M \min_{w \in V} \N{ \operator^{-1} f-w}_{\LtG};
\eeq
i.e., $k$-uniform quasi-optimality is achieved (for all possible data) with a choice of subspace dimension proportional to $k^{d-1}$.

\begin{informaltheorem}\mythmname{The $h$-BEM suffers from the pollution effect}
\label{t:informalR3b}
Let $\operator\in \{A_k,A_k'\}$. There is $\Gamma_-$ smooth such that if \eqref{e:basicForm} is solved using the Galerkin BEM with degree $p $ polynomials (with approximation space $\cV_h$) then 
given $\epsilon>0$ there are sequences of $h\to 0$ and $k\to\infty$ with $hk\leq \epsilon$ 
and a sequence of $v\in L^2(\Gamma_-)$ 
such that the Galerkin approximation, $v_h$, to $v$ in $\cV_h$, if it exists satisfies
\begin{equation}
\label{e:thePollution}
\|v-v_h\|_{L^2(\Gamma_-)}\geq \e^{-1} \min_{w_h\in \cV_h}\|v-w_h\|_{L^2(\Gamma_-)}.
\end{equation}
Furthermore, if $\Omega_-=B(0,1)$ and $\operator\in \{ \Breg,\Breg'\}$, then there is $c>0$ such that if \eqref{e:basicForm} is solved using the Galerkin BEM with degree $p$ polynomials (with approximation space $\cV_h$) there is $v\in L^2(\Gamma_-)$ such that the Galerkin approximation, $v_h$, to $v$ in $\cV_h$, if it exists satisfies
\begin{equation}
\label{e:thePollution2}
\begin{aligned}
&\|v-v_h\|_{L^2(\Gamma_-)}\geq  c\Big(1+\min\big\{(hk)^{p+1}\rhoB(k), (hk)^{-p-1}\big\} \Big)\min_{w_h\in \cV_h}\|v-w_h\|_{L^2(\Gamma_-)}.
\end{aligned}
\end{equation}
\end{informaltheorem}

The bound \eqref{e:thePollution} implies that the $h$-BEM with $\operator\in\{A_k,A_k'\}$ suffers from the pollution effect (in the sense of Definition \ref{d:pollution}). Moreover, since for $\operator\in\{\Breg,\Breg'\}$ and $\Omega_-=B(0,1)$, $\rhoB(k)\sim k^{1/3}$
\cite[\S14]{GaRaSp:25}, if $hk=\e$, then~\eqref{e:thePollution2} implies that the quasioptimality constant is \emph{at least} $c\e^{-p-1}$ and hence that decreasing $\e$ can actually make the quasioptimality constant worse until the point that $\e^{2p+2}k^{1/3}\lesssim1$. The bound \eqref{e:thePollution2}
also shows that Informal Theorem~\ref{t:informalR3a} is optimal over its full range of validity, \eqref{e:approxRequirementsConcreteIntro}.

In the rigorous versions of Informal Theorem~\ref{t:informalR3b} --  Theorems~\ref{t:neumannDisk}, \ref{t:diamond}, and Theorem~\ref{t:diamondNew} below -- we quantify~\eqref{e:thePollution} and show that the phenomenon of decreasing $hk$ resulting in \emph{increasing} the quasioptimality constant also occurs for certain $\Omega_-$ with $\operator\in\{A_k,A_k'\}$. Furthermore,  Figure~\ref{f:aReallyCoolPicture} gives a numerical illustration of pollution for the $h$-BEM, with further numerical results demonstrating pollution in $h$-BEM in \cite[\S2.3]{GaRaSp:25}.

\paragraph{The context of Informal Theorem~\ref{t:informalR3a}.}
Informal Theorem~\ref{t:informalR3a} was proved in \cite{GaRaSp:25}.
Of the previous investigations into proving quasioptimality of the Galerkin method with piecewise polynomials applied to $A_k'$ and $A_k$, 
namely \cite{BaSa:07, LoMe:11, GrLoMeSp:15, SpKaSm:15, GaMuSp:19, GS2}, the best results are the following. 
\bit\item The result of \cite[Lemma 3.1]{GS2} that if $\Gamma_-$ is smooth and
$\rhoB(k) (hk)$ is sufficiently small, then the Galerkin method is quasioptimal with constant proportional to $\rhoB(k)$.
\item
The  result of \cite[Theorem 3.17]{LoMe:11} that if $\Gamma$ is analytic then the Galerkin method applied to $A_k'$ is quasioptimal (with constant independent of $k$) if
$(hk)^{p+1} k^5 \rhoB(k)$
is sufficiently small \cite[Equation 3.22]{LoMe:11}; a similar result holds for $A_k$ with $k^5$ replaced by $k^6$ \cite[Equation 3.26]{LoMe:11}.
\eit
The quasioptimality constants in these two sets of results are larger than that in \eqref{e:qoBEM} and the thresholds for existence are more restrictive than that in \eqref{e:approxRequirementsConcreteIntro}.

As far as the authors are aware, there are no other results stated in the literature about quasioptimality of the Galerkin method applied to either $\Breg$ or $\Breg'$ defined by \eqref{e:NBIEs}.

\bre[Analysis of collocation and Nystr\"om methods]
We highlight that \cite{GaRaSp:25} contains 
results analogous to Informal Theorem \ref{t:informalR3a} for the Galerkin method in 2-d with trigonometric polynomials, the collocation method in 2-d with \emph{either} piecewise polynomials \emph{or} trigonometric polynomials, and for a 2-d Nystr\"om method -- i.e., a fully-discrete method -- based on trigonometric polynomials. For brevity, we do not state these results here, but note that for all the methods involving trigonometric polynomials, up to potential factors of $k^\e$ for any $\e>0$, 
there is no pollution, even for trapping obstacles (see \cite[Table 2.1]{GaRaSp:25} for an overview of these results).
\ere

\paragraph{The context of Informal Theorem~\ref{t:informalR3b}.}

As discussed in \S\ref{s:informalR1}, the pollution effect was famously identified in~\cite{IhBa:95,IhBa:97} for the FEM applied to constant-coefficient Helmholtz problems in $1$-d and has since been observed in many numerical experiments involving the FEM in higher dimensions.

In contrast, there is a common belief in both engineering~\cite{Ma:02,Ma:17,Ma:16a} and numerical analysis~\cite{BaSa:07, LoMe:11,GrLoMeSp:15} that boundary integral methods do not suffer from the pollution effect. Indeed, it is standard in engineering to ask whether a specific number of points per wavelength (e.g., six~\cite{Ma:02}) is sufficient to obtain accurate solutions. Informal Theorem~\ref{t:informalR3b} disproves this common belief, showing that no fixed number of points per wavelength is sufficient.  In fact, to the best of the authors' knowledge, Informal Theorem~\ref{t:informalR3b} (or rather its rigorous analogues Theorems~\ref{t:neumannDisk}, \ref{t:diamond}, and Theorem~\ref{t:diamondNew}) is the first rigorous proof of pollution 
for a numerical method applied to a
Helmholtz problem with non-empty scatterer.


\subsection{\ref{R4}:~Non-uniform $h$-FEM meshes for scattering problems with PML truncation, where the meshes are defined by ray dynamics}
\label{s:informalR4}

We consider the scattering problem in \eqref{e:informalPDE} with smooth coefficients, smooth obstacle, the radiation condition approximated by a radial PML, and under the assumption  that $\rho(k)$ is polynomially bounded in $k$.

Theorem \ref{t:informalR1} shows that the $h$-FEM solution 
\bit
\item is quasioptimal if $\rho(k) (hk)^p$ is sufficiently small (see \eqref{e:highGalerkinintro}), and 
\item has controllably-small $H^1_k$ relative error (for data with frequency $\lesssim k$) if $\rho(k) (hk)^{2p}$ is sufficiently small (see \eqref{e:relError}).
\eit 
We now show that the properties of quasioptimality 
and controllably-small relative error
can be achieved using certain non-uniform meshes that violate the conditions above on $h$ and involve coarser meshes away from trapping and in the perfectly matched layer (PML). 

\paragraph{Subsets of $\Omega$ defined by ray dynamics.}
We define billiard trajectories
(i.e., the natural analogue of geometric-optic rays in the presence of a boundary) 
to be geodesics for the metric $g^{-1}=n^{-1}A$ in $\Omega_+$ continued by reflection with respect to $g$ at the boundary 
of $\Omega_+$; when $A=I$ and $n=1$, these are straight-line paths continued using the Snell--Descartes law at the boundary. 
Next, we define the \emph{cavity} $\cavity\subset \overline{\Omega}_+$ as the set of points $x\in \overline{\Omega}_+$ such that there is a billiard trajectory passing over $x$ that remains in a compact set for all positive and negative times. We also define the \emph{visible set} $\visible\subset \overline{\Omega}_+$ as those points $x\in \overline{\Omega}_+$ such that there is a billiard trajectory passing over $x$ that remains in a compact set for all positive times or all negative times. Finally, we define the \emph{invisible set} $\invisible:=\overline{\Omega}_+\setminus (\visible\cup \cavity)$ (the adjectives visible and invisible are relative to the cavity). 
Let $\Gamma_{\tr}:=\partial\Omega\setminus\Gamma_-$ be the truncation boundary, and 
let $\Omega_{\pml}\subset \Omega$ be an open neighbourhood of $\Gamma_{\tr}$ that is strictly contained in the PML. Next, let $\Omega_\cavity$, $\Omega_\visible$ and $\Omega_\invisible$ be open neighbourhoods of the intersections with $\overline{\Omega}$ of, respectively, $\cavity$, $\visible\setminus (\cavity\cup\Omega_\pml)$, and $\invisible\setminus \Omega_\pml$ in the subspace topology of $\overline{\Omega}$ such that $\Omega_{\cavity}\cap \Gamma_{\tr}=\Omega_{\visible}\cap\partial \Omega_{\tr}=\Omega_{\invisible}\cap\Gamma_{\tr}=\emptyset$; for an illustration of these domains, see Figure \ref{f:domainsTrappingMeshing}.
\begin{figure}
\begin{center}
\begin{tikzpicture}
\begin{scope}[scale=.9]
\draw[fill=orange ](0,0)circle(3);
\node at(0,2.55){$\Omega_{\pml}$};
\fill[fill=light-blue] (0,0)circle(2.35);
\fill[fill=verde,domain=-155:-25] plot({2.35*cos(\x)},{2.35*sin(\x)})--cycle;
\fill[fill=verde,domain=155:25] plot({2.35*cos(\x)},{2.35*sin(\x)})--cycle;
\fill[fill=pink] (-1,-1)rectangle (1,1);
    \draw[rounded corners=0.2cm, fill=gray] (-1.6,-1) rectangle (-.5,1);

    \draw[rounded corners=0.2cm, fill=gray](1.6,-1) rectangle (.5,1);
    \node at(0,0){$\Omega_\cavity$};
\node at(0,-1.5){$\Omega_\visible$};
\node at(0,1.5){$\Omega_\visible$};
\node at(2,0){$\Omega_\invisible$};
\node at(-2,0){$\Omega_\invisible$};
    \end{scope}
\end{tikzpicture}
\end{center}
\caption{\label{f:domainsTrappingMeshing}The domains $\Omega_\cavity,\Omega_\visible,\Omega_\invisible,$ and $\Omega_\pml$, when $\Omega_-$ consists of two (rounded) aligned rectangles}
\end{figure}

\paragraph{The finite-element space.}
Given a triangulation, $\mathcal{T}$, of $\Omega$, we define 
\begin{equation}
\label{e:meshWidths}
h_\star:=\sup\big\{ h_{T}\,:\, T\in \mathcal{T},\, T\cap \Omega_\star\neq \emptyset\big\}, \, \text{ with } \,\star\in \{\cavity,\visible,\invisible,\pml\},
\end{equation}
where $h_T$ is the width of the element $T$. We also define $h:=\max_{T\in\mathcal{T}}h_T$.

\renewcommand{\arraystretch}{2}
\begin{table}[htbp]
\footnotesize
\hspace{-.5cm}
	$\rotatebox[origin=c]{90}{fewer DoFs \qquad\quad more DoFs}%
	\left\updownarrow
    \vcenter{\hbox{
	\begin{tabular}{|l|l|l|}
		\hlineb
		{\bf Mesh threshold} & {\bf Asymptotic DoFs} & {\bf Theoretical guarantee} \\
		\hlineb
		$(h_\cavity k)^p \rho(k) + (h_\visible k)^p \rho(k) + (h_\invisible k)^p \rho(k) =c$ & $\textup{vol}(\Omega) k^d (\rho(k))^{\frac{d}{p}}$ & $k$-QO \\
		\hlineb
		$(h_\cavity k)^p \rho(k) + (h_\visible k)^p \sqrt{k\rho(k)} + (h_\invisible k)^p k=c$ & $\textup{vol}(\Omega_\cavity) k^d (\rho(k))^{\frac{d}{p}}$ & $k$-QO \\
		\hlineb
		$(h_\cavity k)^p \sqrt{k\rho(k)} + (h_\visible k)^p k+ (h_\invisible k)^p k =c$ & $\textup{vol}(\Omega_\cavity) k^{d+\frac{1}{2p}}(\rho(k))^{\frac{d}{2p}}$ & $k$-QO away from trapping \\
		\hlineb
		$(h_\cavity k)^{2p} \rho(k) + (h_\visible k)^{2p} \rho(k) + (h_\invisible k)^{2p} \rho(k) =c$ & $\textup{vol}(\Omega) k^d (\rho(k))^{\frac{d}{2p}}$ & CRE \\
		\hlineb
		$(h_\cavity k)^{2p} \rho(k) + (h_\visible k)^{2p} \sqrt{k \rho(k)} + (h_\invisible k)^{2p} k = c$ & $\textup{vol}(\Omega_\cavity) k^d (\rho(k))^{\frac{d}{2p}}$ & CRE\\
		\hlineb
		$(h_\cavity k)^{2p} \rho(k) + (h_\visible k)^p k + (h_\invisible k)^p k = c$ & $\textup{vol}(\Omega_\cavity) k^d (\rho(k))^{\frac{d}{2p}}$ & CRE away from trapping \\
		\hlineb
	\end{tabular}}}
	\right.%
	\rotatebox[origin=c]{90}{}$
	\normalsize
    \renewcommand{\arraystretch}{1}
    	\caption{\label{tab:regimes}
         The table summarises of special cases of Theorem \ref{t:R4} with $\cavity\neq \emptyset$ (i.e., the problem is trapping). In all cases we require $h_Pk= c$ (with this not shown in the table for brevity) which does not contribute to the asymptotic number of degrees of freedom (DoFs). The notations $k$-QO, $k$-QO away from trapping, CRE, and CRE away from trapping are explained in the text.}
\end{table}

\paragraph{Results about the $h$-FEM applied with non-uniform meshes.}

Theorem \ref{t:R4} below details how varying the meshwidth in one of the regions $\Omega_\cavity,\Omega_\visible,\Omega_\invisible,$ and $\Omega_\pml$ 
affects errors both in that region and elsewhere. 
Table \ref{tab:regimes} gives some important special cases of Theorem \ref{t:R4}, showing, in particular, how the conditions $``\rho(k)(hk)^p$ sufficiently small" for quasioptimality and 
``$\rho(k) (hk)^{2p}$ is sufficiently small for controllably-small $H^1$ error can be violated away from $\Omega_\cavity$.

In 
   Table~\ref{tab:regimes}, \emph{$k$-QO} stands for $k$-uniform quasioptimality and means that there is $C>0$ independent of $k$ such that 
    $$
    \|u-u_h\|_{H_k^1(\Omega)}\leq C\inf_{w\in \mathcal{V}_h}\|u-w_h\|_{H_k^1(\Omega)},
    $$
    \emph{$k$-QO away from trapping} means that if $U$ is any open set not intersecting $\cavity$, then
    $$
    \|u-u_h\|_{H_k^1(U)}\leq C\inf_{w_h\in \mathcal{V}_h}\|u-w_h\|_{H_k^1(\Omega)},
    $$
    \emph{CRE} stands for controllably-small relative error and means that if $\|u\|_{H_k^{p+1}(\Omega)}\leq C\|u\|_{H_k^1(\Omega)}$ i.e., $u$ is oscillating at frequency $\lesssim k$, then by making $c$ in the left-most column small as a function $\e$, one can obtain
    $$
    \|u-u_h\|_{H_k^1(\Omega)}\leq \e\|u\|_{H_k^1(\Omega)},
    $$
    and, similarly, \emph{CRE away from trapping} means that if $u$ is oscillating at frequency $k$, then for any $U$ not intersecting $\cavity$ one can obtain
    $$
    \|u-u_h\|_{H_k^1(U)}\leq \e\|u\|_{H_k^1(\Omega)},
    $$
    by making $c$ small.

\paragraph{The context of the results summarised in Table \ref{tab:regimes}.}

Estimates on the local FEM error for second-order linear elliptic PDEs were pioneered in \cite{NiSc:74}; see also \cite{De:75}, \cite{ScWa:77}, \cite{ScWa:82}, \cite[Chapter 9]{Wa:91}, and \cite{DeGuSc:11}. 
The standard interpretation of these results 
is that the FEM solution is locally quasioptimal up to a lower-order term that allows
error to propagate in from the rest of the domain, sometimes
called the ``slush'' term in the literature.

The $k$-explicit analogues of these classic local-FEM-error results were then obtained in \cite{AGS1}, with the result being, essentially, that if $\Omega_0\subset \Omega_1\subset\Omega$, then 
	\begin{equation}
\|u - u_h\|_{H^1_k(\Omega_0)}
\leq C
		\big(\min_{w_h \in V_h} 
		\|u - w_h\|_{H^1_k(\Omega_1)} 
		+\|u-u_h\|_{(H^p_k{(\Omega_1})^*}\big).
        \label{e:QOmodulo}
	\end{equation}
The interpretation of \eqref{e:QOmodulo} is that the 
``slush" term (i.e., the second term on the right-hand side) consists of low frequencies, since a bound on a high $k$-weighted Sobolev norm of a function in
terms of a low $k$-weighted Sobolev norm, with the constant independent of $k$, implies
that the function is controlled by its frequencies $\lesssim k$; i.e., the Helmholtz FEM solution is locally quasioptimal, modulo low frequencies. However, the results of~\cite{AGS1} do not give any information about what impacts the low frequency term in~\eqref{e:QOmodulo}.

To the best of the authors’ knowledge, Theorem \ref{t:R4} (and its more sophisticated analogue
\cite[Theorem 3.11]{AGS2})
  are the first results
 specifying in a $k$-explicit way how the low-frequency errors propagate around the domain (see the discussion in \S\ref{s:interpret} below), and then using this information to design $k$-dependent non-uniform finite-element meshes
that maintain accuracy (either in the sense of $k$-QO or CRE) as $k\to\infty$. In particular, one can interpret the estimates in Theorem~\ref{t:R4} (and their more sophisticated analogues~\cite[Theorem 3.11]{AGS2}) as effectively estimating the low frequency term in~\eqref{e:QOmodulo}.

\subsection{\ref{R5} Analysis of overlapping Schwarz methods 
with PMLs at the subdomain boundaries}    
\label{s:informalR5}

Let $P$ be as in \eqref{e:informalPDE} with $A\equiv I$ and $\Omega_-=\emptyset$.
Given a hyperrectangle $\Omega$, let $P_{\rm s}$ be the operator $P$ with the Sommerfeld radiation condition approximated by a Cartesian PML near $\partial\Omega$.
We consider the problem:~given $f \in H^{-1}(\Omega)$, 
find $u\in H^1_0(\Omega)$ such that 
$P_{\rm s}u =f$.

We split $\Omega$ into overlapping hyperrectangles $\Omega_j$, and let $P^j_{\rm s}$ be the operator $P_{\rm s}$ on $\Omega_j$ with additional Cartesian PMLs contained in the overlap between $\Omega_j$ and its neighbours.
Let $\{\chi_j\}_{j=1}^N$ be 
a partition of unity 
with $\chi_j \in C^\infty_c(\Omega_j)$ and $\chi_j$ supported away from the PML regions of $P^j_{\rm s}$ that are not PML regions of $P_{\rm s}$. The parallel overlapping Schwarz method is then:~let $u_0=0$ and for $n\geq 0$ define  $u^{n+1} \in H^1_0(\Omega)$ by
\beq\label{e:SchwarzIntro}
u^{n+1} := u^n + \sum_{j=1}^N \chi_j 
(P_{\rm s}^j)^{-1}
(f- P_{\rm s}u^n )|_{\Omega_j};
\eeq
see, e.g., \cite[\S1.1]{DoJoNa:15}. 

\begin{informaltheorem}\label{t:informalR5}
Suppose that $P$ is nontrapping.
Let $\mathcal{N}$ be the maximum number of subdomains (counted with multiplicity) that a geometric optics ray can travel through before leaving $\Omega$. After $\mathcal{N}$ iterations, the error $u-u^\mathcal{N}$ is both smooth and $O(k^{-\infty})$ (i.e., superalgebraically small in $k$). 
\end{informaltheorem}

The rigorous statement of Informal Theorem \ref{t:informalR5} is Theorem \ref{t:R5} below. 
We make the following immediate remarks:
\bit
\item
The number $\mathcal{N}$ is defined more precisely in \S\ref{s:Ncurly} below,
but we highlight here that when $n\equiv 1$ and $\{\Omega_j\}_{j=1}^N$ is a decomposition of $\Omega$ into $N$ strips, then $\mathcal{N}=N$.
\item \cite{GGGLS2} (where Informal Theorem \ref{t:informalR5} was originally proved) also contains analogous results about sequential overlapping Schwarz methods, but, for brevity, we do not state these here.
\item All the results in \cite{GGGLS2}, while initially stated for Cartesian PML truncation on $\Omega$ and Cartesian PMLs on the subdomains, apply to truncation methods based on complex absorption and other types of PML; see the discussion in \cite[Appendix A]{GGGLS2}.
\eit

\paragraph{The context of Informal Theorem \ref{t:informalR5}.}
The design and analysis of DD methods for solving the Helmholtz equation 
is a very active area; see, e.g., the reviews \cite{Er:08, ErGa:12, LaTaVu:17, GaZh:19, GaZh:22}. 

A key question is:~what boundary conditions should be imposed on the DD subdomains? It is known that the optimal boundary condition on a particular DD subdomain 
is the Dirichlet-to-Neumann map for the Helmholtz equation posed in the exterior of that subdomain (as a subset of the whole domain)
 \cite[\S2]{NaRoSt:94} \cite[\S2]{EnZh:98}, \cite[\S2.4]{DoJoNa:15}. However, complete knowledge of these maps is 
equivalent to knowing the solution operator for the original problem.
 The design of good, practical subdomain boundary conditions is then the goal of \emph{optimised Schwarz methods}; see, e.g., \cite{Ga:06}.

Using PML as a subdomain boundary condition for Helmholtz DD was first advocated for in \cite{To:98}, and there are now many DD methods for Helmholtz using PML on subdomain boundaries with impressive empirical performance; see, e.g.,  
\cite{EnYi:11c, St:13, ChXi:13, ZeDe:16, TaZeHeDe:19, LeJu:19, RoGeBeMo:22, BoNaTo:23} and the reviews \cite{GaZh:19, GaZh:22}.

It is now understood in a $k$-explicit way how PML approximates the Dirichlet-to-Neumann map associated with the Sommerfeld radiation condition. Indeed, when $A\equiv I$, $n \equiv 1$, and $d=2$, the error between the true Helmholtz solution of~\eqref{e:informalPDE} and the Cartesian PML approximation decreases exponentially in $k$, the width of the PML, and the strength of the scaling by \cite[Lemma 3.4]{ChXi:13} (this result also holds for general scatterers in $d\geq 2$ with a radial PML by \cite{GLS2} -- see also Theorem~\ref{t:blackBoxErr}). For general smooth $n$ and any $d\geq 2$, this error is smaller than any negative power of $k$ by~\cite[Theorem 4.6]{GGGLS2}.

However, to the best of the authors' knowledge, Informal Theorem~\ref{t:informalR5} (or rather its rigorous statement~Theorem~\ref{t:R5}) is the first result to provide a rigorous understanding of how well PML 
approximates the (more complicated) Dirichlet-to-Neumann maps corresponding to the optimal DD boundary conditions for general decompositions and non-empty scatterers.

For example, in the particular case of a strip decomposition with $N$ subdomains, 
 the parallel overlapping Schwarz method with a strip decomposition with $N$ subdomains converges in $N$ iterations with the optimal boundary conditions \cite[Proposition 2.4]{NaRoSt:94} and \cite[Lemma 2.2]{EnZh:98}.
For such a strip decomposition with $n\equiv 1$, 
 Informal Theorem \ref{t:informalR5}/Theorem \ref{t:R5} shows that the parallel method  with PML at the subdomain boundaries has $O(k^{-\infty})$ error after $N$ iterations; i.e., PML approximates well the optimal boundary condition in this case. 
However, 
we highlight that $\mathcal N$ can be much larger than $N$ when $n \not\equiv 1$ and the rays are not straight lines (e.g., if the rays ``turn around" at one end of the domain), and PML does not approximate well the optimal boundary condition in this case. 

The only other existing $k$-explicit rigorous convergence results are for the sequential ``source transfer" DD methods (which \cite{GaZh:19} showed can be considered as a particular type of optimised Schwartz method) applied to~\eqref{e:informalPDE} with $A\equiv I$, $n\equiv 1$, and $\Omega_-=\emptyset$ (i.e., no scattering). 
These methods involve subdomains that only overlap via the PMLs (i.e., the physically-relevant parts of the subdomains do not overlap). For these methods applied with 
$A\equiv I$, $n\equiv 1$, and $\Omega_-=\emptyset$ (i.e., no scattering), $k$-explicit convergence of the DD method at the continuous level can be obtained from a $k$-explicit result about how well (Cartesian) PML approximates the Sommerfeld radiation condition -- this is precisely because the optimal boundary conditions on the subdomains in this case are the Dirichlet-to-Neumann maps associated with the Sommerfeld radiation condition. Therefore, accuracy of Cartesian PML  when 
$A\equiv I$, $n\equiv 1$, and $\Omega_-=\emptyset$ 
\cite[Lemma 3.4]{ChXi:13} is then the heart of the $k$-explicit convergence proofs of the source-transfer-type methods in \cite{ChXi:13, LeJu:19, DuWu:20, LeJu:21, LeJu:22} in the case of no scattering.

\section{Definition of the scattering problem and its approximation by a radial PML}\label{s:scat}

\subsection{The scattering problem.}\label{s:problem}

We consider scattering by a combination of an impenetrable obstacle and a penetrable obstacle (the latter corresponding to variable coefficients in the PDE). Let $\Omega_-\Subset \mathbb{R}^d$ be a bounded Lipschitz open set with connected open complement, $\Omega_+:=\mathbb{R}^d\setminus \overline{\Omega}_-$. Let 
$A$ be an $L^\infty$ symmetric, positive-definite matrix with real coefficients and $n\in L^\infty(\overline{\Omega_+};\mathbb{R}_+)$ such that $\supp (A-I)\cup \supp (n-1)\Subset \overline{\Omega}_+$. 
Recall that the \emph{conormal derivative} $\partial_{n,A}$ is such that 
$\partial_{n,A}u := n \cdot (A\nabla u)$ when $u\in H^2(\Omega)$, and that $\partial_{n,A} u$ can be defined for $u\in H^1(\Omega)$ with $\divergence (A\nabla u)\in L^2(\Omega)$ by Green's identity; see, e.g., \cite[Lemma 4.3]{Mc:00}.

Given $f\in L^2_{\rm comp}(\overline{\Omega_+})$, 
the scattering problem is to
find $u\in H^1_{\rm loc}(\overline{\Omega_+}):=\{ v : \chi v \in H^1(\Omega_+) \tfa \chi \in C^\infty_c(\Rea^d)\}$  
such that
\begin{equation}
\label{e:edp}\begin{gathered}
Pu:=-k^{-2}\operatorname{div}( A\nabla u)-n u=f\text{ in }\Omega_+,\quad (Bu)|_{\partial\Omega_+}=0,\\ 
\|(k^{-1}\partial_r-i)u\|_{L^\infty(|x|=r)}=o_{r\to \infty}(r^{\frac{1-d}{2}}),
\end{gathered}
\end{equation}
where $Bu=u$ in the sound-soft case and $Bu=\partial_{\nu,A}u$
in the sound-hard case. The condition as $r\to \infty$ in \eqref{e:edp} is the \emph{Sommerfeld radiation condition}; this condition ensures that the corresponding solutions of the wave equation created via \eqref{e:wave1} travel away from the obstacle as time increases.

\subsection{Existence and uniqueness of the solution of the scattering problem}

Since the boundary value problem \eqref{e:edp} is Fredholm with index zero (see, e.g., \cite[Theorem 4.4]{DyZw:19}) existence of a solution to~\eqref{e:edp} follows from uniqueness. 
Rellich's uniqueness theorem (see \cite{Re:43}, \cite[\S3.3]{CoKr:83}, \cite[Theorem 4.17]{DyZw:19}) shows that any two solutions to~\eqref{e:edp} must agree away from $\supp (A-I)\cup \supp( n-1)\cup \Omega_-$.
Uniqueness can then be shown via the unique continuation principle (UCP); i.e., that any compactly supported solution of~\eqref{e:edp} with $f\equiv 0$ must be 0. The standard tool used to prove such a result is a \emph{Carleman estimate}, going back to \cite{Ca:39} (where a UCP is proved for the case $d=2$, $A\equiv I$, and $n\in L^\infty$).

There is a large literature on proving UCPs under progressively weaker conditions on $A$ and $n$. Furthermore, given a UCP for $A$ and $n$ with certain regularity, \cite{BaCaTs:12},
proves uniqueness for the problem when $A$ and $n$ satisfy that regularity ``piecewise" 
(this argument is formulated in \cite{BaCaTs:12} for the time-harmonic Maxwell equations, but holds more generally, in particular for the Helmholtz equation).

The upshot is that in 2-d, the UCP holds when 
$A \in L^\infty$ and $n \in L^p$ with $p>1$ \cite{Al:12}. In 3-d, the UCP holds when $A$ is Lipschitz (and hence piecewise Lipschitz by \cite{BaCaTs:12}) and $n\in L^\infty$ \cite{Ho:59}; much more singular $n$ are allowed (although this is not important for this article) -- see 
\cite{JeKe:85, 
Wo:92}
and the surveys in \cite[Chapter 8]{Le:19}, \cite{Wo:93}.
Examples of $A \in C^{0,\alpha}$ for
all $\alpha<1$ for which the solution of \eqref{e:edp} is not unique at a particular
$k>0$ (and thus the UCP fails) are given in 
\cite{Pl:63,Mi:74,Fi:01}; see also \cite{Ch:25}.

\subsection{Recap of the $k$-dependence of the solution operator to \eqref{e:edp}}\label{s:rho}

Let $\mathcal{R}$ be the solution operator for the problem \eqref{e:edp} (i.e., the map $f\mapsto u$, aka, the resolvent).
Let $R_{\rm scat}$ be such that
\beq\label{e:Rscat}
\supp(A-I) \cup \supp(n-1) \cup \Omega_-\Subset B(0,\Rscat)
\eeq
let $\chi \in C_c^\infty(\mathbb{R}^d)$ with $\chi \equiv 1$ on $B(0,\Rscat)$, and let $\rho(k)$ be defined by \eqref{e:rho}.

By considering solutions of the form $e^{i k x_1} \psi(x)$ for $\psi \in C_c^\infty(\Rea^d)$
with $\supp \psi \subset B(0,\Rscat)
\setminus \supp(A-I) \cup \supp(n-1) \cup \Omega_-$, one can show that there exists $c>0$ such that $\rho(k)$ defined by \eqref{e:rho} satisfies
\beqs
\rho(k) \geq c k \quad\tfa k>0.
\eeqs

If the problem is nontrapping (i.e., all geometric-optic rays escape to infinity) then 
given $k_0>0$ there exists $C>0$ such that 
\beqs
\rho(k) \leq Ck\quad \tfa k\geq k_0
\eeqs
\cite{Mo:75, Va:75,MeSj:82,Bu:02, ChMo:08}, \cite[Theorem 4.43]{DyZw:19}, with the constant $C$ proportional to the length of the longest ray \cite{GaSpWu:20}.

When the problem is trapping, $\rho(k)$ grows faster than $k$ \cite{BoBuRa:10}, \cite[Theorem 7.1]{DyZw:19} but if the boundary and coefficients are sufficiently smooth then given $k_0>0$ there exist $C_1,C_2>0$ such that 
\beqs
\rho(k) \leq C_1 \exp(C_2 k)\quad \tfa k\geq k_0
\eeqs
\cite{Bu:98, Vo:00, Bel:03}. Furthermore, without any smoothness requirements, the resolvent can, in fact, grow arbitrarily fast \cite{ChSa:25}.

Nevertheless, for ``most" frequencies, the resolvent is polynomially bounded in $k$:~for all $\delta,k_0,\varepsilon>0$ there is $C>0$ and a set $\mathcal{J}\subset \mathbb{R}$ with Lebesgue measure $\leq  \delta$ 
such that 
\beq\label{e:polybound}
\rho(k) \leq C k^{5d/2+\varepsilon} \quad\tfa k\in [k_0,\infty)\setminus \mathcal{J}
\eeq
\cite{LSW1}.

\subsection{The PML problem.}\label{s:PML}

We approximate the Sommerfeld radiation condition using a radial perfectly matched layer (PML).
With $R_{\rm scat}$ defined by \eqref{e:Rscat}, 
let $\RPMLo>R_{\rm scat}$ and let $\Omega_{\tr}$ be such that $B_{\RPMLo }\Subset \Omega_{\tr}$.
Let $\Omega:=\Omega_{+}\cap \Omega_{\tr}$ and $\Gamma_{\tr}:=\partial\Omega_{\tr}$.

For $0\leq \theta<\pi/2$ (the \emph{scaling angle}), let the PML scaling function $\PMLs_\theta\in C^{\infty}([0,\infty);\mathbb{R})$ be defined by $\PMLs_\theta(r):=\PMLs(r)\tan\theta$ for some $\PMLs$ satisfying
\begin{equation}
\label{e:fProp}
\begin{gathered}
\big\{\PMLs(r)=0\big\}=\big\{\PMLs'(r)=0\big\}=\big\{r\leq \RPMLo\big\},\quad \PMLs'(r)\geq 0;
\end{gathered}
\end{equation}
i.e., the scaling ``turns on'' at $r=\RPMLo$. 
Given $\PMLs(r)$,  let 
\beqs
\alpha(r) := 1 + \ri \PMLs_\theta'(r) \quad \tand\quad \beta(r) := 1 + \ri \PMLs_\theta(r)/r.
\eeqs

We now define two possible PML problems; both are formed by first replacing $\Delta$ in \eqref{e:edp} by 
\begin{align*}
\Delta_\theta&= \Big(\frac{1}{1+\ri \PMLs_\theta'(r)}\pdiff{}{r}\Big)^2+\frac{d-1}{(r+\ri \PMLs_\theta(r))(1+\ri \PMLs_\theta'(r))}\pdiff{}{r} +\frac{1}{(r+\ri \PMLs_\theta(r))^2}\Delta_\omega,\\
&= \frac{1}{(1+\ri \PMLs_\theta'(r))(r+\ri \PMLs_\theta(r))^{d-1}}\frac{\partial}{\partial r}
\left( \frac{ (r+\ri \PMLs_\theta(r))^{d-1}}{1+\ri \PMLs_\theta'(r)}\frac{\partial}{\partial r}
\right)+\frac{1}{(r+\ri \PMLs_\theta(r))^2}\Delta_\omega
\end{align*}
(with $\Delta_\omega$ the surface Laplacian on $S^{d-1}$) and then \emph{either} multiplying by $ \alpha \beta^{d-1}$ \emph{or} not. 

In both cases, the PML problem has the form:~given $f\in L^2(\Omega)$, find $u\in H^1(\Omega)$ such that
\begin{equation}
\label{e:PML1}
P_\theta u:=-k^{-2}\operatorname{div}(A_\theta\nabla u)+k^{-2} b_\theta\cdot\nabla u-n_\theta u=f\text{ in }\Omega,\, (Bu)|_{\partial\Omega_+}=0,\, u|_{\partial\Omega_{\tr}}=0,
\end{equation} 
with the associated sesquilinear form
\begin{equation}
\label{e:def_ak}
a(u,v) := \int_{\Omega} \Big(k^{-2}A_\theta(x) \nabla u(x) \cdot \overline{\nabla v(x)} 
+k^{-2}\langle b_\theta(x),\nabla u\rangle \overline{v(x)}
- n_\theta(x) u(x) \overline{v(x)}\Big)\,dx.
\end{equation}
Finally, since $\alpha= \beta=1$ for $r\leq \RPMLo$, when the data $f$ in \eqref{e:edp} is supported in this region, the right-hand sides of the two different PML problems are the same.

\paragraph{Comparison of the two different formulations.}

Before we give the precise definitions of the two different formulations, we discuss their advantages and disadvantages. 

The multiplication by $\alpha \beta^{d-1}$ ensures that $b_\theta=0$ in \eqref{e:PML1}; i.e., the resulting operator is in divergence form. However, one requires additional assumptions for $P_k$ to satisfy the G\aa rding inequality:~there exists $\omega\in \Rea$ such that 
for all $k_0 > 0$ there are $c,C > 0$ such that for all $k \geq k_0$, 
\beq\label{e:Garding1}
\Re( \re^{\ri \omega} a_k(u,u)) \geq c\|u\|^2_{H^1_k(\Omega)} - C\|u\|^2_{L^2(\Omega)} \quad\tfa u \in H^1(\Omega).
\eeq
In particular,~\cite[Lemma 2.3]{GLSW1} shows that~\eqref{e:Garding1} holds for any $\PMLs_\theta(r)$ satisfying the above assumptions in $d=2$ and holds in $d=3$ when $\PMLs_\theta(r)/r$ is, in addition, non-decreasing and~\cite[Remark 2.1]{GLSW1} shows that such an additional assumption is needed.

If one instead integrates by parts the complex-scaled PDE directly (i.e., 
one does not multiply by $\alpha \beta^{d-1}$), 
then the resulting sesquilinear form satisfies the G\aa rding inequality after multiplication by $\re^{\ri \omega}$ for some suitable constant $\omega$~\cite[Lemma A.6]{GLS1}. 

We highlight that, in other papers on PMLs, the scaled variable, which in our case is $r+\ri \PMLs_\theta(r)$, is often written as $r(1+ \ri \widetilde{\sigma}(r))$ with $\widetilde{\sigma}(r)= \sigma_0$ for $r$ sufficiently large; see, e.g., \cite[\S4]{HoScZs:03}, \cite[\S2]{BrPa:07}. Therefore, to convert from our notation, set $\widetilde{\sigma}(r)= \PMLs_\theta(r)/r$ and $\sigma_0= \tan\theta$.
In this alternative notation, the assumption that $\PMLs_\theta(r)/r$ is nondecreasing is therefore that $\widetilde{\sigma}$ is nondecreasing -- see \cite[\S2]{BrPa:07}.

\paragraph{The sesquilinear form after multiplication by $\alpha \beta^{d-1}$.}
Let $b_\theta:=0$,
\beq\label{e:firstPML}
A_\theta := 
\begin{cases}
A
\hspace{-1ex}
& \tin \Omega,\\
HDH^T 
\hspace{-1ex}
&\tin (B_{\RPMLo})^c
\end{cases}
\quad\tand\quad
n_\theta := 
\begin{cases}
 n &\tin \Omega \cap B_{\RPMLo},\\
\alpha(r) \beta(r)^{d-1} 
\hspace{-1ex}
&\tin (B_{\RPMLo})^c,
\end{cases}
\eeq
where, in polar coordinates $(r,\varphi)$,
\beqs
D :=
\left(
\begin{array}{cc}
\beta(r)\alpha(r)^{-1} &0 \\
0 & \alpha(r) \beta(r)^{-1}
\end{array}
\right) 
\quad\tand\quad
H :=
\left(
\begin{array}{cc}
\cos \varphi & - \sin\varphi \\
\sin \varphi & \cos\varphi
\end{array}
\right) 
\tfor d=2,
\eeqs
and, in spherical polar coordinates $(r,\varphi, \phi)$,
\beqs
D :=
\left(
\begin{array}{ccc}
\beta(r)^2\alpha(r)^{-1} &0 &0\\
0 & \alpha(r) &0 \\
0 & 0 &\alpha(r)
\end{array}
\right) 
\tand\,
H :=
\left(
\begin{array}{ccc}
\sin \varphi \cos\phi & \cos \varphi \cos\phi & - \sin \phi \\
\sin \varphi \sin\phi & \cos \varphi \sin\phi & \cos \phi \\
\cos \varphi & - \sin \varphi & 0 
\end{array}
\right) 
\eeqs
for $d=3$.
(observe that then $A=I$ and $n=1$ when $r=\RPMLo$ and thus $A_\theta$ and $n_\theta$ are continuous at $r=\RPMLo$).

\ble[{\cite[Lemma 2.3]{GLSW1}}]\label{l:PMLGard1}
Let $\PMLs_\theta$ satisfy \eqref{e:fProp} 
and the additional assumption when $d=3$ that $\PMLs_\theta(r)/r$ is nondecreasing.
Given $\epsilon>0$
there exists $c>0$ such that, 
for all $\epsilon \leq \theta\leq \pi/2-\epsilon$, $A_\theta$ defined by \eqref{e:firstPML} satisfies
\beqs
\Re \big( A_\theta(x) \xi, \xi\big)_2 \geq c \|\xi\|_2^2 \quad\tfa \,\xi \in \mathbb{C}^d\, \tand \,x \in \Omega;
\eeqs
thus the G\aa rding inequality \eqref{e:Garding1} holds with $\omega=0$.
\ele

\paragraph{The sesquilinear form without multiplication by $\alpha \beta^{d-1}$.}

Let
\beq\label{e:secondPML}
A_\theta := 
\begin{cases}
A
\hspace{-1ex}
& \tin \Omega,\\
HDH^T 
\hspace{-1ex}
&\tin (B_{\RPMLo})^c
\end{cases}
\quad\tand\quad
n_\theta  := 
\begin{cases}
 n &\tin \Omega \cap B_{\RPMLo},\\
1
\hspace{-1ex}
&\tin (B_{\RPMLo})^c,
\end{cases}
\eeq
where, in polar coordinates $(r,\varphi)$,
\beqs
D :=
\left(
\begin{array}{cc}
\alpha(r)^{-2} &0 \\
0 & \beta(r)^{-2}
\end{array}
\right) 
\quad\tand\quad
H :=
\left(
\begin{array}{cc}
\cos \varphi & - \sin\varphi \\
\sin \varphi & \cos\varphi
\end{array}
\right) 
\tfor d=2,
\eeqs
and, in spherical polar coordinates $(r,\varphi, \phi)$,
\beqs
D :=
\left(
\begin{array}{ccc}
\alpha(r)^{-2} &0 &0\\
0 & \beta(r)^{-2} &0 \\
0 & 0 &\beta(r)^{-2}
\end{array}
\right) 
\,\tand
\,
H :=
\left(
\begin{array}{ccc}
\sin \varphi \cos\phi & \cos \varphi \cos\phi & - \sin \phi \\
\sin \varphi \sin\phi & \cos \varphi \sin\phi & \cos \phi \\
\cos \varphi & - \sin \varphi & 0 
\end{array}
\right) 
\eeqs
for $d=3$. 
(Since $A_\theta=I$ and $n_\theta=1$ when $r=\RPMLo$, $A$ and $n$ are continuous at $r=\RPMLo$.)
In addition, for $d=2$, 
$$
b_\theta(r):=\begin{cases}0&\tin \Omega\cap B_{\RPMLo},\\
H \begin{pmatrix} 
\alpha^{-2}\big( \log (\alpha\beta)\big)'
\\0\end{pmatrix}&\tin (B_{{\RPMLo}})^c,\end{cases}
$$
and for $d=3$
$$
b_\theta(r):=\begin{cases}0&\tin \Omega\cap B_{{\RPMLo}},\\
H\begin{pmatrix}
\alpha^{-2}\big( \log (\alpha\beta^2)\big)'\\0\\0\end{pmatrix}&\tin(B_{{\RPMLo}})^c.\end{cases}
$$

\ble[{\cite[Lemma A.6]{GLS2}}]\label{l:PMLGard2}
Let $\PMLs_\theta$ satisfy \eqref{e:fProp}. 
Given $\epsilon>0$
there exists $\omega\in \Rea, c>0$ such that, 
for all $\epsilon \leq \theta\leq \pi/2-\epsilon$, $A_\theta$ defined by \eqref{e:secondPML} satisfies
\beqs
\Re \big( \re^{\ri \omega} A_\theta(x) \xi, \xi\big)_2 \geq c\|\xi\|_2^2 \quad\tfa \xi \in \mathbb{C}^d \,\tand\, x \in \Omega;
\eeqs
thus the G\aa rding inequality \eqref{e:Garding1} holds with $\omega=0$.
\ele

\subsection{The PML problem is exponentially accurate at high-frequency and inherits the bound on the solution operator of the scattering problem}\label{s:PMLaccuracy}

We define the exponential rate of growth for the solution operator away from $\mathcal{J}\subset \mathbb{R}$:
\begin{equation}
\label{e:Lambda1}
\Lambda(P,\Rea\setminus\mathcal{J}):=
\limsup_{\substack{k\to \infty\\k\in \mathbb{R}\setminus \mathcal{J}}}\frac{1}{k}\log \rho(k)=\limsup_{\substack{k\to \infty\\k\in \mathbb{R}\setminus \mathcal{J}}}\frac{1}{k}\log \|\chi \mathcal{R}(k)\chi\|_{L^2\to L^2}.
\end{equation}
We write $\Lambda(P)$ for $\Lambda(P,\mathbb{R})$. 
The bounds on $\rho(k)$ recapped in \S\ref{s:rho} imply that 
\bit
\item if the boundary and coefficients are sufficiently smooth, $\Lambda(P)<\infty$,
\item if the problem is nontrapping then $\Lambda(P)=0$,
\item for all $\delta>0$ there is a set $\mathcal{J}\subset \mathbb{R}$ with $|\mathcal{J}|\leq \delta$ such that $\Lambda(P,\Rea\setminus\mathcal{J})=0$.
\eit

The following are simplified versions of \cite[Theorems 1.5 and 1.6]{GLS2}.
Both results involve a quantity $\theta_0(P,\mathcal{J},R_{\tr})$ defined in \cite[Equation 1.10]{GLS2}. For our purposes, we only need that $\theta_0(P,\mathcal{J},R_{\tr})<\pi/2$ with $\theta_0(P,\mathcal{J},R_{\tr})=0$ if $\Lambda(P,\mathbb{R}\setminus\mathcal{J})=0$. 

\begin{theorem}\mythmname{The PML error decreases exponentially in $k$, the PML width, and the scaling angle}
\label{t:blackBoxErr}
Let $\mathcal{J}\subset \mathbb{R}$ and suppose that $P$ is such that $\Lambda(P,\Rea\setminus\mathcal{J})<\infty$. Let $\chi \in C_c^\infty(B(0,\RPMLo))$ with $\chi \equiv 1$ in a neighbourhood of $B(0,\Rscat)$, and $\e>0$. Then there are $C,c,k_0>0$ such that for all $R_{\tr}>\RPMLo+\e$, $\Omega_{\tr}$ with $B(0,R_{\tr})\subset \Omega_{\tr}\subset \mathbb{R}^d$ and Lipschitz boundary, $\theta_0(P,\mathcal{J},R_{\tr})+\e<\theta<\pi/2-\e$, $f\in L^2$, $k>k_0$ and $k\notin \mathcal{J}$, 
the solution $v$ of $P_\theta v=\chi f$ exists, is unique, and satisfies
\begin{align*}\nonumber
&\|\chi( u-v)\|_{H^1_k}+\|(1-\chi)(u-v)\|_{H_k^2(B(0,\RPMLo))}
\\
&\qquad\leq C
\exp\bigg(-k
c(R_{\tr}- R_1 + \e)\tan \theta -2\Lambda(P,\Rea\setminus\mathcal{J})\Big)\bigg)
\|\chi f\|_{L^2},
\label{e:blackBoxErr}
\end{align*}
where $u$ is the solution of $Pu=\chi f$.
\end{theorem}

\begin{theorem}\mythmname{The PML problem inherits the solution-operator bound of the scattering problem}
\label{t:blackBoxResolve}
Let $J\subset \mathbb{R}$ and suppose that $P$ is such that $\Lambda(P,J)<\infty$.
Let $\chi \in C_c^\infty(B(0,\RPMLo))$ with $\chi \equiv 1$ in a neighbourhood of $B(0,\Rscat)$, and $\e>0$. Then there are $C,k_0>0$ such that 
the following holds. For all $R_{\tr}>R_1+\e$, 
$\Omega_{\tr}$ with $B(0,R_{\tr})\subset \Omega_{\tr}\subset \mathbb{R}^d$ and Lipschitz boundary, $\theta_0(P,J,R_{\tr})+\e<\theta<\pi/2-\e$, all $f\in L^2$ with $\supp f\subset \Omega_{\tr}$, all $k>k_0$ and $k\notin \cJ$, the solution $v$ to $P_\theta v =f$ exists, is unique, and satisfies
\beq\label{e:blackBoxResolve}
\|v\|_{H^1_k(\Omega_{\tr})}\leq C\|\chi \mathcal{R}(k)\chi\|_{L^2\to L^2}\|f\|_{L^2}.
\eeq
\end{theorem}

\bre[Results on the accuracy of PML truncation at fixed $k$]\label{r:PMLfixedk}
Theorem \ref{t:blackBoxErr} shows that, for a PML with any positive width, the PML solution exists and is unique for sufficiently large $k$ and sufficiently large $\theta$ (with $\theta$ only needing to be $>0$ for nontrappping problems). 

For a fixed $k$, 
the PML solution exists and is unique when the width is sufficiently large by \cite[Theorem 2.1]{LaSo:98}, \cite[Theorem A]{LaSo:01}, \cite[Theorem 5.8]{HoScZs:03}, and the $H^1_k$ errors decreases exponentially in the width 
\cite[Theorem 2.1]{LaSo:98}, \cite[Theorem A]{LaSo:01}, \cite[Theorem 5.8]{HoScZs:03}.
\ere

\subsection{Definition of the Helmholtz PML problem considered in this article}\label{s:defPML}

We collect here the definitions of the Helmholtz PML problem considered elsewhere in the article.

We write $u\in \overline{C^{\ell,1}}(\Omega)$ if for all connected components $\Omega_j$ of $\Omega$ there exists $U_j \in C^{\ell,1}(\Rea^d)$ such that $U_j|_{\Omega_j}= u|_{\Omega_j}$ (i.e., $u$ is $C^{\ell,1}$ ``up to the boundary"). 


\begin{definition}[$H^\ell$ Helmholtz problem]
\label{d:helmholtzProblem}
The problem~\eqref{e:edp} is a \emph{$H^{\ell}$ Helmholtz problem} if $\Omega_-$ is $C^{\ell-1,1}$, and there is $\Gamma_{\rm p}\Subset \Omega_+$ a $C^{\ell-1,1}$ closed, embedded hypersurface such that $A,n\in \overline{C^{\ell-2,1}}(\Omega_+\setminus \Gamma_{\rm p})$. 
The problem \eqref{e:edp} is an $H^\infty$ Helmholtz problem if it is an 
$H^\ell$ Helmholtz problem for all $\ell$.
\end{definition}

The whole point of this notation is that if \eqref{e:edp} is an \emph{$H^{\ell}$ Helmholtz problem} then, by elliptic regularity, $u \in 
H^{\ell}(\Omega_+\setminus \Gamma_{\rm p})$; see, e.g., \cite[Theorems 4.18 and 4.20]{Mc:00}.

\begin{definition}[$H^\ell$ Helmholtz problem truncated using a PML]
\label{d:pmlTruncation}
An $H^{\ell}$ Helmholtz problem is truncated using a $C^{n,1}$ PML with a $C^{r,1}$ truncation boundary if $\PMLs_{\theta} \in C^{n,1}$ and $\Gamma_{\tr}\in C^{r,1}$ 
and, for the formulation corresponding to multiplying by $\alpha\beta^{d-1}$, $\PMLs_\theta(r)/r$ is nondecreasing.
\end{definition}

\section{A primer on semiclassical analysis}
\label{s:SCA}

\subsection{Symbols, quantisation, and semiclassical pseudodifferential operators on $\Rea^d$}\label{s:symbol}

It is standard in the analysis of partial differential equations (PDEs) to use the position variables, $x\in\mathbb{R}^d$, to decompose the solution 
of a partial differential equation 
into component parts localized at different positions (via, e.g., a partition of unity). Similarly, Fourier analysis allows one to do the same in the momentum/frequency variables, $\xi\in\mathbb{R}^d$. Microlocal analysis -- or, more specifically, since we work in the high-frequency regime, semiclassical analysis -- allows one to do both of these things simultaneously, up to the limit imposed by the Heisenberg uncertainty principle\footnote{The mathematical formulation of this principle is actually not important in the rest of this article, but can be found in, e.g., \cite[Theorem 3.9]{Zw:12}.} 

As described in \S\ref{s:how}, the combination of physical and Fourier variables is known as \emph{phase space}. In this section, we consider $x\in \Rea^d$, and treat phase space as $\Rea^{2d}$ (strictly speaking $(x,\xi)\in T^*\Rea^d:= \Rea^d \times (\Rea^d)^*$, i.e., $\Rea^d$  times its dual, 
but this distinction is not important in what follows).
We work with the small parameter $\hbar =k^{-1}$; typically in SCA the small parameter is denoted by $h$, but we instead use $\hbar$ to distinguish it from the FEM/BEM meshwidth.

The operators that allow one to work microlocally are ``quantisations'' of symbols, where a symbol is a function of $(x,\xi)$; this terminology comes from quantum mechanics, where the symbols are classical observables (i.e., functions of position and momentum) and their quantisations are quantum measurements (i.e., operators acting on wavefunctions). 

We work with quantisations of symbols from the following classes: given $m\in \Rea$,
$$
S^m(\mathbb{R}^{2d}):=\Big\{ a\in C^\infty(\mathbb{R}^{2d})\,:\,  
\forall\alpha,\beta \in \mathbb{N}^d
\sup_{\substack{(x,\xi)\in\mathbb{R}^{2d}\\ 0<\hbar<1
}
}\langle \xi\rangle^{-m+|\beta|}|\partial_x^\alpha \partial_{\xi}^\beta a(x,\xi)|<\infty\Big\},
$$
where $\langle\xi\rangle:= (1+|\xi|^2)^{1/2}$; i.e., for a symbol in $S^m$, 
taking a derivative in $\xi$ 
results in 
one power of $\langle\xi\rangle$ faster decay in $\xi$, 
but taking a derivative in $x$ does not change the decay in $\xi$. (Note that we allow the symbols $a(x,\xi)$ to depend on $\hbar$, but do not write this dependence explicitly in the notation).

\begin{example}[Examples of symbols]\label{ex:symbol} \ \\

\vspace{-1em}

(i) $a(x,\xi) := \sum_{|\alpha|\leq m}a_\alpha(x) \xi^\alpha$, where 
$m\in\mathbb{N}$, $a_\alpha \in C^\infty$, and $\partial^\gamma a_\alpha \in L^\infty$ for all $\gamma$ and $\alpha$, is in $S^m$.

(ii) $\langle \xi\rangle^{-m}:= (1+ |\xi|^2)^{-m/2}\in S^{-m}$.

(iii) If $\chi \in C^{\infty}_{c}(\Rea^{2d})$, then $\chi \in S^{-N}$ for every $N\geq 1$.
\end{example}

Given a symbol $a(x,\xi)$, the most-basic quantisation map
defines an operator for each $\hbar>0$ by
\begin{align}\label{e:quant}
(\widetilde{\Op}(a)u)(x)
&:=(2\pi\hbar)^{-d}\int_{\Rea^d}\int_{\Rea^d} e^{\frac{i}{\hbar}\langle x-y,\xi\rangle}a(x,\xi)u(y)dyd\xi
\\
&=(2\pi\hbar)^{-d}\int_{\Rea^d} e^{\frac{i}{\hbar}\langle x,\xi\rangle}a(x,\xi)(\mc{F}_\hbar u)(\xi)d\xi
\nonumber
\end{align}
for $x\in \Rea^d$ and $u\in \mathcal{S}(\Rea^d)$, with $\widetilde{\Op}(a)$ acting on $\mathcal{S}'(\Rea^d)$ then defined by duality.

\begin{example}[Examples of quantisations]\label{ex:quant} \ \\

\vspace{-1em}

(i) If $a(x,\xi)=1$, then $(\widetilde{\Op}(a)v)(x)= v(x)$, i.e., $\widetilde{\Op}(1)= I$.

(ii) If $a(x,\xi)=a(x)$, then $(\widetilde{\Op}(a)v)(x) = a(x)v(x)$.

(iii) If $a(x,\xi)=a(\xi)$, then $(\widetilde{\Op}(a)v)(x)= \mathcal{F}^{-1}_\hbar \big(a(\cdot)(\mathcal{F}_\hbar v)(\cdot)\big)(x)$, i.e., $\widetilde{\Op}(a)$ is a Fourier multiplier.

(iv) If $a(x,\xi) := \sum_{|\alpha|\leq m}a_\alpha(x) \xi^\alpha$, where $m\in\mathbb{N}$, $a_\alpha \in C^\infty$, then $(\widetilde{\Op}(a)v)(x)= \sum_{|\alpha|\leq m}a_\alpha(x) (\hbar D)^\alpha v$, where $D:= -i \partial$.
\end{example}

As special cases of (iv), we see that $-\hbar^2\Delta -1= \widetilde{\Op}( |\xi|^2-1)$
and
\beq\label{e:quantex}
P:= -\hbar^2\sum_{i,j=1}^d\partial_{x^i}A^{ij}(x)\partial_{x^j}-n(x)
=\widetilde{\Op}\Big( \langle A \xi,\xi\rangle -i \hbar\xi_j \partial_{x^i}A^{ij}(x) -n(x)\Big).
\eeq

Unfortunately, $\widetilde{\Op}(a)$ does not have the property that it is \emph{properly supported}; recall that an operator $B$ is properly supported if and only if for any
$\chi \in C_c^\infty(\Rea^d)$ there exist $\chi_1, \chi_2 \in C_c^\infty(\Rea^d)$ such that 
$$
\chi B = \chi B \chi_1\quad\tand \quad B \chi = \chi_2 B \chi.
$$
Having the quantisation be properly supported is helpful when dealing with functions that are only locally in $L^2$ (such as Helmholtz solutions).
We therefore let $\psi\in C_c^\infty(\mathbb{R})$ with $0\notin \supp (1-\psi)$ and use the quantisation map
\beq\label{e:quant2}
({\Op}(a)u)(x):=(2\pi\hbar)^{-d}\int_{\Rea^d}\int_{\Rea^d} e^{\frac{i}{\hbar}\langle x-y,\xi\rangle}a(x,\xi)\psi(|x-y|)u(y)dyd\xi.
\eeq
With this quantisation, Parts (i), (ii), and (iv) of Example \ref{ex:quant} still hold. Moreover, 
by repeatedly integrating by parts the integral defining $\widetilde{\Op}(a)-\Op(a)$, we find that 
for all $a\in S^m(\mathbb{R}^{2d})$ and all $N>0$ there is $C>0$ such that for $0<\hbar<1$
\beq\label{e:residual}
\|\widetilde{\Op}(a)-\Op(a)\|_{H_k^{-N}(\Rea^d)\to H_k^N(\Rea^d)}\leq C\hbar^N,
\eeq
where the spaces $H_k^s(\Rea^d)$ are defined in \S\ref{s:Sobolev} below.

We say that an ($\hbar$-dependent) operator $B: C_c^\infty(\mathbb{R}^d)\to \mathcal{D}'(\mathbb{R}^d)$ is $ O(\hbar^\infty)_{\Psi_h^{-\infty}}$ if for all $N>0$ there is $C>0$ such that 
\beq\label{e:residualclass}
\|B\|_{H_k^{-N}(\Rea^d)\to H_k^N(\Rea^d)}\leq C\hbar^N.
\eeq
In particular, $\widetilde{\Op}(a) = \Op(a) +O(\hbar^\infty)_{\Psi_h^{-\infty}}$ by \eqref{e:residual}.

We then define the class of semiclassical pseudodifferential operator of order $m$, $\Psi^m_\hbar(\mathbb{R}^d)$ by
\beq\label{e:pseudos}
\Psi^m_\hbar(\mathbb{R}^d):=\big\{ \Op(a)+O(\hbar^\infty)_{\Psi_\hbar^{-\infty}}\,:\, a\in S^m(\Rea^{2d})\big\}
\eeq
(without the inclusion of the residual term in \eqref{e:pseudos}, this class of operators cannot contain, e.g., $(-\hbar^2\Delta+1)^{-1}$, since this operator is not properly supported).


\bre[Why do we work in $S^m$?]\label{r:Sm}  
$S^m$ -- 
known as the Kohn--Nirenberg symbol class \cite{KoNi:65} -- is a natural class of symbols that includes both polynomials and, when they are invertible, the inverse of polynomials. This results in the fact that the quantisation of $S^m$ includes (up to residual terms) both the (semiclassical) differential operators and, when the differential operator is elliptic in an appropriate sense, the inverse of such a differential operator. 

Much wider classes of (smooth) symbols exist; see, e.g., \cite[Chapter 7, \S1]{Ta:96}, \cite[\S4.4.1]{Zw:12}, \cite[Equation E.1.48]{DyZw:19}.
\ere

\bre[Other symbol classes]
There are many calculi of (semiclassical) pseudodifferential operators that allow for the treatment of various geometric singularities; e.g., boundaries, corners, cone points, etc.
In this article, we discuss only the standard semiclassical calculus, although we use some results from the calculus associated to boundaries in \S\ref{s:R4} (when stating the analogues of Theorems \ref{l:measureAwayTrapping} and \ref{l:everythingawaytrapping} below for the PML problem -- see Table \ref{ta:resolve} below).

In general, non-smooth symbol classes are more technically complicated and have worse remainder terms than smooth symbol classes; see e.g.~\cite[\S 13.9]{Ta:23}
\ere

\subsection{Semiclassical Sobolev spaces}\label{s:Sobolev}

For $s\geq 0$, let 
$$
|u|_{H_k^{s}(\mathbb{R}^d)}^2:=\frac{1}{(2\pi \hbar)^d}\int_{\Rea^d}|\xi|^{2s}|\mathcal{F}_k(u)(\xi)|^2 d\xi
$$
(where $\mathcal{F}_k$ -- the semiclassical Fourier transform -- is defined by \eqref{e:Fourier} above) 
and 
\beq\label{e:norm2}
\|u\|_{H_k^s(\Rea^d)}^2:= \sum_{\ell=0}^s \frac{1}{(\ell!)^2}|u|_{H_k^\ell(\Rea^d)}^2
\eeq
(compare to \eqref{e:norm}). Let $H^{-s}_k(\Rea^d):= (H^s_k(\Rea^d))^*$, with corresponding norm defined by duality.

Given $s_0>0$, for all $|s|\leq s_0$ these norms  are then equivalent
 -- up to $s_0$- and $d$-dependent constants -- to the more-standard Sobolev norms on $\Rea^d$ defined by Fourier transform:
\beq\label{eq:Hhnorm}
\vertiii{u}_{H_k^s(\Rea^d)} ^2 := (2\pi \hbar)^{-d} \int_{\Rea^d} \langle \xi \rangle^{2s}
 |\mathcal F_\hbar u(\xi)|^2 \, d \xi.
\eeq
Throughout this article we work with the definition \eqref{e:norm}/\eqref{e:norm2}, but the precise values of the constants in the norm is only important for the results on $h$- and $p$-explicit polynomial approximation in \S\ref{a:poly} (and thus, in particular, not important in this section).

\subsection{Fundamental properties of semiclassical pseudodifferential operators}


\begin{informaltheorem}[Properties of semiclassical pseudodifferential calculus]\label{t:rules}

\ 
\vspace{-1em}

\begin{enumerate}
\item (Quantisation.) The  \emph{quantisation} map $\Op:S^m(\mathbb{R}^{2d})\to \Psi^{m}_\hbar(\Rea^d)$ 
defined by \eqref{e:quant2} is 
such that $\Op(a)$ 
is properly supported, $\Op(1)=I$, and $[\Op(a)u](x)=a(x)u(x)$ when $a$ is independent of $\xi$.
\item (Principal symbol.)
There is 
a \emph{principal symbol} map $\sigma_\hbar:\Psi^{m}_\hbar(\Rea^d)\to S^m(\mathbb{R}^d)$ such that 
$$\sigma_\hbar\circ \Op = I +\lot $$ 
(where ``\lot " stands for ``lower-order terms" --- both smaller in $\hbar$, and more smoothing).
\item (Boundedness in Sobolev norms.)
Given $a\in S^m(\mathbb{R}^{2d})$,
for all $s\in \Rea$ there is $C_s>0$ such that $$\|\Op(a)\|_{H_k^s(\Rea^d)\to H_k^{s-m}(\Rea^d)}\leq C_s$$
(i.e., the quantisation of a symbol of order $m$ decreases Sobolev regularity by $m$). 
Furthermore, for $a\in S^0(\mathbb{R}^{2d})$, $$\|\Op(a)\|_{L^2(\Rea^d)\to L^2(\Rea^d)}\leq \sup_{(x,\xi)\in \Rea^{2d}}|a(x,\xi)|+\lot $$
\item (Composition.) $$\Op(a)\Op(b)=\Op(ab)+\lot$$ 
Furthermore, $\Op(a)\Op(b)=\Op(c)+O(\hbar^\infty)_{\Psi_{\hbar}^{-\infty}}$
with $\supp c\subset \supp a\cap \supp b$.
\item (Adjoints.) $(\Op(a))^* =\Op(\overline{a}) + \lot $
\item (Commutators.) 
\begin{align*}
i\hbar^{-1} [\Op(a),\Op(b)]
&:=i\hbar^{-1}\big(\Op(a)\Op(b)- \Op(b)\Op(a)\big)\\
&=\Op(\{ a,b\})+\lot,
\end{align*}
where the Poisson bracket $\{\cdot,\cdot\}$ is defined by
$$\{a,b\}:=\langle \partial_{\xi}a,\partial_x b\rangle -\langle \partial_{\xi}b,\partial_x a\rangle.$$
\item (Sharp G\aa rding inequality.) If $a\in S^m(\mathbb{R}^{2d})$, $w\in S^{m/2}(\mathbb{R}^{2d})$ and $a\geq w^2$, then for all $u\in H^{m/2}_k(\Rea^d)$
\begin{equation}
\label{e:gaarding}\langle \Op(a)u,u\rangle\geq \|\Op(w)u\|_{L^2(\Rea^d)}^2+\lot 
\end{equation}
\end{enumerate}
\end{informaltheorem}

The main way in which Informal Theorem \ref{t:rules} is non-rigorous is that we have not been precise about what the lower-order terms are. 
The precise statement of what these lower-order terms are is given in Theorem \ref{t:rules2} below. However, we emphasise that the sketch proofs in the rest of this section can be understood using only Informal Theorem \ref{t:rules}.


As an example of Point 2 of Informal Theorem \ref{t:rules},
\beq\label{e:p}
p(x,\xi):=\sigma_\hbar(P)(x,\xi)=
\langle A\xi,\xi\rangle - n(x) =
\sum_{i,j=1}^dA^{ij}(x)\xi_i\xi_j-n(x);
\eeq
i.e., the term on the right-hand side of \eqref{e:quantex} involving $\hbar$ has been neglected. 

\bre[Pseudodifferential operators on a manifold]
\label{r:manifold}
In \S\ref{s:R3} we use pseudodifferential operators defined on the boundary of $\Omega_-$ (with the boundary assumed smooth). 
We highlight that the analogue of Informal Theorem~\ref{t:rules} (and its precise version Theorem \ref{t:rules2}) holds on any smooth, compact manifold, $M$, with $\mathbb{R}^{2d}$ replaced by the cotangent bundle $T^*M$, and $\mathbb{R}^d$ replaced by $M$ \cite[Appendix E]{DyZw:19}, except that now the quantisation map is defined by \eqref{e:quant2}
on coordinate charts (see \cite[Equation E.1.31]{DyZw:19}).
\ere

\subsection{Ellipticity}
\label{s:elliptic}
Ellipticity is the notion of ``microlocal invertibility":~at elliptic points, the solution operator for $P$ is very well behaved, in that it both has good estimates and is (almost) local.

From now on, to keep the notation concise, we abbreviate
 $\|\cdot\|_{H^s_k(\Rea^d)}$ to $\|\cdot\|_{H^s_k}$
 and $\|\cdot\|_{L^2(\Rea^d)}$ to $\|\cdot\|_{L^2}$.

The next result involves an arbitrary pseudodifferential operator $Q=\Op(q)$.
We highlight that in almost all applications of this result in this article, $Q$ is taken to be either $P$ \eqref{e:quantex} or $P^*$.

\begin{theorem}[Elliptic estimate]
\label{t:elliptic}
Let $m_1,m_2\in \mathbb{R}$ and $q\in S^{m_1}(\mathbb{R}^{2d})$, $a\in S^{m_2}(\mathbb{R}^{2d})$, and $b\in S^{0}(\mathbb{R}^{2d})$ be such that there is $c>0$ such that  
\beq\label{e:ellipticp}
|q(x,\xi)|\geq c\langle \xi\rangle^{m_1},\qquad (x,\xi)\in \supp a,
\eeq
and $\supp a\cap \supp (1-b)=\emptyset$. Then for all $\chi\in C_c^\infty(\mathbb{R}^d)$ there are $e\in S^{m_2-m_1}(\mathbb{R}^{2d})$, $\psi\in C_c^\infty(\mathbb{R}^{d})$ such that for all $N>0$ there is $C>0$ such that for $0<\hbar<1$ and $u\in \mathcal{D}'(\Rea^d)$
\beq\label{e:elliptic0}
\big\|\chi\big(\Op(a)-\Op(e)\Op(b) \Op(q)\big)u\big\|_{ H_k^N}\leq C_N\hbar^N\|\psi u\|_{H_k^{-N}}.
\eeq
In particular,
\beq\label{e:elliptic1}
\|\chi\Op(a) u\|_{H_k^s}\leq C\|\Op(b) \Op(q) u\|_{H_k^{s+m_2-m_1}}+C\hbar^N\|\psi u\|_{H_k^{-N}}.
\eeq
\end{theorem}
\begin{proof}[Sketch of proof]
\begin{enumerate}
\item By \eqref{e:ellipticp}, $e:=a/q\in S^{m_2-m_1}$ and $a /q=a b/q$. Hence, by the composition property and principal symbol property from Informal Theorem \ref{t:rules},
$$\Op(a)=
\Op(\tfrac{a}{q} bq)
+\lot
=\Op(\tfrac{a}{q}) \Op(b)\Op(q)+\lot.$$
The result \eqref{e:elliptic0} then follows by iterating this argument (to improve the regularity and size of the remainder term) 
and using the fact that 
\begin{align*}
&\chi\big(\Op(\tfrac{a}{q} bq)
-\Op(\tfrac{a}{q}) \Op(b)\Op(q)
\big)
\\
&\qquad=\chi\big(\Op(\tfrac{a}{q} bq)
-\Op(\tfrac{a}{q}) \Op(b)\Op(q)
\big)\psi,
\end{align*}
for some $\psi$, since the quantisation is properly supported.
\item
The bound \eqref{e:elliptic1} follows from \eqref{e:elliptic0} since $\|\Op(e)\|_{H^{s+m_2-m_1}_k \to H^s_k}\leq C$
by the boundedness property in Informal Theorem \ref{t:rules}.
\end{enumerate}
\end{proof}

\begin{example}
The prototypical semiclassically-elliptic differential operator is 
$\widetilde{P}:= -\hbar^2 \Delta +1$. Indeed, $\widetilde{P}=\Op(\widetilde{p})$ with $\widetilde{p}(x,\xi) =|\xi|^2+1 = \langle \xi\rangle^{2}$. 
Since this operator is a Fourier multiplier, 
directly by the Fourier transform,
\beqs
\|u\|_{H^s_k}\leq C \| \widetilde{P}u\|_{H^{s-2}_k},
\eeqs
which is analogous to the elliptic estimate \eqref{e:elliptic1} with $a=b=1$, no cutoffs, no remainder term, and $\Op(q)=\widetilde{P}$.
\end{example}

\subsection{The Hamiltonian flow}

In \S\ref{s:how} we stated that SCA describes how Helmholtz solutions in the limit $k\to\infty$ are governed by the rays. The rays appear via the commutator property of Informal Theorem \ref{t:rules} which implies that
\begin{equation}
\label{e:comm}
i\hbar^{-1}[P,\Op(a)]=\Op(H_pa)+\lot .
    \end{equation}
where $\lot$ stands for lower-order terms and 
$$
H_p a:=
\{ p, a\}=
\langle \partial_{\xi}p(x,\xi),\partial_x a\rangle -\langle \partial_{x}p(x,\xi),\partial_\xi a \rangle.
$$
The key point is that 
\beq\label{e:key}
\begin{minipage}{\textwidth}
    \begin{center}
    $H_p a$ is the derivative of $a(x,\xi)$ when $(x,\xi)$ evolve according to Hamilton's equations with Hamiltonian $p(x,\xi)$.
    \end{center}
\end{minipage}
\eeq
To see this, recall that given a function $\mathscr{H}(x,\xi)$ (the \emph{Hamiltonian}), Hamilton's equations are 
\beq\label{eq:Hamilton_equations}
\diff{x_i}{t}(t) = \pdiff{}{\xi_i}\mathscr{H}\big(x(t), \xi(t) \big), \qquad
\diff{\xi_i}{t}(t)
 = -\pdiff{}{x_i}\mathscr{H}\big(x(t), \xi(t) \big).
\eeq
By the chain rule and the definition of the Poisson bracket $\{\cdot,\cdot\}$, 
\beq\label{eq:Hamilton1}
\diff{}{t}\big( f(x(t),\xi(t);t)\big) = \bigg(\{ \mathscr{H} , f\} +\pdiff{f}{t}\bigg) \big(x(t), \xi(t);t\big).
\eeq
Given $(x_0,\xi_0)$, we define the Hamiltonian flow 
\beq\label{e:flow}
\varphi_t(x_0,\xi_0):=(x(t),\xi(t)),
\eeq
where $(x(t),\xi(t))$ is the solution of \eqref{eq:Hamilton_equations} with initial condition $(x(0),\xi(0))= (x_0,\xi_0)$.
(Another notation for $\varphi_t$ is $\exp(tH_p)$.)
With this notation, given $a(x,\xi)$, \eqref{eq:Hamilton1} can be rewritten as 
\beq\label{eq:a_flow}
\diff{}{t}(a\circ \varphi_t) = \{\mathscr{H},a\}\circ \varphi_t.
\eeq
When $\mathscr{H}=p$, \eqref{eq:a_flow} therefore justifies the statement \eqref{e:key}.

When $P$ is given by \eqref{e:quantex} with $A\equiv I$ and $n\equiv 1$, by \eqref{e:p}
$p(x,\xi)= |\xi|^2-1$, and Hamilton's equations \eqref{eq:Hamilton_equations} become
$$
\dot{x}_i = 2 \xi_i, \quad \dot{\xi}_i =0,
$$
with solution 
$$
x=x_0 +2 t \xi_0, \quad \xi= \xi_0;
$$
i.e., straight-line motion with speed $2\xi_0$. We see below that special importance is played by the set $\Sigma:=\{p=0\}$; on this set $|\xi_0|=1$ so that the flow restricted to $\Sigma$ has speed $2$.

\bre[Generalised broken bicharacteristics]
When the domain has a boundary, the Hamiltonian trajectories need to be replaced by \emph{generalised broken bicharacteristics} \cite[Section 24.3]{Ho:85}, which, roughly speaking, extend trajectories that intersect the boundary by the Snell--Descartes law of reflection. When $A\equiv I$ and $n\equiv 1$, these are the usual billiard trajectories except possibly when the trajectory is tangent to a point on the boundary that is flat to infinite order.
\ere

\subsection{Propagation}
\label{s:propagate}

\subsubsection{From energy estimates to commutators}\label{s:energy}

In PDE theory, \emph{energy estimates} usually refer to bounds obtained by multiplying the PDE by a combination of the solution and its derivatives and integrating by parts. If $P$ is self-adjoint and $V$ is a vector field, then
\begin{align}\label{e:vf}
2\Re \big\langle Pu,(V-V^*) u\big\rangle &=    
\big\langle (V-V^*) u, Pu \big\rangle
+\big\langle Pu, (V-V^*) u \big\rangle
= \big\langle [P,V-V^*]u,u\big\rangle 
\end{align}
so that if 
\beq\label{e:finalMorawetz}
[P,V-V^*]\geq c(V-V^*)^2,
\eeq
then one can bound $(V-V^*)u$ in terms of $Pu$. In other words, standard arguments used to obtain energy estimates can/should be thought of as \emph{positive commutator arguments}.

For example, 
if $V= x\cdot \nabla+d/2$ then $V^* = -x\cdot \nabla -d/2$, so that $V-V^*= 2 x\cdot \nabla$,
and then $[V-V^*, \Delta]= 2[x\cdot \nabla, \Delta] = -4\Delta$ and \eqref{e:finalMorawetz} is satisfied; this is the basis of so-called Rellich/Morawetz identities \cite{Re:40, MoLu:68, Mo:75} which have been the main way the numerical-analysis community has obtained propagation information about Helmholtz solutions (see Remark \ref{r:nt} below for more discussion on this).


In semiclassical and microlocal analysis, propagation estimates can be thought of as ``microlocalized energy estimates'', where one multiplies by the quantisation of certain carefully-chosen symbols. 
If $A$ and $P$ are both self-adjoint, then 
\begin{align}\label{e:multiplier}
-2i\Im \langle Pu,A u\rangle &= \big( \langle A u, Pu \rangle - \langle Pu, A u \rangle\big)= \langle [P,A]u,u\rangle.
\end{align}
By the key relation \eqref{e:comm}, 
the commutator term in \eqref{e:multiplier} has a 
sign (up to $\lot$) provided that the symbol $a=\sigma_\hbar(A)$ either increases or decreases along the flow with Hamiltonian $p=\sigma_\hbar(P)$.


\subsubsection{The basic propagation argument}

The elliptic estimate from Theorem \ref{t:elliptic} shows that Helmholtz solutions are well-behaved at high frequency (since $P$ is semiclassically elliptic there). 
For simplicity, therefore, we only consider propagation estimates with compactly supported symbols (i.e., compact support in both space and frequency). For propagation estimates with more general symbols; see \cite[Appendix E.4]{DyZw:19}.


\begin{theorem}[Basic propagation estimate]
\label{t:basicPropagate}
Suppose that $P$ is given by \eqref{e:quantex} and let $\varphi_t:=\exp(tH_p)$ be the Hamiltonian flow of $p=\sigma_\hbar(P)$ (given by \eqref{e:p}).
Let $T\in\mathbb{R}$, $(x_0,\xi_0)\in \mathbb{R}^{2d}$, and $(y_0,\eta_0):=\varphi_{-T}(x_0,\xi_0)$. Let $U$ be a neighborhood of $(y_0,\eta_0)$ and $U_1$ be a neighborhood of 
$$
\bigcup_{t=0}^T \varphi_{-t}(x_0,\xi_0).
$$
Then there is a neighborhood, $V$, of $(x_0,\xi_0)$ and $\chi \in C_c^\infty(\mathbb{R}^d)$ such that for all $a\in C_c^\infty(V;\Rea)$ and $b,b_1\in C_c^\infty(\mathbb{R}^{2d})$ 
with $\supp (1-b)\cap U=\emptyset$, $\supp(1-b_1)\cap U_1=\emptyset$ and all $N>0$, there is $C>0$ such that for $0<\hbar<1$ and for all $u\in \mathcal{D}'(\Rea^d)$ with $\Op(b_1)Pu\in L^2(\Rea^d)$
\beq\label{e:propagate}
\|\Op(a) u\|_{L^2}\leq C\hbar^{-1}\|\Op(b_1)Pu\|_{L^2}+C\|\Op(b)u\|_{L^2}+C\hbar^N\|\chi u\|_{H_k^{-N}}.
\eeq
\end{theorem}

\begin{figure}[htbp]
  \centering
  \begin{tikzpicture}
\draw[domain=.01:6, smooth, variable=\x, samples=300, thick] (-.8,1)--plot({\x}, {1+20*exp(- 2/(max(\x,.001)))*exp(-\x )*exp(-1/ (6.001-\x)})--(7,1);
\fill (.25,1)circle (.1cm);
\draw (.25,1)node[below]{$(y_0,\eta_0)$};
\fill (4.5,1.05)node[below]{$(x_0,\xi_0)$}circle (.1cm);
\draw (-.85,1)node[left]{$g$};
\draw (-.85,-0.5)node[left]{$b_1$};
\draw (-.85,-1.5)node[left]{$b$};
\draw (-.85,-2.5)node[left]{$a$};

\draw[domain=-0.8:1.5, smooth, variable=\x, samples=300, thick,xshift=0cm] (-.8,-1.5)--plot({\x}, {-1.5+.75*(exp(-1/(3*max(\x+.475,.0001)))*exp(-1/(3*max(1.1-\x,.0001))))/(exp(-1/(3*max(\x+.475,.0001)))*exp(-1/(3*max(1.1-\x,.0001)))+exp(-1/(3*max(-.1-\x,.0001)))+exp(-1/(3*max(\x-.725,.0001))))})--(7,-1.5);

\draw[domain=-0.8:7, smooth, variable=\x, samples=300, thick,xshift=0cm] (-.8,-.5)--plot({\x}, {-.5+.75*(exp(-1/(3*max(\x+.675,.0001)))*exp(-1/(3*max(5.55-\x,.0001))))/(exp(-1/(3*max(\x+.675,.0001)))*exp(-1/(3*max(5.55-\x,.0001)))+exp(-1/(3*max(-.3-\x,.0001)))+exp(-1/(3*max(\x-5.175,.0001))))})--(7,-.5);

\draw[domain=-.8:7, smooth, variable=\x, samples=300, thick](-.8,-2.5)-- plot({\x}, {-2.5+.75*(exp(-1/(3*max(\x-3.775,.0001)))*exp(-1/(3*max(5.35-\x,.0001))))/(exp(-1/(3*max(\x-3.775,.0001)))*exp(-1/(3*max(5.35-\x,.0001)))+exp(-1/(3*max(4.15-\x,.0001)))+exp(-1/(3*max(\x-4.975,.0001))))})--(7,-2.5);

\draw[thick,->,dashed,>=Stealth] (.45,-4)--(5.95,-4)node[midway,above]{flow direction};
      \end{tikzpicture}\\
  \caption{
 \label{f:escape}Profile of the escape function, $g$,  the control function, $b$,
 and the observation function, $a$, along the flow $\varphi_t$ used in the propagation estimate of Theorem \ref{t:basicPropagate} (when the assumptions hold with $T>0$).}
\end{figure}

\bre[Direction of propagation]\label{r:direction}
Theorem \ref{t:basicPropagate} involves propagation in both directions; this is possible since $\Im p=0$. More generally, if $\Im \sigma_\hbar(P)\leq 0$ then one can propagate forwards 
(i.e., $T\geq 0$)
under the Hamiltonian flow defined by $\Re \sigma_\hbar(P)$ and if $\Im \sigma_\hbar(P)\geq 0$ then one can propagate backwards under this flow (i.e., $T\leq 0$).
\ere

\begin{proof}[Sketch of the proof of Theorem \ref{t:basicPropagate}]
\begin{enumerate}
\item 
Without loss of generality, we assume that $T\geq 0$. (If $T<0$, then we can apply the proof below to $-P$, noting that the backward flow for $-P$ equals the forward flow for $P$.)

By, e.g., \cite[Lemma E.48]{DyZw:19}, 
there exists an ``escape function" $g\in C_c^\infty(\mathbb{R}^{2d};\mathbb{R})$ such that there is $c>0$ satisfying 
\begin{equation}
\label{e:escape}
g\geq 0,\qquad H_pg^2\leq -c g^2+b^2, \qquad g^2\geq ca^2,\qquad \supp g\subset U_1
\end{equation}
i.e., $g$ is decreasing along the flow apart from on the support of $b$ and is supported where $b_1$ is 1.
Figure~\ref{f:escape} gives a sketch of the functions $a$, $b$, $b_1$, and $g$ in the flow direction.
\item Let $G:=\Op(g)$. By \eqref{e:comm},
$$i\hbar^{-1}[P,G^*G]= \Op( H_pg^2) + \lot 
$$
and, by the self-adjointness of $P$,\begin{align}\nonumber
2\hbar^{-1}\Im \langle Pu,G^*Gu\rangle &=i \hbar^{-1} \big( \langle G^* G u, Pu \big\rangle - \langle Pu, G^*Gu \rangle\big)\\
&=i\hbar^{-1} \langle [P,G^*G]u,u\rangle = \langle \Op( H_pg^2)u,u\rangle+\lot 
\label{e:commutator}
\end{align}
(where the boundary terms in the integration by parts vanish since $G$ is compactly supported).
\item By the sharp G\aa rding inequality \eqref{e:gaarding} (with $w=0$)
together with~\eqref{e:escape},
$$
\langle \Op( b^2 - cg^2 - H_p g^2)u,u\rangle \geq \lot.
$$
Thus, 
by the adjoint and composition properties of Informal Theorem \ref{t:rules},
\beq\label{e:dark1}
\langle \Op( H_pg^2)u,u\rangle \leq -c\|Gu\|_{L^2}^2+ \|\Op(b)u\|_{L^2}^2 +\lot 
\eeq
\item The combination of \eqref{e:commutator}, \eqref{e:dark1}, and the Cauchy--Schwarz inequality implies that
\begin{align*}
-2\hbar^{-1}\|GPu\|_{L^2}\|Gu\|_{L^2}&\leq 2\hbar^{-1}\Im \langle Pu,G^*Gu\rangle\\
&\leq -c\|Gu\|_{L^2}^2+\|\Op(b)u\|_{L^2}^2 +\lot 
\end{align*}
\item  
Combining the result of Step 4 with the weighted Young's inequality, 
\beq\label{e:peterPaul}
2ab \leq \epsilon a^2 + \epsilon^{-1} b^2 \tfa a,b,\e>0,
\eeq
gives
\beq\label{e:dark2}
\|Gu\|_{L^2}^2\leq C\hbar^{-2}\|GPu\|_{L^2}^2+C\|\Op(b)u\|_{L^2}^2+\lot 
\eeq
\item 
Since $g^2\geq c a^2$, the sharp G\aa rding inequality \eqref{e:gaarding} (with $w=0$) implies that
$
\langle \Op(g^2) u, u \rangle\geq 
c\langle \Op(a^2) u, u \rangle+ \lot,
$
from which
\beq\label{e:dark3}
c\|\Op(a)u\|_{L^2}^2\leq \| Gu\|_{L^2}^2 + \lot
\eeq
by the composition and adjoint properties of Informal Theorem \ref{t:rules}. 
The combination of \eqref{e:dark2} and \eqref{e:dark3} implies that 
\beqs
\|\Op(a) u\|_{L^2}\leq C\hbar^{-1}\|GPu\|_{L^2}+C\|\Op(b)u\|_{L^2}+\lot
\eeqs
\item Next, Theorem~\ref{t:elliptic} and the support properties of $g$ imply
$$
\|GPu\|_{L^2}\leq \|\Op(b_1)Pu\|_{L^2}+\lot 
$$
and hence
\beq\label{e:propagate2}
\|\Op(a) u\|_{L^2}\leq C\hbar^{-1}\|\Op(b_1)Pu\|_{L^2}+C\|\Op(b)u\|_{L^2}+\lot
\eeq
(compare to \eqref{e:propagate}). 
\item By iterating Steps 1-6, one can make the remainder both smoother and smaller in terms of $\hbar$. Finally, the error term is compactly supported since $a$, $b$, and $b_1$ are.
\end{enumerate} 
\end{proof}

\subsection{Wavefront set}

For our treatment of \ref{R5} (the result about the parallel overlapping Schwarz method) in \S\ref{s:R5} it is useful to have the notion of semiclassical wavefront set. We highlight that this notion is only used in \S\ref{s:R5} (for readers wanting to initially skip this concept). 

\begin{definition}[$\hbar$-tempered distributions]\label{def:tempered}
An $\hbar$-dependent family of distributions $v_\hbar \in \mathcal D'(\mathbb R^d)$ is \emph{$\hbar$-tempered} if
for all $\chi \in C^\infty_c(\mathbb R^d)$ there exist 
$C, M>0$ such that
$$\Vert \chi v_\hbar \Vert_{H^{-M}_\hbar(\Rea^d)} \leq C \hbar^{-M}.$$
\end{definition}

The wavefront set of an $\hbar$-tempered distribution $u$, $\WF(u)$, should be understood as the points in phase space at which $u$ has non-trivial mass. One precise version of this is given in the following definition.

\begin{definition}[Wavefront set]\label{d:wavefront}
Suppose that $u$ is $\hbar$-tempered. $(x_0,\xi_0)\in \mathbb{R}^{2d}$ is \emph{not in the wavefront set of $u$}, $(x_0,\xi_0)\notin \WF(u)$, if there is $a\in C_c^\infty(\mathbb{R}^{2d})$ such that $a(x_0,\xi_0)=1$ and $M>0$ such that
$$
\|\Op(a)u\|_{H_k^{-M}(\mathbb{R}^d)}=O(\hbar^{\infty}).
$$
\end{definition}

\begin{example}[Simple examples of wavefront sets]
By integration by parts and/or stationary phase, 
\begin{enumerate}
\item for $\xi_0\in \mathbb{R}^d$ and $u(x):=\exp(i \hbar^{-1} \langle x,\xi_0\rangle)$, 
$$\WF(u )= \big\{ (x,\xi_0)\,:\, x\in \mathbb{R}^d \big\},$$
\item for $(x_0,\xi_0)\in \mathbb{R}^{2d}$, and $v(x):=\exp\big( i\hbar^{-1} \langle x-x_0,\xi_0\rangle -\frac{1}{2}\hbar^{-1}|x-x_0|^2\big)$, $$\WF(v )=\big\{ (x_0,\xi_0)\big\}.$$
\end{enumerate}
\end{example}

\begin{lemma}[Properties of $\WF$] \label{l:WF1}
Let $u$ be $\hbar$-tempered. If $(x_0,\xi_0)\notin \WF(u)$ then there is a neighbourhood $U$ of $(x_0,\xi_0)$ such that for all $a\in C_c^\infty(U)$ and all $N>0$
$$
\|\Op(a)u\|_{H_k^N(\mathbb{R}^n)}=O(\hbar^\infty).
$$\end{lemma}

Observe that Lemma \ref{l:WF1} implies that $\WF(u)$ is closed.

\begin{proof}[Proof of Lemma \ref{l:WF1}]
Let $(x_0,\xi_0)\notin \WF(u)$.
By the definition of $\WF(u)$ there are
$M\geq 0$, and $b\in C_c^\infty(\mathbb{R}^{2d})$ with $b(x_0,\xi_0)=1$ such that 
\beq\label{e:WFcombine1}
\|\Op(b)u\|_{H_k^{-M}(\mathbb{R}^d)}=O(\hbar^\infty).
\eeq
Let 
$$
U:=\big\{ (x,\xi)\,:\, |b(x,\xi)|>\tfrac{1}{2}\big\}.
$$
By the bound \eqref{e:elliptic1} from Theorem~\ref{t:elliptic}, for any $a\in C_c^\infty(U)$, there are $e\in S^{-\infty}(\mathbb{R}^{2d})$ and $\chi_1,\chi_2\in C_c^\infty(\mathbb{R}^d)$ such that 
\beq\label{e:WFcombine2}
\chi_1\Op(a)=\Op(a)
\eeq
and
\beq\label{e:WFcombine3}
\|\chi_1\Op(a)u\|_{H^N_k}
\leq \|\Op(b)u\|_{H_k^{-N}}+C_N\hbar^N\|\chi u\|_{H_k^{-N}(\mathbb{R}^d)}.
\eeq
The result then follows from 
combining \eqref{e:WFcombine1}, \eqref{e:WFcombine2}, \eqref{e:WFcombine3}, the boundedness property of Informal Theorem \ref{t:rules}, and 
the fact that $u$ is $\hbar$-tempered.
\end{proof}

We record the following corollaries of Theorems~\ref{t:elliptic} and~\ref{t:basicPropagate}. 

\begin{corollary}[Ellipticity and the wavefront set]
\label{c:ellipticWavefront}
Suppose that $u\in\mathcal{D}'(\mathbb{R}^d)$ is $\hbar$-tempered, $P$ is given by \eqref{e:quantex}, and $p$ is given by \eqref{e:p}. Then,
$$
\WF(u)\subset \{p=0\}\cup \WF(Pu). 
$$
\end{corollary}
\begin{proof}
It is sufficient to prove that if  $p(x_0,\xi_0)\neq 0$ and $(x_0,\xi_0)\notin \WF(Pu)$, then $(x_0,\xi_0)\notin\WF(u)$. 
Since $(x_0,\xi_0)\notin \WF(Pu)$, by the definition of $\WF(Pu)$
there is $b\in C_c^\infty(\mathbb{R}^{2d})$ with $b(x_0,\xi_0)=1$ such that for all $N$
$$
\|\Op(b)Pu\|_{H_k^N}=O(\hbar^{\infty}).
$$
Furthermore, since $p(x_0,\xi_0)\neq 0$, there is $a\in C_c^\infty(\mathbb{R}^{2d})$ with $a(x_0,\xi_0)=1$ such that $\supp a\subset \{|p|>0\}\cap\{ |b|>0\}$. Then, by Theorem~\ref{t:elliptic} there are $\chi,\psi\in C_c^\infty(\mathbb{R}^d)$ with $\chi(x_0)=1$ such that for all $N$
$$
\|\chi\Op(a)\|_{L^2}\leq C\|\Op(b)Pu\|_{H_k^{-N}}+C\hbar^N\|\psi u\|_{H_k^{-N}}.
$$
Therefore, since $u$ is $\hbar$-tempered, $\|\chi\Op(a)\|_{L^2}=O(\hbar^\infty)$.
Finally, since $\sigma_\hbar(\chi\Op(a))=\chi a$ 
(by the composition and principal symbol properties of Informal Theorem \ref{t:rules}) and $\chi(x_0) a(x_0,\xi_0)=1$, $(x_0,\xi_0)\notin \WF(u)$.
\end{proof}

\begin{corollary}[Propagation and the wavefront set]
\label{c:propagateWavefront}
Suppose that $u\in\mathcal{D}'(\mathbb{R}^d)$ is $\hbar$-tempered and $P$ is as in~\eqref{e:quantex}. If there is $T\in\mathbb{R}$ such that $\varphi_{-T}(x_0,\xi_0)\notin \WF(u)$ and 
$$
\bigcup_{t=0}^T\varphi_{-t}(x_0,\xi_0)\cap \WF(Pu)=\emptyset,
$$
then $(x_0,\xi_0)\notin \WF(u)$.
\end{corollary}
\begin{proof}
By Lemma~\ref{l:WF1} and compactness of $\bigcup_{t=0}^T\varphi_{-t}(x_0,\xi_0)$, there are 
$b,b_1\in C_c^\infty(\mathbb{R}^{2d})$ 
 such that
$\varphi_{-T}(x_0,\xi_0)\notin \supp (1-b)$, 
$$
\bigcup_{t=0}^T\varphi_{-t}(x_0,\xi_0)\cap \supp (1-b_1)=\emptyset,
$$
and
$$
\|\Op(b)u\|_{H_k^N}+\|\Op(b_1)Pu\|_{H_k^N}=O(\hbar^\infty).
$$
By Theorem~\ref{t:basicPropagate}, there are $a\in C_c^\infty(\mathbb{R}^{2d})$ and $\chi\in C_c^\infty(\mathbb{R}^{d})$ such that $a(x_0,\xi_0)=1$ and 
\beqs
\|\Op(a) u\|_{L^2}\leq C\hbar^{-1}\|\Op(b_1)Pu\|_{L^2}+C\|\Op(b)u\|_{L^2}+C\hbar^N\|\chi u\|_{H_k^{-N}}.
\eeqs
Therefore,
$$
\|\Op(a)u\|_{L^2}=O(\hbar^\infty),
$$
and the result follows from the definition of $\WF(u)$.
\end{proof}

\subsection{Application of ellipticity and propagation to solutions of the Helmholtz equation}

\subsubsection{Outgoing solutions of the Helmholtz equation}

Solutions of the Helmhotz satisfying the Sommerfeld radiation condition are semiclassically outgoing; i.e., far from the scatterer they are microlocally concentrated on trajectories traveling towards infinity. This fact is encoded in the following lemma via the support condition on $a$.

\begin{lemma}[The semiclassical outgoing property]\label{l:outgoing}
For all $R>0$ there are $R'>R$, $c>0$ such that for all $a\in S^0(\mathbb{R}^{2d})$ with 
\beq\label{e:suppa}
\supp a\subset \big\{ \langle \tfrac{x}{|x|},\xi\rangle \leq c,\,|x|>R'\big\}\cup\big\{\big||\xi|-1\big|>c>0, |x|>R'\big\}
\eeq
and all $N>0$ there are $
\psi \in C_c^\infty(\mathbb{R}^d)$ and $C>0$ such that if 
$u\in H^2_{\loc}(\mathbb{R}^d)$ satisfies
$$
(-\hbar^2\Delta-1)u=0,\qquad |x|>R,\qquad 
\|(\hbar D_r-1)u\|_{L^\infty(|x|=r)}=o_{r\to \infty}(r^{\frac{1-d}{2}}),
$$
then
\beq\label{e:outgoingWF}
\|\Op(a)u\|_{L^2(\Rea^d)}\leq C\hbar ^N\|\psi u\|_{H_k^{-N}(\mathbb{R}^d)}.
\eeq
In addition, there is $N_0>0$ such that for $\chi \in C_c^\infty(\mathbb{R}^d)$ with $\supp \chi \cap B(0,R)=\emptyset$,  
\begin{equation}
\label{e:weak1}
\|\chi u\|_{L^2(\Rea^d)}\leq C\hbar^{-N_0}\|\psi u\|_{H_k^2(\mathbb{R}^d)}. 
\end{equation}
\end{lemma}

(A much sharper bound than \eqref{e:weak1} holds, but \eqref{e:weak1} is sufficient for 
the later proofs in this section. Indeed,
using the result of Theorem~\ref{t:nontrap} below in the proof of Lemma \ref{l:outgoing}, we see that \eqref{e:weak1} holds with $N_0$ replaced by $1$.)

\begin{proof}[Proof of Lemma \ref{l:outgoing}]
Let 
\beq\label{eq:R_0}
\big(R_0(\hbar) f\big)(x):= \int_{\Rea^d}\Phi_\hbar(x,y) \, f(y)\, d y,
\eeq
where $\Phi_\hbar$ is the outgoing fundamental solution satisfying $(-\hbar^2\Delta-1)\Phi_\hbar(x,y)=\delta(x-y)$, i.e., 
\beq\label{eq:fundH}
\Phi_\hbar(x,y):=
\hbar^{-2}\frac{\ri}{4} \left( \frac{\hbar^{-1}}{2\pi|x-y|}\right)^{(d-2)/2} H^{(1)}_{d/2-1}\big(\hbar^{-1}|x-y|\big);
\eeq
see, e.g., \cite[Chapter 9]{Mc:00}.
Observe that
$$
u=R_0(-\hbar^2\Delta-1)u=:R_0f, 
$$
with $f\in L^2(\mathbb{R}^d)$, $\supp f\subset B(0,R)$. 
Since $f=Pu$, for any $s$, 
\beq\label{e:dark4}
\|f\|_{H_k^{s}}\leq C\|u\|_{H_k^{s+2}(B(0,R))}.
\eeq
Using standard Hankel-function asymptotics for large argument,
for $|z|>R+1$
\beq\label{e:Hankel}
R_0 f(z) = \hbar^{-(d-3)/2-2} \int_{B(0,R)} e^{\frac{i}{\hbar} |z-y|}q(|z-y|) f(y)d y
\eeq
where 
$$
|\partial^\alpha q(s)|\leq C_\alpha |s|^{\frac{1-d}{2}-|\alpha|}.
$$
Thus, by the Cauchy--Schwarz inequality,
\beq\label{e:dark5}
\| \chi u\|_{L^2(\Rea^d)}\leq C 
\hbar^{-(d-3)/2-2}
\|f\|_{L^2},
\eeq
and then \eqref{e:weak1} follows by combining \eqref{e:dark4} (with $s=0$) and \eqref{e:dark5}.

Let $\psi$ be the cutoff function appearing in the definition of $\Op$ \eqref{e:quant2}.
To prove \eqref{e:outgoingWF}, observe that for $R'$ large enough $\supp (a(x,\xi)\psi(|x-\bullet|))\subset \{|\bullet |>R+1\}$, 
and thus, by \eqref{e:Hankel},
\begin{align*}
&[\Op(a)u](x)= [\Op(a)R_0f](x)\\
&= \frac{\hbar^{-2-\frac{d-3}{2}}}{(2\pi \hbar)^d}\int_{|z|>R+1, |y|<R} e^{\frac{i}{\hbar}
\Phi(z,\xi;x,y)
}a(x,\xi)\psi(|x-z|) q(|z-y|)f(y) d\xi dzdy.
\end{align*}
where
$$
\Phi(z,\xi;x,y):=\langle x-z,\xi\rangle +|z-y|.
$$
The critical points of the phase in $(z,\xi)$ occur when 
$\partial_z\Phi = \partial_\xi \Phi=0$; i.e.,  when 
$$
\xi=\frac{z-y}{|z-y|}\quad\tand\quad x=z;
$$
i.e., 
when $\xi = (x-y)/|x-y|$.
We now claim that the support properties of $a$ \eqref{e:suppa} imply that there are no critical points of the phase function on the support of the integrand.
Indeed, when $|x|>R'$ for large enough $R'$, $(x-y)/|x-y|$ is pointing ``outwards" (i.e., close to the direction $x/|x|$) and has modulus one, while, by \eqref{e:suppa}, on $\supp a$ \emph{either} $\xi$  is pointing ``inwards" (or rather ``not too outward" because of the presence of $c>0$)  \emph{or} $|\xi|$ is bounded away from one.
Repeatedly integrating by parts in $z$ and $\xi$ then brings down arbitrarily many powers of $\hbar$ and $\langle\xi\rangle^{-1}$, proving the result. 
\end{proof}

\subsubsection{Trapping and nontrapping}
Define $\pi_x:\mathbb{R}^{2d}\to \mathbb{R}^d$ by $\pi_x(x_0,\xi_0)=x_0$. Then
\beq\label{eq:Gammafw}
\Gamma_{\rm f} := \big\{ (x,\xi) : \big|\big(\pi_x(\varphi_t(x,\xi)\big)\big|\nrightarrow \infty \quad \tas \quad t \to\infty\big\},
\eeq
i.e., $\Gamma_{\rm f}$ is the \emph{forward trapped set}, and 
\beq\label{eq:Gammabw}
\Gamma_{\rm b} := \big\{ (x,\xi) : \big|\big(\pi_x(\varphi_t(x,\xi)\big)\big|\nrightarrow \infty \quad \tas \quad t \to-\infty\big\}
\eeq
i.e., $\Gamma_{\rm b}$ is the \emph{backward trapped set}.
Let $K:= \Gamma_{\rm f}\cap \Gamma_{\rm b}$; i.e., $K$ is the \emph{trapped set}. 

\ble[Properties of $\Gamma_{\rm f}, \Gamma_{\rm b}$, and $K$]\label{l:trapping}
\ \
\vspace{-.8em}
\begin{enumerate}
\item[(i)] $\Gamma_{\rm f}\neq \emptyset$ if and only if $\Gamma_{\rm b}\neq \emptyset$. 
\item[(ii)] $\Gamma_{\rm f}\neq \emptyset$ if and only if $K\neq \emptyset$. 
\item[(iii)] $\Gamma_{\rm f}$, $\Gamma_{\rm b}$, and thus $K$ are closed.
\end{enumerate}
\ele
\bpf[References for the proof]
(i) follows by time-reversal of the flow. (ii) is proved in, e.g., \cite[Proposition 6.4]{DyZw:19}, and (iii) is proved in  \cite[Proposition 6.3]{DyZw:19}.
\epf

Recall that $\mathcal{R}$ is the solution operator for the problem \eqref{e:edp}.

\begin{theorem}[Nontrapping cutoff resolvent bound]
\label{t:nontrap}
Let $P$ be given by \eqref{e:quantex}/\eqref{e:edp} 
with $\Omega_-=\emptyset$, $A$ and $n$ smooth, and $\supp(A-I)$ and $\supp(n-1)$ compact. 
If $K=\emptyset$, then for all $\chi\in C_c^\infty(\mathbb{R}^d)$ there is $C>0$ such that for all $0<\hbar<1$,
\beq\label{e:nt}
\|\chi\mathcal{R}\chi\|_{L^2\to H_k^2} \leq C \hbar^{-1}.
\eeq
\end{theorem}

\begin{proof}[Sketch of the proof]
\begin{enumerate}
    \item (Set up.) Since $\|\chi\cR\chi \|_{L^2\to H^2_k}$ is a continuous function of $\hbar$ and bounded for every fixed $\hbar$, it is sufficient to prove \eqref{e:nt} for $\hbar$ sufficiently small.
Without loss of generality, we can assume that $\supp(A-I)\cup \supp(n-1)\subset \supp \chi$.
Let $u:=\mathcal{R}\chi f$, $R>0$ such that $\supp \chi \subset B(0,R)$, and $\chi_1\in C_c^\infty(\mathbb{R}^d)$ with $\supp (1-\chi_1)\cap B(0,R)=\emptyset$. 

\item (The elliptic region.) 
By the composition property in Informal Theorem \ref{t:rules},
given $a\in S^0$ with $\supp a\cap \{p=0\}=\emptyset$ and $\chi_1$ as above, there exists $a'$ such that $\Op(a)\chi_1 =\Op(a')$
and, up to an $O(\hbar^\infty)_{\Psi_\hbar^{-\infty}}$ error, $\supp a'\subset \supp a$. 
Then, by the fact that $\Op(a)$ is properly supported and the elliptic estimate (Theorem \ref{t:elliptic}, applied with $a$ replaced by $a'$), there are $\psi_1,\psi\in C_c^\infty(\mathbb{R}^d)$ such that for any $M>0$,  
\begin{align}\nonumber
\|\Op(a)\chi_1 u\|_{H_k^2}
=\|\psi_1\Op(a)\chi_1 u\|_{H_k^2}
&=\|\psi_1\Op(a') u\|_{H_k^2}\\
&\leq C\|Pu\|_{L^2} + C \hbar^M\| \psi u\|_{L^2}.
\label{e:res1}
\end{align}

\item (The propagative region.)  Since $K=\emptyset$, $\Gamma_{\rm b}=\emptyset$ (by Part (ii) of Lemma \ref{l:trapping}), and thus for any $R'>R$ 
and all $(x_0,\xi_0)\in \{p=0\}\cap \supp \chi_1$, there is $t>0$ such that 
$$
\varphi_{-t}(x_0,\xi_0)\in \big\{(x,\xi)\,:\, |x|>R',\, \langle \tfrac{x}{|x|},\xi\rangle \leq 0\big\}. 
$$
By the propagation estimate (Theorem~\ref{t:basicPropagate}), there are $a,b\in C_c^\infty(\mathbb{R}^{2d})$ and $\psi\in C_c^\infty(\mathbb{R}^d)$ with $a(x_0,\xi_0)=1$ and 
$$
\supp b\subset \big\{(x,\xi)\,:\, |x|>R',\, \langle \tfrac{x}{|x|},\xi\rangle \leq 0\big\}
$$
such that given $M>0$ there exists $C>0$ such that
$$
\|\Op(a) u\|_{L^2}\leq C\hbar^{-1}\|Pu\|_{L^2}+\|\Op(b)u\|_{L^2}+C\hbar^M\|\psi u\|_{L^2}.
$$
By the outgoing property \eqref{e:outgoingWF} and the support properties of $b$, we obtain that
\beq\label{e:res2}
\|\Op(a) u\|_{L^2}\leq C\hbar^{-1}\|Pu\|_{L^2}+C\hbar^M\|\psi u\|_{L^2}. 
\eeq
\item (Conclusion.) 
The combination of \eqref{e:res1}, \eqref{e:res2}, the compactness of $\{p=0\}\cap \supp\chi_1$, and a partition of unity argument implies that
$$
\|\chi_1 u\|_{H_k^2}\leq C\hbar^{-1}\|Pu\|_{L^2}+C\hbar^M\|\psi u\|_{L^2}
$$
for some $\psi\in C_c^\infty(\mathbb{R}^d)$ and any $M>0$. 
Absorbing part of the right-most term into the left-hand side, we obtain
\beq\label{e:swim1}
\|\chi_1 u\|_{H_k^2}\leq C\hbar^{-1}\|Pu\|_{L^2}+C\hbar^M\|\psi(1-\chi_1) u\|_{L^2}.
\eeq
Since $\psi(1-\chi_1)\in C^\infty_c$ with support away from $B(0,R)$, we can use the crude bound~\eqref{e:weak1} to bound the last term in \eqref{e:swim1}. Then, with $M=N+N_0$,
$$
\|\chi_1 u\|_{H_k^2}\leq C\hbar^{-1}\|Pu\|_{L^2}+C\hbar^N\|u\|_{H_k^2(B(0,R))};
$$
the result then follows by absorbing the right-most term into the left-hand side.
\end{enumerate}
\end{proof}

\bre[History and context of Theorem \ref{t:nontrap}]
Theorem \ref{t:nontrap}, and its generalisation to 
either a Dirichlet or Neumann obstacle, was proved in \cite{Va:75} using the propagation results of \cite{MeSj:78, MeSj:82} (see also 
\cite[Theorem 4.43]{DyZw:19}). 
Alternative proofs appear in \cite{Me:79} and \cite{Bu:02}, using, respectively, Lax--Phillips scattering theory \cite{LaPh:89} and \emph{defect measures} (another way of encoding propagation information; see \cite[\S5]{Zw:12}).
\ere

\bre[Proving cutoff resolvent bounds using vector fields]
\label{r:nt}
When $\Omega_-$ is star-shaped and $A$ and $n$ satisfy certain monotonicity properties, the bound \eqref{e:nt} can be proved using the vector field $V= x\cdot \nabla$ in \eqref{e:vf}; i.e., using Rellich/Morawetz identities (see \S\ref{s:energy}). When $A\equiv I$ and $n\equiv 1$, this result goes back to \cite{MoLu:68,Mo:75}; for the case of variable $A$ and $n$, see \cite{Bl:73, BlKa:77, PeVe:99, GrPeSp:19, MoSp:19}.
The advantage of the vector-field method is that, when it works, the boundary and coefficients can be very rough (see \cite{ChMo:08, GrPeSp:19}), but it 
cannot prove the bound for a general smooth nontrapping obstacle (unlike propagation estimates). 
(For cutoff resolvent estimates proved using non-radial vector fields -- and hence managing to break out of the ``star-shaped paradigm" slightly --   see \cite{BlKa:74, Mo:75}, 
\cite[\S4]{MoRaSt:77}, \cite{ChSpGiSm:20, ChSp:23}.)

The evolution of propagation estimates relying on positive commutators from vector-field-type arguments is discussed in \cite{MoRaSt:77} (where the term ``escape function" originated) and \cite[\S1]{BaLeRa:92}. The relationship of the vector-field method to symmetries of the underlying PDE is given in \cite[\S4.4]{Ol:93}.

We highlight that the vector field $x\cdot \nabla$ has often been used by the numerical-analysis community to 
prove bounds on the solution operator of the interior impedance problem (a ubiquitous model problem in the numerical analysis of the Helmholtz equation) \cite{AzKeSt:88, DoSaShBe:93, MaIhBa:96}, \cite[Prop.~8.2.7]{Me:95}, \cite{CuFe:06, He:07}, \cite[Chapter 2]{Ch:15}, \cite{BrGaPe:17, BaChGo:17, OhVe:18,GrSa:20}
and also to obtain new, coercive variational formulations of this problem \cite{MoSp:14, GaMo:20}.
\ere

\subsubsection{Improved estimates with the measurement and/or data away from trapping}

The next result shows that (provided the cutoff resolvent is polynomially bounded in $\hbar$), the behaviour of the solution operator is improved when the solution is measured away from the trapped set $K$.

\begin{theorem}[Estimates with the measurement away from trapping]
\label{l:measureAwayTrapping}
Let $P$ be given by \eqref{e:quantex}/\eqref{e:edp} 
with $\Omega_-=\emptyset$, $A$ and $n$ smooth, and $\supp(A-I)$ and $\supp(n-1)$ compact. 
Suppose that $a\in C_c^\infty (\mathbb{R}^{2d})$ with $\supp a\cap K=\emptyset$. Then, there is $\chi \in C_c^\infty(\mathbb{R}^d)$ such that for all $N>0$ there is $C>0$ such that for $0<\hbar<1$ and all $u\in L^2_{\loc}$  with $u$ satisfying the Sommerfeld radiation condition and $\supp(1-\chi)\cap \supp Pu=\emptyset$
\begin{align}\label{e:improve1}
\|\Op(a)u\|_{L^2}^2&\leq C\hbar^{-1}\|Pu\|_{L^2}\|\chi u\|_{L^2}+C\hbar^{-2}\|Pu\|_{L^2}^2+C\hbar^N\|\chi u\|_{L^2}^2\\ \nonumber
 &\leq C\hbar^{-1}\Big(\|\chi \cR\chi \|_{L^2\to L^2}+\hbar^{-1} +\hbar^N\|\chi \cR\chi\|_{L^2\to L^2}^2\Big)\|Pu\|_{L^2}^2.
\end{align}
Therefore, if $\|\chi \cR\chi\|_{L^2\to L^2}\leq C \hbar^{-M}$ for some $M>0$, then 
\beqs
\|\Op(a)\cR\chi\|_{L^2\to L^2} \leq C \sqrt{ \hbar^{-1}\|\chi \cR\chi \|_{L^2\to L^2}}.
\eeqs
\end{theorem}

As recalled in \S\ref{s:rho}, $\rho:=\|\chi \cR\chi\|_{L^2\to L^2}\gg\hbar^{-1}$ when $K\neq \emptyset$, so that 
$\sqrt{ \hbar^{-1}\rho }\ll \rho$ in this case; i.e., Lemma \ref{l:measureAwayTrapping} indeed shows that the solution operator is improved when the solution is measured away from $K$.

\begin{proof}[Sketch of the proof of Theorem \ref{l:measureAwayTrapping}]
\begin{enumerate}
\item Let $w\in C_c^\infty(-2,2)$ $w\equiv \frac{1}{2}$ on $[-1,1]$ and $\int w^2\equiv 1$. Then set 
$$
\psi_R(x):=w_R(|x|),\qquad w_R(r):=1-\int_0^{r}w^2(s-R)ds
$$
(i.e., $\psi_R\equiv 1$ on $r\leq R-2$, and $\psi_R\equiv 0$ on $r\geq R+2$).
Let $R>0$ be sufficiently large so that  $\supp (1-\psi_R)\cap \big(\supp (A-I)\cup \supp (n-1)\cup \supp (Pu)\big)=\emptyset$. Then, by \eqref{e:p}, $p(x,\xi)= |\xi|^2 -1$ on $\supp \psi_R$ so that
$$
H_p\psi_R = -2\langle \xi,\tfrac{x}{|x|}\rangle w^2(|x|-R).
$$
\item Since $u$ is outgoing, 
Lemma~\ref{l:outgoing} implies that all of its energy/mass in phase space over $\supp  w^2(|\cdot |-R)$ is contained where $\langle \xi,\tfrac{x}{|x|}\rangle \geq c/2>0$ and $|\xi|\leq 2$.
Since 
$$
H_p \psi_R \leq - c w^2(|x|-R) 
\quad \ton 
\big\{\langle \xi,\tfrac{x}{|x|}\rangle \geq c/2>0\big\} \cup \{|\xi|\leq 2\},
$$
by the sharp G\aa rding inequality \eqref{e:gaarding} and Lemma \ref{l:outgoing},
\begin{align*}
\langle \Op(H_p\psi_R) u,u\rangle
\leq 
-c\|w(|\bullet|-R)u\|_{L^2}^2+\lot 
\end{align*}
\item Arguing as in~\eqref{e:multiplier}/\eqref{e:commutator}, for real-valued $g\in C_c^\infty$,
\beq\label{e:am1}
2\hbar^{-1}\Im \langle Pu,\Op(g) u\rangle 
=i\hbar^{-1} \langle [P,\Op(g)]u,u\rangle = \langle \Op( H_p g)u,u\rangle+\lot 
\eeq
so that 
\begin{equation}
\label{e:upgradeIntermediate}
-2\hbar^{-1}|\Im \langle Pu,\psi_Ru\rangle|\leq -c\|w(|\bullet|-R)u\|_{L^2}^2+\lot 
\end{equation}
Although $\psi_R$ is not compactly supported in $\xi$ and so \eqref{e:am1} cannot immediately be applied with $g=\psi_R$, the elliptic estimate shows that the high-frequencies of $u$ are well-behaved, and so this step can be rigorously implemented by inserting low-frequency cutoffs.
\item By the Cauchy--Schwarz inequality, if $\chi$ is such that  $\supp (1-\chi)\cap \supp \psi_R=\emptyset$,
$$
\|w(|\bullet|-R)u\|^2_{L^2}\leq C\hbar^{-1}\|Pu\|_{L^2}\|\chi u\|_{L^2} +\lot 
$$
\item The result \eqref{e:improve1} now follows by applying Theorem~\ref{t:basicPropagate}, the fact that $w(|x|-R)=\frac{1}{2}$ on $\{R-1\leq |x|\leq R+1\}$, and the fact that the Hamiltonian trajectory from any point in $\supp a$ reaches $\{R-1\leq |x|\leq R+1\}$ in finite time (under one of the forward or backward flow) since $\supp a\cap K=\emptyset$.
(As in the proof of Theorem \ref{t:basicPropagate}, we iterate the argument to show that the lower-order terms can be made $O(\hbar^\infty)$.)
\end{enumerate}
\end{proof}

The next result shows that (provided that the resolvent is polynomially bounded in $\hbar$) if $Pu$ is supported away from the trapped set and $u$ is measured away from trapping, then 
one recovers the nontrapping cutoff resolvent bound.

\begin{theorem}[Estimates with measurement and data away from trapping]\label{l:everythingawaytrapping}
Let $P$ be given by \eqref{e:quantex}/\eqref{e:edp} 
with $\Omega_-=\emptyset$, $A$ and $n$ smooth, and $\supp(A-I)$ and $\supp(n-1)$ compact. 
Suppose that $a,b\in S^0 (\mathbb{R}^{2d})$ with $(\supp a\cup \supp b)\cap K=\emptyset$. Then, there is $\chi \in C_c^\infty(\mathbb{R}^d)$ such that for all $N>0$ there is $C>0$ such that for $0<\hbar<1$ and all $u\in L^2_{\loc}(\Rea^d)$  with $u$ satisfying the Sommerfeld radiation condition 
and $\supp(1-\chi)\cap \supp Pu=\emptyset$
\begin{align}\label{e:improve2}
\|\Op(a)u\|_{L^2}^2&\leq C\hbar^{-2}\|Pu\|_{L^2}^2+C\hbar^{-1}\|\Op(1-b)Pu\|_{L^2}\|\chi u\|_{L^2}+C\hbar^N\|\chi u\|_{L^2}^2\\ \nonumber
&\leq C\hbar^{-2}\big(1+\hbar^N\|\chi \cR\chi \|_{L^2\to L^2}^2\big)\|Pu\|_{L^2}^2\\
&\qquad +C\hbar^{-1}\|\Op(1-b)Pu\|_{L^2}\|\chi \cR\chi \|_{L^2\to L^2} \|Pu\|_{L^2}.
\label{e:improve2a}
\end{align}
Moreover, if $\|\chi \cR\chi\|_{L^2\to L^2}\leq C \hbar^{-M}$ for some $M>0$
then 
\beq\label{e:improve3}
\| \Op(a)\mathcal{R}_P \Op(b)\|_{L^2\to L^2} \leq C \hbar^{-1}.
\eeq
\end{theorem}
\begin{proof}[Sketch of the proof]
\begin{enumerate}
\item We begin from~\eqref{e:upgradeIntermediate}.
As in the proof of Theorem \ref{l:measureAwayTrapping}, 
let $\chi$ be such that $\supp (1-\chi)\cap \supp \psi_R=\emptyset$.
The inequality \eqref{e:upgradeIntermediate} 
together with the Cauchy--Schwarz inequality imply that
\begin{equation}
\label{e:mint}
\begin{aligned}
\|w(|x|-R)u\|_{L^2}^2&\leq C|\hbar^{-1}\Im \langle \Op(b)Pu,\psi_Ru\rangle|\\
&\qquad
+C\hbar^{-1}\| \Op(1-b)Pu\|_{L^2}\|\chi u\|_{L^2}+\lot \\
&= C|\hbar^{-1}\Im \langle Pu,\Op(\bar{b})u\rangle|\\
&\qquad
+C\hbar^{-1}\| \Op(1-b)Pu\|_{L^2}\|\chi u\|_{L^2}+\lot \\
&\leq C\hbar^{-1}\|Pu\|_{L^2}\|\Op(\bar{b})u\|_{L^2}\\
&\qquad
+C\hbar^{-1}\| \Op(1-b)Pu\|_{L^2}\|\chi u\|_{L^2}+\lot 
\end{aligned}
\end{equation}
\item By Theorem \ref{t:basicPropagate} and the fact that $\supp b\cap K=\emptyset$ (and hence that the Hamiltonian trajectory from any point in $\supp b$ reaches $\{R-1\leq |x|\leq R+1\}$ in finite time), 
\begin{equation}
\label{e:pastry}
\|\Op(\bar{b})u\|_{L^2}\leq C\hbar^{-1}\|Pu\|_{L^2}+\|w(|x|-R)u\|_{L^2}+\lot 
\end{equation}
\item  Combining~\eqref{e:mint} and~\eqref{e:pastry} yields
\begin{equation}
\label{e:mint2}
\begin{aligned}
\|w(|x|-R)u\|_{L^2}^2&\leq C\hbar^{-2}\|Pu\|_{L^2}^2+C\hbar^{-1}\|Pu\|_{L^2}\|w(|x|-R)u\|_{L^2}\\
&\qquad
+C\hbar^{-1}\| \Op(1-b)Pu\|_{L^2}\|\chi u\|_{L^2}+\lot 
\end{aligned}
\end{equation}
\item Use the weighted Young's inequality \eqref{e:peterPaul} on the second term on the right-hand-side of~\eqref{e:mint2} and absorbing the resulting $C\epsilon\|w(|x|-R)u\|_{L^2}^2$ on the left-hand-side yields
\begin{equation*}
\begin{aligned}
\|w(|x|-R)u\|_{L^2}^2&\leq C\hbar^{-2}\|Pu\|_{L^2}^2+C\hbar^{-1}\| \Op(1-b)Pu\|_{L^2}\|\chi u\|_{L^2}+\lot 
\end{aligned}
\end{equation*}
\item The result \eqref{e:improve2} follows by applying Theorem~\ref{t:basicPropagate} as in Step 5 of the proof of Theorem \ref{l:measureAwayTrapping}.
\end{enumerate}
To obtain \eqref{e:improve3}, given $b\in C^\infty_c$ with $\supp b\cap K=\emptyset$, since $K$ is closed (by Part (iii) of Lemma \ref{l:trapping}) there exists $\widetilde{b}\in C^\infty_c$ with $\supp \widetilde{b}\cap K=\emptyset$ and $\supp(1-\widetilde{b})\cap \supp b=\emptyset$ (i.e., $\widetilde{b}$ is ``a bit bigger" than $b$).
Applying \eqref{e:improve2a} with $b$ replaced by $\widetilde{b}$ and $u= R_P \Op(b) f$, 
and using both the boundedness property in Theorem \ref{t:rules} and the polynomial bound on the resolvent, 
we obtain that 
\beqs
\|\Op(a)R_P \Op(b) f\|_{L^2}^2\leq C\hbar^{-2}\|f\|_{L^2}^2+C\hbar^{-1-M}\|\Op(1-\widetilde{b})\Op(b)\|_{L^2\to L^2} \|f\|_{L^2}^2.
\eeqs
By the composition property in Theorem \ref{t:rules} $\|\Op(1-\widetilde{b})\Op(b)\|_{L^2\to L^2} =\lot$; in fact one can show that it is $O(\hbar^N)$ for any $N>0$, and the result \eqref{e:improve3} follows.
\end{proof}

\bre[History and context of Theorems \ref{l:measureAwayTrapping} and \ref{l:everythingawaytrapping}]
The result of Theorem \ref{l:measureAwayTrapping} first appeared in \cite{DaVa:12}, and the result of Theorem  \ref{l:everythingawaytrapping} first appeared in \cite{DaVa:12a}.
The result of Theorem \ref{l:everythingawaytrapping} with 
$\Op(a)$ and $\Op(1-b)$ replaced by 
spatial cutoffs supported far away from $\supp(A-I)\cup \supp(n-1)$, but without the assumption that $\|\chi \cR \chi\|_{L^2\to L^2}$ is polynomially bounded, was proved in \cite{Bu:02, CaVo:02} (see also \cite[Theorem 6.22]{DyZw:19}).
The analogues of Theorems \ref{l:measureAwayTrapping} and \ref{l:everythingawaytrapping} for the Helmholtz PML problem were proved in \cite[Appendix C]{AGS2}; these are used to prove \ref{R4} in \S\ref{s:R4}.
\ere

\section{Preasymptotic error estimates for the $h$-FEM (\ref{R1})}\label{s:R1}

In this section, we review and slightly improve the results of~\cite{GS3} concerning preasymptotic error estimates in the $h$-FEM,
with the main result applied to the Helmholtz PML problem Theorem \ref{t:R1} below.

The argument is based on the fact that Helmholtz operators are semiclassically elliptic at high frequencies.
As discussed in \S\ref{s:Fouriertopseudo}, it is often possible to prove results based on this property without explicit use of SCA, and, in this section, all SCA
has been removed in favour of the functional calculus of a self-adjoint operator. Nevertheless, we emphasise that the idea of the proof was obtained via ``thinking semiclassically", with this process  illustrated via a model problem in \S\ref{s:ellipticProjection}.

\subsection{Abstract set up}

Let $\mathcal{H}\subset \mathcal{H}_0\subset \mathcal{H}^*$ be  Hilbert spaces with $\mathcal{H}_0$ identified with its dual and let $a:\mathcal{H}\times \mathcal{H}\to \mathbb{C}$ be a sesquilinear form. 

\begin{definition}[Galerkin approximation of $u$ for $a$]
Let $\mathcal{V}_h\subset \mathcal{H}$ and $u\in\mathcal{H}$. A \emph{Galerkin approximation of $u$ in $\mathcal{V}_h$}, is an element $u_h\in\mathcal{V}_h$ such that 
\begin{equation}
\label{e:galerkinOrthogonality}
a(u_h,v_h)=a( u,v_h)\quad\text{ for all }v_h\in\mathcal{V}_h
\end{equation}
(with \eqref{e:galerkinOrthogonality} known as \emph{Galerkin orthogonality}).
\end{definition}

The goal of this section is to analyse the error in approximating $u$ by $u_h$.

We use the subscript to $h$ to indicate that the object comes from the approximation space (e.g., $\cV_h$, $u_h$). This notation is motivated by the fact that the concrete spaces to which this abstract theory is applied in this article are piecewise-polynomial spaces on a triangulation with width $h$.

\begin{definition}[Norm of a sesquilinear form]
\label{def:norm}
For $a:\mathcal{H}\times \mathcal{H}\to \mathbb{C}$ a sesquilinear form,
$$
\|a\|:=\sup_{\substack{u,v\in\mathcal{H}\\u,v\neq0}}\frac{|a(u,v)|}{\|u\|_{\mathcal{H}}\|v\|_{\mathcal{H}}}.
$$
\end{definition}

If $\|a\|<\infty$, we say that $a$ is \emph{bounded}.
We also need the notion of injectivity, surjectivity, and invertibility for sesquilinear forms.
\begin{definition}[Injectivity, surjectivity, and invertibility]
\label{d:injSurBij}
A sesquilinear form is \emph{injective} if there is an operator $\mathcal{R}^*:\mathcal{H}^*\to \mathcal{H}$ such that for all $v\in\mathcal{H}^*$,
\beq\label{e:R*}
a(w,\mathcal{R}^*v)=\langle w,v\rangle\quad \text{for all }w\in\mathcal{H}.
\eeq
$a$ is \emph{surjective} if there is an operator $\mathcal{R}:\mathcal{H}^*\to \mathcal{H}$ such that for all $v\in\mathcal{H}^*$,
$$
a(\mathcal{R}v,w)=\langle v,w\rangle\quad \text{for all }w\in\mathcal{H}.
$$
$a$ is \emph{invertible} if it is both surjective and injective. In this case, $\mathcal{R}^*=(\mathcal{R})^*$. 
\end{definition}

\begin{remark}\label{r:operatorP}
Recall that a sesquilinear form $a:\mathcal{H}\times\mathcal{H}\to \mathbb{C}$ defines a unique operator $P:\mathcal{H}\to \mathcal{H}^*$ such that 
$$
\langle Pu,v\rangle=a(u,v)\quad \text{ for all }u,v\in{H}.
$$
(Note that $\|a\|=\|P\|_{\mathcal{H}\to \mathcal{H}^*}$.)
The existence of $\mathcal{R}^*$ as in Definition~\ref{d:injSurBij} implies that $P^*$ has a right inverse and hence is surjective. In particular, since $\mathcal{H}$ is a Hilbert space (indeed, a reflexive Banach space), this implies that $P$ is injective. Likewise, the existence of $\mathcal{R}$ implies that $P$ is surjective.
\end{remark}

\subsection{Coercivity and C\'ea's lemma}

It is straightforward to analyse the Galerkin method when it is applied to a sesquilinear form that is coercive in the following sense.
\begin{definition}[Coercive sesquilinear form]
A sesquilinear form $a$ is \emph{coercive} if
$$
\mathfrak{C}(a):=\inf_{\substack{u\in\mathcal{H}\\u\neq 0}}\frac{ |a(u,u)|}{\|u\|_{\mathcal{H}}^2}>0.
$$
\end{definition}

By the Lax--Milgram lemma, if $a$ is bounded and coercive, then it is invertible (in the sense of Definition \ref{d:injSurBij}).
In addition, we have the following estimate for the Galerkin approximation known as C\'ea's Lemma  \cite{Ce:64}.
\begin{lemma}[C\'ea's Lemma]
\label{l:cea}
Suppose that $a:\mathcal{H}\times\mathcal{H}\to \mathbb{C}$ is a coercive sesquilinear form and $\mathcal{V}_h\subset \mathcal{H}$ is a finite dimensional subspace. Then for all $u\in\mathcal{H}$, the Galerkin approximation, $u_h$, for $u$ in $\mathcal{V}_h$ exists, is unique, and satisfies
\begin{equation}
\label{e:ceaEstimate}
\|u-u_h\|_{\mathcal{H}}\leq(\mathfrak{C}(a))^{-1}\|a\|\|(I-\Pi_h)u\|_{\mathcal{H}},
\end{equation}
where $\Pi_h:\mathcal{H}\to \mathcal{V}_h$ denotes the orthogonal projection.
\end{lemma}
\begin{proof}
Suppose that a Galerkin approximation, $u_h$, for $u$ in $\mathcal{V}_h$ exists. Then, by coercivity, Galerkin orthogonality, and the definition of $\|a\|$,
\begin{align*}
\mathfrak{C}(a)\|u-u_h\|_{\mathcal{H}}^2\leq | a(u-u_h,u-u_h)|
&= | a(u-u_h,(I-\Pi_h)u)|\\
&\leq \|a\|\|u-u_h\|_{\mathcal{H}}\|(I-\Pi_h)u\|_{\mathcal{H}},
\end{align*}
and~\eqref{e:ceaEstimate} holds.
Now, if $u=0$ then \eqref{e:ceaEstimate} implies that any Galerkin approximation is zero. Hence, since $\mathcal{V}_h$ is finite dimensional, the Galerkin approximation exists and is unique for any $u$, and the proof is complete. 
\end{proof}

\subsection{The Schatz duality argument}

The standard variational formulations of Helmholtz problems are not coercive for large $k$. However, they are coercive up to lower order terms in the following sense.

\begin{assumption}[G\aa rding inequality]
\label{a:Gaarding1}
There are $c_{\rm G}>0$ and $C_{\rm G}>0$ such that for all $u\in \cH$
$$
| a(u,u)|\geq c_{\rm G}\|u\|_{\mathcal{H}}^2-C_{\rm G}\|u\|_{\mathcal{H}_0}^2.
$$
\end{assumption}

The Schatz duality argument then gives sufficient conditions for quasioptimality in 
terms of the following adjoint approximability constant.
\begin{definition}[Adjoint approximability]
Let $\mathcal{V}_h\subset \mathcal{H}$, $a$ be an injective sesquilinear form, and $\mathcal{W}$ be a Hilbert space with $\mathcal{W}\subset \mathcal{H}^*$. The \emph{adjoint approximability constant from $\mathcal{W}$ to $\mathcal{H}$} is defined by
\beq\label{e:eta}
\eta_{_{\mathcal{W}\to \mathcal{H}}}(\mathcal{V}_h):=\|(I-\Pi_h)\mathcal{R}^*\|_{\mathcal{W}\to \mathcal{H}}.
\eeq 
\end{definition}
\begin{lemma}[Schatz duality argument]
\label{l:Schatz}
Let $a:\mathcal{H}\times \mathcal{H}\to \mathbb{C}$ be an injective sesquilinear form satisfying Assumption~\ref{a:Gaarding1}, and let $0<\epsilon<1$.  Then for all $u\in\mathcal{H}$ and $\mathcal{V}_h\subset \mathcal{H}$ finite dimensional subspaces satisfying
\begin{equation}
\label{e:schatzCondition}
C_{\rm G}\|a\|^2\big(\eta_{_{\mathcal{H}_0\to \mathcal{H}}}(\mathcal{V}_h)\big)^2\leq (1-\epsilon)c_{\rm G},
\end{equation}
the Galerkin approximation, $u_h$, for $u$ in $\mathcal{V}_h$ exists, is unique, and satisfies
\begin{equation}
\label{e:SchatzFinal}
\|u-u_h\|_{\mathcal{H}}\leq \epsilon^{-1}c_{\rm G}^{-1}\|a\|\|(I-\Pi_h)u\|_{\mathcal{H}}.
\end{equation}
\end{lemma}
\begin{proof}
Suppose that $u_h$ is a Galerkin approximation for $u$ in $\mathcal{V}_h$.

\noindent{{\bf Step 1:}}  Estimates in $\mathcal{H}_0$.
 Let $v\in\mathcal{H}_0$. By the definition of $\cR^*$ \eqref{e:R*}, Galerkin orthogonality \eqref{e:galerkinOrthogonality}, the definition of $\|a\|$, and the definition of $\eta_{_{\mathcal{H}_0\to \mathcal{H}}}$, 
\begin{align*}
|\langle u-u_h,v\rangle|=|a(u-u_h,\mathcal{R}^*v)|
&=|a(u-u_h,(I-\Pi_h)\mathcal{R}^*v)|\\
&\leq \|a\|\|u-u_h\|_{\mathcal{H}}\|(I-\Pi_h)\mathcal{R}^*v\|_{\mathcal{H}}\\
&\leq \|a\|\|u-u_h\|_{\mathcal{H}}\eta_{_{\mathcal{H}_0\to \mathcal{H}}}\|v\|_{\mathcal{H}_0}.
\end{align*}
Hence,
\begin{equation}
\label{e:H0Estimate}
\|u-u_h\|_{\mathcal{H}_0}\leq \|a\|\eta_{_{\mathcal{H}_0\to \mathcal{H}}}(\mathcal{V}_h)\|u-u_h\|_{\mathcal{H}}.
\end{equation}

\noindent{{\bf Step 2:}} Estimates in $\mathcal{H}$  using the G\aa rding inequality.
By Assumption~\ref{a:Gaarding1},~\eqref{e:galerkinOrthogonality}, and~\eqref{e:H0Estimate},
\begin{align*}
c_{\rm G}\|u-u_h\|_{\mathcal{H}}^2&\leq | a(u-u_h,u-u_h)|+ C_{\rm G}\|u-u_h\|_{\mathcal{H}_0}^2\\
&\leq  | a(u-u_h,(I-\Pi_h)u)|+ C_{\rm G}\|a\|^2\big(\eta_{_{\mathcal{H}_0\to \mathcal{H}}}(\mathcal{V}_h)\big)^2\|u-u_h\|^2_{\mathcal{H}}\\
&\leq  \|a\|\|u-u_h\|_{\mathcal{H}}\|(I-\Pi_h)u\|_{\mathcal{H}}+ C_{\rm G}\|a\|^2\big(\eta_{_{\mathcal{H}_0\to \mathcal{H}}}(\mathcal{V}_h)\big)^2\|u-u_h\|^2_{\mathcal{H}}.
\end{align*}
Therefore, using~\eqref{e:schatzCondition}, and subtracting the last term on the right-hand side to the left-hand side, we obtain
$$
\epsilon c_{\rm G}\|u-u_h\|_{\mathcal{H}}^2\leq \|a\|\|u-u_h\|_{\mathcal{H}}\|(I-\Pi_h)u\|_{\mathcal{H}},
$$
and thus~\eqref{e:SchatzFinal} holds. Now, if $u=0$, then this implies that any Galerkin approximation is zero. Hence, since $\mathcal{V}_h$ is finite dimensional, the Galerkin approximation exists and is unique for any $u$, and the proof is complete.
\end{proof}

\bre[The history of the Schatz argument]\label{rem:Schatz}
The bound \eqref{e:H0Estimate} when $a(\cdot,\cdot)$ is coercive was proved 
by Nitsche \cite{Ni:68} and 
Aubin \cite[Theorem 3.1]{Au:67},
and appears in the form \eqref{e:H0Estimate} in, e.g., \cite[Theorem 3.2.4]{Ci:02}.
Schatz was then the first to use this argument in conjunction with a G\aa rding inequality to prove quasioptimality; see \cite{Sc:74},  \cite[Lemma 4]{ScWa:96}.
The notation $\eta(\mathcal{V}_h)$ was introduced -- and the importance of adjoint approximability emphasised --  by \cite{Sa:06}.
\ere

\subsection{The elliptic projection}
\label{s:ellipticProjection}

The idea behind the elliptic projection argument, as first introduced in~\cite{FeWu:09, FeWu:11}, is to replace a  sesquilinear form $a$ satisfying a G\aa rding inequality
by one that is coercive with the cost of this replacement controlled. 
In~\cite{GS3}, the original argument is improved by substantially decreasing the cost of this replacement. To explain the argument, we consider a model problem.

\paragraph{Model problem.}

As a model problem, we consider the Galerkin solution for the problem: given $f\in H^{-1}(\mathbb{T}^d)$, find $u\in H^1(\mathbb{T}^d)$ such that 
$$
(-k^{-2}\Delta-n(x)-ik^{-1})u=f\text{ on }\mathbb{T}^d
$$
(where the imaginary term in the PDE is introduced just to ensure that the solution is unique for all $k$).
More precisely, let $\mathcal{H}_0:=L^2(\mathbb{T}^d)$ and $\mathcal{H}:=H^1_k(\mathbb{T}^d)$, and let $a_k:\mathcal{H}\times \mathcal{H}\to \mathbb{C}$ be defined by
\begin{equation}
\label{e:modelForm}
a_k(u,v):=\langle k^{-1}\nabla u,k^{-1}\nabla v\rangle-\langle n(x)u,v\rangle -ik^{-1}\langle u,v\rangle.
\end{equation}
By this definition, the following G\aa rding inequality holds:
\begin{align*}
\Re  a_k(u,u)&\geq \|k^{-1}\nabla u\|_{L^2(\mathbb{T}^d)}^2-\|n\|_{L^\infty}\|u\|_{L^2(\mathbb{T}^d)}^2\\
&\geq \|u\|_{H_k^{1}(\mathbb{T}^d)}^2-(1+\|n\|_{L^\infty})\|u\|_{L^2(\mathbb{T}^d)}^2.
\end{align*}
In particular, this implies that 
$$
\widetilde{a}_k(u,v):=a_k(u,v)+\langle \widetilde{S}u,v\rangle,\quad \text{ with }\,\widetilde{S}=(1+\|n\|_{L^\infty})I
$$
is coercive and hence C\'ea's Lemma applies. 
This is the key observation (applied to this particular model) used in~\cite[Lemma 5.1]{FeWu:09}, \cite[Lemma 4.1]{FeWu:11} (albeit there in the context of discontinuous Galerkin methods).

To produce the estimates from~\cite{GS3}, one needs to reduce the `cost' of obtaining a coercive operator. It is not possible to reduce the norm of $\|\widetilde{S}\|_{\mathcal{H}_0\to \mathcal{H}_0}$, but it \emph{is} possible to improve its regularity properties.

\begin{lemma}\mythmname{Construction of $S_k$ for the model problem with constant coefficients}
\label{l:model}
Let $a_k$ be as in~\eqref{e:modelForm} with $n\equiv 1$. 
For all $k_0> 0$ there is $c>0$ such that for  $k\geq k_0$ there exists $S_k$ such that 
\begin{equation}
\label{e:modelPerturb}
a_{S_k}(u,v):=a_k(u,v)+\langle S_ku,v\rangle\,\text{ is coercive  with }\,\mathfrak{C}(a_{S_k})\geq c.
\end{equation}
 Furthermore, for all $N>0$ there exists $C_N>0$ such that for all $k\geq k_0$,
\begin{equation}
\label{e:modelSmoothingEstimates}
\|S_k\|_{H_k^{-N}\to H_k^N}\leq C_N.
\end{equation}
\end{lemma}
\begin{proof}
Let $\chi\in C_c^\infty((-2,2);[0,1])$ with $\supp (1-\chi)\cap [-1,1]=\emptyset$ and define  $S:\cup_N H_k^{-N}(\mathbb{T}^d)\to \cap _NH_k^N(\mathbb{T}^d)$ by
$$
\widehat{S_ku}(m):=2\chi(k^{-1}|m|)\widehat{u}(m),\qquad m\in\mathbb{Z}^d. 
$$
where 
\beqs
\widehat{u}(m) := \int_{\mathbb{T}^d} e^{-i m x} u(x) d x.
\eeqs
The bound~\eqref{e:modelSmoothingEstimates} follow easily from this definition of $S_k$. 
To prove \eqref{e:modelPerturb}, we express $a_k$ using Fourier series. By Plancherel's theorem,
$$
a_k(u,v)=(2\pi)^{-d}\Big(\big\langle \widehat{u}(m)ik^{-1}m,\widehat{v}(m)ik^{-1}m\big\rangle_{\ell^2(\mathbb{Z}^d)}-(1+ik^{-1})\big\langle \widehat{u},\widehat{v}\big\rangle_{\ell^2(\mathbb{Z}^d)}\Big),
$$
so that 
$$
\Re a_{S_k}(u,u)=(2\pi)^{-d}\big\langle \big(k^{-2}|m|^2-1+2\chi(k^{-1}|m|)\big)\widehat{u},\widehat{u}\big\rangle_{\ell^2(\mathbb{Z}^d)}.
$$
Observe that since $\supp(1-\chi)\cap [-1,1]=\emptyset$, 
$$
\inf_{x\in[0,2]}
\big(x^2-1+2\chi(x)\big)>0,\qquad \inf_{x\in[2,\infty)} (1+x^2)^{-1}(x^2-1)>0,
$$
and hence
$$
\Re a_{S_k}(u,u)\geq c\big\langle (1+k^{-2}|m|^2)\widehat{u},\widehat{u}\big\rangle_{\ell^2(\mathbb{Z}^d)}=c\|u\|_{H_k^1(\mathbb{T}^d)}^2;
$$
i.e., $\mathfrak{C}(a_{S_k})\geq c$. 
\end{proof}

When the wave-speed is not constant, the Fourier analysis used in Lemma~\ref{l:model} does not apply. One route to proving the analogue  of Lemma~\ref{l:model} is to use a pseudodifferential operator $S_k$ rather than a Fourier multiplier. 

Indeed, in the variable wave-speed case, one could replace $S_k$ by 
$$
S_k:=\Op\big(2n(x)\chi(n(x)^{-1}|\xi|^2)\big)
$$
where $\chi\in C_c^\infty(\mathbb{R};[0,1])$, $\supp (1-\chi)\cap[-1,1]=\emptyset$. 
Since
$$
\sigma_\hbar(S_k) = 2n(x) \quad\text{ when } \quad |\xi|^2 \leq n(x),
$$
there exists $\mu>0$ such that
$$
\Re\sigma_\hbar(-k^{-2}\Delta -n(x)-ik^{-1} +S_k)
= |\xi|^2 - n(x) + \sigma_\hbar(S_k)
\geq \mu\langle \xi\rangle^2,
$$
so that, by 
the sharp G\aa rding inequality \eqref{e:gaarding} from Informal Theorem \ref{t:rules},
$$
\Re \big\langle (-k^{-2}\Delta -n(x)-ik^{-1} +S_k)u,u\big\rangle \geq \frac{\mu}{2}\|u\|_{H_k^1}^2.
$$
Furthermore, since $\chi\in C_c^\infty(\mathbb{R})$, 
$$
\|S_k\|_{H_k^{-N}\to H_k^N}\leq C_N
$$
by the boundedness property from Informal Theorem \ref{t:rules}.

However, we show in \S\ref{s:abstractHelmholtz} (see Theorem~\ref{t:abstractHelmholtz}) that the procedure in the proof in Lemma~\ref{l:model}, applied there for constant coefficients, can also be implemented in much more general situations by using the functional calculus of a self-adjoint operator (i.e., we do not need to use pseudodifferential operators to generalise to variable coefficients). Before doing this, we provide an analysis of the Galerkin error when one has the property \eqref{e:modelPerturb}.

\subsection{An abstract elliptic projection setup}\label{s:absep}

To demonstrate the main idea behind the $h$-FEM preasymptotic error estimates, we introduce the notion of coercivity up to a subspace $\mathcal{Z}\subset\mathcal{H}$. This notion replaces the standard G\aa rding type inequality in the elliptic projection argument. 
\begin{assumption}[Coercivity up to $\mathcal{Z}^*$]
\label{a:regularizer}
Let $a:\mathcal{H}\times \mathcal{H}\to \mathbb{C}$ be a sesquilinear form and $\mathcal{Z}\subset \mathcal{H}$. There is $S:\mathcal{Z}^*\to \mathcal{Z}$ such that 
\begin{equation}
\label{e:sCoercive}
\begin{gathered}
\mathfrak{C}(a_S)>0,\qquad a_S(u,v):=a(u,v)+\langle Su,v\rangle.
\end{gathered}
\end{equation}
\end{assumption}
\begin{remark}
Notice that Assumption~\ref{a:regularizer} implies the weaker G\aa rding-type inequality
\begin{equation}
\label{e:weakGarding}
|a(u,u)|\geq c_{\rm G}\|u\|_{\mathcal{H}}^2-C_{\rm G}\|u\|_{\mathcal{Z}^*}^2,
\end{equation}
hence the terminology ``coercivity up to $\mathcal{Z}^*$".
\end{remark}

\begin{remark}
 We show in \S~\ref{s:abstractHelmholtz} that, under elliptic-regularity-type assumptions on $a$, \eqref{e:weakGarding} holds with $\mathcal{Z}=H_k^\ell$ for some large $\ell$.  In this context, one should think of the $\|\bullet\|_{\mathcal{Z}^*}$ norm as measuring ``low-frequencies'' and hence that~\eqref{e:weakGarding} states that $a$ is coercive up to low frequencies (as discussed in \S\ref{s:how}).
\end{remark}

If $a$ satisfies Assumption~\ref{a:regularizer} and $\|a\|<\infty$, then by the Lax--Milgram Lemma, there is a unique operator $\mathcal{R}_S:\mathcal{H}^*\to \mathcal{H}$ such that 
\begin{equation}
\label{e:ellipticInverse}
a_S(\mathcal{R}_S g,w)=\langle g,w\rangle\quad \text{for all } w\in\mathcal{H},\,g\in\mathcal{H}^*.
\end{equation}

\begin{theorem}[Abstract elliptic projection argument]
\label{t:abstractEllipticProjection}
Let $a:\mathcal{H}\times \mathcal{H}\to \mathbb{C}$ be an injective sesquilinear form satisfying Assumption~\ref{a:regularizer}, let $\mathcal{V}_h$ be a finite dimensional subspace of $\mathcal{H}$, and $\epsilon>0$. Suppose that
\begin{equation}
\label{e:ellipticProjectionThreshold}
\|S\|_{\mathcal{Z}^*\to \mathcal{Z}}(\mathfrak{C}(a_S))^{-1}\|a_S\|^2\|(I-\Pi_h)\mathcal{R}_S\|_{\mathcal{Z}\to \mathcal{H}}\eta_{_{\mathcal{Z}\to \mathcal{H}}}(\mathcal{V}_h)\leq 1-\epsilon.
\end{equation}
Then for all $u\in\mathcal{H}$, the Galerkin approximation, $u_h$, for $u$ in $\mathcal{V}_h$ exists, is unique, and, for all $\mathcal{Y}\subset \mathcal{Z}$, satisfies
\begin{align}\nonumber
&\|u-u_h\|_{\mathcal{H}}\\
&\leq (\mathfrak{C}(a_S))^{-1}\Big(\|a\|+\epsilon^{-1} \sqrt{2(\mathfrak{C}(a_S))^{-1}\|S\|_{\mathcal{Z}^*\to\mathcal{Z}}} \|a_S\|^2\eta_{_{\mathcal{Z}\to \mathcal{H}}}(\mathcal{V}_h)\Big)\|(I-\Pi_h)u\|_{\mathcal{H}},\label{e:highGalerkin}\\
&\|u-u_h\|_{\mathcal{Y}^*}\nonumber\\ \nonumber
&\leq
\Big(
1 +\epsilon^{-1}
\|S\|_{\mathcal{Z}^*\to \mathcal{Z}}
(\mathfrak{C}(a_S))^{-1}\|a_S\|^2
\|(I-\Pi_h)\mathcal{R}_S\|_{\mathcal{Z}\to \mathcal{H}}\eta_{_{\mathcal{Z}\to \mathcal{H}}}(\mathcal{V}_h)\Big)
\\
&\qquad\cdot 
\|a_S\|^2(\mathfrak{C}(a_S))^{-1}
\eta_{_{\mathcal{Y}\to \mathcal{H}}}(\mathcal{V}_h)
\|(I-\Pi_h)u\|_{\mathcal{H}},
\label{e:intermediateGalerkin}
\end{align}
and
\begin{align}
\|u-u_h\|_{\mathcal{Z}^*}\leq \epsilon^{-1} \|a_S\|^2(\mathfrak{C}(a_S))^{-1}\eta_{_{\mathcal{Z}\to \mathcal{H}}}(\mathcal{V}_h)\|(I-\Pi_h)u\|_{\mathcal{H}}\label{e:lowGalerkin}.
\end{align}
\end{theorem}

\bre
If $\mathcal{Y}=\mathcal{H}^*$, then \eqref{e:intermediateGalerkin} gives a bound on $\|u-u_h\|_\cH$ complementary to that in \eqref{e:highGalerkin}. When applied to concrete Helmholtz problems, these two bounds both give \eqref{e:highGalerkinintro}; the only difference being the value of the constant $C$.
 \ere
\begin{proof} 

\

\noindent{{\bf{Step 1: }}}Estimates for $a_S$.
By Assumption~\ref{a:regularizer}, $a_S$ is coercive. Hence, by Cea's Lemma (Lemma~\ref{l:cea}),  there is $\Pi_S:\mathcal{H}^*\to \mathcal{V}_h$ such that 
$$
a_S(w_h,\Pi_Sg)=a_S(w_h,g)\, \tfa w_h\in\mathcal{V}_h,\,g\in\mathcal{H}^*.
$$
Moreover,
\begin{equation}
\label{e:coerciveAS}
\|(I-\Pi_S)w\|_{\mathcal{H}}\leq (\mathfrak{C}(a_S))^{-1}\|a_S\|\|(I-\Pi_h)w\|_{\mathcal{H}}.
\end{equation}

Next, we improve~\eqref{e:coerciveAS} when $(I-\Pi_S)w$ is measured in $\mathcal{Z}^*$ using the Schatz duality argument (Lemma \ref{l:Schatz}).  For $v\in\mathcal{Z}$,
by \eqref{e:ellipticInverse}, Galerkin orthogonality, and \eqref{e:coerciveAS},
\begin{align*}
\big|\langle v,(I-\Pi_S)w\rangle \big|&=\big|a_S(\mathcal{R}_Sv,(I-\Pi_S)w\rangle \big|\\
&=\big|a_S((I-\Pi_h)\mathcal{R}_Sv,(I-\Pi_S)w)\big|\\
&\leq \|a_S\|\|(I-\Pi_h)\mathcal{R}_S\|_{\mathcal{Z}\to \mathcal{H}}\|(I-\Pi_S)w\|_{\mathcal{H}}\|v\|_{\mathcal{Z}}\\
&\leq (\mathfrak{C}(a_S))^{-1}\|a_S\|^2\|(I-\Pi_h)\mathcal{R}_S\|_{\mathcal{Z}\to \mathcal{H}}\|(I-\Pi_h)w\|_{\mathcal{H}}\|v\|_{\mathcal{Z}}.
\end{align*}
Hence,
\begin{equation}
\label{e:aSLowNorm}
\|(I-\Pi_S)w\|_{\mathcal{Z}^*}\leq  (\mathfrak{C}(a_S))^{-1}\|a_S\|^2\|(I-\Pi_h)\mathcal{R}_S\|_{\mathcal{Z}\to \mathcal{H}}\|(I-\Pi_h)w\|_{\mathcal{H}}.
\end{equation}

\noindent{{\bf{Step 2:}}} Proof of existence, uniqueness, and~\eqref{e:lowGalerkin}.
Suppose that $u_h$ is a Galerkin approximation for $u$ in $\mathcal{V}_h$ and let $v\in\mathcal{Z}$. Then, by \eqref{e:R*}, Galerkin orthogonality \eqref{e:galerkinOrthogonality}, and 
\eqref{e:sCoercive},
\begin{equation}
\label{e:soSAAD}
\begin{aligned}
\big|\langle u-u_h,v\rangle\big|&=\big|a(u-u_h,\mathcal{R}^*v)\big|\\
&=\big|a(u-u_h,(I-\Pi_S)\mathcal{R}^*v)\big|\\
&=\big|a_S((I-\Pi_h)u,(I-\Pi_S)\mathcal{R}^*v)-\langle S(u-u_h),(I-\Pi_S)\mathcal{R}^*v\rangle\big|\\
&\leq \|a_S\|\|(I-\Pi_h)u\|_{\mathcal{H}}\|(I-\Pi_S)\mathcal{R}^*v\|_{\mathcal{H}}\\
&\qquad +\|S\|_{\mathcal{Z}^*\to \mathcal{Z}}\|u-u_h\|_{\mathcal{Z}^*}\|(I-\Pi_S)\mathcal{R}^*v\|_{\mathcal{Z}^*}.
\end{aligned}
\end{equation}
Hence, by~\eqref{e:coerciveAS} and~\eqref{e:aSLowNorm},
\begin{align*}
&\big|\langle u-u_h,v\rangle\big|\\
&\leq \|a_S\|\|(I-\Pi_h)u\|_{\mathcal{H}}(\mathfrak{C}(a_S))^{-1}|\|a_S\|\|(I-\Pi_h)\mathcal{R}^*\|_{\mathcal{Z}\to \mathcal{H}}\|v\|_{\mathcal{Z}}\\
&\,
+\|S\|_{\mathcal{Z}^*\to \mathcal{Z}}\|u-u_h\|_{\mathcal{Z}^*}(\mathfrak{C}(a_S))^{-1}\|a_S\|^2\|(I-\Pi_h)\mathcal{R}_S\|_{\mathcal{Z}\to \mathcal{H}}\|(I-\Pi_h)\mathcal{R}^*\|_{\mathcal{Z}\to\mathcal{H}}\|v\|_{\mathcal{Z}}
\end{align*}
so that the condition \eqref{e:ellipticProjectionThreshold} implis that \eqref{e:lowGalerkin} holds. Since~\eqref{e:lowGalerkin} implies uniqueness of the Galerkin solution and $\mathcal{V}_h$ is finite dimensional, the Galerkin solution exists, is unique, and satisfies~\eqref{e:lowGalerkin}

\noindent{{\bf{Step 3:}}} Proof of~\eqref{e:intermediateGalerkin}. We begin from~\eqref{e:soSAAD}, this time using~\eqref{e:lowGalerkin},~\eqref{e:coerciveAS}, and~\eqref{e:aSLowNorm}, to obtain
\begin{align*}
&\big|\langle u-u_h,v\rangle\big|\\
&\leq \Bigg(1+\epsilon^{-1}\|(I-\Pi_h)\mathcal{R}_S\|_{\mathcal{Z}\to \mathcal{H}} \eta_{_{\mathcal{Z}\to \mathcal{H}}}(\mathcal{V}_h)\|S\|_{\mathcal{Z}^*\to \mathcal{Z}}(\mathfrak{C}(a_S))^{-1}\|a_S\|^2\Bigg)
\\
&\hspace{3cm}\cdot \|a_S\|^2(\mathfrak{C}(a_S))^{-1}\eta_{_{\mathcal{Y}\to\mathcal{H}}}(\mathcal{V}_h)\|(I-\Pi_h)u\|_{\mathcal{H}}\|v\|_{\mathcal{Y}},
\end{align*}
which implies \eqref{e:intermediateGalerkin}.

\noindent{{\bf{Step 4:}}} Proof of~\eqref{e:highGalerkin}.
By \eqref{e:sCoercive} and  Galerkin orthogonality~\eqref{e:galerkinOrthogonality},
\begin{align*}
\|u-u_h\|_{\mathcal{H}}^2&\leq (\mathfrak{C}(a_S))^{-1}|a_S(u-u_h,u-u_h)|\\
&\leq  (\mathfrak{C}(a_S))^{-1}|\big(| a(u-u_h,(I-\Pi_h)u)|+|\langle u-u_h,S(u-u_h)\rangle|\big)\\
&\leq (\mathfrak{C}(a_S))^{-1}\big(\|a\|\|u-u_h\|_{\mathcal{H}}\|(I-\Pi_h)u\|_{\mathcal{H}}+\|S\|_{\mathcal{Z}^*\to\mathcal{Z}}\|u-u_h\|_{\mathcal{Z}^*}^2\big)\\
&\leq  \big(\delta \|u-u_h\|^2_{\mathcal{H}}+\tfrac{1}{4}\|a\|^2(\mathfrak{C}(a_S))^{-2}\delta^{-1}\|(I-\Pi_h)u\|^2_{\mathcal{H}}\\
&\qquad+(\mathfrak{C}(a_S))^{-1}\|S\|_{\mathcal{Z}^*\to\mathcal{Z}} \epsilon^{-2} \|a_S\|^4\|(I-\Pi_h)u\|^2_{\mathcal{H}}(\mathfrak{C}(a_S))^{-2}\eta^2_{_{\mathcal{Z}\to\mathcal{H}}}(\mathcal{V}_h)\big).
\end{align*}
Optimizing in $\delta$, and using $\sqrt{a^2+b^2}\leq a+b$, we obtain~\eqref{e:highGalerkin}.
\end{proof}

\subsection{Ellipticity for abstract Helmholtz operators}
\label{s:abstractHelmholtz}

In this section we consider a \emph{family} of sesquilinear operators $a_k:\mathcal{H}\times \mathcal{H}\to \mathbb{C}$, $k\geq 0$, where the norms on all spaces may implicitly depend on $k$. 
\begin{assumption}[Uniform boundedness]
\label{a:continuity}
For all $k_0>0$ 
$$
\sup_{k>k_0}\|a_k\|<\infty.
$$
\end{assumption}

\begin{assumption}[Uniform G\aa rding inequality]
\label{a:Gaarding}
For all $k_0>0$ there are $c_{\rm G}>0$ and $C_{\rm G}>0$ such that for $k>k_0$, 
$$
\Re a(u,u)\geq c_{\rm G}\|u\|_{\mathcal{H}}^2-C_{\rm G}\|u\|_{\mathcal{H}_0}^2.
$$
\end{assumption}

\begin{assumption}[Abstract elliptic regularity for $a_k$ and $\Re a_k$]
\label{a:ellipticRegularity}
Let $\mathcal{Z}_0:=\mathcal{H}_0$, $\mathcal{Z}_1:=\mathcal{H}$, and $\mathcal{Z}_j\subset \mathcal{Z}_{j-1}$ for $j=2,3,\dots, \ell+1$ such that $\mathcal{Z}_j$ is dense in $\mathcal{Z}_{j-1}$. For all $k_0 \geq 0$, there are $C_{\rm ell}>0$ such that for all $u\in\mathcal{H}$ and $k>k_0$
\begin{equation}
\label{e:fullReg}
\|u\|_{\mathcal{Z}_j}\leq C_{\rm ell}\Big(\|u\|_{\mathcal{H}_0}+\sup_{v\in\mathcal{H},\, \|v\|_{(\mathcal{Z}_{j-2})^*}=1}\big|a_k(u,v)\big|\big),\quad j=2,\dots, \ell+1,
\end{equation}
and
\begin{equation}
\label{e:realReg}
\|u\|_{\mathcal{Z}_j}\leq C_{\rm ell}\bigg(\|u\|_{\mathcal{H}_0}+\sup_{v\in\mathcal{H},\, \|v\|_{(\mathcal{Z}_{j-2})^*}=1}\Big|\big(\Re a_k\big)(u,v)\Big|\bigg),\quad j=2,\dots, \ell+1,
\end{equation}
where $(\Re a_k)(u,v)= \frac{1}{2}(a_k(u,v)+\overline{a_k(v,u)}).$ 
\end{assumption}
\begin{remark}
In Assumpion \ref{a:ellipticRegularity}, we mean that if either of the right-hand sides of the bounds \eqref{e:fullReg} and \eqref{e:realReg}
are finite, then $u\in\mathcal{Z}_j$ and the bounds hold.
\end{remark}

We now state the main theorem of this section.
\begin{theorem}
\label{t:abstractHelmholtz}
Suppose that $a_k:\mathcal{H}\times\mathcal{H}\to \mathbb{C}$ is a sesquilinear form satisfying Assumptions~\ref{a:continuity}, \ref{a:Gaarding}, and~\ref{a:ellipticRegularity} for some $\ell \in\mathbb{Z}^+$. Then for all $\epsilon>0$ and $k_0>0$, there are $C>0$ and $S_k$ such that, for $k>k_0$, $a_k$ satisfies Assumption~\ref{a:regularizer} with $S=S_k$ and for all $k>k_0$, 
\begin{align}
  \mathfrak{C}(a_{k,S_k})&\geq (1-\epsilon)c_{\rm G},\label{e:helmholtzCoercive}\\
  \|S_k\|_{(\mathcal{Z}_{\ell+1})^*\to \mathcal{Z}_{\ell+1}}&\leq C,\label{e:sIsSmoothing}\\
\|\mathcal{R}_{S_k}\|_{\mathcal{Z}_{j-1}\to \mathcal{Z}_{j+1}}&\leq C,\qquad j=0,1,\dots,\ell, \label{e:rIsReasonable}
\end{align}
where $\mathcal{R}_{S_k}$ is the operator defined by~\eqref{e:ellipticInverse} with $a=a_k$ and $S=S_k$.  
\end{theorem}
\begin{remark}
In fact, the proof shows that there is $\chi\in C_c^\infty(\mathbb{R})$ such that $S_k=\chi(\mathcal{P})$, where $\mathcal{P}$ is the self-adjoint operator generated by $\Re a_k$. This operator is also used in the proofs of the results \ref{R2} and \ref{R4}.
\end{remark}
\begin{proof}

\

\noindent {\bf{Step 1:}} Construction of $S_k$.
By Assumptions~\ref{a:continuity} and~\ref{a:Gaarding} and the Friedrichs extension theorem (see, e.g., \cite[Theorem VIII.15, Page 278]{ReSi:80},  \cite[Theorem 12.24, Page 360]{Gr:08}, with  \cite{Fr:34} the original paper), there is a (potentially unbounded) self-adjoint operator $\mathcal{P}:\mathcal{H}_0\to \mathcal{H}_0$ with domain
\beq\label{e:domain}
\mathcal{D}(\mathcal{P})=\bigg\{ u\in\mathcal{H}\,:\, \sup_{\substack{v\in\mathcal{H}\\v\neq 0}}\frac{|\Re a_k(u,v)|}{\|v\|_{\mathcal{H}_0}}<\infty\bigg\}
\eeq
such that 
$$
\Re a_k(u,v)=\langle \mathcal{P}u,v\rangle\, \tfor u\in\mathcal{D}(\mathcal{P}).
$$

By the spectral theorem for unbounded operators 
(see, e.g., \cite[Chapter VII]{ReSi:80}), there is a unitary operator $U:\mathcal{H}_0\to L^2(\mathbb{R};\mu)$ for some Borel measure $\mu$ such that for $u\in\mathcal{D}(\mathcal{P})$, 
$$
\mathcal{P}u=U^*M_{\lambda} Uu,
$$
where, for $f:\mathbb{R}\to \mathbb{C}$ $\mu$ measurable, $(M_fw)(\lambda):=f(\lambda)w(\lambda)$. In addition,
 $u\in \mathcal{D}(P)$ if and only if $M_\lambda Uu\in L^2(\mu)$. 
 
Therefore, for $u\in\mathcal{H}$, 
$$
\Re a_k(u,u)= \langle \lambda Uu,Uu\rangle_{L^2(d\mu)}\geq c_{\rm G}\|u\|_{\mathcal{H}}^2-C_{\rm G}\|u\|^2_{\mathcal{H}_0}\geq (c_{\rm G}-C_{\rm G})\|u\|_{\mathcal{H}_0}^2,
$$
which implies that $\supp \mu \subset[c_{\rm G}-C_{\rm G},\infty)$. 

Let $M\geq C_{\rm G}$, to be chosen large depending on $\epsilon$, and $\chi\in C_c^\infty(\mathbb{R};[0,1])$ with $\supp (1-\chi)\cap [-C_{\rm G},M]=\emptyset$ and define 
$$
S_k u:= C_{\rm G}\chi(\mathcal{P}),
$$ 
where for $f\in C^0(\mathbb{R},d\mu)$, we define $f(\mathcal{P}):=U^*M_{f}U$.
Note that when $f$ is real-valued, $f(\mathcal P)$ is then self-adjoint.
Thus, by the definition of $\chi$ and the support properties of $\mu$,
\begin{align*}
\Re a_{k,S_k}(u,u)&= \big\langle (\lambda +C_{\rm G}\chi(\lambda))Uu,Uu\big\rangle_{L^2(d\mu)} \\
&\geq  \Big\langle \frac{\lambda +C_{\rm G}\chi(\lambda)}{\lambda+C_{\rm G}} (\lambda+C_{\rm G})Uu,Uu\Big\rangle_{L^2(d\mu)}\\
&\geq \inf_{\lambda \geq M}\frac{\lambda}{\lambda+C_{\rm G}}\big\langle (\lambda+C_{\rm G})Uu,Uu\big\rangle_{L^2(d\mu)}\\
&\geq \inf_{\lambda \geq M}\frac{\lambda}{\lambda+C_{\rm G}} c_{\rm G}\|u\|_{\mathcal{H}}^2.
\end{align*}
For $M$ large enough, this implies~\eqref{e:helmholtzCoercive}.

\noindent {\bf{Step 2:}} Estimates for $S_k$.
To prove~\eqref{e:sIsSmoothing}, we claim that for any $\psi\in C_c^\infty(\mathbb{R})$, there is $C>0$ such that 
\begin{equation}
\label{e:compactBound}
\|\psi(\mathcal{P})\|_{( \mathcal{Z}_{\ell+1})^*\to \mathcal{Z}_{\ell+1}}\leq C;
\end{equation} 
since $S=C_{\rm G}\chi(\mathcal{P})$ for some $\chi\in C_c^\infty$, this implies~\eqref{e:sIsSmoothing}.

To prove~\eqref{e:compactBound} we use Assumption~\ref{a:ellipticRegularity}. First, observe that for all $m \geq 0$ and any $\psi\in C_c^\infty(\mathbb{R})$, 
$$
\|\lambda^m\psi(\lambda)\|_{L^\infty}\leq C. 
$$
Hence, 
$$
\|\mathcal{P}^m\psi(\mathcal{P})\|_{\mathcal{H}_0\to \mathcal{H}_0}\leq C. 
$$
Now, by~\eqref{e:realReg}, for $2\leq j\leq \ell+1$, 
$$
\|\psi(\mathcal{P})\|_{\mathcal{H}_0\to \mathcal{Z}_j}\leq C_{\rm ell}\Big(\|\psi(\mathcal{P})\|_{\mathcal{H}_0\to \mathcal{H}_0}+\|\mathcal{P}\psi(\mathcal{P})\|_{\mathcal{H}_0\to \mathcal{Z}_{j-2}}\Big).
$$
Hence, by induction,
$$
\|\psi(\mathcal{P})\|_{\mathcal{H}_0\to \mathcal{Z}_{\ell+1}}\leq C_{\ell}\sum_{j=0}^{\lceil (\ell+1)/2\rceil}\|\mathcal{P}^j\psi(\mathcal{P})\|_{\mathcal{H}_0\to \mathcal{H}_0}\leq C.
$$
Now, let $\tilde{\psi}\in C_c^\infty(\mathbb{R})$ with $\supp \psi\cap \supp (1-\tilde{\psi})=\emptyset$. Then, since $\tilde{\psi}(\mathcal{P})$ is self-adjoint, 
$$
\|\tilde{\psi}(\mathcal{P})\|_{( \mathcal{Z}_{\ell+1})^*\to \mathcal{H}_0}\leq C
$$
and hence, since $\psi(\mathcal{P})=\psi(\mathcal{P})\tilde{\psi}(\mathcal{P})$, 
\begin{align*}
\|\psi(\mathcal{P})\|_{( \mathcal{Z}_{\ell+1})^*\to \mathcal{Z}_{\ell+1}}&=\|\psi(\mathcal{P})\tilde{\psi}(\mathcal{P)}\|_{( \mathcal{Z}_{\ell+1})^*\to \mathcal{Z}_{\ell+1}}\\
&\leq \|\psi(\mathcal{P})\|_{\mathcal{H}_0\to \mathcal{Z}_{\ell+1}}\|\tilde{\psi}(\mathcal{P})\|_{( \mathcal{Z}_{\ell+1})^*\to \mathcal{H}_0}\leq C,
\end{align*}
which proves~\eqref{e:compactBound}.

\noindent {\bf{Step 3:}} Estimates for $\mathcal{R}_{S_k}$.
By \eqref{e:helmholtzCoercive} and the Lax--Milgram lemma, $\mathcal{R}_{S_k}$ defined by 
$$
a_{k,S_k}(\mathcal{R}_{S_k}g,v)=\langle g,v\rangle, \quad g\in \mathcal{H}^*, 
$$
satisfies
$$
\|\mathcal{R}_{S_k}g\|_{\mathcal{H}}\leq 
\mathfrak{C}(a_{k,S_k})^{-1}
\|g\|_{\mathcal{H}^*}.
$$
In particular,  for $g\in\mathcal{Z}_{j-2}$, $j=2,\dots,\ell+1$, $v\in\mathcal{H}$, 
\begin{align*}
|a_k(\mathcal{R}_{S_k}g,v)|&=\big|a_{k,S_k}(\mathcal{R}_{S_k}g,v)-\langle S_k\mathcal{R}_{S_k}g,v\rangle\big|\\
&=\big|\langle g,v\rangle-\langle S_kg,v\rangle\big|\\
&\leq \big(\|g\|_{\mathcal{Z}_{j-2}}+\|S_k\mathcal{R}_{S_k}g\|_{\cZ_{j-2}}\big)\|v\|_{(\cZ_{j-2})^*}\\
&\leq \big(\|g\|_{\mathcal{Z}_{j-2}}+C\|\mathcal{R}_{S_k}g\|_{\mathcal{H}_0}\big)\|v\|_{(\cZ_{j-2})^*}\\
&\leq \big(\|g\|_{\mathcal{Z}_{j-2}}+C\|g\|_{\mathcal{H}^*}\big)\|v\|_{(\cZ_{j-2})^*}\\
&\leq C\|g\|_{\mathcal{Z}_{j-2}}\|v\|_{(\cZ_{j-2})^*},
\end{align*}
which, together with~\eqref{e:fullReg} implies~\eqref{e:rIsReasonable}.
\end{proof}

Next, we bound the adjoint-approximability constant under the following additional elliptic-regularity assumption.
\begin{assumption}[Abstract adjoint elliptic regularity for $a_k$]
\label{a:adjointEllipticRegularity}
Let $\mathcal{Z}_0:=\mathcal{H}_0$, $\mathcal{Z}_1:=\mathcal{H}$, and $\mathcal{Z}_j\subset \mathcal{Z}_{j-1}$ for $j=2,3,\dots, \ell+1$ such that $\mathcal{Z}_j$ is dense in $\mathcal{Z}_{j-1}$. For all $k_0 \geq 0$, there are $C_{\rm ell}>0$ such that for all $u\in\mathcal{H}$, and $k>k_0$
\begin{equation}
\label{e:adjointReg}
\|u\|_{\mathcal{Z}_j}\leq C_{\rm ell}\Big(\|u\|_{\mathcal{H}_0}+\sup_{v\in\mathcal{H},\, \|v\|_{(\mathcal{Z}_{j-2})^*}=1}\big|a_k(v,u)\big|\big),\quad j=2,\dots, \ell+1.
\end{equation}
\end{assumption}
\begin{lemma}[Bound on $\eta$ \eqref{e:eta}]
\label{l:adjointApproximability}
Suppose that $a_k:\mathcal{H}\times\mathcal{H}\to \mathbb{C}$ is a sesquilinear form satisfying Assumptions~\ref{a:continuity}, \ref{a:Gaarding}, \ref{a:ellipticRegularity}, and~\ref{a:adjointEllipticRegularity} for some $\ell \in\mathbb{Z}^+$. Let $k_0>0$. Then there is $C>0$ such that for all $k>k_0$ with $a_k$ invertible, all $\mathcal{V}_h\subset \mathcal{H}$, and 
$m\in\{-1,\ldots,\ell\}$,
\beq\label{e:boundoneta}
\eta_{_{\mathcal{Z}_{m+1}\to \mathcal{H}}}(\mathcal{V}_h)
\leq C\Big(\|(I-\Pi_h)\|_{\mathcal{Z}_{\max\{m+3,\ell+1\}}\to \mathcal{H}}+\|(I-\Pi_h)\|_{\mathcal{Z}_{\ell+1}\to \mathcal{H}}\|\mathcal{R}\|_{\mathcal{Z}_{\ell+1}\to (\mathcal{Z}_{m+1})^*}\Big)
\eeq
and 
$$
\eta_{_{\mathcal{H}^*\to \mathcal{H}}}(\mathcal{V}_h)\leq C\Big(1+\|(I-\Pi_h)\|_{\mathcal{Z}_{\ell+1}\to \mathcal{H}}\|\mathcal{R}\|_{\mathcal{Z}_{\ell+1}\to \cH}\Big).
$$
\end{lemma}
\begin{proof}
Let  $\epsilon>0$ and let $S_k$ be the operator constructed in Theorem~\ref{t:abstractHelmholtz}. 
With $P_k$ the operator associated to $a_k$ via Remark \ref{r:operatorP},
$$
(P_k^*+S_k)\cR^* = I + S_k \cR^*,$$
so that
\begin{equation}
\label{e:inverseFormula}
\begin{aligned}
\mathcal{R}^*=\mathcal{R}^*_{S_k}+\mathcal{R}^*_{S_k}S_k\mathcal{R}^*.
\end{aligned}
\end{equation}
Arguing as in Step 3 of the proof of Theorem~\ref{t:abstractHelmholtz}, but using Assumption~\ref{a:adjointEllipticRegularity} (instead of Assumption \ref{a:ellipticRegularity}), yields 
$$
\|\mathcal{R}^*_{S_k}\|_{\mathcal{Z}_j\to \mathcal{Z}_{j+2}}\leq C,\qquad 0\leq j\leq \ell-1.
$$
Using this in~\eqref{e:inverseFormula}, we obtain
\begin{align*}
&\eta_{_{\mathcal{Z}_{m+1}\to \mathcal{H}}}(\mathcal{V}_h)=\|(I-\Pi_h)\mathcal{R}^*\|_{\mathcal{Z}_{m+1}\to \mathcal{H}}\\
&\qquad\leq C\|(I-\Pi_h)\|_{\mathcal{Z}_{\max\{m+3,\ell+1\}}\to \mathcal{H}}
+
C\|(I-\Pi_h)\|_{\mathcal{Z}_{\ell+1}\to \mathcal{H}}
\|\mathcal{R}^*\|_{\mathcal{Z}_{m+1}\to (\mathcal{Z}_{\ell+1})^*}\\
&\qquad= 
C\|(I-\Pi_h)\|_{\mathcal{Z}_{\max\{m+3,\ell+1\}}\to \mathcal{H}}
+
C\|(I-\Pi_h)\|_{\mathcal{Z}_{\ell+1}\to \mathcal{H}}
\|\mathcal{R}\|_{\mathcal{Z}_{\ell+1}\to (\mathcal{Z}_{m+1})^*}
\end{align*}
as claimed. The proof of the bound on $\eta_{_{\cH^* \to \cH}}(\cV_h)$ is very similar, using that $\| \mathcal{R}^*_{S_k}\|_{\cH^*\to \cH}\leq C$ by coercivity of $a_{k,S_k}$ \eqref{e:helmholtzCoercive}.
\end{proof}

\begin{lemma}[Oscillatory data implies oscillatory solution]\label{l:osci}
    Suppose that $a_k:\mathcal{H}\times\mathcal{H}\to \mathbb{C}$ is a sesquilinear form satisfying Assumption \ref{a:ellipticRegularity} for some $\ell \in\mathbb{Z}^+$. 
   If $f \in \cZ_{\ell-1}$ satisfies
   \beqs
\| f\|_{\cZ_{\ell-1}}\leq C_{\rm osc} \|f\|_{\cH^*},
   \eeqs
   then the solution $u\in \cH$ to $a_k(u,v) = \langle f, v\rangle$ for all $v\in \cH$ satisfies
   \beqs
\| u\|_{\cZ_\ell} \leq C_{\rm ell}\big( 1 + C_{\rm osc} \| a_k\|\big) \|u \|_\cH.
   \eeqs
\end{lemma}

\bpf
By \eqref{e:fullReg} and the fact that $a(u,v)=\langle g, v\rangle$ for all $v\in \cH$,
\begin{align*}
\| u\|_{\cZ_{\ell+1}}\leq C_{\rm ell} \big( \| u\|_{\cH_0} + \| f\|_{\cZ_{\ell-1}}\big)
& \leq C_{\rm ell}\big( \| u\|_{\cH_0} + C_{\rm osc} \| f\|_{\cH^*}\big),\\
& \leq C_{\rm ell} \big( \| u\|_{\cH_0} + C_{\rm osc} \| a_k\| \| u\|_\cH \big),\\
& \leq C_{\rm ell} \big( 1 + C_{\rm osc} \| a_k\| \big) \| u\|_{\cH}.
\end{align*}
\epf

\begin{corollary}[Abstract preasymptotic error bound for the Galerkin method]
\label{c:preasymptotic}
Suppose that $a_k: \cH\times \cH\to \Com, k\geq 0$ is a family of injective sesquilinear forms satisfying Assumptions \ref{a:continuity}, \ref{a:Gaarding}, \ref{a:ellipticRegularity}, and \ref{a:adjointEllipticRegularity} for some $\ell\in\mathbb{Z}^+$
and spaces $\mathcal{Z}_0:=\mathcal{H}_0$, $\mathcal{Z}_1:=\mathcal{H}$, and $\mathcal{Z}_j\subset \mathcal{Z}_{j-1}$ for $j=2,3,\dots, \ell+1$ such that $\mathcal{Z}_j$ is dense in $\mathcal{Z}_{j-1}$.

There exists $c,C>0$ such that the following is true.
Let $\cV_h$ be a finite-dimensional subspace of $\cH$.
If 
\beq\label{e:threshold}
\| (I- \Pi_h)\|_{\cZ_{\ell+1}\to \cH}^2 \big( 1 + \|\cR\|_{\cZ_{\ell+1}\to (\cZ_{\ell+1})^*}\big) \leq c,
\eeq
then for all $u\in \cH$ the Galerkin approximation, $u_h$, for $u$ in $\cV_h$ exists, is unique, and, for all $m\in \{-1,\ldots,\ell\}$, 
\beq\label{e:highGalerkinfinal}
\|u-u_h\|_{\cH}\leq C \Big(1+ \| (I- \Pi_h)\|_{\cZ_{\ell+1}\to \cH}
\big(1+ \|\cR\|_{\cZ_{\ell+1}\to (\cZ_{\ell+1})^*}\big)\Big) \| (I-\Pi_h) u\|_{\cH}
\eeq
and 
\begin{align}\nonumber
&\|u-u_h\|_{(\cZ_{m+1})^*}\\ \nonumber
&\leq C 
\Big(
\| (I- \Pi_h)\|_{\cZ_{\max\{m+3,\ell+1\}}\to \cH} + \| (I- \Pi_h)\|_{\cZ_{\ell+1}\to \cH}
\big(1+\|\cR\|_{\cZ_{\ell+1}\to (\cZ_{m+1})^*}\big)
\Big)
\\
&\hspace{4cm}\cdot\| (I-\Pi_h) u\|_{\cH}.
\label{e:intermediateGalerkinfinal}
\end{align}
Furthermore, given $C_{\rm osc}>0$ there exists $C'>0$ such that, with $P:\cH\to \cH^*$ the operator defined by $a_k$, if 
 \beqs
 \N{Pu}_{\cZ_{\ell-1}}\leq C_{\rm osc} \N{Pu}_{\cH^*}
  \eeqs
then 
\beq\label{e:relerrorfinal}
\frac{
\N{u-u_h}_{\cH}
}{
\N{u}_{\cH}
}
\leq C' 
\Big(1+ \| (I- \Pi_h)\|_{\cZ_{\ell+1}\to \cH}
\big(1+ \|\cR\|_{\cZ_{\ell+1}\to (\cZ_{\ell+1})^*}\big)\Big)\| (I- \Pi_h)\|_{\cZ_{\ell+1}\to \cH}.
\eeq
\end{corollary}

\bpf
The combination of Theorem \ref{t:abstractEllipticProjection} (applied with $\cZ =\cZ_{\ell+1}$ and $\cY = \cZ_{m+1}$, $m\in\{-1,\ldots, \ell\}$) and Theorem \ref{t:abstractHelmholtz} imply that there exist $c,C>0$ such that if 
\beq\label{e:thresholdpf}
\| (I- \Pi_h)\|_{\cZ_{\ell+1}\to \cH} \,\eta_{_{\cZ_{\ell+1}\to \cH}}(\cV_h)\leq c
\eeq
then 
\beq\label{e:highGalerkinpf}
\|u-u_h\|_{\cH}\leq C \big(1+ \eta_{_{\cZ_{\ell+1}\to \cH}}(\cV_h)\big) \| (I-\Pi_h) u\|_{\cH},
\eeq
\beq\label{e:intermediateGalerkinpf}
\|u-u_h\|_{(\cZ_{m+1})^*}\leq 
C \eta_{_{\cZ_{m+1}\to \cH}}(\cV_h) 
\| (I-\Pi_h) u\|_{\cH},
\eeq
and 
\beq\label{e:lowGalerkinpf}
\|u-u_h\|_{(\cZ_{\ell+1})^*}\leq C \eta_{_{\cZ_{\ell+1}\to \cH}}(\cV_h) 
\| (I-\Pi_h) u\|_{\cH},
\eeq
where \eqref{e:highGalerkinpf}, \eqref{e:intermediateGalerkinpf}, and \eqref{e:lowGalerkinpf} come from  \eqref{e:highGalerkin}, \eqref{e:intermediateGalerkin}, and \eqref{e:lowGalerkin} respectively.  
Observe that the bound \eqref{e:lowGalerkinpf} is contained in the bound \eqref{e:intermediateGalerkinpf} when $m=\ell$.
The results \eqref{e:highGalerkinfinal} and \eqref{e:intermediateGalerkinfinal} then follow from combining 
\eqref{e:highGalerkinpf} and \eqref{e:intermediateGalerkinpf} with the bound on $\eta_{_{\cZ_{m+1}\to \cH}}(\cV_h)$ from \eqref{e:boundoneta}.

The relative error bound \eqref{e:relerrorfinal} then follows from combining Lemma \ref{l:osci} with \eqref{e:highGalerkinpf}.
\epf

\subsection{Application to the $h$-FEM in scattering problems}

The limiting factor in applying the abstract results of Theorems \ref{t:abstractEllipticProjection} and \ref{t:abstractHelmholtz} (combined in Corollary \ref{c:preasymptotic}) is checking the abstract elliptic-regularity assumptions (Assumptions \ref{a:ellipticRegularity}, \ref{a:adjointEllipticRegularity}).
We now do this for 
the Helmholtz PML problem of Definition \ref{d:pmlTruncation}, but we emphasise that these elliptic-regularity assumptions are also satisfied for Helmholtz problems with \emph{either} impedance boundary conditions, \emph{or} when the exact Dirichlet-to-Neumann map is imposed on the boundary of a large ball; see \cite[\S3]{GS3}.

Let
\beq\label{e:cH_PML}
\cH:=H^1_0(\Omega)  \quad \text{ or } \quad \big\{ v \in H^1(\Omega) \,:\, v=0 \text{ on } \Gamma_{\tr}\big\},
\eeq
depending on whether $\Gamma_-$ has Dirichlet or Neumann boundary conditions, and let
\beq\label{e:cZj}
\cZ_j := H^j(\Omega\setminus \Gamma_{\rm p})\cap \cH
\eeq
with norm
\beqs
\| v\|_{\cZ_j}:= \| v\|_{H^j_k(\Omega\setminus \Gamma_{\rm p})}.
\eeqs

\begin{theorem}\mythmname{Rigorous statement of \ref{R1}:~$k$-explicit error analysis of the $h$-FEM}\label{t:R1}
Let $p\in \mathbb{Z}^+$, and let $a$ be the sesquilinear form defining an $H^{p+1}$ Helmholtz problem approximated using a $C^{p,1}$ PML with $C^{p,1}$ truncation boundary (in the sense of Definitions \ref{d:helmholtzProblem} and \ref{d:pmlTruncation}). Given $k_0, \e>0$, and $\Upsilon>0$ there exist $c,C>0$  such that for all affine-conforming
$C^{p+1}$ simplicial triangulations $\mathcal{T}$ (in the sense of Definition \ref{d:Crtriang}) with constant $\Upsilon>0$ such that 
$$
\partial \Omega\cup \Gamma_{\rm p}\subset \bigcup_{T\in\mathcal{T}}\partial T, 
$$
 $k\geq k_0$, and $\e \leq \theta\leq \pi/2-\e$ the following is true. 
If 
$h:=h(\mathcal{T})$ satisfies 
\beq\label{e:R1threshold}
(hk)^{2p}\rho(k) \leq c,
\eeq
then for all $u\in H_k^1$, the Galerkin approximation, $u_h$, for $u$ in $\mathcal{P}_{\mathcal{T}}^p$ (defined by \eqref{e:polySpace}) exists, is unique, and satisfies 
\begin{align}\label{e:H1boundnew}
\N{u-u_h}_{H^{1}_k(\Omega)}&\leq C \Big(1 + \rho(k)(hk)^p \Big) \min_{v_h \in \mathcal{P}_{\mathcal{T}}^p} \N{u-v_h}_{H^1_k(\Omega)}
\end{align}
and, for $\newm\in\{1,\ldots,p\}$,
\beq\label{e:lowfreqbound}
\N{u-u_h}_{(\cZ_{\newm-1})^*}\leq C \Big((hk)^\newm + \rho(k)(hk)^p 
\Big) \min_{v_h \in \mathcal{P}_{\mathcal{T}}^p} \N{u-v_h}_{H^1_k(\Omega)}.
\eeq
Furthermore, given $C_{\rm osc}>0$ there exists $C'>0$ such that, with $P:H_k^1(\Omega)\to (H_k^1(\Omega))^*$ be the operator defined by $a$, if 
 \beq\label{e:koscil}
 \N{Pu}_{H^{p-1}_{k}(\Omega\setminus\Gamma_{\rm p})}\leq C_{\rm osc} \N{Pu}_{(H_k^{1}(\Omega))^*}
  \eeq
then 
\beq\label{e:rel_error}
\frac{\N{u-u_h}_{H^1_k(\Omega)}}
{
\N{u}_{H^1_k(\Omega)}
}
\leq C' \Big(1  + \rho(k)(hk)^p \Big)(hk)^p.
\eeq
\end{theorem}

\begin{remark}\label{r:negativeNorms}
Note that $H_{0}^{j}(\Omega)\subset \cZ_j$ and for all $k_0>0$ there is $C>0$ such that for $k>k_0$ for $u\in H_0^j$, 
$$
\|u\|_{H_k^j}\leq C\|u\|_{\mathcal{Z}_j}.
$$
Therefore, $\|u\|_{H_k^{-j}(\Omega)}\leq C\|u\|_{(\cZ_j)^*}$ and the bound \eqref{e:lowfreqbound} implies bounds on $\|u-u_h\|_{H^{1-n}_k(\Omega)}$ (consistent with \eqref{e:lowGalerkin}).
\end{remark}

To prove Theorem \ref{t:R1}, we first recap local elliptic-regularity results near a Dirichlet or Neumann boundary, and for the transmission problem.

Given coefficients $A, b,$ and $n$, let 
\beq\label{e:L}
L u := -k^{-2} \divergence (A \nabla u ) - k^{-1} b\cdot \nabla u - n u.
\eeq

\begin{assumption}\label{ass:McLean}
For all $x\in \Omega$, 
$A_{j\ell}(x)= A_{\ell j}(x)$ and 
\beqs
\Re \big( A(x) \xi, \xi\big)_2=\Re\sum_{j=1}^d\sum_{\ell=1}^d A_{j\ell}(x) \xi_\ell \overline{\xi_j} \geq c|\xi|^2 \quad\tfa \xi \in \Com^d.
\eeqs
\end{assumption}

\begin{theorem}[Local elliptic regularity near a Dirichlet or Neumann boundary]\label{t:erDN}
Let $\Omega$ be a Lipschitz domain and let $G_1, G_2$ be open subsets of $\Rea^d$ with $G_1 \Subset G_2$ and $G_1 \cap \partial \Omega \neq \emptyset$. Let 
\beq\label{e:Omegaj}
\Omega_j := G_j \cap \Omega, \,\,j=1,2, \quad\tand \quad \Gamma_2:= G_2 \cap \partial\Omega.
\eeq
Suppose that $A$ satisfies Assumption \ref{ass:McLean}, $A, b, n \in C^{m,1}(\overline{\Omega_2})$, and $\Gamma_2 \in C^{m+1,1}$ for some $m\in \mathbb{N}$.
Given $k_0>0$, there exists $C>0$ such that if $k\geq k_0$, 
 $u\in H^1(\Omega_2)$, $L u \in H^m(\Omega_2)$, and either $u=0$ or $\partial_{n,A}u=0$ on $\Gamma_2$, then 
\beq\label{e:erDN}
\N{u}_{H^{m+2}_k(\Omega_1)} \leq C \big( \N{u}_{H^1_k(\Omega_2)} + \N{L u}_{H^m_k(\Omega_2)}\big).
\eeq
\end{theorem}

\bpf
The bound in unweighted norms
\beqs
\N{u}_{H^{m+2}(\Omega_1)} \leq C \big( \N{u}_{H^1(\Omega_2)} + \N{L u}_{H^m(\Omega_2)}\big)
\eeqs
is proved in e.g., \cite[Theorems 4.7 and 4.18]{Mc:00}; the proof of the bound 
\beq\label{e:erDN2}
\N{u}_{H^{m+2}_k(\Omega_1)} \leq C \big( \N{u}_{H^1_k(\Omega_2)} + \N{L u}_{H^m_k(\Omega_2)}\big)
\eeq
is then very similar. By considering the pairing $\langle L u, \chi u\rangle_\Omega$, where $\chi \in C^\infty_c(\Rea^d)$ and $\chi \equiv 1$ on $\Omega_1$, integrating by parts, and using Assumption \ref{ass:McLean}, one obtains that 
\beqs
\N{u}_{H^1_k(\Omega_1)} \leq C \big( \N{u}_{L^2(\Omega_2)} + \N{L u }_{L^2(\Omega_2)}\big).
\eeqs
This bound then allows us to replace $\|u\|_{H^1_k(\Omega_2)}$ in \eqref{e:erDN} by $\|u\|_{L^2(\Omega_2)}$ (by introducing an intermediate domain between $\Omega_1$ and $\Omega_2$), and hence obtain the result \eqref{e:erDN}.
\epf

\begin{theorem}[Local elliptic regularity for the transmission problem]\label{t:erT}
Let $\Omega_{\rm in}$ be a Lipschitz domain, and let $\Omega_{\rm out}:= \Rea^d \setminus \overline{\Omega_{\rm in}}$. 
Let $G_1, G_2$ be open subsets of $\Rea^d$ with $G_1 \Subset G_2$ and $G_1 \cap \partial \Omega_{\rm in} \neq \emptyset$. Let 
\beqs
\Omega_{\rm in/out, j}:= G_j \cap \Omega_{\rm in/out},\quad j=1,2, \quad\tand\quad \Gamma_2 := G_{2} \cap \partial \Omega_{\rm in}.
\eeqs
Suppose that $A$ satisfies Assumption \ref{ass:McLean}, $A|_{\Omega_{\rm in/out, 2}}, 
b|_{{\Omega_{\rm in/out, 2}}},
n|_{{\Omega_{\rm in/out, 2}}}\in C^{m,1}(\overline{\Omega_{\rm in/out, 2})}$, and $\Gamma_2 \in C^{m+1,1}$ for some $m\in \mathbb{N}$.
Given $k_0>0$, there exists $C>0$ such that if $k\geq k_0$, $u_{\rm in/out}\in H^1(\Omega_{\rm in/out})$, $L u \in H^m(\Omega_{\rm in/out, 2})$, and
$\uin =\uout$ and $\partial_{n,A}u_{\rm in} = \partial_{n,A} u_{\rm out}$ on $\Gamma_2$.
then 
\begin{align}\nonumber
&\N{u_{\rm in}}_{H^{m+2}_k(\Omega_{\rm in, 1})}
+\N{u_{\rm out}}_{H^{m+2}_k(\Omega_{\rm out, 1})}\\
&\quad \leq C \Big( \N{\uin}_{H^1_k(\Omega_{\rm in, 2})}
+  \N{\uout}_{H^1_k(\Omega_{\rm out, 2})}
 + \N{L\uin}_{H^m_k(\Omega_{\rm in, 2})}
  + \N{L\uout}_{H^m_k(\Omega_{\rm out, 2})}\Big).\label{e:erT}
\end{align}
\end{theorem}
\bpf
The analogue of \eqref{e:erT} in unweighted norms and with $H^1$ norms on the right-hand side (instead of $L^2$ norms) is, e.g.,
\cite[Theorem 5.2.1(i)]{CoDaNi:10}, \cite[Theorems 4.7 and 4.20]{Mc:00}. The proof of the bound in the weighted norms is very similar, and then the $H^1_k$ norms can be replaced by $L^2$ norms by considering $\langle L u, \chi u\rangle_\Omega$, where $\chi \in C^\infty_c(\Rea^d)$ and $\chi \equiv 1$ on $\Omega_{\rm out, 1} \cap \Omega_{\rm in, 1}$ (similar to in the proof of Theorem \ref{t:erDN}).
\epf

\bpf[Proof of Theorem \ref{t:R1}]
The plan is to apply Corollary \ref{c:preasymptotic} (with $\ell=p$) to the sesquilinear form corresponding to the PML problem (i.e., defined by \eqref{e:def_ak}); we start by 
verifying the assumptions of these results, namely that the sesquilinear form is continuous (Assumption \ref{a:continuity}), satisfies a G\aa rding inequality (Assumption \ref{a:Gaarding}), and that the abstract elliptic-regularity assumptions (Assumptions \ref{a:ellipticRegularity} and \ref{a:adjointEllipticRegularity}) hold. (Without loss of generality we can assume that $a$ is injective, since otherwise $\rho=\infty$ and the theorem is trivially true.)

Let $\cH$ and $\cZ_j$ be defined by \eqref{e:cH_PML} and \eqref{e:cZj}, respectively. 

The sesquilinear form \eqref{e:def_ak} of the PML problem  is continuous, uniformly in $k$. 
By Lemmas \ref{l:PMLGard1} and \ref{l:PMLGard2}, there exists $\omega\in\Rea$ such that $e^{i \omega}a$ satisfies the G\aa rding inequality of Assumption \ref{a:Gaarding} (with $\omega=0$ for the first PML formulation in \S\ref{s:PML}). 
Since, with $\omega\in \Rea$, 
\beqs
a_k(u-u_h, v_h)=0 \quad\text{ if and only if }\quad 
e^{i \omega} a_k (u-u_h, v_h)=0,
\eeqs
without loss of generality we can assume that $\omega=0$.


In both PML formulations, the matrix $A_\theta$ (defined by \eqref{e:firstPML} and \eqref{e:secondPML}, respectively) is symmetric.
Therefore Assumption \ref{ass:McLean} is satisfied for both $P_\theta$ and $P_\theta^*$ by Lemmas \ref{l:PMLGard1} and \ref{l:PMLGard2} (again noting that we can take $\omega=0$ without loss of generality). 

Definitions \ref{d:helmholtzProblem} and \ref{d:pmlTruncation} imply that $\partial\Omega$ is $C^{p,1}$, the coefficients of the PDE are $C^{p-1,1}$ and that $\Gamma_{\rm p}$ is $ C^{p,1}$. 

The elliptic-regularity assumptions for $a_k$ \eqref{e:fullReg} and its adjoint \eqref{e:adjointReg} with $\ell=p$ then follow by combining  Theorem \ref{t:erDN} (used near $\partial\Omega_-$ and $\Gamma_{\tr}$) and Theorem \ref{t:erT} (used near $\Gamma_{\rm p}$), both with $m=p-1$. 
The operator associated with the sesquilinear form $\Re a_k$ is 
\begin{align*}
&\left(\frac{P + P^*}{2}\right)u  \\
&\quad=-k^{-2} \divergence \left(\frac{A_\theta+\overline{A_\theta}^T}{2} \nabla u \right) 
 + k^{-2}\bigg( \frac{\nabla \cdot( b_\theta u) - b_\theta \cdot\nabla u}{2}\bigg)
 - \left(\frac{n_\theta+ \overline{n_\theta}}{2}\right) u.
\end{align*}
Since $A_\theta^T=A_\theta$ (as noted above), the sesquilinear form $\Re a_k$ also satisfies Assumption \ref{ass:McLean}, and the elliptic-regularity assumption for $\Re a_k$ \eqref{e:realReg} holds by the combination of Theorems \ref{t:erDN} and \ref{t:erT} used above.

Finally, since $c$ in both Lemma \ref{l:PMLGard1} and \ref{l:PMLGard2} is independent of $\e \leq \theta\leq \pi/2-\e$ and since the coefficients $A_\theta, b_\theta$, and $n_\theta$ can be bounded uniformly in $\e \leq \theta\leq \pi/2-\e$, the constant $C_{\rm \ell}$ in Assumptions \ref{a:ellipticRegularity} and \ref{a:adjointEllipticRegularity}, and the constants $c_{\rm G}$ and $C_{\rm G}$ in Assumption \ref{a:Gaarding} can be taken to be independent of $\e \leq \theta\leq \pi/2-\e$. Furthermore the bound on $\|a_k\|$ in Assumption \ref{a:continuity} can also be taken to be independent of $\e \leq \theta\leq \pi/2-\e$.

We apply Corollary \ref{c:preasymptotic} with $\cV_h = \mathcal{P}^p_{\mathcal{T}}$. The assumptions on $\mathcal{T}$ and Theorem 
\ref{t:approxHighLowReg} 
imply that there exists $C>0$ such that for all $m\in \{0,\ldots,p\}$
\beqs
\| (I-\Pi_h)\|_{\cZ_{m+1}\to \cH}\leq C (hk)^m.
\eeqs
Furthermore, for $m\in \{-1,\ldots,\ell\}$,
\beqs
\| \cR\|_{\cZ_{\ell+1} \to (\cZ_{m+1})^*} \leq 
\| \cR\|_{\cH_0\to \cH_0}
=\| \cR\|_{L^2(\Omega)\to L^2(\Omega)}\leq C \rho,
\eeqs
by Theorem \ref{t:blackBoxResolve}.
By these last two displayed inequalities, the bound \eqref{e:highGalerkinfinal} then becomes \eqref{e:H1boundnew}.
The bound \eqref{e:lowfreqbound} with $n\in \{1,\ldots,p\}$ then follows from \eqref{e:intermediateGalerkinfinal} applied with $m=n-2$ (so that $m\in \{-1,\ldots,p-2\}$ corresponds to $n\in\{1,\ldots,p\}$).

\epf

\bre[Analysis of the $h$-FEM including geometric error]
Theorem \ref{t:R1} assumes that the triangulations resolve 
both $\partial\Omega$ and any boundaries across which $A$ and $n$ jump. The paper \cite{ChSp:25} considers the case these interfaces are approximated to order $q$ by the mesh (so that $q=1$ for straight elements and $q=p$ for isoparametric elements) and proves that if
\beq\label{e:geoThreshold}
\rho(k)(hk)^{2p}
+ \rho(k) h^q \,\,\text{ is sufficiently small, }
\eeq
(compare to \eqref{e:R1threshold}) 
then the Galerkin solution $u_h$ exists, is unique, and satisfies, for $k$-oscillatory data (in the sense of \eqref{e:koscil}),
\beq\label{e:geoR1}
\frac{
\N{u-u_h}_{H^1_k(\Omega)}
}{
\N{u}_{H^1_k(\Omega)}
}
\leq C \bigg[ 
 \Big( 1 +  \rho(k)(hk)^p\Big) (hk)^p + \rho(k) h^q
\bigg]
\eeq
(compare to \eqref{e:rel_error}); 
this result is obtained via proving an abstract analogue of Corollary \ref{c:preasymptotic} allowing for a variational crime \cite[Theorem 3.2]{ChSp:25}. In \eqref{e:geoR1}, near $\partial \Omega$, the error $\|u-u_h\|_{H^1_k(\Omega)}$ is measured on the subset of $\Omega$ where both $u$ and $u_h$ are well-defined; when $A$ and $n$ are discontinuous, defining precisely where the error $\|u-u_h\|_{H^1_k(\Omega)}$ is measured is more involved (see \cite[Theorem 4.11]{ChSp:25}).
Observe that \eqref{e:geoThreshold} and \eqref{e:geoR1} imply that when the problem is nontrapping (i.e., $\rho(k)\leq C$), even with straight elements ($q=1$), 
 in the limit $k\to\infty$ with $h$ chosen to control the pollution error (via \eqref{e:R1threshold}), the geometric
error is smaller than the pollution error.
\ere

\bre
The analogue of Theorem \ref{t:R1} for the PML problem for the time-harmonic Maxwell equations is proved in \cite{CGS1}.
\ere

\section{The $hp$-FEM does not suffer from the pollution effect (\ref{R2})}\label{s:R2}

In this section we analyse the Galerkin error for the  $hp$-FEM. In this case, we require strong regularity assumptions on the coefficients and boundaries in the Helmholtz problem, as well as on the mesh defining the piecewise-polynomial spaces. However, under these strong assumptions, one obtains $k$-uniform quasioptimality without pollution for piecewise-polynomial spaces with high degree.

We saw in \S\ref{s:R1} that both the Schatz duality argument (Lemma~\ref{l:Schatz}) and the elliptic-projection argument from Theorem~\ref{t:abstractEllipticProjection} give sufficient conditions for quasioptimality. We use the former here; using the latter 
would lead to slightly better constants.

The results in this section build on the $hp$-FEM analyses in \cite{LSW3, LSW4} (for truncation by the exact Dirichlet-to-Neumann map on a ball) and \cite{GLSW1} (for PML truncation), although the details of the arguments are different, drawing on the results of \cite{GS3} (appearing in \S\ref{s:R1}) and \cite{AGS2}.

\subsection{An abstract $hp$-FEM convergence result}

The main goal of this section is the following abstract estimate on the adjoint approximability constant. 
Let \begin{equation}\label{e:weightednorm1}
\|u\|_{H_k^p(\Omega)}^2:= \sum_{\ell=0}^p \frac{1}{(\ell!)^2}|u|_{H_k^\ell(\Omega)}^2,
\eeq
where
\beqs
|u|_{H_k^{p}(\Omega)}^2:=\sum_{\substack{|\alpha|=p\\\alpha\in\mathbb{N}^d}}\frac{p!}{\alpha!}\|(k^{-1}D_{x_1})^{\alpha_1}(k^{-1}D_{x_2})^{\alpha_2}\dots (k^{-1}D_{x_d})^{\alpha_d}u\|_{L^2(\Omega)}^2
\end{equation*}
(observe that this definition is the same as in \eqref{e:norm}).
The key point about the constants in the seminorm definition is that 
\beq\label{e:Fouriernorm1}
|u|_{H_k^{p}(\mathbb{R}^d)}^2:=\frac{1}{(2\pi)^d}\big\||\xi|^{p}\mathcal{F}_k(u)(\xi)\big\|^2_{L^2(\mathbb{R}^d)},
\eeq
and similarly for $|\cdot|_{H^p(\Rea^d)}$.
We let $|\cdot|_{H^p(\Omega)}
:=|\cdot|_{H^p_1(\Omega)}
$
and 
$\|\cdot\|_{H^p(\Omega)}
:=\|\cdot\|_{H^p_1(\Omega)}
$.

\begin{theorem}\mythmname{$hp$-FEM quasioptimality under an abstract assumption on low frequencies near interfaces}
\label{t:hpFEMAbstract}
Suppose $a$ is the sesquilinear form defining an $H^\infty$ Helmholtz  problem approximated by a $C^\infty$ PML with $C^{1,1}$ truncation boundary (in the sense of Definitions \ref{d:helmholtzProblem} and \ref{d:pmlTruncation}). 
There is $\psi \in C_c^\infty(\mathbb{R})$ such that for all $W\subset \overline{\Omega}$ open in the subspace topology of $\overline{\Omega}$ with $\Gamma_{\rm p}\cup \Gamma_-\Subset W$, 
given $M,N,k_0>0$ and $\Upsilon>0$ there exist $C,h_0>0$ such that 
for all affine-conforming
$C^\omega$ simplicial triangulations $\mathcal{T}$ with constant $\Upsilon>0$ (in the sense of Definition~\ref{d:Crtriang}) such that 
\beq\label{e:meshGammap}
\partial \Omega\cup \Gamma_{\rm p}\subset \bigcup_{T\in\mathcal{T}}\partial T, 
\eeq
$p\geq 1$, $0<h(\mathcal{T})<h_0$, and $k\geq k_0$, 
\begin{equation}
\label{e:hpApproximability}
\begin{aligned}
&\eta_{L^2\to H_k^1}(\mathcal{P}^p_{\mathcal{T}})
\\
&\leq C\Bigg[\frac{hk}{p}+
\Big(\frac{hk}{p}\Big)^Mk^{-N} \rho(k) \\
&\quad +C^{p}(hk)^{p}k^{-p-1}
\Bigg(
\max\Big\{\frac{k}{p+1},1\Big\}^{p+1}
+\|\psi(\mathcal{P})\|_{L^2\to H^{p+1}(W\setminus \Gamma_{\rm p})\cap H_k^1}
\Bigg)\rho(k)\Bigg],
\end{aligned}
\end{equation}
where 
 $h:= h(\mathcal{T})$.
\end{theorem}

\paragraph{The ideas of the proof.}

The estimate~\eqref{e:hpApproximability} is obtained by splitting the adjoint solution operator $\mathcal{R}^*$ into its action on high and low frequencies
and using the estimates from Theorem~\ref{t:approxHighLowReg}:
\begin{align}
\label{e:pExplicitApproximation2}
\| I-\mathcal{C}_{\mathcal{T}}^{p}\|_{H_k^{m}(\Omega\setminus \Gamma_{\rm p})\cap H_k^1\to H_k^1}&\leq C_{m} \Big(\frac{hk}{p}\Big)^{m-1},\\
\label{e:pExplicitApproximation}
\| I-\mathcal{C}_{\mathcal{T}}^{p}\|_{H^{p+1}(\Omega\setminus \Gamma_{\rm p})\cap H_k^1\to H_k^1}&\leq C^{p+1}  (hk)^{p}k^{-p-1},
\end{align}
where $\mathcal{C}_{\mathcal{T}}^p$ is a 
triangulation-preserving operator, in the sense of Definition \ref{d:triangulationPreserving} below.

 The main idea is that $P$ is semiclassically elliptic at high frequencies and hence $\mathcal{R}^*:L^2\to H_k^2$ is bounded uniformly in $k$ there. On the other hand, functions with frequency $\lesssim k$ are highly regular (with quantitative bounds on the derivatives) and thus should be well approximated by piecewise polynomials.
 
 The key question is then:~how should one define high and low frequencies to make this heuristic argument rigorous? Below, we use two different answers to this question.
\begin{itemize}
\item[Answer 1.] The first answer (and best answer when it is available) is that a function with low frequencies is any function whose (semiclassical) Fourier transform is compactly supported; i.e., there is $\psi\in C_c^\infty(\mathbb{R})$ such that $u=\psi(|\hbar D|)u$, where $D:= -i \partial$. For $\psi\in C_c^\infty(\mathbb{R})$, by \eqref{e:Fouriernorm1} with $k=1$, 
\beq\label{e:Fouriermult}
(\psi(|\hbar D|)u)(x):= \mathcal{F}^{-1}_\hbar\big( \psi(|\bullet|)(\mathcal{F}u)(\bullet)\big)(x),
\eeq
satisfies
\beqs
| \psi(|\hbar D|) u|_{H^\ell(\Rea^d)}\leq (Ck)^\ell \| u\|_{L^2(\Rea^d)},
\eeqs
so that by the definition of $\|\cdot\|_{H^p(\Rea^d)}$ \begin{align}\label{e:analyticityBound0}
\|\psi(|\hbar D|)\|_{L^2\to H^{p+1}}\leq \sqrt{\sum_{\ell=0}^{p+1} \frac{1}{(\ell!)^2}(Ck)^{2\ell}}
&\leq \sum_{\ell=0}^{p+1} \frac{1}{\ell!}(Ck)^{\ell}.
\end{align}
We show below (in the proof of Theorem~\ref{t:hpFEMAbstract}) that, by 
Stirling's approximation,
\beq
\|\psi(|\hbar D|)\|_{L^2\to H^{p+1}}
\leq 
C^{p+1}
\max\Big\{\frac{k}{p+1},1\Big\}^{p+1}.
\label{e:analyticityBound1a}
\eeq
Combining this with \eqref{e:pExplicitApproximation}, we obtain that
\begin{align*}
\| (I-\Pi_h^p) \psi(|\hbar D|) \|_{L^2\to H^1_k}
& \leq C^{p+1}  (hk)^pk^{-p-1}
\max\Big\{\frac{k}{p+1},1\Big\}^{p+1},\\
& \leq C^{p+1}  \Big(\frac{hk}{p+1}\Big)^p \quad\text{ when $p+1\leq k$};
\end{align*}
i.e., low frequencies are very well approximated by piecewise polynomials with both $h$ and $hk/p$ small and $p$ large.
The main issue to resolve is then:~is there $\psi\in C_c^\infty(\mathbb{R})$ such that $(1-\psi(|\hbar D|))\mathcal{R}^*$ is uniformly bounded? When there are no boundaries or interfaces, semiclassical ellipticity (together  with a bound of the form $\rho(k)\leq Ck^L$ for some $L$) provides a positive answer to this question, even when all coefficients are only smooth (with this first shown in \cite{LSW3}). However, in the presence of boundaries or interfaces, one cannot work with Fourier multipliers and one needs a second definition of low frequencies.
\item[Answer 2.] When the Fourier transform is unavailable, it is natural to define low frequencies in terms of the operator $P$, or more precisely its real part, $\mathcal{P}$. In this case, we say a function is low frequency if there is $\psi\in C_c^\infty(\mathbb{R})$ such that $u=\psi(\mathcal{P})u$; i.e., $u$ can be written as a sum $\sum_{j=1}^{N_k}a_ju_{\lambda_j}$, of eigenfunctions of $\mathcal{P}$ with eigenvalue, $\lambda_j\leq C$. In this case, it is relatively straightforward to check that there is $\psi\in C_c^\infty(\mathbb{R})$ such that $(1-\psi(\mathcal{P}))\mathcal{R}^*$ is uniformly bounded.
\eit
To prove Theorem \ref{t:hpFEMAbstract}, we combine Answers 1 and 2, using Answer 1 away from all boundaries and interfaces, and Answer 2 in a neighborhood of all non-PML boundaries or interfaces.
Finally, near the truncation boundary, we use that for $\chi\in \overline{C^\infty}(\Omega)$ supported inside the PML, $\chi \mathcal{R}^*:L^2\to H_k^2$ is uniformly bounded. 

\subsection{Application of Theorem \ref{t:hpFEMAbstract}}

To apply Theorem \ref{t:hpFEMAbstract} to concrete settings, one must estimate 
$\|\psi(\mathcal{P})\|_{L^2\to H^{p+1}(W)}$.
We show below in Theorem~\ref{t:hpFEMEllipticRegularity} that if 
$\Gamma_{\rm p}=\emptyset$, $\Gamma_-\in C^\omega$, and $A$ and $n$ are analytic in a neighbourhood of $\Gamma_-$, then
for $\psi\in C_c^\infty(\mathbb{R})$
\beq\label{e:analyticityBound2}
\|\psi(\mathcal{P})\|_{L^2\to H^{p+1}(W)}\leq \Big(C
\max\Big\{\frac{k}{p+1},1\Big\}\Big)^{p+1};
\eeq
i.e., the same bound as in \eqref{e:analyticityBound1a}.
Combining~\eqref{e:analyticityBound2} with~\eqref{e:hpApproximability}, we obtain the rigorous statement of Informal Theorem \ref{t:informalR2}.
\begin{theorem}
\mythmname{Rigorous statement of \ref{R2}:~$hp$-FEM quasioptimality for analytic obstacles with variable coefficients}
\label{t:concreteHPFEM}
Suppose $a$ is the sesquilinear form defining an $H^\infty$ Helmholtz  problem approximated using a $C^\infty$ PML with $C^{1,1}$ truncation boundary (in the sense of Definitions \ref{d:helmholtzProblem} and \ref{d:pmlTruncation}) and, in addition, suppose that $\Gamma_{\rm p}=\emptyset$, $\Gamma_-\in C^\omega$, a Dirichlet condition is imposed on $\Gamma_-$, and there is a neighborhood, $U$ of $\Gamma_-$ such that $A,n\in C^\omega(U)$.

(i) Given $M,N,k_0>0$ and $\Upsilon>0$ there exists $C,h_0>0$ such that 
for all affine-conforming
$C^\omega$ simplicial triangulations $\mathcal{T}$ with constant $\Upsilon>0$ (in the sense of Definition~\ref{d:AffineConforming}) such that 
\beq\label{e:goodtri}
\partial \Omega \subset \bigcup_{T\in\mathcal{T}}\partial T, 
\eeq
$p\geq 1$, $0<h(\mathcal{T})<h_0$, and $k\geq k_0$, 
\begin{align*}
&\eta_{L^2\to H_k^1}(\mathcal{P}_{\mathcal{T}}^p)\\
&
\qquad
\leq C\Bigg[\frac{hk}{p}+
\Big(\frac{hk}{p}\Big)^Mk^{-N} \rho(k)  +C^{p}(hk)^{p}k^{-p-1}
\max\Big\{\frac{k}{p+1},1\Big\}^{p+1}
\rho(k)\Bigg],
\end{align*}
where $h:=h(\mathcal{T})$. 

(ii) Given $k_0>0$ $\e>0$, and $\Upsilon>0$, for any $\mathcal{J}$ and $L>0$ such that 
$$
\sup\{k^{-L}\rho (k)\,:\, k\in (k_0,\infty)\setminus \mathcal{J}\}<\infty
$$
there exist $C,c,h_0>0$ such that 
for all $k\geq k_0$, $k\notin \mathcal{J}$,
$C^\omega$ simplicial triangulations $\mathcal{T}$ with constant $\Upsilon>0$ such that \eqref{e:goodtri} holds, 
$$
0<h(\mathcal{T})<h_0,\qquad \frac{h(\mathcal{T})k}{p}<c,\quad\tand\quad p\geq 1+\e \log k,
$$
then for all $u\in H_k^1(\Omega)$, the Galerkin approximation $u_h$ in $\mathcal{P}_{\mathcal{T}}^p$ satisfies
$$
\|u-u_h\|_{H_k^1(\Omega)}\leq C\min_{w_h\in\mathcal{P}_{\mathcal{T}}^p}\|u-w_h\|_{H_k^1(\Omega)}.
$$
\end{theorem}

\subsection{Recap of pseudolocality results from \cite{AGS2}}\label{s:pseudo}

To prove Theorem~\ref{t:hpFEMAbstract}, we need results on pseudolocality of functions of a self-adjoint operator. To do this, 
with $P_\theta$ defined by  \eqref{e:PML1} and $\mathcal{P}=\Re P_\theta$,
we fix $\psi_0\in C_c^\infty(\Rea)$ 
such that there is $c>0$ such that with $S_k:=\psi_0(\mathcal{P})$
$$
\Re \big\langle \big(P_\theta+S_k\big)v,v\big\rangle \geq c\|v\|_{H_k^1}^2,\quad\tfa  v\in H_k^1(\Omega).
$$
As in \eqref{e:ellipticInverse}, let $\mathcal{R}_{S_k}:= (P_\theta+S_k)^{-1}$.

Before we state the pseudolocality results, we introduce the following notation.
\begin{definition}
Let $X,Y$ be Banach spaces, $f:(0,\infty)\to (0,\infty)$, and $B_k:X\to Y$ a $k$-dependent operator. 
$$
B_k=O(f(k))_{X\to Y}
$$
if for all $k_0>0$ there is $C>0$ such that 
$$
\|B_k\|_{X\to Y}\leq Cf(k),\qquad k>k_0.
$$
Furthermore, 
$$
B_k=O(k^{-\infty})_{X\to Y}
$$
if for all $k_0>0$ and $N>0$ there is $C>0$ such that 
$$
B_k=O(k^{-N})_{X\to Y}.
$$
\end{definition}

\begin{lemma}\mythmname{Spatial Pseudolocality \cite[Theorems 5.27 and 6.2]{AGS2}}
\label{l:spatialPseudolocalFinal}
Suppose that $\psi\in \mathcal{S}(\mathbb{R})$ and 
$\chi_1,\chi_2 \in \overline{C^\infty}(\Omega)$  with 
$\supp \chi_1\cap \supp \chi_2=\emptyset$. Then for all $N>0$, 
\begin{align*}
\chi_1\psi(\mathcal{P})\chi_2&=O(k^{-\infty})_{(H_k^N(\Omega\setminus \Gamma_{\rm p}))^*\to H_k^N(\Omega\setminus \Gamma_{\rm p})},\\
\chi_1\mathcal{R}_{S_k}\chi_2&=O(k^{-\infty})
_{
 (H_k^N(\Omega\setminus \Gamma_{\rm p}))^*
\to H_k^N(\Omega\setminus \Gamma_{\rm p})},\\ 
\chi_1\mathcal{R}_{S_k}^*\chi_2&=O(k^{-\infty})_{( H_k^N(\Omega\setminus \Gamma_{\rm p}))^*\to H_k^N(\Omega\setminus \Gamma_{\rm p})}.
\end{align*}
\end{lemma}

\begin{lemma}
\mythmname{Pseudolocality in Frequency
\cite[Theorem 5.37 and Lemmas 6.7 and 6.9]{AGS2}}
\label{l:frequencyPseudolocalFinal}
Let $\psi_1,\psi_2\in C^\infty(\mathbb{R})$ with either $\psi_1\in C_c^\infty(\mathbb{R})$ or $\psi_2\in C_c^\infty(\mathbb{R})$, and $\supp\psi_1\cap \supp\psi_2=\emptyset$. Suppose that $\chi \in C^\infty(\Omega)$ and $\supp \nabla \chi\cap (\partial \Omega\cup \Gamma_{\rm p}) =\emptyset$, and $\supp \chi\cap \Gamma_{\tr}=\emptyset$. Then for all $N>0$, 
\begin{align*}
\psi_1(\mathcal{P})\chi\psi_2(\mathcal{P})&=O(k^{-\infty})_{\mathcal{D}(\mathcal{P}^{-N})\to \mathcal{D}(\mathcal{P}^{N})},\\
\psi_1(\mathcal{P})\chi \mathcal{R}_{S_k}\chi\psi_2(\mathcal{P})&=O(k^{-\infty})_{\mathcal{D}(\mathcal{P}^{-N})\to \mathcal{D}(\mathcal{P}^{N})},\\
\psi_1(\mathcal{P})\chi \mathcal{R}_{S_k}^*\chi\psi_2(\mathcal{P})&=O(k^{-\infty})_{\mathcal{D}(\mathcal{P}^{-N})\to \mathcal{D}(\mathcal{P}^{N})}.
\end{align*}
\end{lemma}


\paragraph{The ideas behind the proofs of Lemmas  \ref{l:spatialPseudolocalFinal}
 and \ref{l:frequencyPseudolocalFinal}.}
 
\cite[Theorems 5.27 and 5.37]{AGS2} respectively provide abstract conditions under which $S_k$, $\mathcal{R}_{S_k}$, and $\mathcal{R}_{S_k}^*$ are pseudolocal in space and abstract under which an operator $B$ acts pseudolocally in frequency. These conditions boil down to estimating repeated commutators of either spatial cutoffs with $P_{\theta}$ and $P_{\theta}^*$ (i.e., commutators of the form $[\chi,[\chi,[\chi,\cdots[\chi,P_\theta]]]$) in the case of spatial pseudolocality, and estimating repeated commutators of $P_\theta$ and $P_\theta^*$ with $B$ (i.e., commutators of the form $[P_\theta,[P_\theta,[P_\theta,\cdots[P_\theta,B]]]$) in the case of frequency pseudolocality.  These abstract conditions are checked in~\cite[Lemma 6.2]{AGS2} for spatial pseudolocality and~\cite[Lemmas 6.7 and 6.9]{AGS2} for frequency pseudolocality.

To see why commutators arise in the case of spatial pseudolocality of $S_k:=\psi_0(\cP)$, first notice that if $\supp \chi_1\cap \supp \chi_2=\emptyset$ then there is $\chi\in \overline{C^\infty}(\Omega)$ such that $\supp \chi \cap \supp \chi_2=\emptyset$ and $\supp (1-\chi)\cap \supp \chi_1=\emptyset$ and hence
$$
\chi_1S_k\chi_2= \chi_1\chi S_k\chi_2=\chi_1[\chi,S_k]\chi_2.
$$
Repeating this argument, we have
$$
\chi_1S_k\chi_2=\chi_1[\chi,[\chi,[\chi,\dots[\chi,S_k]]]]\chi_2.
$$
Next, recall from the Helffer--Sj\"ostrand formula (see, e.g., \cite[Theorem 8.1]{DiSj:99}) that 
	for all $f \in \mathcal{S}(\Rea)$, then
	\begin{equation*}
		f(\cP) = \frac{1}{\pi}\int_{\mathbb{C}} \partial_{\bar{z}}\tilde{f}(z) (\cP - z)^{-1} d m_{\mathbb{C}}(z),
	\end{equation*}
	where $dm_{\mathbb{C}}(x + iy)= dx\,dy$,  $\partial_{\bar{z}}:=\frac{1}{2}(\partial_{x}+i\partial_y)$ and $\tilde{f}$ is an almost analytic extension of $f$ in the sense of~\cite[Theorem 3.6]{Zw:12}.

Then, since
\begin{equation}
\label{e:commuteInverse}
[\chi,(\mathcal{P}-z)^{-1}]=-(\mathcal{P}-z)^{-1}[\chi,\mathcal{P}](\mathcal{P}-z)^{-1},
\end{equation}
and $\mathcal{P}=\frac{1}{2}(P_\theta+P_\theta^*)$, iterated commutators of $\chi$ with $P_\theta$ and $P_\theta^*$ appear naturally.

Analogously, frequency pseudolocality involves repeated commutators of a frequency cutoff, $\psi(\mathcal{P})$, with the operator $B$ and hence, again via the Helffer--Sj\"ostrand formula, can be reduced to iterated commutators with $(\mathcal{P}-z)^{-1}$ and finally to iterated commutators with $\mathcal{P}$ via~\eqref{e:commuteInverse} with $\chi$ replaced by $B$.

\subsubsection{Bound on the high frequencies of the solution operator}

We now use the pseudolocality results of the previous subsection (i.e., Lemmas ~\ref{l:spatialPseudolocalFinal} and~\ref{l:frequencyPseudolocalFinal}) to 
show that 
the (adjoint) solution operator behaves well on high frequencies as measured by $\psi(\mathcal{P})$.

\begin{lemma}[High frequency elliptic-type estimate]\label{l:trainLate1}
Given $\psi_0 \in C_c^\infty$, 
let $\psi\in C_c^\infty(\mathbb{R})$ be such that $\supp (1-\psi)\cap \supp \psi_0=\emptyset$ and $\chi \in \overline{C^\infty}(\Omega)$ such that $\supp \chi\cap \Gamma_{\rm tr}=\emptyset$, $\supp \nabla \chi \cap (\Gamma_{\rm p}\cup \Gamma_-
)=\emptyset$. Then, for all $N>0$ $k_0>0$ there is $C>0$ such that for all $k>k_0$
\begin{equation}
\label{e:highElliptic1}
\|(1-\psi(\mathcal{P}))\chi\mathcal{R}^*\|_{L^2\to H_k^2(\Omega\setminus \Gamma_{\rm p})\cup H_k^1(\Omega)}\leq C\big(1+k^{-N}\|\mathcal{R}^*\|_{L^2\to L^2}\big).
\end{equation}
\end{lemma}
\begin{proof}
By \eqref{e:inverseFormula},
$$
\mathcal{R}^*=\mathcal{R}_{S_k}^*+\mathcal{R}_{S_k}^*\psi_0(\mathcal{P})\mathcal{R}^*,
$$
so that 
\beq\label{e:slough0}
(1-\psi(\mathcal{P}))\chi\mathcal{R}^*= (1-\psi(\mathcal{P}))\chi \mathcal{R}_{S_k}^*+(1-\psi(\mathcal{P}))\chi\mathcal{R}_{S_k}^*\psi_0(\mathcal{P})\mathcal{R}^*.
\eeq
Let $\tilde{\chi}\in \overline{C^\infty}(\Omega)$ with $\supp \nabla \tilde{\chi}\cap (\Gamma_{\rm p}\cup \partial\Omega)=\emptyset$, $\supp (1-\tilde{\chi})\cap \supp \chi=\emptyset$, and $\supp \tilde{\chi}\cap \Gamma_{\tr}=\emptyset$. Then, by Lemma \ref{l:spatialPseudolocalFinal}, 
\begin{align}\nonumber
&(1-\psi(\mathcal{P}))\chi\mathcal{R}_{S_k}^*\psi_0(\mathcal{P})\mathcal{R}^*\\
&
=(1-\psi(\mathcal{P}))\chi\mathcal{R}_{S_k}^*\tilde{\chi}\psi_0(\mathcal{P})\mathcal{R}^* +(1-\psi(\mathcal{P}))O(k^{-\infty})_{L^2\to H_k^N(\Omega\setminus \Gamma_{\rm p})}
\psi_0(\mathcal{P})\mathcal{R}^*.
\label{e:slough1}
\end{align}
Now let $\psi_1\in C_c^\infty(\mathbb{R})$ be such that $\supp \psi_0\cap \supp(1-\psi_1)=\emptyset$ and $\supp \psi_1\cap \supp(1-\psi)=\emptyset$. 
By Lemma \ref{l:frequencyPseudolocalFinal} and then Lemma \ref{l:spatialPseudolocalFinal},
\begin{align}\nonumber
(1-\psi(\mathcal{P}))\chi\mathcal{R}_{S_k}^*\tilde{\chi}\psi_0(\mathcal{P})\mathcal{R}^*
&=
(1-\psi(\mathcal{P}))\chi(1-\psi_1(\mathcal{P}))\mathcal{R}_{S_k}^*\tilde{\chi}\psi_0(\mathcal{P})\mathcal{R}^* \\ \nonumber
&\qquad
+O(k^{-\infty})_{\mathcal{D}(\mathcal{P}^{-N})\to \mathcal{D}(\mathcal{P}^N)}
\mathcal{R}_{S_k}^*\tilde{\chi}\psi_0(\mathcal{P})\mathcal{R}^*\\ \nonumber
&=(1-\psi(\mathcal{P}))\chi(1-\psi_1(\mathcal{P}))\tilde{\chi}\mathcal{R}_{S_k}^*\tilde{\chi}\psi_0(\mathcal{P})\mathcal{R}^* \\
&\,\quad
+O(k^{-\infty})_{L^2\to H_k^N(\Omega\setminus \Gamma_{\rm p})}
\mathcal{R}_{S_k}^*\tilde{\chi}\psi_0(\mathcal{P})\mathcal{R}^*
\label{e:slough2}
\end{align}
(where we have combined the remainder terms).
Finally, by Lemma \ref{l:frequencyPseudolocalFinal},
\begin{align}\nonumber
&(1-\psi(\mathcal{P}))\chi(1-\psi_1(\mathcal{P}))\tilde{\chi}\mathcal{R}_{S_k}^*\tilde{\chi}\psi_0(\mathcal{P})\mathcal{R}^*\\
&\qquad\qquad=(1-\psi(\mathcal{P}))\chi O(k^{-\infty})_{\mathcal{D}(\mathcal{P}^{-N})\to \mathcal{D}(\mathcal{P}^N)}\psi_0(\mathcal{P})\mathcal{R}^*;
\label{e:slough3}
\end{align}
the result then follows by combining \eqref{e:slough0}, \eqref{e:slough1}, \eqref{e:slough2}, and \eqref{e:slough3}.
\end{proof}

\subsection{Proof of Theorem~\ref{t:hpFEMAbstract}}

\begin{proof}
Recall from the text underneath \eqref{e:sums} that the plan is to split into high and low frequencies using (i) the Fourier transform away from all boundaries and interfaces and (ii) $\psi(\cP)$ near all non-PML boundaries and interfaces. The contribution from the PML is dealt with using ellipticity of $P_\theta$ (and hence good solution-operator bounds) in the PML region.

 Let $\varphi_{P}\in \overline{C^\infty}(\Omega)$ (recall that this notation is defined in \S\ref{s:defPML})
with 
$\supp \varphi_P$ contained inside the PML and 
$\supp (1-\varphi_P)\cap \Gamma_{\tr}=\emptyset$.
Let $\chi \in C_c^\infty(\Omega)$ with $\supp (1-\chi)\cap \supp(1-\varphi_P)\subset W$ and $\supp \chi\cap \Gamma_{p}=\emptyset$. 
Let also $\psi\in C_c^\infty(\mathbb{R})$ with $\supp(1-\psi)\cap [-M,M]=\emptyset$ for some $M>0$ to be chosen large enough. We then write 
\begin{align*}
\mathcal{R}^*&=\chi \mathcal{R}^* + (1-\chi)\mathcal{R}^*= \big((1-\psi(|\hbar D|)\big)\chi \mathcal{R}^*+ \psi(|\hbar D|)\chi\mathcal{R}^*+ (1-\chi)\mathcal{R}^*,
\end{align*}
Since the coefficients of $P_\theta$ are smooth on $\supp\chi\Subset \Omega$,  by the semiclassical ellipticity of $P_\theta$ on high frequencies and the elliptic estimate (Theorem \ref{t:elliptic}),
\beq\label{e:useSCelliptic}
\begin{gathered}
 (1-\psi(|\hbar D|))\chi \mathcal{R}^*=\mathcal{H}+\mathcal{L},\\
\text{ where } \quad\|\mathcal{H}\|_{L^2(\Omega)\to H_k^2(\Omega)}\leq C\quad \tand\quad\|\mathcal{L}\|_{L^2(\Omega)\to H_k^N(\Omega)}\leq C\hbar^N\rho(k).
 \end{gathered}
\eeq
Indeed, let $\chi_{\Omega}\in C_c^\infty(\mathbb{R}^d)$ with $\supp (1-\chi_{\Omega})\cap \Omega=\emptyset$ and $\chi_i\in C_c^\infty(\Omega)$, $i=1,2$ with $\supp(1-\chi_i)\cap \supp \chi_{i-1}=\emptyset$ and $\chi_0:=\chi$. 
Since a Fourier multiplier (such as $\psi(|\hbar D|)$ \eqref{e:Fouriermult}) can be written in terms of the quantisation $\widetilde{\Op}$ \eqref{e:quant}, the bound \eqref{e:residual} implies that 
\beq\label{e:early1}
\big(1 - \psi(|\hbar D|)\big) \chi
= 
\Op(a) + O(\hbar^\infty)_{\Psi_\hbar^{-\infty}},
\eeq
where $a\in C^\infty_c(\Rea^{2d})$ with $\supp a\subset \{ (x,\xi)\,:\, x\in \supp \chi,\,\xi\in \supp (1-\psi)\}$.
Then, by Theorem \ref{t:elliptic} applied with $P_\theta$ replaced by $P_\theta^*$, $b=\chi_1$, and $u$ replaced by $\chi_2\mathcal{R}^*f$, 
there is $e\in S^{-2}(\Rea^d)$ such that
\begin{align}\label{e:early2}
\big\|\chi_{\Omega}\big(\Op(a) - \Op(e)\chi_1P_\theta^*\big) \chi_2\mathcal{R}^*f\big\|_{H_k^N}
\leq C_N\hbar^N \|\chi_2\mathcal{R}^*f\|_{H_k^{-N}}.
\end{align}
Now, with $\mathscr{R}= O(\hbar^\infty)_{\Psi_\hbar^{-\infty}}$, by \eqref{e:early2},
\begin{align*}
\chi_\Omega(1- \psi(|\hbar D|)) \chi\cR^* f 
&=\chi_\Omega (1- \psi(|\hbar D|))\chi \chi_2 \cR^* f\\
&= \chi_\Omega \Op(a) \chi_2 \cR^* f + \chi_\Omega \mathscr{R}
\chi_2 \cR^* f.
    \end{align*}
Therefore, if 
$
\mathcal{H}:=\Op(e)\chi_1
$
and 
$$
\mathcal{L}:=\big(
\Op(a) \chi_2\mathcal{R}^* - \Op(e)\chi_1\big) + \mathscr{R}\chi_2 \cR^*,
$$
then the mapping properties in  \eqref{e:useSCelliptic} hold by \eqref{e:early1} and the boundedness property of Informal Theorem \ref{t:rules}.

Combining \eqref{e:useSCelliptic} with \eqref{e:pExplicitApproximation}, \eqref{e:pExplicitApproximation2}and~\eqref{e:analyticityBound0}, 
we obtain
\begin{align*}
&\|(I-\Pi_h^p)\chi \mathcal{R}^*\|_{L^2\to H_k^1}\\
&\leq \|(I-\Pi_h^p)\mathcal{H}\|_{L^2\to H_k^1}+\|(I-\Pi_h^p)\mathcal{L}\|_{L^2\to H_k^1}+\|(I-\Pi_h^p)\psi(|\hbar D|)\chi\mathcal{R}^*\|_{L^2\to H_k^1}\\
&\leq C\|(I-\Pi_h^p)\|_{H_k^2\to H_k^1}+C_{N,M}k^{-N}\rho(k) \|(I-\Pi_h^p)\|_{H_k^M\to H_k^1}\\
&\qquad +\|(I-\Pi_h^p)\|_{H^{p+1}\to H_k^1}
\rho(k)
\sum_{\ell=0}^{p+1}\frac{1}{\ell!}(Ck)^{\ell},\\
&\quad\leq C\frac{hk}{p}+ C_{N,M}k^{-N}\rho(k) \Big(\frac{hk}{p}\Big)^M+ C^{p+1}(hk)^{p}k^{-p-1}
\rho(k)
\sum_{\ell=0}^{p+1}\frac{1}{\ell!}(Ck)^{\ell}.
\end{align*}
We now show that, as claimed in \eqref{e:analyticityBound1a} above,
\beq\label{e:Stirling1}
\sum_{\ell=0}^{p+1} \frac{1}{\ell!}(Ck)^{\ell} \leq
C^{p+1}
\max\Big\{\frac{k}{p+1},1\Big\}^{p+1},
\eeq
so that 
\begin{align}\nonumber
&\|(I-\Pi_h^p)\chi \mathcal{R}^*\|_{L^2\to H_k^1}\\
&\leq C\frac{hk}{p}+ C_{N,M}k^{-N}\rho(k) \Big(\frac{hk}{p}\Big)^M+ C^{p+1}(hk)^{p}k^{-p-1}
C^{p+1}
\max\Big\{\frac{k}{p+1},1\Big\}^{p+1}
\rho(k).
\label{e:highInterior}
\end{align}
First, by Stirling's approximation, 
\beqs
\sum_{\ell=0}^{p+1}\frac{1}{\ell!}C^{\ell}k^{\ell}
\leq \sum_{\ell=0}^{p+1}\Big(\frac{Ck}{\ell}\Big)^{\ell}.
\eeqs
Next, for $Ck/\ell \geq e$ (which is ensured if $Ck\geq p$, with a larger value of $C$) the function $(Ck/\ell)^\ell$ is increasing. Therefore if $Ck\geq p$ then
\beq\label{e:boundingsums1}
\sum_{\ell=0}^{p+1}\Big(\frac{Ck}{\ell}\Big)^{\ell}
\leq (p+2) \Big(\frac{Ck}{p+1}\Big)^{p+1} \leq 
\Big(\frac{Ck}{p+1}\Big)^{p+1},
\eeq
(where we have increased $C$ in the last inequality). 
On the other hand for any $p\geq C k$, 
$$
\sum_{\ell=0}^{p+1}\frac{1}{\ell!} C^\ell k^\ell \leq \sum_{\ell=0}^{\infty} \frac{1}{\ell!}(Ck)^{\ell}=e^{Ck}\leq C^{p+1},
$$
and then \eqref{e:Stirling1} holds.

To handle $(1-\chi)\mathcal{R}^*$, 
we write
$$
(1-\chi)\mathcal{R}^*=\varphi_P(1-\chi)\mathcal{R}^*+(1-\varphi_P)(1-\chi)\mathcal{R}^*
$$
By ellipticity of $P_\theta$ in the PML region (see e.g.~\cite[Theorem 4.2 and \S C.3]{AGS2}), 
\begin{equation}
\label{e:PMLEstimate}
\|\varphi_P(1-\chi)\mathcal{R}^*\|_{L^2\to H_k^2}\leq C.
\end{equation}
Thus, by~\eqref{e:pExplicitApproximation},
\begin{equation}
\label{e:PMLApprox}
\|(I-\Pi_h^p)\varphi_P(1-\chi)\mathcal{R}^*\|_{L^2\to H_k^1}\leq C\frac{hk}{p},
\end{equation}
and it remains to bound $(I-\Pi_h^p)(1-\varphi_P)(1-\chi)\mathcal{R}^*$.
To do this, observe that by~\eqref{e:highElliptic1}, for $M$ large enough,
\begin{equation}
\label{e:highElliptic}
\|(1-\psi(\mathcal{P}))(1-\varphi_P)(1-\chi)\mathcal{R}^*\|_{L^2\to H_k^2(\Omega\setminus \Gamma_{\rm p})\cap H_k^1(\Omega)}\leq C,
\end{equation}
where we have used that $\supp(1-\varphi_P)(1-\chi)\cap \Gamma_{\rm tr}=\emptyset$ and $(1-\varphi_P)(1-\chi)$ is constant on $\Gamma_{\rm p}\cup \Gamma_-$.
By~\eqref{e:highElliptic} and~\eqref{e:pExplicitApproximation},
\begin{equation}
\label{e:highPart}
\|(I-\Pi_h^p)(1-\psi(\mathcal{P}))(1-\varphi_P)(1-\chi)\mathcal{R}^*\|_{L^2\to H_k^1}\leq C\frac{hk}{p},
\end{equation}
and it now remains to bound $(I-\Pi_h^p)\psi(\mathcal{P})(1-\varphi_P)(1-\chi)\mathcal{R}^*$.
To do this, 
let $N\geq 0$ and 
 $\chi_0\in \overline{C^\infty}(\Omega)$ with 
 $$\supp \chi_0\subset W \tand \supp(1- \chi_0)\cap \supp (1-\chi)\cap \supp (1-\varphi_P)=\emptyset.$$ 
  Then, by Lemma~\ref{l:spatialPseudolocalFinal}, for any $N$ and $M$,
\begin{equation}
\label{e:pseudolocal1}
\|(1-\chi_0)\psi(\mathcal{P})(1-\varphi_P)(1-\chi)\|_{L^2\to H_k^{M+1}}\leq Ck^{-N}.
\end{equation}
   Next, define 
 \begin{gather*}
  W_{\mathcal{T}}:=\bigcup_{T \in \cT\,:\,T \subset W} T.
 \end{gather*}
 Then, let $h_0>0$ be small enough such that for $0<h<h_0$, 
 $$
 \supp \chi_0\subset W_{\mathcal{M}}.
 $$
 By Theorem~\ref{t:approxHighLowReg}, applied with $\Omega_{\mathcal{T}_1}=\overline{W_{\mathcal{T}}^c}$, $\Omega_{\mathcal{T}_2}=W_{\mathcal{T}}$, $t=M+1$, and $s=1$, 
together with~\eqref{e:pseudolocal1}, we obtain 
\begin{align*}
&\|(I-\Pi_h^p)\psi(\mathcal{P})(1-\varphi_P)(1-\chi)\|_{L^2\to H_k^1(\Omega)}\\
&\qquad\leq C\Big(\frac{h k}{p}\Big)^{M}k^{-M-1}\|\psi(\mathcal{P})(1-\varphi_P)(1-\chi)\|_{L^2\to H^{M+1}(W_{\mathcal{T}}^c)}\\
&\qquad\qquad\qquad+(Chk)^{p}k^{-p-1}\|\psi(\mathcal{P})\|_{L^2\to H^{p+1}(W_{\mathcal{T}}\setminus \Gamma_{\rm p})\cap H^1(W_{\mathcal{T}})},\\
&\qquad\leq C\Big(\frac{h k}{p}\Big)^{M}\|(1-\chi_0)\psi(\mathcal{P})(1-\varphi_P)(1-\chi)\|_{L^2\to H^{M+1}_k(W_{\mathcal{T}}^c)}\\
&\qquad\qquad\qquad+(Chk)^{p}k^{-p-1}\|\psi(\mathcal{P})\|_{L^2\to H^{p+1}(W\setminus \Gamma_{\rm p})\cap H^1(W)},\\
&\qquad\leq C\Big(\frac{h k}{p}\Big)^{M}k^{-N}+(Chk)^{p}k^{-p-1}\|\psi(\mathcal{P})\|_{L^2\to H^{p+1}(W\setminus \Gamma_{\rm p})\cap H^1(W)}.
\end{align*}
Together with \eqref{e:highInterior}, \eqref{e:PMLApprox}, and~\eqref{e:highPart}, this implies~\eqref{e:hpApproximability}.

(Strictly speaking, when applying Theorem \ref{t:approxHighLowReg}, we should split the domain into more than two sets -- depending on $\Gamma_{\rm p}$ -- and  use \eqref{e:meshGammap}. The argument, however, is identical.)
\end{proof}

\subsection{Analytic regularity estimates for $\psi(\mathcal{P})$}\label{s:analyticestimates}

\begin{theorem}
\label{t:hpFEMEllipticRegularity}
Suppose $a$ is the sesquilinear form defining an $H^\infty$ Helmholtz  problem approximated using a $C^\infty$ PML with $C^{1,1}$ truncation boundary (in the sense of Definitions \ref{d:helmholtzProblem} and \ref{d:pmlTruncation}) and, in addition, suppose that $\Gamma_{\rm p}=\emptyset$, $\Gamma_-\in C^\omega$, a Dirichlet condition is imposed on $\Gamma_-$, and there is a neighborhood, $U$ of $\Gamma_-$ such that $A,n\in C^\omega(U)$. Then for all $W\Subset U$, $\psi\in C_c^\infty(\mathbb{R})$ and $k_0>0$, there is $C>0$ such that
\beq\label{e:sums}
\|\psi(\mathcal{P})\|_{L^2\to H^{p+1}(W)}\leq \Big(C\max\big\{\tfrac{k}{p+1},1\big\}\Big)^{p+1}.
\eeq
\end{theorem}
\begin{proof}
Let $\chi\in C_c^\infty(U)$ with $\supp(1-\chi)\cap W=\emptyset$. Since $\psi$ is compactly supported,
\begin{align}
\|\chi \partial^\alpha \psi(\mathcal{P})\|_{L^2\to L^2}&= \|\chi \partial^{\alpha} e^{-t\mathcal{P}} e^{t\mathcal{P}}\psi(\mathcal{P}) \|_{L^2\to L^2}\leq  \|\chi \partial^{\alpha} e^{-t\mathcal{P}}\|_{L^2\to L^2} Ce^{Ct}.\label{e:rain3}
\end{align}
By the heat-equation bounds of \cite{EsMoZh:17} (as summarised in \cite[Theorem 4.3]{LSW4}), for $0\leq t\leq k^2$ and $\tau\in [0,1]$,
\beq\label{e:heat1}
\|\chi \partial^{\alpha} e^{-t\mathcal{P}}\|_{L^2\to L^2}\leq \exp \big((tk^{-2})^{-\tau}\big)|\alpha|!C^{|\alpha|}(tk^{-2})^{(\tau-1)|\alpha|/2}.
\eeq
We now prove that 
\beq\label{e:rain1}
\inf_{t\in [0,k^2]}\|\chi \partial^{\alpha} e^{-t\mathcal{P}}\|_{L^2\to L^2}e^{Ct}\leq C^{|\alpha|}\max(|\alpha|^{|\alpha|}, k^{|\alpha|}). 
\eeq
When $|\alpha|\geq k$, let $\tau=1$ and $t=k$. Then, by \eqref{e:heat1}, 
\begin{align*}
\|\chi \partial^{\alpha} e^{-t\mathcal{P}}\|_{L^2\to L^2}e^{Ct}
\leq C e^{Ck} |\alpha|!C^{|\alpha|}e^{Ck},
\end{align*}
which implies \eqref{e:rain1} when $|\alpha|\geq k$. 
When $|\alpha|\leq k$, we seek $t$ and $\tau$ such that
\beq\label{e:rain2}
(k^{-2}t)^{(\tau-1)|\alpha|/2} = k^{|\alpha|}|\alpha|^{-|\alpha|}
\quad\tand\quad
t=(k^{-2} t)^{-\tau}
\eeq
(with the second equality making the arguments of the two exponentials equal). The second equality in \eqref{e:rain2} is that 
\beqs
\tau =\frac{\log|\alpha|}{\log  k^{2}|\alpha|^{-1}},
\eeqs
and inputting this into the first equality in \eqref{e:rain2} implies that $t=|\alpha|$. With these choices, \eqref{e:rain1} and Stirling's approximation imply that 
\begin{align*}
\|\chi \partial^{\alpha} e^{-t\mathcal{P}}\|_{L^2\to L^2}e^{Ct}
\leq C e^{C|\alpha|} C^{|\alpha|}k^{|\alpha|},
\end{align*}
which implies \eqref{e:rain1} when $|\alpha|\leq k$.
With $t=\min(k,|\alpha|)$ and $\tau=\min(\frac{\log|\alpha|}{\log  k^{2}|\alpha|^{-1}},1)$, this bound implies that
$$
C\|\chi \partial^{\alpha} e^{-t\mathcal{P}}\|_{L^2\to L^2}e^{Ct}\leq C^{|\alpha|}\max(|\alpha|^{|\alpha|}, k^{|\alpha|}). 
$$
Combining \eqref{e:rain3}, \eqref{e:rain1}, the definition of $\|\cdot\|_{H^{p+1}(W)}$
(i.e., \eqref{e:weightednorm1} with $k=1$), 
the equality 
\beq\label{e:consequencemulti}
\sum_{|\alpha|=\ell}\frac{1}{\alpha!} =\frac{d^\ell}{\ell!}
\eeq
(proved using the multinomial formula), and Stirling's approximation,
we obtain that
\begin{align}\nonumber
\|\psi(\mathcal{P})\|_{L^2\to H^{p+1}(W)}^2&\leq \sum_{\ell=0}^{p+1}\sum_{|\alpha|=\ell}\frac{1}{\ell!\alpha!}\|\chi \partial^\alpha \psi(\mathcal{P})\|^2_{L^2\to L^2}\\
&\leq \sum_{\ell=0}^{p+1}\sum_{|\alpha|=\ell}\frac{1}{\ell!\alpha!}C^{2\ell}\max(\ell^{2\ell},k^{2\ell})\leq \sum_{\ell=0}^{p+1}C^{2\ell}\max\Big(1,\frac{k^{2\ell}}{{\ell}^{2\ell}}\Big).
\label{e:sums1}
\end{align}
Now, 
by arguing as in \eqref{e:boundingsums1}, 
if $k\geq p+1$ then 
\begin{align}
\sum_{\ell=0}^{p+1} C^{2\ell} \max\Big(1,\frac{k^{2\ell}}{{\ell}^{2\ell}}\Big)
\leq
\sum_{\ell=0}^{p+1} C^{2\ell} \frac{k^{2\ell}}{{\ell}^{2\ell}}
\leq C^{2(p+1)} \Big( \frac{k}{p+1}\Big)^{2(p+1)}.
\label{e:sums2}
\end{align}
Furthermore, if $k\leq p+1$, then 
\begin{align}\nonumber
\sum_{\ell=0}^{p+1} C^{2\ell}  \max\Big(1,\frac{k^{2\ell}}{{\ell}^{2\ell}}\Big)
&= \sum_{\ell=0}^k C^{2\ell} \Big(\frac{k}{\ell}\Big)^{2\ell} + \sum_{\ell=k+1}^{p+1} C^{2\ell}
\\
&\leq C^{2k} + C^{2(p+1)}\leq C^{2(p+1)}.
\label{e:sums3}
\end{align}
The result \eqref{e:sums} then follows by combining \eqref{e:sums1}, \eqref{e:sums2}, and \eqref{e:sums3}.
\end{proof}

\section{Pollution and preasymptotic error estimates for the $h$-BEM (\ref{R3})}\label{s:R3}

\subsection{Second-kind BIEs for Helmholtz scattering problems}

In this section, we consider the $h$-BEM for the standard second-kind boundary integral equation formulations for the Helmholtz scattering problems:~given $f\in L^2(\Gamma)$, find $v\in L^2(\Gamma)$ such that
\begin{equation}
\label{e:basicForm2}
\operator v=f,\qquad \operator := c_0 (I+\pert),\qquad c_0\in\mathbb{C}\setminus \{0\},
\end{equation}
where the operator $\pert:L^2(\Gamma_-)\to L^2(\Gamma_-)$ is compact. 
Theorem \ref{thm:BIEs} below recaps the standard result that the Helmholtz exterior Dirichlet and Neumann problems can be reformulated as integral equations of the form~\eqref{e:basicForm2}
where $f$ is given in terms of the known Dirichlet/Neumann boundary data and $\operator$ is one of the boundary integral operators $A_k'$, $A_k$, $\Breg$, or $\Breg'$ defined in~\eqref{e:DBIEs} and~\eqref{e:NBIEs}.
Standard mapping properties of these operators (see, e.g., \cite[Theorems 2.17 and 2.18]{ChGrLaSp:12}) imply that $A'_k, A_k, \Breg, \Breg'$ are bounded on $L^2(\Gamma_-)\to L^2(\Gamma_-)$. Furthermore, each of  $A'_k, A_k, \Breg, \Breg'$ is invertible on $L^2(\Gamma_-)\to L^2(\Gamma_-)$; moreover they are each equal to a multiple of the identity plus a compact operator on $L^2(\Gamma_-)\to L^2(\Gamma_-)$ (see, e.g., \cite[\S2.6]{ChGrLaSp:12} for $A'_k, A_k$ and \cite[Theorem 2.2]{GaMaSp:21N} for $\Breg, \Breg'$). \footnote{In the real-valued $L^2(\Gamma_-)$ inner product, $A'_k$ and $A_k$ are each other's adjoints (hence the $'$ notation); similarly for $\Breg, \Breg'$.}
      As before, we use the notation
\beq\label{e:rhoB}
\rhoB(k):=\|\operator^{-1}\|_{L^2(\Gamma_-)\to L^2(\Gamma_-)},
\eeq
and observe that since $\pert$ is compact, $\rhoB(k) \geq c_0^{-1}$;
see Lemma \ref{lem:inversebound} below.

We study the Galerkin approximation to $u$ with $a:L^2(\Gamma_-)\times L^2(\Gamma_-)\to \mathbb{C}$ defined by
$$
a(v,w):= \langle \operator v,w\rangle_{L^2(\Gamma_-)}
$$
under assumptions on $\operator$ that allow us to include any of $A_k$, $A_k'$, $\Breg$, or $\Breg'$. Since $\operator:L^2(\Gamma_-)\to L^2(\Gamma_-)$, for $\mathcal{V}\subset L^2(\Gamma_-)$, and $u\in L^2(\Gamma_-)$, the Galerkin approximation, $v_h$, to $v$ in $\mathcal{V}$ satisfies, 
\begin{equation}
\label{e:galerkinApproximationBEM0}
\Pi_{\mathcal{V}}\operator v_h=(I+\Pi_{\mathcal{V}}\pert)v_h=\Pi_{\mathcal{V}}\operator v,
\end{equation}
where $\Pi_{\mathcal{V}}$ is the $L^2$-orthogonal projector onto $\mathcal{V}$. In fact, we prove error bounds for the solution, $v_h\in\mathcal{V}_h$ to 
\begin{equation}
\label{e:galerkinApproximationBEM}
(I+P_{\mathcal{V}}\pert)v_h=P_{\mathcal{V}}\operator v,
\end{equation}
where $P_{\mathcal{V}}:L^2(\Gamma)\to \mathcal{V}$ is any projector onto $\mathcal{V}$; i.e, in principle, we consider any projection method for approximating $v$.

\bre[The history of the operators $A_k'$, $A_k$, $\Breg$, $\Breg'$]
\label{r:historyBIEs}
     The BIEs involving the operators $A_k'$ and $A_k$ \eqref{e:DBIEs} were introduced in \cite{BrWe:65, Le:65, Pa:65}. The subscript ``reg" on the Neumann boundary-integral operators \eqref{e:NBIEs} indicates that these are not the 
``combined-field'' (or ``combined-potential") Neumann BIEs introduced by \cite{BuMi:71} (denoted by $B'_k$ and $B_k$ in, e.g., \cite[\S2.6]{ChGrLaSp:12}); the operators  introduced in \cite{BuMi:71} are given by \eqref{e:NBIEs} with $S_{\ri k}$ removed, and thus are not bounded on $\LtG$ because $H_k: \LtG\to H^{-1}(\Gamma)$. The idea of preconditioning $H_k$ with an order $-1$ operator goes back to \cite{Bu:76} (see, e.g., the discussion in \cite{AmHa:90}),
with the use of $S_{\ri k}$ proposed in \cite{BrElTu:12}, and then advocated for in
\cite{BoTu:13,ViGrGi:14} (for more details, see the discussion in, e.g., \cite[\S2.1.1]{GaMaSp:21N}).
The results in this section hold for a wider class of regularising operators, of which $S_{\ri k}$ is the prototypical example; see \cite[Assumption 1.1]{GaMaSp:21N}. 
\ere

\bre[Implementing the operator product for the Neumann BIEs \eqref{e:NBIEs}]
Because of the operator product $S_{\ri k}H_k$ in $\Breg$, in practice, approximations to the solution of $\Breg v= f$ are computing by applying the 
projection 
method, not to $\Breg v= f$, 
but to the system
\begin{align*}
\left(
\begin{array}{cc}
\ri (\tfrac{1}{2} I - K) & S_{\ri k}\\
H_k & - 1
\end{array}
\right)
\left(
\begin{array}{c}
v\\
w
\end{array}
\right)
=
\left(
\begin{array}{c}
f \\
0
\end{array}
\right);
\end{align*}
similarly for $\Breg'$.
For simplicity, when studying $\Breg$ and $\Breg'$ 
we consider the idealised situation of 
\eqref{e:galerkinApproximationBEM}; 
i.e., we ignore the issue of discretising the operator product in Galerkin or collocation methods.
We emphasise, however, that 
\cite{GaRaSp:25} \emph{does} analyse discretising this operator product in the Nystr\"om method.
\ere

\subsection{Assumptions on $\operator$ and the functional calculus of $-\Delta_{\Gamma_-}$}

Recall that since $\Gamma_-$ is a smooth, compact hypersurface in $\mathbb{R}^d$, the tangential Laplacian, $-\Delta_{\Gamma_-}$ (i.e., the Laplacian induced on $\Gamma_-$ from the flat Laplacian on $\mathbb{R}^d$) is a self-adjoint, non-negative, elliptic, second-order differential operator on $L^2$ with domain $H^2(\Gamma_-)$. In particular, there are $0=\lambda_1< \lambda_2\leq\dots$ with $\lambda_j\to \infty$ and an orthonormal basis of $L^2(\Gamma_-)$, $\{\phi_{\lambda_j}\}_{j=1}^\infty$ satisfying
$$
(-\Delta_{\Gamma_-}-\lambda_j^2)\phi_{\lambda_j}=0.
$$
For $f\in L^\infty(\mathbb{R})$ we define
$$
f(-k^{-2}\Delta_{\Gamma_-})u:=\sum_{j=1}^\infty f(k^{-2}\lambda_j^2)\langle u,\phi_{\lambda_j}\rangle_{L^2(\Gamma_-)}\phi_{\lambda_j}.
$$
We recall that elliptic regularity in $H_k^s$ implies that for any basis of vector fields $\{X_i\}_{i=1}^N$ on $\Gamma_-$ and $j\geq 0$, there is $C>0$ such that 
\begin{align*}
\tfrac{1}{C}\|(-k^{-2}\Delta_{\Gamma_-}+1)^{j/2}u\|_{L^2(\Gamma_-)}&\leq \sum_{|\alpha|\leq j}\|k^{-|\alpha|}X_{1}^{\alpha_1}\dots X_{N}^{\alpha_N}u\|_{L^2(\Gamma_-)}\\
&\leq C\|(-k^{-2}\Delta_{\Gamma_-}+1)^{j/2}u\|_{L^2(\Gamma_-)}.
\end{align*}
Therefore, we define 
\beq\label{e:weightedNormGamma}
\|u\|_{H_k^s(\Gamma_-)}:=\|(-k^{-2}\Delta_{\Gamma_-}+1)^{s/2}u\|_{L^2(\Gamma_-)}.
\eeq
We then have the following estimate for functions of the Laplacian.
\begin{lemma}\label{l:functionLaplace}
If $f\in L^\infty_{\comp}(\mathbb{R})$ then for all $N>0$ there is $C>0$ such that 
$$
\|f(-k^{-2}\Delta_{\Gamma_-})\|_{H_k^{-N}(\Gamma_-)\to H_k^N(\Gamma_-)}\leq C.
$$
Moreover, for all $f\in L^\infty(\mathbb{R})$ and $s\in\mathbb{R}$, 
$$
\|f(-k^{-2}\Delta_{\Gamma_-})\|_{H_k^s(\Gamma_-)\to H_k^s(\Gamma_-)}\leq \|f\|_{L^\infty}.
$$
\end{lemma}

We make the following assumption on $\operator$.
\begin{assumption}[Assumptions on $\operator$ and its high-frequency components]
\label{ass:abstract1}

\

(i) For all $k>0$, the map $k\mapsto \pert$ is continuous, 
$\operator:= c_0(I+\pert)$ is bounded and invertible on $H^s(\Gamma_-)$ for all $s\in \Rea$, and $\pert :H^s(\Gamma_-)\to H^{s+1}(\Gamma_-)$. 

(ii) For any $\chi \in C^\infty_{c}(\Rea;[0,1])$ with $\supp(1-\chi)\cap [-1,1]=\emptyset$, 
$(I-\chi(-k^{-2}\Delta_\Gamma))\pert$ and $\pert(I-\chi(-k^{-2}\Delta_{\Gamma}))$ are both in $\Psi^{-1}_{\hbar}(\Gamma)$.

(iii) There is $L_{\max}>0$ such that for any $\chi \in C^\infty_{c}(\Rea;[0,1])$ with $\supp(1-\chi)\cap [-1,1]=\emptyset$,
$$
\big|\sigma_\hbar\big(I+\big(I-\chi(-k^{-2}\Delta_{\Gamma_-})\big)\pert\big)\big|\geq \tfrac{1}{L_{\max}}.
$$
\end{assumption}

\begin{remark}
Since multiplication by $c_0$ maps $\mathcal{V}$ to $\mathcal{V}$ for any space of functions, $\mathcal{V}$, we assume, without loss of generality,  that $c_0=1$ from now on. 
\end{remark}

Under Assumption~\ref{ass:abstract1}, the calculus of semiclassical pseudodifferential operators has the following three consequences.
Roughly speaking, the first (Lemma \ref{l:forward}) states that, uniformly in $k$, the operator $\pert$ increases Sobolev regularity by one on high frequencies, the third (Lemma \ref{l:inverse2}) states that 
$(I+(1-\psi(-k^{-2}\Delta_{\Gamma_-}))\pert)^{-1}$ does not move frequencies, and the second (Lemma \ref{l:inverse1}) states that the symbol of this operator controls its norm.

\begin{lemma}
\label{l:forward}Suppose Assumption~\ref{ass:abstract1} holds. Then for all $\psi\in C_c^\infty(\mathbb{R};[0,1])$ with $\supp (1-\psi)\cap [-1,1]=\emptyset$ and all $k_0>0$, $s\in\mathbb{R}$, is $C>0$ such that for $k>k_0$, 
\begin{equation}
\label{e:forwardNorm}
\begin{aligned}
\big\|\pert \big(I-\psi(-k^{-2}\Delta_{\Gamma_-})\big)\big\|_{H_k^s(\Gamma_-)\to H_k^{s+1}(\Gamma_-)}&\leq C \\
\big\| \big(I-\psi(-k^{-2}\Delta_{\Gamma_-})\big)\pert\big\|_{H_k^s(\Gamma_-)\to H_k^{s+1}(\Gamma_-)}&\leq C.
\end{aligned}
\end{equation}
\end{lemma}
\begin{proof}
This follows from Part (ii) of Assumption~\ref{ass:abstract1}
and the boundedness property of 
Informal Theorem~\ref{t:rules} (or, more precisely, its analogue for semiclassical pseudodifferential operators on a smooth manifold; see Remark \ref{r:manifold}).
\end{proof}
\begin{lemma}
\label{l:inverse1}Suppose Assumption~\ref{ass:abstract1} holds. Then for all $\psi\in C_c^\infty(\mathbb{R};[0,1])$ with $\supp (1-\psi)\cap [-1,1]=\emptyset$ there is $k_0>0$ such that for all $s\in \mathbb{R}$ there is $C>0$ such that for $k>k_0$
$$\big(I+\big(I-\psi(-k^{-2}\Delta_{\Gamma_-})\big)\pert\big)^{-1}\in\Psi_{\hbar}^0(\Gamma_-)$$
with 
\begin{equation}
\label{e:inverseNorm}
\big\|\big(I+\big(I-\psi(-k^{-2}\Delta_{\Gamma_-})\big)\pert\big)^{-1}\big\|_{H_k^s(\Gamma_-)\to H_k^s(\Gamma_-)}\leq L_{\max}+Ck^{-1}. 
\end{equation}
\end{lemma}
\begin{proof}
By Theorem~\ref{t:elliptic} (or rather its analogue on a manifold) together with Parts (ii) and (iii) of Assumption~\ref{ass:abstract1} and the composition property of (the manifold analogue of) Informal Theorem \ref{t:rules}, 
there is $E\in \Psi^0_{\hbar}(\Gamma_-)$ with $|\sigma_\hbar(E)|\leq L_{\max}$ such that 
$$
E\big(I+(I-\chi(-k^{-2}\Delta_{\Gamma_-}))\pert \big)=I+R,\quad\text{ with }\, R=O(k^{-\infty})_{\Psi_{\hbar}^{-\infty}}. 
$$
Therefore, for $\hbar$ small enough (and hence $k_0>0$ large enough), 
$$
(I+(1-\psi(-k^{-2}\Delta_{\Gamma_-}))\pert)^{-1}=(I+R)^{-1}E=E-R(I+R)^{-1}E=E+O(k^{-\infty})_{\Psi^{-\infty}_\hbar}.
$$
Hence, $(I+(1-\psi(-k^{-2}\Delta_{\Gamma_-}))\pert)^{-1}\in \Psi^{0}_{\hbar}(\Gamma_-)$ and $|\sigma_\hbar((I+(1-\psi(-k^{-2}\Delta_{\Gamma_-}))\pert)^{-1})|\leq L_{\max}$. The bound \eqref{e:inverseNorm} then follows from the boundedness property of Informal Theorem~\ref{t:rules}, using that the $L^2\to L^2$ bound in terms of the symbol recorded there can be converted into an  analogous $H^s_k\to H^s_k$ bound; see \cite[Lemma 3.2]{GaRaSp:25}.
\end{proof}

\begin{lemma}
\label{l:inverse2}
Suppose Assumption~\ref{ass:abstract1} holds and let $\psi\in C_c^\infty(\mathbb{R};[0,1])$ with $\supp (1-\psi)\cap [-1,1]=\emptyset$. Then there is $k_0>0$ such that for all $N>0$, $\psi_1,\psi_2\in C_c^\infty(\mathbb{R})$ with $\supp \psi_1\cap \supp (1-\psi_2)=\emptyset$, and $\supp (1-\psi_1)\cap [-1,1]=\emptyset$,  there is $C>0$ such that for $k>k_0$
\begin{equation}
\label{e:inverseLocal}
\begin{aligned}
&\|\psi_1(-k^{-2}\Delta_{\Gamma})\pert(1-\psi_2(-k^{-2}\Delta_{\Gamma_-}))\|_{H_k^{-N}(\Gamma_-)\to H_k^N(\Gamma_-)}\leq Ck^{-N},\\
&\|\psi_1(-k^{-2}\Delta_{\Gamma})(I+(1-\psi(-k^{-2}\Delta_{\Gamma_-}))\pert)^{-1}(1-\psi_2(-k^{-2}\Delta_{\Gamma_-})\|_{H_k^{-N}(\Gamma_-))\to H_k^N(\Gamma_-)}\\
&\hspace{8cm}\leq Ck^{-N}\\
&\|(1-\psi_2(-k^{-2}\Delta_{\Gamma}))(I+(1-\psi(-k^{-2}\Delta_{\Gamma_-}))\pert)^{-1}\psi_1(-k^{-2}\Delta_{\Gamma_-})\|_{H_k^{-N}(\Gamma_-))\to H_k^N(\Gamma_-)}\\
&\hspace{8cm}\leq Ck^{-N}.
\end{aligned}
\end{equation}
\end{lemma}
\begin{proof}[Sketch proof of Lemma~\ref{l:inverse2}]
To prove the result, we use a notion of frequency pseudolocality similar to that in Lemma~\ref{l:frequencyPseudolocalFinal}. In particular, one can check using \emph{either} \cite[Theorem 5.37]{AGS2} \emph{or} the fact that by, e.g.,~\cite[Lemma 3.9]{GaRaSp:25} for $\varphi\in C_c^\infty(\mathbb{R})$, 
$$
\varphi(-k^{-2}\Delta_{\Gamma_-})=\Op(a)+O(k^{-\infty})_{\Psi_{\hbar}^{-\infty}}
$$
with
$$
\supp a\subset \big\{(x',\xi')\in T^*\Gamma_-\,:\,|\xi'|_{g_{\Gamma_-}}^2\in \supp \varphi\big\},
$$
together with 
the composition property from 
Informal Theorem~\ref{t:rules} that for any $m\in\mathbb{R}$, any $\psi_1,\psi_2$ as in the statement, and any $A\in\Psi_{\hbar}^m$ 
$$
\psi_1(-k^{-2}\Delta_{\Gamma_-})A(1-\psi_2(-k^{-2}\Delta_{\Gamma_-})=O(k^{-\infty})_{\Psi_{\hbar}^{-\infty}}
$$
and 
$$
(1-\psi_2(-k^{-2}\Delta_{\Gamma_-})
A\psi_1(-k^{-2}\Delta_{\Gamma_-})
=O(k^{-\infty})_{\Psi_{\hbar}^{-\infty}}.
$$
The last two bounds in \eqref{e:inverseLocal} immediately follow since 
$(I+(1-\psi(-k^{-2}\Delta_{\Gamma_-}))\pert)^{-1}\in\Psi^0_{\hbar}$ by Lemma~\ref{l:inverse1}.
For the first bound in \eqref{e:inverseLocal}, let 
$\psi_3\in C_c^\infty(\mathbb{R})$ with $\supp(1-\psi_3)\cap [-1,1]=\emptyset$, and $\supp \psi_3\cap \supp (1-\psi_2)=\emptyset$. These support properties imply that
\begin{align*}
&\psi_1(-k^{-2}\Delta_{\Gamma_-})\pert(1-\psi_2(-k^{-2}\Delta_{\Gamma_-})
\\
&\hspace{2cm}=
\psi_1(-k^{-2}\Delta_{\Gamma_-})\pert
(1-\psi_3(-k^{-2}\Delta_{\Gamma_-})
(1-\psi_2(-k^{-2}\Delta_{\Gamma_-}),
\end{align*}
and then the first bound in \eqref{e:inverseLocal} follows since 
$\pert(1-\psi_3(-k^{-2}\Delta_{\Gamma_-}))\in\Psi_{\hbar}^{-1}$ by Part (ii) of Assumption \ref{ass:abstract1}.
\epf



We now show that the high frequency components of $(I+\pert)^{-1}$ are well behaved; this result can be viewed as a boundary-integral analogue of the fact that the Helmholtz operator is semiclassically elliptic on high frequencies (discussed in \S\ref{s:how}). 
\begin{lemma}
\label{l:hfGood}
Suppose Assumption~\ref{ass:abstract1} holds and let $\chi\in C_c^\infty(\mathbb{R};[0,1])$ with $\supp (1-\chi)\cap [-1,1]=\emptyset$. Then for all $k_0>0$, $s\in \Rea$, and $N>0$ there is $C>0$ such that for all $k>k_0$ 
\begin{align}
\|(I+\pert)^{-1}\|_{H_k^s\to H_k^s}&\leq C\rhoB(k)\label{e:HsInverseNorm},\\
\big\|(I+\pert)^{-1}\big(I-\chi(-k^{-2}\Delta_{\Gamma_-})\big)\big\|_{H_k^{s}(\Gamma_-)\to H_k^s(\Gamma_-)}&\leq L_{\max}+Ck^{-1}+Ck^{-N}\rhoB(k),\label{e:hfInverseNorm1}\\
\big\|\big(I-\chi(-k^{-2}\Delta_{\Gamma_-})\big)(I+\pert)^{-1}\big\|_{H_k^{s}(\Gamma_-)\to H_k^s(\Gamma_-)}&\leq L_{\max}+Ck^{-1}+Ck^{-N}\rhoB(k).\label{e:hfInverseNorm2}
\end{align}
\end{lemma}
Before proving Lemma \ref{l:hfGood}, we highlight that the rest of this section only uses Lemmas \ref{l:forward} and \ref{l:hfGood} and the following corollary of these two results (i.e., Lemmas \ref{l:inverse1} and \ref{l:inverse2} are only used to prove Lemma \ref{l:hfGood}.

\begin{corollary}\label{c:hfGood}
Suppose Assumption~\ref{ass:abstract1} holds and let $\chi\in C_c^\infty(\mathbb{R};[0,1])$ with $\supp (1-\chi)\cap [-1,1]=\emptyset$. Then for all $k_0>0$ and $s,N>0$ there is $C>0$ such that for all $k>k_0$ 
\begin{align*}
\big\| \pert (I+\pert)^{-1} 
\big(I-\chi(-k^{-2}\Delta_{\Gamma_-})\big)\big\|_{H^s_k(\Gamma_-) \to H^{s+1}_k(\Gamma_-)} \leq C + C k^{-N}\rhoB(k).
\end{align*}
    \end{corollary}

\bpf
Since $I= (I-\chi(-k^{-2}\Delta_{\Gamma_-})) + 
\chi(-k^{-2}\Delta_{\Gamma_-})$,
\begin{align*}
&\pert (I+\pert)^{-1} 
\big(I-\chi(-k^{-2}\Delta_{\Gamma_-})\big)\\
&\qquad = 
\big(I-\chi(-k^{-2}\Delta_{\Gamma_-})\big)
\pert (I+\pert)^{-1} 
\big(I-\chi(-k^{-2}\Delta_{\Gamma_-})\big)
\\
&\hspace{3cm}
+ \chi(-k^{-2}\Delta_{\Gamma_-})
\pert (I+\pert)^{-1} 
\big(I-\chi(-k^{-2}\Delta_{\Gamma_-})\big)\\
&\qquad= 
\big(I-\chi(-k^{-2}\Delta_{\Gamma_-})\big)
\pert (I+\pert)^{-1} 
\big(I-\chi(-k^{-2}\Delta_{\Gamma_-})\big)
\\
&\hspace{3cm}
+ \chi(-k^{-2}\Delta_{\Gamma_-})
\big( I - (I+\pert)^{-1} \big)
\big(I-\chi(-k^{-2}\Delta_{\Gamma_-})\big),
\end{align*}
and then the result follows from Lemma \ref{l:forward}, \eqref{e:hfInverseNorm1}, and 
Lemma \ref{l:functionLaplace}.
\epf

\begin{proof}[Proof of Lemma \ref{l:hfGood}]
We first prove \eqref{e:HsInverseNorm}. First observe that, by Part (i)
of Assumption \ref{ass:abstract1} it is sufficient to prove the result for $k_0$ sufficiently large. 
Next observe that once \eqref{e:HsInverseNorm} is proved for $s\geq 0$, since $\|(I+\pert)^{-1}\|_{H^{s}_\hbar(\Gamma) \to H^{s}_\hbar(\Gamma)} = \|((I+\pert)^*)^{-1} \|_{H^{-s}_\hbar(\Gamma) \to H^{-s}_\hbar(\Gamma)}$, the result for $s<0$ follows by applying the result for $s\geq 0$ to $(I+\pert)^*$.

Given $g\in H^s_k$, let $\phi\in H^s_k$ satisfy $(I+\pert)\phi=g$. We need to show that 
\beq\label{e:sweet1}
\|\phi\|_{H^s_k} \leq C \rhoB(k) \|g\|_{H^s_k}.
\eeq
Let $\widetilde{\chi} \in C_c^\infty(\Rea;[0,1])$ with $\supp(1-\widetilde{\chi})\cap [-1,1]=\emptyset$ and $\supp(1-\widetilde{\chi})\cap \supp\chi=\emptyset$.
By Lemma \ref{l:functionLaplace} and the definition of $\rhoB(k)$ \eqref{e:rhoB},
\begin{align}\label{e:GRSold1}
\big\| \widetilde{\chi}(-k^{-2}\Delta_{\Gamma_-})\phi\big\|_{H^s_k}\leq C \| \phi\|_{L^2} \leq C\rhoB(k) \|g\|_{L^2}\leq C\rhoB(k) \|g\|_{H^s_k}.
\end{align}
We now claim that 
\beq\label{e:ellipPara1}
\big\|\big(I-\widetilde{\chi}(-k^{-2}\Delta_{\Gamma_-}\big)\phi\big\|_{H^s_k}\leq
C\big\|\big(I-\chi(-k^{-2}\Delta_{\Gamma_-})\big)(I+\pert)\phi\big\|_{H^s_k}+ Ck^{-M}\|\phi\|_{H^s_k}.
\eeq
Once this is established, the bound
\eqref{e:sweet1} (and hence
\eqref{e:HsInverseNorm}) for $k_0$ sufficiently large follows by combining \eqref{e:GRSold1} with \eqref{e:ellipPara1}.

To establish \eqref{e:ellipPara1}, let $\psi\in C_c^\infty(\Rea;[0,1])$ with $\supp(1-\psi)\cap [-1,1]=\emptyset$ and $\supp(1-\chi)\cap \supp\psi=\emptyset$. By Part (iii) of Assumption \ref{ass:abstract1}, $I + (I-\psi(-k^{-2}\Delta_{\Gamma_-})\pert$ is elliptic, and thus 
\beq\label{e:ellipticHF}
\big(I-\chi(-k^{-2}\Delta_{\Gamma_-})\big)
\big(I + (I-\psi(-k^{-2}\Delta_{\Gamma_-})\pert\big)
=
\big(I-\chi(-k^{-2}\Delta_{\Gamma_-})\big)(I+\pert)
\eeq
is elliptic on the support of $I-\widetilde{\chi}(-k^{-2}\Delta_{\Gamma_-})
$. 
More precisely by~\cite[Lemma 3.9]{GaRaSp:25} 
\begin{gather*}I-\tilde{\chi}(-k^{-2}\Delta_{\Gamma_-})=\Op(\tilde{a})+O(k^{-\infty})_{\Psi_{\hbar}^{-\infty}},\\ I-\chi(-k^{-2}\Delta_{\Gamma_-})=\Op(a)+O(k^{-\infty})_{\Psi_{\hbar}^{-\infty}}
 \end{gather*} for some $a,\tilde{a}\in S^0(T^*\Gamma)$ with $\supp \tilde{a}\subset \{ |a|>\frac{1}{2}\}$. Therefore, by the composition property of Informal Theorem~\ref{t:rules}, 
 $$
 (I-\chi(-k^{-2}\Delta_{\Gamma_-}))(I+\pert)=\Op(b)+O(k^{-\infty})_{\Psi_{\hbar}^{-\infty}}
 $$ 
 with $\{|a|>\frac{1}{2}\}\subset \{ |b|>c>0\}$ and hence $\supp \tilde{a}\subset \{ |b|>c>0\}$.
The bound \eqref{e:ellipPara1} then follows from the manifold analogue of the elliptic estimate of Theorem \ref{t:elliptic}.

Next observe that the bound \eqref{e:HsInverseNorm} implies 
both~\eqref{e:hfInverseNorm1} and~\eqref{e:hfInverseNorm2} for all $k_0<k\leq k_1$ for any $k_1>0$ fixed. Thus, it remains to check~\eqref{e:hfInverseNorm1} and~\eqref{e:hfInverseNorm2} for $k>k_1$. 

To do this, let $\psi\in C_c^\infty(\mathbb{R})$ with $\supp (1-\psi)\cap [-1,1]=\emptyset$, and $\supp \psi\cap \supp (1-\chi)=\emptyset$. Then, by Lemma~\ref{l:inverse1} and Lemma~\ref{l:inverse2}, there are $k_1>0$ and $C>0$ such that \eqref{e:inverseNorm}
and \eqref{e:inverseLocal} hold for $k>k_1$.

Now, since $I + \pert = I + (1-\psi(-k^{-2}\Delta_{\Gamma_-}))\pert + \psi(-k^{-2}\Delta_{\Gamma_-})\pert$,
\begin{align*}
&(I+\pert)\big(I+\big(1-\psi(-k^{-2}\Delta_{\Gamma_-})\big)\pert\big)^{-1}
\big(I-\chi(-k^{-2}\Delta_{\Gamma_-})\big)\\
&=  I-\chi(-k^{-2}\Delta_{\Gamma_-})
+\psi(-k^{-2}\Delta_{\Gamma_-})\pert\big(I+\big(1-\psi(-k^{-2}\Delta_{\Gamma_-})\big)\pert\big)^{-1}
\big(I-\chi(-k^{-2}\Delta_{\Gamma_-})\big)\\
&=I-\chi(-k^{-2}\Delta_{\Gamma_-})+ O(k^{-\infty})_{H_k^{-N}\to H_k^N},
\end{align*}
where we have used~\eqref{e:inverseLocal} in the last equality.
Hence,
\begin{align*}
\big\|\big(I+\big(1-\psi(k^{-2}\Delta_{\Gamma_-})
\big)\pert\big)^{-1}\big(I-\chi(-k^{-2}\Delta_{\Gamma_-})\big)&-(I+\pert)^{-1}\big(I-\chi(-k^{-2}\Delta_{\Gamma_-})\big)\big\|_{H_k^{-N}\to H_k^N}\\
&\leq C k^{-N}\rhoB(k).
\end{align*}
Combining this with~\eqref{e:inverseNorm} implies~\eqref{e:hfInverseNorm1} for $k>k_1$
Similarly, 
\begin{align*}
&\big(I-\chi(-k^{-2}\Delta_{\Gamma_-})\big)\big(I+\big(1-\psi(-k^{-2}\Delta_{\Gamma_-})\big)\pert\big)^{-1}(I+\pert)\\
&= \big(I-\chi(-k^{-2}\Delta_{\Gamma_-})\big)\big(I+\big(1-\psi(-k^{-2}\Delta_{\Gamma_-})\big)\pert\big)^{-1}\psi(-k^{-2}\Delta_{\Gamma_-})\pert+ (I-\chi(-k^{-2}\Delta_{\Gamma_-})\\
&=\big(I-\chi(-k^{-2}\Delta_{\Gamma_-})\big)+ O(k^{-\infty})_{H_k^{-N}\to H_k^N},
\end{align*}
which implies
\begin{align*}
\big\|\big(I-\chi(-k^{-2}\Delta_{\Gamma_-})\big)\big(I+\big(1-\psi(-k^{-2}\Delta_{\Gamma_-})\big)\pert\big)^{-1}&-\big(I-\chi(-k^{-2}\Delta_{\Gamma_-})\big)(I+\pert)^{-1}\big\|_{H_k^{-N}\to H_k^N}\\
&\leq C\rhoB(k) k^{-N}.
\end{align*}
Combining this with~\eqref{e:inverseNorm} implies~\eqref{e:hfInverseNorm2} for $k>k_1$.
\end{proof}

\subsection{Preasymptotic error estimates for the $h$-BEM}

Let $s,t\in\mathbb{R}$, $\mathcal{V}\subset H_k^s(\Gamma_-)$, and $P_{\mathcal{V}}:H_k^t(\Gamma_-)\to \mathcal{V}$ be a projector onto $\mathcal{V}$. Let
\beq\label{e:projnotation}
\bestt{t}{s}(P_{\mathcal{V}}):=\|I-P_{\mathcal{V}}\|_{H_k^t(\Gamma_-)\to H_k^s(\Gamma_-)}.
\eeq
\begin{theorem}[Abstract result on convergence of the projection method]
\label{t:abstracthBEM}
Let $s\geq 0$, $M>0$ $k_0>0$, suppose that $\operator$ satisfies Assumption~\ref{ass:abstract1}, and $\mathcal{J}\subset (0,\infty)$ satisfies 
$$
\sup\big\{ k^{-M}\rhoB(k)\,:\, k\in (k_0,\infty)\setminus \mathcal{J}\big\}<\infty.
$$
For all $N\geq s$, $\e> 0$, there are $c,C>0$ such that 
the following holds. If $k>k_0$, $k\notin \mathcal{J}$, $\mathcal{V}\subset H_k^s(\Gamma_-)$, and $P_{\mathcal{V}}:H_k^{s}(\Gamma_-)\to\mathcal{V}$ is a projector onto $\mathcal{V}$ satisfying
\begin{equation}
\label{e:approxRequirements}
\Big(\bestt{N}{- N}(P_{\mathcal{V}}) + 
\bestt{N}{s}(P_{\mathcal{V}})\bestt{s+1}{-N}(P_{\mathcal{V}})
\Big)\rhoB(k)
+\bestt{s+1}{ s}(P_{\mathcal{V}})<c,
\end{equation}
then for all $v\in H_k^s(\Gamma_-)$, the solution $v_h\in\mathcal{V}$ to~\eqref{e:galerkinApproximationBEM} exists, is unique, and satisfies
\begin{align}
&\big\|1_{[-1-\e,1+\e]}(k^{-2}\Delta_{\Gamma})(v-v_h)\big\|_{H_k^N(\Gamma_-)} \label{e:BEMLF}
\leq C \bestt{s}{s}\bestt{s}{-N}\rhoB(k)\|(I-P_{\mathcal{V}})v\|_{H_k^s(\Gamma_-)},\\ \nonumber
&\big\|(I-1_{[-1-\e,1+\e]}(k^{-2}\Delta_{\Gamma}))(v-v_h)\big\|_{H_k^s(\Gamma_-)}\\ \nonumber§
&\leq (L_{\max}+Ck^{-1})\bestt{s}{s}\Bigg(\bestt{s}{s}+C\Big[
\bestt{N}{s}\bestt{s}{-N}\rhoB(k) 
\\
&\hspace{1cm}+\bestt{s}{s}\Big((\bestt{N}{-N}+\bestt{N}{s}\bestt{s+1}{-N})\rhoB(k)+\bestt{s+1}{s}\Big)\Big]\Bigg)
\|(I-P_{\mathcal{V}})v\|_{H_k^s(\Gamma_-)}.
\label{e:BEMHF}
\end{align}
\end{theorem}

Observe that \eqref{e:BEMLF} is a bound on the low frequencies of the Galerkin error, and \eqref{e:BEMHF} is a bound on the high frequencies.

In the concrete setting of the $h$-BEM, with piecewise polynomials Theorem~\ref{t:abstracthBEM} has the following consequence. 
\begin{theorem}
\mythmname{Rigorous statement of Informal Theorem \ref{t:informalR3a}:~$k$-explicit upper bounds on the $h$-BEM error}
\label{t:concretehBEM}
Let $\Omega_-\Subset \mathbb{R}^d$ with smooth boundary, $M>0$, $\Upsilon>0$ $p\geq 0$, $k_0>0$. Suppose that $\operator\in\{ A_k, A_k',\Breg,\Breg'\}$ and $\mathcal{J}\subset (0,\infty)$ satisfies 
$$
\sup\big\{ k^{-M}\rhoB(k)\,:\, k\in (k_0,\infty)\setminus \mathcal{J}\big\}<\infty.
$$
Then for all $\e> 0$, there are $c,C>0$ such that for all $k>k_0$, $k\notin \mathcal{J}$, and all $C^{p+1}$ triangulations with constant $\Upsilon$, $\mathcal{T}$ (in the sense of Definition~\ref{d:Crtriang}), satisfying
\begin{equation}
\label{e:approxRequirementsConcrete}
(hk)^{2p+2}\rhoB(k)+hk<c
\end{equation}
(where $h:=h(\mathcal{T})$) 
 all $v\in L^2(\Gamma_-)$, the Galerkin approximation, $v_h$ to $v$ in $\mathcal{P}_{\mathcal{T}}^p$ exists, is unique, and satisfies
\begin{align}\label{e:BEMLFconcrete}
&\|1_{[-1-\e,1+\e]}(k^{-2}\Delta_{\Gamma})(v-v_h)\|_{L^2(\Gamma_-)} \leq C(hk)^{p+1}\rhoB(k)\min_{w_h\in\mathcal{P}_{\mathcal{T}}^p}\|v-w_h\|_{L^2(\Gamma_-)},\\ \nonumber
&\|(I-1_{[-1-\e,1+\e]}(k^{-2}\Delta_{\Gamma}))(v-v_h)\|_{L^2(\Gamma_-)}\\
&\qquad\qquad\leq (1+Ck^{-1})\big(1+C(hk)^{2p+2}\rhoB(k) +Chk\big)\min_{w_h\in\mathcal{P}_{\mathcal{T}}^p}\|v-w_h\|_{L^2(\Gamma_-)}.
\label{e:BEMHFconcrete}
\end{align}
\end{theorem}

The bounds \eqref{e:BEMLFconcrete} and \eqref{e:BEMHFconcrete} together imply \eqref{e:qoBEM}.
The bound \eqref{e:PPMR2} on the relative error then follows from the combination of \eqref{e:qoBEM}, the polynomial-approximation result of 
Theorem~\ref{t:approxHighLowRegL2}
and the following result.

\begin{lemma}
\mythmname{Oscillatory behaviour of BIE solution under plane-wave scattering}\label{thm:oscil}
Suppose that $\operator$ is one of $A_k, A_k'$, $\Breg$, or $\Breg'$ and the right-hand side $f$ is as described in Theorem \ref{thm:BIEs}.
Then $v\in H^t(\Gamma)$ for all $t\geq 0$ and
given $t\geq s\geq 0$ and $k_0>0$ there exists $C>0$ such that
\beq\label{e:oscil}
\N{v}_{H^t_k(\Gamma)} \leq C \N{v}_{H^s_k(\Gamma)} \quad\tfa k\geq k_0.
\eeq
\end{lemma}

We do not prove Lemma \ref{thm:oscil}, but instead give 
a sketch (with references).

\bpf[Sketch proof]
The steps of the proof are the following.
\ben
\item Show that the high-frequency components of $f$ are $O(k^{-\infty})$, via an explicit calculation using integration by parts/stationary phase; see \cite[Lemma 7.3]{GaRaSp:25}.
\item Show that the high-frequency components of $v$ are $O(k^{-\infty})$ using the result of Step 1 and the fact that 
$(I-\chi(-k^{-2}\Delta_{\Gamma_-}))(I+\pert)
$ is elliptic on high-frequencies via Part (iii) of Assumption \ref{ass:abstract1} and \eqref{e:ellipticHF}; see \cite[Lemma 7.4]{GaRaSp:25}.
\item Obtain the bound \eqref{e:oscil} from the result of Step 2 and the smoothing property of low-frequency cutoffs (from Lemma \ref{l:functionLaplace}, as in \eqref{e:projectionbounds} below);
see
\cite[Theorem 7.2]{GaRaSp:25}.
\een
\epf

\subsection{Proof of Theorem \ref{t:abstracthBEM}}

We start with the following lemma.

\ble[Quasioptimality in terms of the norm of the discrete inverse]\label{lem:QOabs}
Let $s\geq 0$, $\mathcal{V}\subset H_k^s(\Gamma_-)$ and suppose that $P_{\mathcal{V}}: H_k^s (\Gamma_-)\to \mathcal{V}$ is a projection onto $\mathcal{V}$ such that $I+ P_{\mathcal{V}}\pert : H^s_k(\Gamma_-)\to H^s_k(\Gamma_-)$ is invertible, then for all $v\in H_k^s(\Gamma_-)$, the Galerkin approximation, $v_h$ to $v$ in $\mathcal{V}$  exists, is unique, and satisfies
\beq\label{e:QOabs}
v-v_h= (I+P_{\mathcal{V}} \pert)^{-1}(I-P_{\mathcal{V}}) (I-P_{\mathcal{V}})v.
\eeq
\ele

\bpf
We first consider the equation \eqref{e:galerkinApproximationBEM} as an equation in $H^s_k(\Gamma_-)$.
Since $I+P_{\mathcal{V}} \pert:H^s_k(\Gamma)\to H^s_k(\Gamma)$ is invertible, the solution $v_h$ to~\eqref{e:galerkinApproximationBEM} in $H^s_k(\Gamma)$ exists and is unique as an element of $H^s_k(\Gamma)$. Applying $(I-P_{\mathcal{V}})$ to \eqref{e:galerkinApproximationBEM}, we see that $(I-P_{\mathcal{V}})v_h=0$ and thus $v_h\in \mathcal{V}$; i.e., the equation \eqref{e:galerkinApproximationBEM} has a unique solution in $\mathcal{V}$.
Then, \eqref{e:galerkinApproximationBEM},
\begin{align*}
(I+P_{\mathcal{V}} \pert) (v-v_h) &= (I+P_{\mathcal{V}} \pert) v - P_{\mathcal{V}}(I+\pert)v \\
& = v + P_{\mathcal{V}} \pert v - P_{\mathcal{V}} (I+\pert)v  = (I-P_{\mathcal{V}})v.
\end{align*}
Therefore, since $ (I-P_{\mathcal{V}})=  (I-P_{\mathcal{V}})^2$,
\begin{align*}
(v-v_h) = (I+P_{\mathcal{V}} \pert)^{-1} (I-P_{\mathcal{V}})v=(I+P_{\mathcal{V}} \pert)^{-1} (I-P_{\mathcal{V}})(I-P_{\mathcal{V}})v.
\end{align*}
\epf

To obtain conditions under which $I+ P_{\mathcal{V}}\pert : H^s_k(\Gamma)\to H^s_k(\Gamma)$ is invertible (and thus under which we can use Lemma \ref{lem:QOabs}),
we write
\beq\label{e:JeffFav1}
(I+P_{\mathcal{V}}\pert) = I +\pert + (P_{\mathcal{V}} -I)\pert = \Big( I+ (P_{\mathcal{V}}-I)\pert (I+\pert)^{-1}\Big) (I+\pert),
\eeq
and study when $ I+ (P_{\mathcal{V}}-I)\pert (I+\pert)^{-1}$ is invertible. 
We then use that 
\beqs
\Big( I+ (P_{\mathcal{V}}-I) \pert(I+\pert)^{-1}\Big)^{-1}(I-P_{\mathcal{V}}) = (I-P_{\mathcal{V}})\Big( I+ (P_{\mathcal{V}}-I) \pert(I+\pert)^{-1}\Big)^{-1}(I-P_{\mathcal{V}});
\eeqs
i.e., we can add an $(I-P_{\mathcal{V}})$ on the left.
Indeed, if $( I+ (P_{\mathcal{V}}-I) \pert(I+\pert)^{-1})v= (I-P_{\mathcal{V}})f$, then $P_{\mathcal{V}} v =0$ so that $v=(I-P_{\mathcal{V}})v$.
Therefore the operator on the right-hand side of \eqref{e:QOabs} can be written as 
\beq\label{e:thething}
(I+P_{\mathcal{V}}\pert )^{-1}(I-P_{\mathcal{V}}) = (I+\pert)^{-1} (I-P_{\mathcal{V}})\Big( I+ (P_{\mathcal{V}}-I) \pert(I+\pert)^{-1}\Big)^{-1}(I-P_{\mathcal{V}}).
\eeq

The first natural attempt is to show that $ I+ (P_{\mathcal{V}}-I)\pert (I+\pert)^{-1}$ is invertible is to show that
\begin{equation}
\label{e:Neumann0}
\|(P_{\mathcal{V}}-I)\pert (I+\pert)^{-1}\|_{H^s_k(\Gamma_-)\to H^s_k(\Gamma_-)}<\frac{1}{2}, 
\end{equation}
in which case $\|(I+(P_N-I)\pert (I+\pert)^{-1})^{-1}\|_{H^s_k(\Gamma_-)\to H^s_k(\Gamma_-)}\leq 2$. 

To understand when the condition \eqref{e:Neumann0} holds, given $\e>0$, we define low- and high-frequency projectors, 
\beq\label{e:highandlow}
\Pi_{\mathscr{L}}:= 1_{[-1-\e,1+\e]}(-k^{-2}\Delta_{\Gamma_-})\quad\tand\quad
\Pi_{\mathscr{H}}:=(I-\Pi_{\mathscr{L}}),
\eeq
 and decompose $H_{k}^s$ into an $L^2$-orthogonal sum $\Pi_{\mathscr{L}}H_{k}^s\oplus \Pi_{\mathscr{H}}H_{k}^s$. 
That is, 
\beq\label{e:split1}
\begin{aligned}
&(P_{\mathcal{V}}-I)\pert (I+\pert)^{-1}\\
&\qquad = \begin{pmatrix}\Pi_{\mathscr{L}}(P_{\mathcal{V}}-I)\pert (I+\pert)^{-1}\Pi_{\mathscr{L}}&\Pi_{\mathscr{L}}(P_N-I)\pert (I+\pert)^{-1}\Pi_{\mathscr{H}}\\

\Pi_{\mathscr{H}}(P_{\mathcal{V}}-I)\pert (I+\pert)^{-1}\Pi_{\mathscr{L}}&\Pi_{\mathscr{H}}(P_{\mathcal{V}}-I)\pert (I+\pert)^{-1}\Pi_{\mathscr{H}}\end{pmatrix}.
\end{aligned}
\eeq
\begin{definition}
We write
$$
\|A\|_{H^s_k(\Gamma_-)\to H^s_k(\Gamma_-)}\leq C_1\begin{pmatrix} a_{\mathscr{LL}}&a_{\mathscr{LH}}\\a_{\mathscr{HL}}&a_{\mathscr{HH}}\end{pmatrix},$$
if
\beq\label{e:matrix}
A=
\begin{pmatrix}
 \Pi_{\mathscr{L}} A \Pi_{\mathscr{L}}
 &\Pi_{\mathscr{L}} A \Pi_{\mathscr{H}}
 \\\Pi_{\mathscr{H}} A \Pi_{\mathscr{L}}&\Pi_{\mathscr{H}} A \Pi_{\mathscr{H}}\end{pmatrix}, 
\eeq
with
$$
\begin{aligned}
\|\Pi_{\mathscr{L}} A \Pi_{\mathscr{L}}\|_{H^s_k(\Gamma_-)\to H^s_k(\Gamma_-)}&\leq C_1a_{\mathscr{LL}},&\|\Pi_{\mathscr{H}} A \Pi_{\mathscr{L}}\|_{H^s_k(\Gamma_-)\to H^s_k(\Gamma_-)}&\leq C_1a_{\mathscr{HL}},\\
\|\Pi_{\mathscr{L}} A \Pi_{\mathscr{H}}\|_{H^s_k(\Gamma_-)\to H^s_k(\Gamma_-)}&\leq C_1a_{\mathscr{LH}},&\|\Pi_{\mathscr{H}} A \Pi_{\mathscr{H}}\|_{H^s_k(\Gamma_-)\to H^s_k(\Gamma_-)}&\leq C_1a_{\mathscr{HH}}.
\end{aligned}
$$
\end{definition}
One can then easily check that if $\|A\|\leq C_1 \mathscr{A}$ and $\|B\|\leq C_2 \mathscr{B}$, then $\|AB\|\leq C_1 C_2 \mathscr{A}\mathscr{B}$.

\begin{lemma}
\label{l:perturbationBlock}Let $s\geq 0$, $k_0>0$ and suppose $\operator$ satisfies Assumption~\ref{ass:abstract1}. Then for any $N\in\mathbb{R}$ there is $C>0$ such that for all $k>k_0$, $\mathcal{V}\subset H_k^s(\Gamma_-)$ and $P_{\mathcal{V}}: H_k^s (\Gamma_-)\to \mathcal{V}$ projectors onto $\mathcal{V}$, 
\begin{equation}
\label{e:blockDecomp1}
\begin{aligned}
&\big\|(P_{\mathcal{V}}-I)\pert (I+\pert)^{-1}\big\|_{H^s_k(\Gamma_-)\to H^s_k(\Gamma_-)}\leq 
C\begin{pmatrix}\bestt{N}{-N}  \rhoB(k) &\bestt{s+1}{-N}\\
\bestt{N}{s}\rhoB(k)&\bestt{s+1}{s}
\end{pmatrix}.
\end{aligned}
\end{equation}
\end{lemma}

We prove Lemma \ref{l:perturbationBlock} below, but first we use it to understand when 
\eqref{e:Neumann0} holds. Indeed, \eqref{e:Neumann0} holds if all entries on the right-hand side of \eqref{e:blockDecomp1} are $\ll 1$.
However, this imposes too strong a restriction on the space $\mathcal{V}$ and we can instead change the norm on $H_{k}^{s}$ using an invertible operator $\mathcal{C}:H^s_k(\Gamma_-)\to H^s_k(\Gamma_-)$ to improve the situation. More precisely, we observe that 
$$
 I+ (P_{\mathcal{V}}-I)\pert (I+\pert)^{-1}= \mathcal{C}^{-1} \Big(I +\mathcal{C}(P_{\mathcal{V}}-I)\pert (I+\pert)^{-1}\mathcal{C}^{-1}\Big)\mathcal{C}.
$$
We now choose $\mathcal{C}$ to be a diagonal matrix with entries chosen so that the bottom-left and top-right entries of 
$\mathcal{C}(P_{\mathcal{V}}-I)\pert (I+\pert)^{-1}\mathcal{C}^{-1}$ are equal; i.e., compared to $(P_{\mathcal{V}}-I)\pert (I+\pert)^{-1}$ we make the
bottom-left  entry smaller, at the cost of making the top-right entry bigger.
The choice of $\mathcal{C}$ that achieves this is 
\begin{equation}
\mathcal{C}:=\begin{pmatrix} I&0\\0&\bestt{N}{s}^{-1/2}\bestt{s+1}{-N}^{1/2}\rhoB(k) ^{-1/2}I\end{pmatrix},
\label{e:newNorm}
\end{equation}
with then
\begin{equation}
\label{e:blockDecomp}
\begin{aligned}
&\|\mathcal{C}(P_{\mathcal{V}}-I)\pert (I+\pert)^{-1}\mathcal{C}^{-1}\|_{H^s_k(\Gamma_-)\to H^s_k(\Gamma_-)}\\
& \hspace{3cm}\leq 
C\begin{pmatrix}\bestt{N}{-N}  \rhoB(k) &\bestt{N}{s}^{1/2}\bestt{s+1}{-N}^{1/2}\rhoB(k)^{1/2}\\
\bestt{N}{s}^{1/2}\bestt{s+1}{-N}^{1/2}\rhoB(k)^{1/2}&\bestt{s+1}{s}
\end{pmatrix}.
\end{aligned}
\end{equation}
The inequality~\eqref{e:approxRequirements} then implies that
\begin{equation}
\label{e:Neumann1}
\big\|\mathcal{C}(P_{\mathcal{V}}-I)\pert (I+\pert)^{-1}\mathcal{C}^{-1}\big\|_{H^s_k(\Gamma_-)\to H^s_k(\Gamma_-)}<\frac{1}{2},
\end{equation}
and so $I+(P_{\mathcal{V}}-I)\pert (I+\pert)^{-1}$ is invertible and 
\beq\label{e:magicC1}
\Big(I+(P_{\mathcal{V}}-I)\pert (I+\pert)^{-1}\Big)^{-1}= \mathcal{C}^{-1}\Big(I+\mathcal{C}(P_{\mathcal{V}}-I)\pert (I+\pert)^{-1}\mathcal{C}^{-1}\Big)^{-1}\mathcal{C}.
\eeq

We now prove Theorem \ref{t:abstracthBEM}, assuming both Lemma \ref{l:perturbationBlock} and the following lemma.

\begin{lemma}
\label{l:inverseBlock}Let $s\geq 0$, $k_0>0$, and suppose $\operator$ satisfies Assumption~\ref{ass:abstract1}. Then for any $N\in\mathbb{R}$ there is $C>0$ such that for all $k>k_0$, $\mathcal{V}\subset H_k^s(\Gamma_-)$ and $P_{\mathcal{V}}: H_k^s (\Gamma_-)\to \mathcal{V}$ projectors onto $\mathcal{V}$ 
\begin{equation}
\label{e:blockDecomp2}
\begin{aligned}
&\big\| (I+\pert)^{-1}(I-P_{\mathcal{V}})\big\|_{H^s_k(\Gamma_-)\to H^s_k(\Gamma_-)}
\leq 
\begin{pmatrix}C\bestt{N}{-N}  \rhoB(k) &C\bestt{s}{-N}\rhoB(k)\\
C\bestt{N}{s}&(L_{\max}+Ck^{-1})\bestt{s}{s}
\end{pmatrix}.
\end{aligned}
\end{equation}
\end{lemma}

\begin{proof}[Proof of Theorem~\ref{t:abstracthBEM}]
By \eqref{e:QOabs}, we need to bound the high- and low-frequency components of $(I+P_{\mathcal{V}} \pert)^{-1}(I-P_{\mathcal{V}})$. Furthermore, by the combination of \eqref{e:thething} and \eqref{e:magicC1}, 
\begin{align}\nonumber
&(I+P_{\mathcal{V}}\pert )^{-1}(I-P_{\mathcal{V}}) \\  \label{e:middle1}
&
=(I+\pert)^{-1} (I-P_{\mathcal{V}})
\Big(I+(P_{\mathcal{V}}-I)\pert (I+\pert)^{-1}\Big)^{-1}
(I-P_{\mathcal{V}})\\
&=
 (I+\pert)^{-1} (I-P_{\mathcal{V}})
\mathcal{C}^{-1}\Big(I+\mathcal{C}(P_{\mathcal{V}}-I)\pert (I+\pert)^{-1}\mathcal{C}^{-1}\Big)^{-1}\mathcal{C}(I-P_{\mathcal{V}}).
 \nonumber
\end{align}
By~\eqref{e:blockDecomp}, under the assumption~\eqref{e:approxRequirements}, 
\begin{align*}
&\Big\|\Big(I+\mathcal{C}(P_{\mathcal{V}}-I)\pert (I+\pert)^{-1}\mathcal{C}^{-1}\Big)^{-1}\Big\|_{H_k^s\to H_k^s}\\
&\leq  \sum_{j=0}^\infty \Big\|[\mathcal{C}(P_{\mathcal{V}}-I)\pert (I+\pert)^{-1}\mathcal{C}^{-1}]\Big\|_{H_k^s\to H_k^s}^j\\
&\leq \sum_{j=0}^\infty C^j\begin{pmatrix}\bestt{N}{-N}  \rhoB(k) &\bestt{N}{s}^{1/2}\bestt{s+1}{-N}^{1/2}\rhoB(k)^{1/2}\\
\bestt{N}{s}^{1/2}\bestt{s+1}{-N}^{1/2}\rhoB(k)^{1/2}&\bestt{s+1}{s}
\end{pmatrix}^j\\
&= \left(I-C\begin{pmatrix}\bestt{N}{-N}  \rhoB(k) &\bestt{N}{s}^{1/2}\bestt{s+1}{-N}^{1/2}\rhoB(k)^{1/2}\\
\bestt{N}{s}^{1/2}\bestt{s+1}{-N}^{1/2}\rhoB(k)^{1/2}&\bestt{s+1}{s}
\end{pmatrix}\right)^{-1}\\
&=\frac{1}{(1-C\bestt{N}{-N}\rhoB(k))(1-C\bestt{s+1}{s})-C^2\bestt{N}{s}\bestt{s+1}{-N}\rhoB(k)}\cdot\\
&\hspace{4.5cm}\begin{pmatrix} 1-C\bestt{s+1}{s}&C\bestt{N}{s}^{1/2}\bestt{s+1}{-N}^{1/2}\rhoB(k)^{1/2}\\C\bestt{N}{s}^{1/2}\bestt{s+1}{-N}^{1/2}\rhoB(k)^{1/2}&1-C\bestt{N}{-N}\rhoB(k)\end{pmatrix}\\
&\leq  I+C\footnotesize{\begin{pmatrix} (\bestt{N}{-N}+\bestt{N}{s}\bestt{s+1}{-N})\rhoB(k) +\bestt{s+1}{s}&\bestt{N}{s}^{1/2}\bestt{s+1}{-N}^{1/2}\rhoB(k)^{1/2}\\\bestt{N}{s}^{1/2}\bestt{s+1}{-N}^{1/2}\rhoB(k)^{1/2}&(\bestt{N}{-N}+\bestt{N}{s}\bestt{s+1}{-N})\rhoB(k) +\bestt{s+1}{s}\end{pmatrix}},
\end{align*}
where in the last step we use that, for $Ca<1$, $(1-Ca)^{-1}\leq 1+ C' a$ in the diagonal entries and $(1-Ca)^{-1}\leq C'$ in the off-diagonal entries.
Therefore, by 
\eqref{e:magicC1} and 
the definition of $\mathcal{C}$ \eqref{e:newNorm},
\begin{align*}
&\big\|\big(I+(P_{\mathcal{V}}-I)\pert (I+\pert)^{-1}\big)^{-1}\big\|_{H_k^s\to H_k^s}\\
&\leq  I+C\footnotesize{\begin{pmatrix} (\bestt{N}{-N}+\bestt{N}{s}\bestt{s+1}{-N})\rhoB(k) +\bestt{s+1}{s}&\bestt{s+1}{-N}\\\bestt{N}{s}\rhoB(k)&(\bestt{N}{-N}+\bestt{N}{s}\bestt{s+1}{-N})\rhoB(k) +\bestt{s+1}{s}\end{pmatrix}}.
\end{align*}
Now, with $\Pi_{\mathscr{L}}$ and $\Pi_{\mathscr{H}}$ defined by \eqref{e:highandlow}, 
by Lemma \ref{l:functionLaplace}
\beq\label{e:projectionbounds}
\|\Pi_{\mathscr{L}}\|_{H_{\hbar}^{-N}\to H_{\hbar}^{N}}\leq C\quad\text{ and }\quad \|\Pi_{\mathscr{H}}\|_{H_{\hbar}^t\to H_{\hbar}^t}\leq 1,
\eeq
and thus, by the notation \eqref{e:projnotation},
\begin{align*}
&\|I-P_{\mathcal{V}}\|_{H_k^s\to H_k^s}\leq  \begin{pmatrix} C\bestt{N}{-N}&C\bestt{s}{-N}\\C\bestt{N}{s}&\bestt{s}{s}\end{pmatrix},
\end{align*}
Therefore, 
under the assumption~\eqref{e:approxRequirements}, 
\begin{align*}
&\big\|\big(I+(P_{\mathcal{V}}-I)\pert (I+\pert)^{-1}\big)^{-1}(I-P_{\mathcal{V}})\big\|_{H_k^s\to H_k^s}\\
&\leq  \scriptsize{\begin{pmatrix} C(\bestt{N}{-N}+\bestt{N}{s}\bestt{s+1}{-N})&C(\bestt{s}{-N}+\bestt{s+1}{-N}\bestt{s}{s})\\C\bestt{N}{s}&\!\!\!\!\!\!\!\!\!\!\!\!C\bestt{N}{s}\bestt{s}{-N}\rhoB(k) +\bestt{s}{s}(1+C(\bestt{N}{-N}+\bestt{N}{s}\bestt{s+1}{-N})\rhoB(k)+C\bestt{s+1}{s}) \end{pmatrix}}.
\end{align*}

Finally, combining this last displayed inequality with \eqref{e:middle1}, Lemma~\ref{l:inverseBlock}, the assumption \eqref{e:approxRequirements}, and the fact that 
$\bestt{N}{-N}\leq \bestt{s}{-s}$ for $s\leq N$ by the definition \eqref{e:projnotation},
we obtain that
\begin{align*}
&\|\Pi_{\mathscr{L}}(I+P_{\mathcal{V}}\pert)^{-1}(I-P_{\mathcal{V}})\Pi_{\mathscr{L}}\|_{H_k^s(\Gamma)\to H_k^s(\Gamma)}\\
&\hspace{2.25cm}\leq C\bestt{N}{-N}
\rhoB(k)\big(\bestt{N}{-N}+\bestt{N}{s}\bestt{s+1}{-N}\big)+C\bestt{N}{s}\bestt{s}{-N}\rhoB(k),\\
&\|\Pi_{\mathscr{L}}(I+P_{\mathcal{V}}\pert)^{-1}(I-P_{\mathcal{V}})\Pi_{\mathscr{H}}\|_{H_k^s(\Gamma)\to H_k^s(\Gamma)}\\
&\hspace{1.25cm}\leq C\bestt{N}{-N}\rhoB(k) \big(\bestt{s}{-N}+\bestt{s+1}{-N}\bestt{s}{s}\big)
 \\
 &\hspace{1.45cm}+C\bestt{s}{-N}\rhoB(k)
 \Big[C\bestt{N}{s}\bestt{s}{-N}\rhoB(k)
 \\
 &\hspace{1.65cm}+\bestt{s}{s}\big(1+C(\bestt{N}{-N}+\bestt{N}{s}\bestt{s+1}{-N})\rhoB(k)+C\bestt{s+1}{s}\big)\Big],\\
&\hspace{1.25cm}\leq C\bestt{N}{-N}\rhoB(k) \big(\bestt{s}{-N}+\bestt{s+1}{-N}\bestt{s}{s}\big)+C\bestt{s}{s}\bestt{s}{-N}\rhoB(k),\\
&\hspace{1.25cm}\leq C\bestt{s}{s}\bestt{s}{-N}\rhoB(k),\\
&\|\Pi_{\mathscr{H}}(I+P_{\mathcal{V}}\pert)^{-1}(I-P_{\mathcal{V}})\Pi_{\mathscr{L}}\|_{H_k^s(\Gamma)\to H_k^s(\Gamma)}
\leq C\bestt{N}{s}\big( \bestt{N}{s}\bestt{s+1}{-N}+\bestt{s}{s}\big),\\
&\|\Pi_{\mathscr{H}}(I+P_{\mathcal{V}}\pert)^{-1}(I-P_{\mathcal{V}})\Pi_{\mathscr{H}}\|_{H_k^s(\Gamma)\to H_k^s(\Gamma)}\\
&\leq (L_{\max}+Ck^{-1})\bestt{s}{s}\Big[
C\bestt{N}{s}\bestt{s}{-N}\rhoB(k) 
\\
&\hspace{4cm}+\bestt{s}{s}\Big(1+C(\bestt{N}{-N}+\bestt{N}{s}\bestt{s+1}{-N})\rhoB(k)+C\bestt{s+1}{s}\Big)\Big],
\\
&\leq (L_{\max}+Ck^{-1})\bestt{s}{s}\Bigg(\bestt{s}{s}+C\Big[
\bestt{N}{s}\bestt{s}{-N}\rhoB(k) 
\\
&\hspace{4cm}+\bestt{s}{s}\Big((\bestt{N}{-N}+\bestt{N}{s}\bestt{s+1}{-N})\rhoB(k)+\bestt{s+1}{s}\Big)\Big]\Bigg).
\end{align*}
The bound \eqref{e:BEMLF} then follows from the first two displayed inequalities above, using that $\bestt{N}{-N}\leq \bestt{s}{-N}$, $\bestt{N}{s}\leq \bestt{s}{s}$, and $\rhoB(k)\geq c$. 
The bound \eqref{e:BEMHF} then follows from the third and fourth displayed inequalities in a similar way, using that $\bestt{s}{s}\geq 1$ (since $P_{\cV}$ is a projection).
\end{proof}

It therefore remains to prove Lemmas~\ref{l:perturbationBlock} and~\ref{l:inverseBlock}.
\begin{proof}[Proof of Lemma~\ref{l:perturbationBlock}]
For the top-left  entry
of the matrix corresponding to 
$(P_{\mathcal{V}}-I)\pert(I+\pert)^{-1}$ via \eqref{e:matrix},
by \eqref{e:projectionbounds} and \eqref{e:HsInverseNorm},
\begin{align*}
&\|\Pi_{\mathscr{L}}(P_{\mathcal{V}}-I)\pert(I+\pert)^{-1}\Pi_{\mathscr{L}}\|_{H_k^s\to H_k^s}\\
&= \|\Pi_{\mathscr{L}}(P_{\mathcal{V}}-I)(I-(I+\pert)^{-1})\Pi_{\mathscr{L}}\|_{H_k^s\to H_k^s}\\
&\leq \|\Pi_{\mathscr{L}}\|_{H_{k}^{-N}\to H_k^s}
\bestt{N}{-N}\|I-(I+\pert)^{-1}\|_{H_k^N\to H_k^N}\|\Pi_{\mathscr{L}}\|_{H_k^s\to H_k^N}\leq C\rhoB(k)\bestt{N}{-N}.
\end{align*}
For the top-right entry, 
let $\chi \in C_{c}^\infty((-1-\epsilon,1+\epsilon))$ with $\supp (1-\chi)\cap [-1,1]=\emptyset$ and observe that, by the definition of $\Pi_{\mathscr{H}}$ \eqref{e:highandlow}, $
\Pi_{\mathscr{H}}=
(I-\chi(-k^{-2}\Delta_\Gamma))\Pi_{\mathscr{H}}$. 
Then, by Corollary \ref{c:hfGood} and \eqref{e:projectionbounds}, 
\begin{align*}
&\|\Pi_{\mathscr{L}}(P_{\mathcal{V}}-I)\pert(I+\pert)^{-1}\Pi_{\mathscr{H}}\|_{H_k^s\to H_k^s}\\
&\leq \|\Pi_{\mathscr{L}}\|_{H_{k}^{-N}\to H_k^s}
\bestt{s+1}{-N}
\|\pert(I+\pert)^{-1}
(I-\chi(-k^{-2}\Delta_\Gamma))\|_{H_k^s\to H_k^{s+1}}
\|\Pi_{\mathscr{H}}\|_{H_k^{s}\to H_k^{s}}
\\
&\leq C\bestt{s+1}{-N}.
\end{align*}
For the bottom-left  entry, by \eqref{e:projectionbounds} and \eqref{e:HsInverseNorm},
\begin{align*}
&\|\Pi_{\mathscr{H}}(P_{\mathcal{V}}-I)\pert(I+\pert)^{-1}\Pi_{\mathscr{L}}\|_{H_k^s\to H_k^s}\\
&= \|\Pi_{\mathscr{H}}(P_{\mathcal{V}}-I)(I-(I+\pert)^{-1})\Pi_{\mathscr{L}}\|_{H_k^s\to H_k^s}\\
&\leq \|\Pi_{\mathscr{H}}\|_{H_{k}^{s}\to H_k^s}
\bestt{N}{s}\|I-(I+\pert)^{-1}\|_{H_k^N\to H_k^N}\|\Pi_{\mathscr{L}}\|_{H_k^s\to H_k^N}\leq C\rhoB(k)\bestt{N}{s}.
\end{align*}
Finally, for the bottom-right entry, 
by 
\eqref{e:projectionbounds} and 
Corollary \ref{c:hfGood}, 
\begin{align*}
&\|\Pi_{\mathscr{H}}(P_{\mathcal{V}}-I)\pert(I+\pert)^{-1}\Pi_{\mathscr{H}}\|_{H_k^s\to H_k^s}\\
&\leq \|\Pi_{\mathscr{H}}\|_{H_{k}^{s}\to H_k^s}
\bestt{s+1}{s}
\|\pert(I+\pert)^{-1}
(I-\chi(-k^{-2}\Delta_\Gamma))\|_{H_k^s\to H_k^{s+1}}
\|\Pi_{\mathscr{H}}\|_{H_k^{s}\to H_k^{s}}
\\
&\leq C\bestt{s+1}{s}.
\end{align*}
\end{proof}

\begin{proof}[Proof of Lemma~\ref{l:inverseBlock}]
For the top-left  entry of the matrix corresponding to $(I+L)^{-1}(I-P_{\mathcal{V}})$ via \eqref{e:matrix}, by \eqref{e:projectionbounds} and \eqref{e:HsInverseNorm},
\begin{align*}
&\|\Pi_{\mathscr{L}}(I+\pert)^{-1}(I-P_{\mathcal{V}})\Pi_{\mathscr{L}}\|_{H_k^s\to H_k^s}\\
&\leq \|\Pi_{\mathscr{L}}\|_{H_{k}^{-N}\to H_k^s}
\|(I+\pert)^{-1}\|_{H_k^{-N}\to H_k^{-N}}\bestt{N}{-N}\|\Pi_{\mathscr{L}}\|_{H_k^s\to H_k^N}\leq C\rhoB(k)\bestt{N}{-N}
\end{align*}
For the top-right entry, 
again by \eqref{e:projectionbounds} and \eqref{e:HsInverseNorm},
\begin{align*}
&\|\Pi_{\mathscr{L}}(I+\pert)^{-1}(I-P_{\mathcal{V}})\Pi_{\mathscr{H}}\|_{H_k^s\to H_k^s}\\
&\leq \|\Pi_{\mathscr{L}}\|_{H_{k}^{-N}\to H_k^s}
\|(I+\pert)^{-1}\|_{H_k^{-N}\to H_k^{-N}}\bestt{s}{-N}\|\Pi_{\mathscr{H}}\|_{H_k^s\to H_k^s}\leq C\rhoB(k)\bestt{s}{-N}
\end{align*}
For the bottom-left entry, let $\chi \in C_{c}^\infty((-1-\epsilon,1+\epsilon))$ with $\supp (1-\chi)\cap [-1,1]=\emptyset$ and observe that, by the definition of $\Pi_{\mathscr{H}}$ \eqref{e:highandlow}, $
\Pi_{\mathscr{H}}=
\Pi_{\mathscr{H}}(I-\chi(-k^{-2}\Delta_\Gamma))$. Using this, along with \eqref{e:projectionbounds} and \eqref{e:hfInverseNorm2}, 
we obtain that
\begin{align*}
&\|\Pi_{\mathscr{H}}(I+\pert)^{-1}(I-P_{\mathcal{V}})\Pi_{\mathscr{L}}\|_{H_k^s\to H_k^s}\\
&\qquad\leq 
\|\Pi_{\mathscr{H}}\|_{H_k^s\to H_k^s}
\|(I-\chi(-k^{-2}\Delta_\Gamma))(I+\pert)^{-1}\|_{H_k^s\to H_k^s}\bestt{N}{s}\|\Pi_{\mathscr{L}}\|_{H_k^s\to H_k^N}
\\
&\qquad\leq C\bestt{N}{s}.
\end{align*}
Finally, for the bottom-right entry, 
by \eqref{e:projectionbounds} and \eqref{e:hfInverseNorm2} again,
\begin{align*}
&\|\Pi_{\mathscr{H}}(I+\pert)^{-1}(I-P_{\mathcal{V}})\Pi_{\mathscr{H}}\|_{H_k^s\to H_k^s}\\
&\qquad\leq 
\|\Pi_{\mathscr{H}}\|_{H_k^s\to H_k^s}
\|(I-\chi(-k^{-2}\Delta_\Gamma))(I+\pert)^{-1}\|_{H_k^s\to H_k^s}\bestt{s}{s}\|\Pi_{\mathscr{H}}\|_{H_k^s\to H_k^s}
\\
&\qquad\leq (L_{\max} +Ck^{-1})\bestt{s}{s}.
\end{align*}
\end{proof}

\subsection{Proof of Theorem~\ref{t:concretehBEM}:~references for Assumption~\ref{ass:abstract1}}

Given Theorem \ref{t:abstracthBEM}, 
to prove Theorem~\ref{t:concretehBEM} it only remains to check that Assumption~\ref{ass:abstract1} holds and to obtain estimates on $\best(\Pi_{\mathcal{P}_{\mathcal{T}}^p})_{s\to t}$ when $\Pi_{\mathcal{P}_{\mathcal{T}}^p}:L^2(\Gamma_-)\to L^2(\Gamma_-)$ is the orthogonal projector onto $\mathcal{P}_{\mathcal{T}}^p$. 

Using the results of 
~\cite[Theorem 4.3]{GaMaSp:21N} and~\cite[Chapter 4]{Ga:19}, 
\cite[\S5.1.1]{GaRaSp:25} shows that 
Assumption~\ref{ass:abstract1}
holds for $A_k$ and $A_k'$, and 
\cite[\S5.1.2]{GaRaSp:25} shows that 
Assumption~\ref{ass:abstract1}
holds for $\Breg$ and $\Breg'$, in all cases with $L_{\max}=1$.

For the $H_k^t\to L^2$ mapping properties of $\Pi_{\mathcal{P}_{\mathcal{T}}^p}$, we refer to Appendix~\ref{a:poly}. Indeed, by Theorem~\ref{t:approxHighLowRegL2}
and the fact that $\Pi_{\mathcal{P}_{\mathcal{T}}^p}$ is the $L^2$-orthogonal projector, there is $C>0$ such that for $0\leq t\leq p+1$, 
$$
\|(I-\Pi_{\mathcal{P}_{\mathcal{T}}^p})\|_{H_k^t(\Gamma_-)\to L^2(\Gamma_-)}=\|(I-\Pi_{\mathcal{P}_{\mathcal{T}}^p})\|_{ L^2(\Gamma_-)\to H_k^{-t}(\Gamma_-)}\leq C (hk)^{t}. 
$$
Hence,
\begin{align*}
\bestt{1}{0}&\leq Chk,&\bestt{0}{0}&=1,\\
\bestt{p+1}{0}&\leq C(hk)^{p+1},&\bestt{0}{-p-1}&\leq C(hk)^{p+1},\\
\bestt{1}{-p-1}&\leq \bestt{1}{0}\bestt{0}{-p-1}\leq C(hk)^{p+2},\\
\bestt{p+1}{-p-1}&\leq \bestt{p+1}{0}\bestt{0}{-p-1}\leq C(hk)^{2p+2}.
\end{align*}
Theorem~\ref{t:concretehBEM} now follows by inserting these bounds into Theorem~\ref{t:abstracthBEM} with $N=p+1$ and $s=0$, and recalling that $L_{\max}=1$.

\subsection{Pollution in the Galerkin $h$-BEM with piecewise polynomials}

Pollution in the $h$-FEM (i.e., the fact that $\gg k^d$ degrees of freedom is required to accurately approximate  a scattering solution at frequency $k$ using the $h$-FEM) was famously identified in~\cite{IhBa:95,IhBa:97}. However, it was long conjectured that the $h$-BEM does not suffer from the pollution effect. Theorem~\ref{t:concretehBEM} shows that, indeed, there is no pollution for the $h$-BEM applied with a sound-soft, nontrapping obstacle (with this result first proved in ~\cite{GS2}). However, the rigorous versions of Informal Theorem \ref{t:informalR3b}, namely \cite[Theorems 2.10, 2.12, 2.13, 2.14, and 2.16]{GaRaSp:25}, show that the $h$-BEM \emph{does} suffer from pollution when applied to many trapping, sound-soft obstacles and even for sound-hard scattering by the disk.

We do not attempt to give complete proofs of these results here. Instead, we opt for a sketch of the arguments.

\paragraph{Step 1:} Quasimodes imply pollution. 
By Lemma~\ref{lem:QOabs}, the Galerkin approximation, $v_h$, to $v$ satisfies
\beqs
v-v_h= (I+\Pi_{\mathcal{P}_{\mathcal{T}}^p}\pert)^{-1}(I-\Pi_{\mathcal{P}_{\mathcal{T}}^p}) (I-\Pi_{\mathcal{P}_{\mathcal{T}}^p})v
\eeqs
(see \eqref{e:QOabs} below). Therefore, to prove pollution, we need to show that $(I+\Pi_{\mathcal{P}_{\mathcal{T}}^p} \pert)^{-1}(I-\Pi_{\mathcal{P}_{\mathcal{T}}^p})$ is large. A simple calculation shows that if 
\beq\label{e:introv}
v := (I + \pert)^{-1} (I-\Pi_{\mathcal{P}_{\mathcal{T}}^p}) \widetilde{f}
\eeq
then 
\beqs
(I+ \Pi_{\mathcal{P}_{\mathcal{T}}^p} \pert) v= (I-\Pi_{\mathcal{P}_{\mathcal{T}}^p}) v =(I-\Pi_{\mathcal{P}_{\mathcal{T}}^p})^2 v
\eeqs
i.e., we need to find $v$ with $\|v\|_{L^2}=1$ such that \eqref{e:introv} holds and $(I-\Pi_{\mathcal{P}_{\mathcal{T}}^p})v$ is ``small". 

To do this, we assume that there is a \emph{quasimode}, $u$ for $\operator$; i.e., that there are $u,f\in L^2(\Gamma_-)$ with 
\begin{equation}
\label{e:quasi1}
(I+\pert)u=f,\qquad \|u\|_{L^2(\Gamma_-)}=1,\quad\tand\quad\|f\|_{L^2(\Gamma_-)}\ll 1.
\end{equation}
Moreover, we assume that both $f$ and $u$ are 
oscillating at frequency $\lesssim k$, but not frequency $\ll k$, i.e., 
there is is $\chi\in C_c^\infty(\mathbb{R})$ with $0\notin \supp \chi$ such that 
\begin{equation}
\label{e:quasi2}
\big\|\big(I-\chi(-k^{-2}\Delta_{\Gamma_-})\big)f
\big\|_{L^2}+
\big\|\big(I-\chi(-k^{-2}\Delta_{\Gamma_-})\big)u
\big\|_{L^2}\ll 1
\end{equation}
and thus by Theorem~\ref{t:approxHighLowRegL2},
$\| (I-\Pi_{\mathcal{P}_{\mathcal{T}}^p})u\|_{L^2(\Gamma_-)} \leq C (hk)^{p+1} \| u\|_{L^2(\Gamma_-)}$.

If there is $\tilde{f}$ such that $f = (I-\Pi_{\mathcal{P}_{\mathcal{T}}^p})\widetilde{f}$, then
\beqs
\big\|(I+ \Pi_{\mathcal{P}_{\mathcal{T}}^p} \pert)^{-1} (I-\Pi_{\mathcal{P}_{\mathcal{T}}^p})\big\|_{L^2(\Gamma_-)\to L^2(\Gamma_-)}\geq c (hk)^{-p-1}.
\eeqs
However, this cannot be true for all $h$ since it contradicts the upper bound on the quasioptimality constant in Theorem \ref{t:concretehBEM}; i.e., we cannot solve $f = (I-\Pi_{\mathcal{P}_{\mathcal{T}}^p})\widetilde{f}$.

Motivated by this discussion, our goal is to find $\widetilde{f}$ that is as close as possible to satisfying $(I-\Pi_{\mathcal{P}_{\mathcal{T}}^p})\widetilde{f}=f$. Notice that, since $v=(I+\pert )^{-1}(I-\Pi_{\mathcal{P}_{\mathcal{T}}^p})\tilde{f}$ and $(I+ \pert)^{-1}$ is $O(1)$ on high frequencies (see Lemma \ref{l:hfGood}), low-frequency errors cost more than high-frequency errors. Therefore, we aim to solve $(I-\Pi_{\mathcal{P}_{\mathcal{T}}^p})\tilde{f}=f+e$, where $e$ has as few low-frequencies as possible.

The lower bounds on Galerkin piecewise-polynomial approximations of $k$-oscillating functions from \cite{Ga:22} (see more specifically~\cite[Lemma 12.7 and Corollary 12.8]{GaRaSp:25} imply that there exists $\widetilde{f}$ such that 
\beq\label{e:blueCar1}
\big\| (I-\Pi_{\mathcal{P}_{\mathcal{T}}^p})\widetilde{f} \big\|_{L^2(\Gamma_-)} \leq C(hk)^{-p-1}\big\| f\big\|_{L^2(\Gamma_-)}
\eeq
and
\begin{gather*}
(I-\Pi_{\mathcal{P}_{\mathcal{T}}^p})\widetilde{f}=f +\chi_{\mathscr{L}}(I-\Pi_{\mathcal{P}_{\mathcal{T}}^p})\widetilde{f}+\chi_{\mathscr{H}}(I-\Pi_{\mathcal{P}_{\mathcal{T}}^p})\widetilde{f}.
\end{gather*}
Here $\chi_{\mathscr{L}}$ is a finite-rank operator (which can be taken to be zero when $p=0$) and 
\beqs
\big\| \chi_{\mathscr{L}}\big\|_{H^{-M}(\Gamma)\to H^M(\Gamma)} \leq C_M 
\quad\tand\quad 
\big\| (I+ \pert)^{-1} \chi_{\mathscr{H}}\big\|_{L^2(\Gamma_-)\to L^2(\Gamma_-)} \leq C 
\eeqs
(i.e., $\chi_{\mathscr{L}}$ and $\chi_{\mathscr{H}}$ are low- and high-frequency cutoffs, respectively).
Thus 
\begin{align}\label{e:blueCar2}
v &= (I+\pert)^{-1} f + (I+\pert)^{-1} \chi_{\mathscr{L}}(I-\Pi_{\mathcal{P}_{\mathcal{T}}^p})\widetilde{f}+(I+\pert)^{-1}\chi_{\mathscr{H}}(I-\Pi_{\mathcal{P}_{\mathcal{T}}^p})\widetilde{f}.
\end{align}
In the rest of this sketch, we ignore the contribution from $\chi_{\mathscr{L}}$; including this contribution leads to a finite-dimensional possible obstruction to pollution.
Using \eqref{e:blueCar2}, \eqref{e:blueCar1}, and the fact that $u= (I+L)^{-1}f$ (with $\| u\|_{L^2(\Gamma)}= 1$), we see that $\|v\|_{L^2(\Gamma_-)}\geq 1/2$ if $(hk)^{-p-1} \|f \|_{L^2(\Gamma_-)} \ll 1$.
Finally, by \eqref{e:introv},
\begin{align*}
(I- \Pi_{\mathcal{P}_{\mathcal{T}}^p}) v = (I- \Pi_{\mathcal{P}_{\mathcal{T}}^p}) \chi_{\rm H}(I+\pert)^{-1} (I- \Pi_{\mathcal{P}_{\mathcal{T}}^p}) \widetilde{f} + (I-\Pi_{\mathcal{P}_{\mathcal{T}}^p}) \chi_{\rm L} v
\end{align*}
where $\chi_{\rm H}$ and $\chi_{\rm L}$ are high- and low-frequency cutoffs, so that
\begin{align*}
\big\|(I- \Pi_{\mathcal{P}_{\mathcal{T}}^p}) v\big\|_{L^2(\Gamma_-)} \leq C (hk)^{-p-1} \| f\|_{L^2(\Gamma_-)} + (hk)^{p+1} \|v\|_{L^2(\Gamma_-)}.
\end{align*}
In summary, 
\begin{align*}
&\big\|(I+ \Pi_{\mathcal{P}_{\mathcal{T}}^p} \pert)^{-1} (I-\Pi_{\mathcal{P}_{\mathcal{T}}^p})\big\|_{L^2(\Gamma_-)}\\
&\geq
\frac{\|v\|_{L^2(\Gamma_-)}}{\| (I-\Pi_{\mathcal{P}_{\mathcal{T}}^p})v\|_{L^2(\Gamma_-)}}\\
&\geq 
 c
\begin{cases}
(hk)^{-p-1}, & \|f\|_{L^2(\Gamma_-)} \leq (hk)^{2(p+1)} ,\\
(hk)^{p+1} \| f\|^{-1}_{L^2(\Gamma_-)}, & (hk)^{2(p+1)}\leq \|f\|_{L^2(\Gamma_-)} \ll (hk)^{p+1} .
\end{cases} 
\end{align*}
Since $\|v\|_{L^2(\Gamma_-)}/\| (I-\Pi_{\mathcal{P}_{\mathcal{T}}^p})v\|_{L^2(\Gamma_-)}\geq 1$, by reducing $c$ if necessary, 
\begin{align*}
\big\|(I+ \Pi_{\mathcal{P}_{\mathcal{T}}^p} \pert)^{-1} &(I-\Pi_{\mathcal{P}_{\mathcal{T}}^p})\big\|_{L^2(\Gamma_-)}\\
&\geq 
\frac 12 + 
c
\begin{cases}
(hk)^{-p-1}, & 1\leq (hk)^{2(p+1)} \|f\|_{L^2(\Gamma_-)}^{-1},\\
(hk)^{p+1} \| f\|^{-1}_{L^2(\Gamma_-)}, & \|f\|_{L^2(\Gamma_-)}^{-1}(hk)^{2(p+1)}\leq 1  ,
\end{cases} 
\end{align*}
which, if $\|f\|_{L^2(\Gamma_-)}\ll 1$, implies that the Galerkin $h$-BEM suffers from pollution.

\paragraph{Step 2:} Construction of quasimodes
The proof that pollution occurs for the $h$-BEM is then complete if we can construct $u$, $f$ satisfying~\eqref{e:quasi1} and~\eqref{e:quasi2}. The existence of $u$ and $f$ satisfying~\eqref{e:quasi1} is guaranteed whenever $\rhoB(k_j)\to \infty$ for some $k_j\to \infty$. However, the requirement~\eqref{e:quasi2} does not automatically hold for these quasimodes. In the case of the unit disk, the fact that the boundary layer operators are diagonal in the basis of trigonometric polynomials enables the use of classical Bessel function asymptotics together for the construction of quasimodes, leading to the following thoerem.

\begin{theorem}[Pollution for the Neumann BIEs on the unit disk]
\label{t:neumannDisk}
Let $\Upsilon>0$, $p\geq 0$, and $\Omega_-= B(0,1)\subset \mathbb{R}^2$, and $k_0>0$. Then there is $c>0$ such that for all  $k>k_0$, and all $C^{p+1}$ simplicial triangulations, $\mathcal{T}$ with constant $\Upsilon$ (in the sense of Definition~\ref{d:Crtriang}),  there is $v\in L^2(\Gamma_-)$ such that the Galerkin approximation, $v_h$ to $v$ in $\mathcal{P}_{\mathcal{T}}^p$ with $\operator$ given by $\Breg$ or $\Breg'$, if it exists, satisfies
\begin{equation*}
\frac{\|v_h-v\|_{L^2(\Gamma_-)}}{\|(I-\Pi_{\mathcal{P}_{\mathcal{T}}^p})v\|_{L^2(\Gamma_-)}}\geq 
\frac 12 + 
c \begin{cases} (hk)^{-p-1}, &  1\leq (hk)^{2(p+1)}k^{\frac{1}{3}},\\
(hk)^{p+1}k^{\frac{1}{3}},&k^{\frac{1}{3}}(hk)^{2(p+1)}\leq 1,
\end{cases}
\end{equation*}
where $h:=h(\mathcal{T})$.
In particular, Theorem~\ref{t:concretehBEM} is optimal in this case.
\end{theorem}

However, except in the case of the unit ball, constructing quasimodes satisfying~\eqref{e:quasi2} in addition to~\eqref{e:quasi1}, is accomplished by working with the PDE rather than the BIEs. Indeed, in~\cite[\S13]{GaRaSp:25}, compactly supported functions, $u\in H_k^2(\Omega_+)\cap H_0^1(\Omega_+)$ are constructed such that 
$$
(-k^{-2}\Delta-1)u=g,\qquad \|g\|_{L^2}\ll k^{-1},\qquad \|u\|_{L^2(\Omega_+)}=1,
$$
and such that the microlocal concentration properties of $v$ and $g$  are well-understood. One then shows that the right-hand-side $g$ can be solved away at the cost of boundary data that is $k$-oscillating. It is then possible to use the description of $\operator^{-1}$ in terms of the outgoing Dirichlet-to-Neumann map and the interior impedance to Dirichlet map (see~\cite[Equation 13.1]{GaRaSp:25} or 
\cite[Theorem 2.33]{ChGrLaSp:12} and \cite[Lemma 7.4]{GaMaSp:21N}) to find $u$ and $f$ satisfying~\eqref{e:quasi1} and~\eqref{e:quasi2}.

As an example of this, we consider the following type of domain.
 \begin{definition}[Four-diamond domain]
$\Omega_-$ is a \emph{four-diamond domain} if there exists 
$0<\e<\pi/2$ such that $\Omega_-=\cup_{i=1}^4\Omega_i\Subset \mathbb{R}^2$, where $\Omega_i$ are open, convex, disjoint, and have smooth boundary so that 
\begin{gather*}
\begin{aligned}
\partial \big((-\pi,\pi)\times(-\pi, \pi)\big)\setminus \big(\cup_{a,b\in\pm 1}&B((a \pi,b \pi),\e)\big)\\
&\subset \Gamma_-\subset \mathbb{R}^2\setminus \big(\cup_{a,b\in\pm 1}B((a \pi,b \pi),\e/2)\big),
\end{aligned}\\[1em]
\Omega_-\cap (-\pi,\pi)\times (-\pi,\pi)=\emptyset.
\end{gather*}
\end{definition}
Figure \ref{f:bReallyCoolPicture} shows an example of a four-diamond domain.

Using the method of images, it is easy to construct a compactly supported function $u_1\in H_0^1(\Omega_+)$ that is supported in a small neighbourhood of the square with corners at $(0,\pm \pi)$ and $(\pm \pi, 0)$ such that $u_1$ is microlocally concentrated near this square and with momentum tangent to the square, and $u$ satisfies
$$
\|(-k^{-2}\Delta-1)u_1\|_{L^2(\Omega_+)}\leq Ck^{-2},\qquad \|u_1\|_{L^2(\Omega_+)}=1,\qquad \|k^{-1}\partial_{\nu}u_1\|_{L^2}\geq c>0.
$$

Using this microlocal information, one can construct $\tilde{u}\in H^2(\Omega_+)$ satisfying the Sommerfeld radiation condition such that 
$$
(-k^{-2}\Delta-1)\tilde{u}=(-k^{-2}\Delta-1)u_1,
$$
with
$$
\|\tilde{u}|_{\Gamma_-}\|_{L^2}+\|k^{-1}\partial_{\nu}\tilde{u}|_{\Gamma_-}\|_{L^2}\leq Ck\|(-k^{-2}\Delta-1)u_1\|_{L^2}\leq Ck^{-1}
$$
and $\tilde{u}_{\Gamma_-}$
is microlocally concentrated at frequencies $|\xi|\sim \frac{\sqrt{2}}{2}$. This is done use a propagation of singularities argument for an auxiliary Helmholtz type problem with additional damping.

Since $u:=u_1-\tilde{u}$ satisfies the Sommerfeld radiation condition and
$$
(-k^{-2}\Delta-1)u=0,\qquad \frac{\|u|_{\Gamma_-}\|_{L^2}}{\|k^{-1}\partial_{\nu}u|_{\Gamma_-}\|_{L^2}}\geq ck;
$$
this is, together with~\cite[Equation 13.1]{GaRaSp:25}, is immediately enough to construct a quasimode $v$ satisfying~\eqref{e:quasi1} and~\eqref{e:quasi2} in the case of $A_k$. Slightly more analysis is necessary in the case of $A_k'$ because one needs to find data for the impedance problem in $\Omega_-$ so that the resulting Dirichlet data is, in some sense, close to $\tilde{u}|_{\Gamma_-}$.  This is again done using a propagation argument for an auxiliary Helmholtz type problem with additional damping in $\Omega_-$.

The results of this analysis for the four-diamond domains are as follows.
\begin{theorem}[Pollution for the four-diamond domain when $p=0$]
\label{t:diamond}
Let $p=0$, $k_0>0$, $\Upsilon>0$, $\Omega_-$ be a four-diamond domain, $M>0$, $\mathcal{J}\subset \mathbb{R}$ such that
$$
\sup\big\{ k^{-M}\rhoB(k)\,:\, k>k_0,\,k\notin \mc{J}\big\}<\infty
$$
with $\operator=A_{k}$ or $A_k'$, and $k_n:=\sqrt{2}n$. Then there is $c>0$, such that for all $k\in \mathbb{R}\setminus \mathcal{J}$ and $n$ such that $k_n^{-1}\leq \delta \leq 1$, $|k-k_n|<\delta^{-1} k_n^{-1}$, and all $C^1$ simplicial triangulations (in the sense of Definition~\ref{d:Crtriang}),
$\mathcal{T}$ with constant $\Upsilon$, there is $v\in L^2(\Gamma)$ such that the Galerkin approximation, $v_h$, to $v$ in $\mathcal{P}_{\mathcal{T}}^0$, if it exists, satisfies
$$
\frac{\|v_h-v\|_{L^2(\Gamma_-)}}{\|(I-\Pi_{\mathcal{P}_{\mathcal{T}}^p})v\|_{L^2(\Gamma_-)}}\geq 
\frac 12 + 
c \begin{cases} (hk_n)^{-1}, &  1\leq (hk_n)^{2} \delta k_n,\\
(hk_n)\delta k_n,&\delta k_n(hk_n)^{2}\leq 1,
\end{cases}
$$
where $h:=h(\mathcal{T})$. 
\end{theorem}

\begin{theorem}[Pollution for the four-diamond domain when $p\geq 1$]
\label{t:diamondNew}
Let $\Upsilon>0$, $p\geq 1$, $\Omega_-$ be a four-diamond domain
$\operator=A_{k}$ or $A_k'$, and $k_n:=\sqrt{2}n$. 

If there are $C_1>0, k_0>0$ such that 
\beq\label{e:fourDiamondsBound}
\rhoB(k) \leq C_1 k^2 \quad \tfa k\geq k_0,
\eeq
then there are $c,C>0$ such that for all $k\geq k_0$ and $n$ such that $Ck_n^{-1}\leq \delta \leq 1$ and $|k-k_n|<\delta^{-1} k_n^{-1}$ and all $C^{p+1}$ simplicial triangulations, $\mathcal{T}$ with constant $\Upsilon$ (in the sense of Definition~\ref{d:Crtriang}), there is $v\in L^2(\Gamma)$ such that the Galerkin approximation, $v_h$, to $v$ in $\mathcal{P}_{\mathcal{T}}^p$, if it exists, satisfies
$$
\frac{\|v_h-v\|_{L^2(\Gamma_-)}}{\|(I-\Pi_{\mathcal{P}_{\mathcal{T}}^0})v\|_{L^2(\Gamma_-)}}\geq \frac 12 +  c \begin{cases} (hk_n)^{-p-1}, &  1\leq (hk_n)^{2(p+1)} \delta k_n,\\
(hk_n)^{p+1}\delta k_n,&\delta k_n(hk_n)^{2(p+1)}\leq 1,
\end{cases}
$$
where $h:=h(\mathcal{T})$. 
\end{theorem}

\begin{remark}[Discussion of Theorems \ref{t:diamond} and \ref{t:diamondNew}]

\

\begin{enumerate}
\item 
When $\Omega_-$ consists of two aligned squares (or rounded squares) -- where the trapping is the same nature (\emph{parabolic}) as for a four-diamond domain (albeit on a smaller set in phase space) -- 
the bound \eqref{e:fourDiamondsBound} is proved in \cite[Corollary 1.14]{ChSpGiSm:20}. Furthermore, from both quasimode considerations and 
the bounds of 
\cite{ChWu:13, Ch:18} in the setting of scattering by metrics of revolution, 
\cite{ChSpGiSm:20}
conjecture that for $\Omega_-$ as in Theorem~\ref{t:diamond}, and $\operator\in\{A_k,A_k'\}$
\beq\label{e:thisMustBeTrue}
\rhoB(k)\leq Ck.
\eeq
If the bound \eqref{e:thisMustBeTrue} holds, then Theorems~\ref{t:diamond} and \ref{t:diamondNew} show that Theorem~\ref{t:concretehBEM} is optimal in this case and $\mathcal{J}=\emptyset$ in Theorem \ref{t:diamond}. 
\item 
We emphasise that the phenomenon behind Theorems~\ref{t:diamond} and \ref{t:diamondNew} is \emph{not} sparse in frequency. Indeed, 
both Theorems~\ref{t:diamond} and \ref{t:diamondNew}
hold with $\delta=1$ for a set of $k$ with infinite measure (since $\sum_{n=1}^\infty n^{-1}=\infty$). Furthermore, although 
we prove Theorem~\ref{t:diamond} with $k_n=\sqrt{2}n$, it is easy to generalize our construction in the proof of Theorem~\ref{t:diamond} to take $k_n=\sqrt{m_n^2+\ell_n^2}$ for any $m_n,\ell_n\in\mathbb{Z}$ with $c<|\frac{m_n}{\ell_n}|<C$ and such that, if $m_n/\ell_n=p_n/q_n$ with $p_n$ and $q_n$ relatively prime, then $|q_n|\leq C$. 
\item 
Theorems~\ref{t:diamond} and \ref{t:diamondNew} are
 illustrated by numerical experiments in Figures~\ref{f:bReallyCoolPicture} and~\ref{f:aReallyCoolPicture} (taken from \cite{GaRaSp:25}).
\end{enumerate}
\end{remark}

\begin{figure}[htbp]
\begin{center}
Piecewise constant elements

\begin{tikzpicture}
\node at(0,0){\includegraphics[width=.8\textwidth]{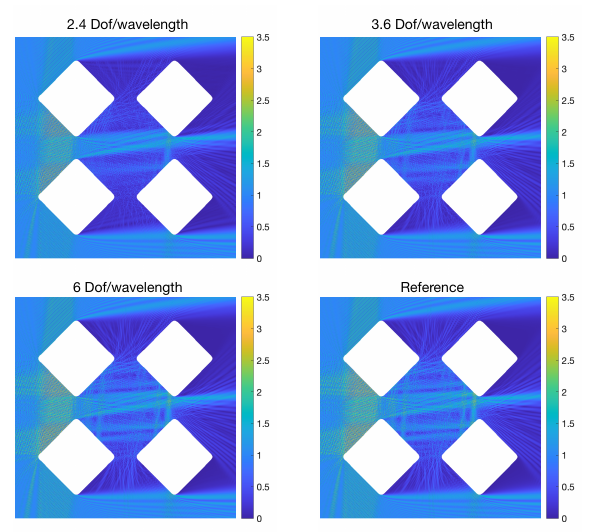}};
\end{tikzpicture}
\end{center}
\caption{\label{f:bReallyCoolPicture} 
The plots depict the absolute value of the total field at $k=40\sqrt{2}$ when the plane wave $e^{ik\langle \omega,x\rangle}$, with $\omega=(\cos (5\pi/180), \sin(5\pi/180))$, is incident on a sound-soft domain consisting of four nearly square obstacles.
The plot shows the Galerkin solutions 
for the BIE involving $A_k$ at 2.4 (top left), 3.6 (top right), and 6 (bottom left) degrees of freedom per wavelength and piecewise constant elements (i.e., $p=0$), with the reference solution (bottom right) computed with $p=11$.
These numbers of degrees of freedom per wavelength correspond to before, on, and after the peak of the quasioptimality constant in Figure~\ref{f:aReallyCoolPicture}. Many of the features of the Galerkin and reference plots are similar. However, in the Galerkin solutions with low numbers of points per wavelength the trapped rays are much less pronounced. 
The plots for piecewise linear elements (i.e., $p=1$) look qualitatively similar. }
\end{figure}

\begin{figure}[htbp]
\begin{center}
\begin{tikzpicture}
\node at(0,0){\includegraphics[width=1\textwidth]{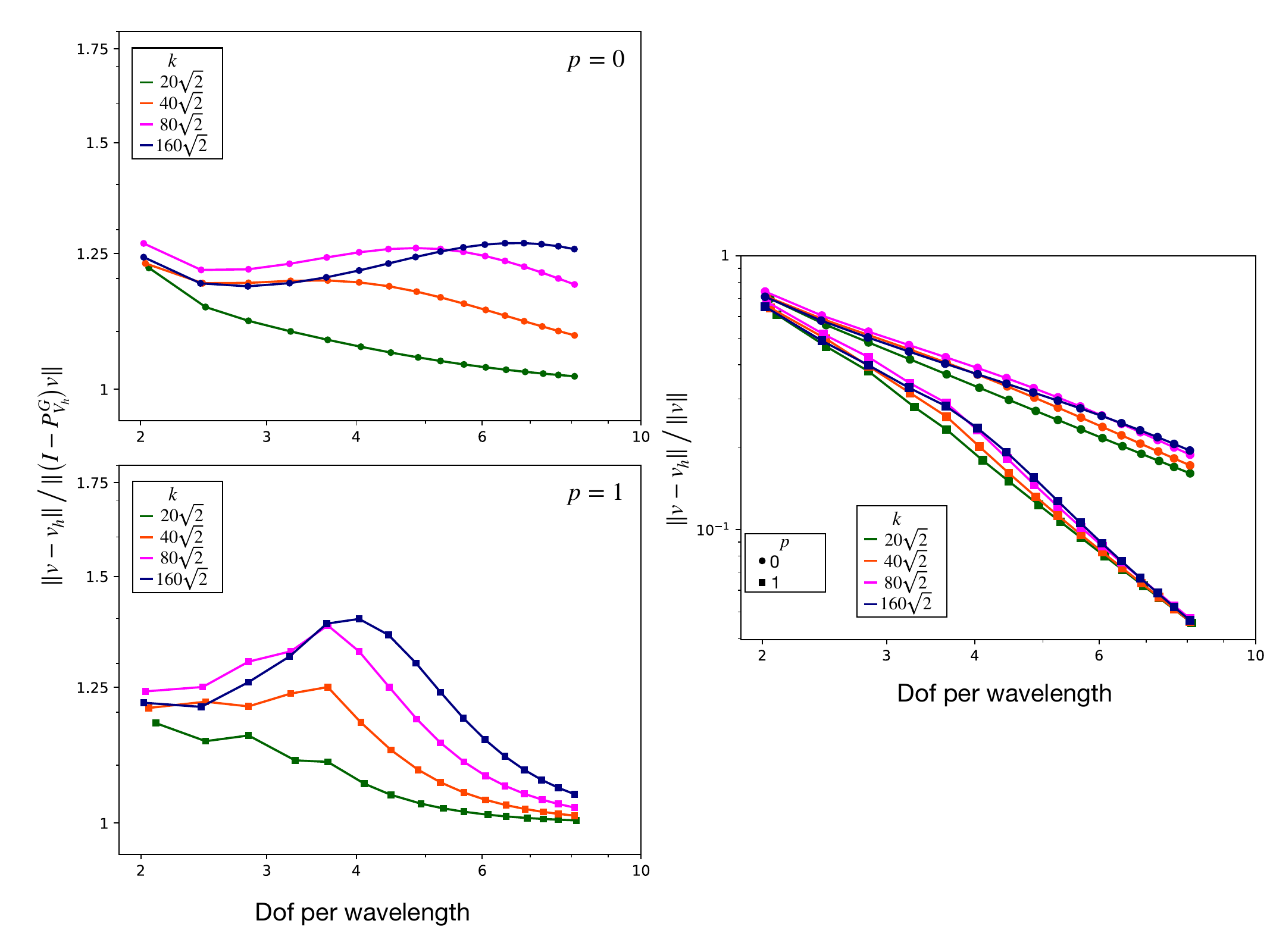}};
\end{tikzpicture}
\end{center}
\caption{\label{f:aReallyCoolPicture} The left plots shows the implied quasioptimality constant
for piecewise constant (top) and piecewise linear (bottom) elements when the plane wave $e^{ik\langle \omega,x\rangle}$ with $\omega=(\cos (5\pi/180), \sin(5\pi/180))$  is incident on a sound-soft domain consisting of four nearly square obstacles (see Figure~\ref{f:bReallyCoolPicture}),
with this problem solved using the BIE involving $A_k$.
  The right plot shows the corresponding $L^2$ relative error in the density. The reference and computed scattered solutions at $k=40\sqrt{2}$ are shown in Figure~\ref{f:bReallyCoolPicture}. The presence of pollution in this example can be seen from both plots:~in the left hand plot, one observes a peak in the implied quasioptimality constant that shifts up and to the right as $k$ increases. In the right hand plot, the error increases for a fixed number of degrees of freedom per wavelength as $k$ increases.}
\end{figure}

\section{$h$-FEM with non-uniform  meshes defined by ray-dynamics (\ref{R4})
}
\label{s:R4}

In this section, we consider an $H^\infty$ Helmholtz problem (in the sense of Definition \ref{d:helmholtzProblem}) with $\Gamma_{\rm p}=\emptyset$. This goal is to present a simple version of the results of~\cite{AGS2}. We  do not give complete proofs; instead, we 
give a sketch proof, 
indicating where extra input is needed.

To begin we recall the notation $\Omega_{\cavity}$, $\Omega_{\visible}$, $\Omega_{\invisible}$, and $\Omega_{\pml}$ for respectively a neighbourhood of the cavity, the visible set, the invisible set, and the PML from \S\ref{s:informalR4} (see also Figure~\ref{f:domainsTrappingMeshing}.)

Given a triangulation, $\mathcal{T}$ of $\Omega$, we also recall the meshwidth parameters $h_\star$, $\star\in \{\cavity,\visible,\invisible,\pml\}$, and $h=\max_{T\in\mathcal{T}}h_{T}$ from~\eqref{e:meshWidths}.
 In addition, we define the following measure of local uniformity of the triangulation at scale $\e>0$:
$$
U(\mathcal{T},\e):= \sup_{x\in \Omega}\underset{\substack{T_1,T_2\in\mathcal{T}\\T_1\cap B(x,\e)\neq \emptyset\\T_2\cap B(x,\e)=\emptyset}}{\sup}\frac{\diam(T_1)}{\diam(T_2)}.
$$
A triangulation $\cT$ is \emph{quasiuniform at scale $\e$ with constant $\gamma_0>0$} if $U(\cT,\e) \leq \gamma_0^{-1}$.

Define the matrices
\begin{gather*}
\mathcal{C}:=\begin{pmatrix} \rho(k)&\sqrt{k \rho(k)}&0&0\\
\sqrt{k\rho(k)}&k&k&0\\
0&k&k&0\\
0&0&0&1
\end{pmatrix},
\quad 
\cH:=\begin{pmatrix} h_\cavity  &0&0&0\\
0&h_\visible &0&0\\
0&0&h_\invisible&0\\
0&0&0&h_\pml\end{pmatrix},
\end{gather*}
\begin{gather*}
\mathscr{F}:=\begin{pmatrix}
1 &1&1& 1 \\
1 &1&1& 1\\
1 &1&1& 1\\
1 &1&1& 1
\end{pmatrix},
\\
\transferIntro:=\begin{pmatrix} 1&(h_\visible k)^{2p}\sqrt{k\rho(k)}&(h_\visible k)^{2p}\sqrt{k \rho(k)}( h_\invisible k)^{2p}k&
0\\
(h_\cavity  k)^{2p}\sqrt{k\rho(k)}&1&(h_\invisible k)^{2p}k&0\\
(h_\cavity  k)^{2p}\sqrt{k\rho}(h_\visible k)^{2p}k&(h_\visible k)^{2p}k&1&0\\
0
&0&0&1\end{pmatrix}.
\end{gather*}
Conceptually, $\mathcal{C}$ is the norm of the localised data-to-solution map (with $\mathcal{C}$ standing for ``communication") and $\mathscr{T}$ controls the propagation of Galerkin errors between subdomains according to the graph in 
Figure \ref{f:graph2} (with a simplified version -- Figure \ref{f:graph1} -- given in the sketch of the proof in \S\ref{s:sketch}). 

\begin{theorem}
\label{t:R4}
Suppose $a$ is the sesqulinear form defining an $H^\infty$ Helmholtz problem approximated using a $C^\infty$ PML with $C^{1,1}$ truncation boundary. Let $k_0,N, \Upsilon>0$, $p\in \mathbb{N}\setminus\{0\}$, $\mathcal{J}\subset \mathbb{R}_+$ such that 
$$
\sup\big\{ k^{-N}\rho(k)\,:\, k\in (k_0,\infty)\setminus \mathcal{J}\big\}<\infty
$$
and
let $\Oursubset{\star}$ be compactly contained in  $\Omega_{\star}$ with respect to the subspace topology of $\overline{\Omega}$, $\star\in\{ \cavity,\visible,\invisible,\pml\}$. 

There exist $c,C>0$ such that 
for all $k\in (k_0,\infty)\setminus \mathcal{J}$, all affine-conforming $C^{p+1}$ simplicial triangulations with constant $\Upsilon$, $\mathcal{T}$ (in the sense of Definition~\ref{d:AffineConforming}), that are quasiuniform at scale $(\Upsilon+k)^{-1}$
and satisfy 
\begin{equation}
\label{e:meshConditions}
(h_\cavity k)^{2p}\rho(k)+(h_\visible k)^{2p}k+(h_\invisible k)^{2p}k+(h_\pml k)^{2p}\leq c,
\end{equation}
all $w_{h,\star} \in V_{\mathcal{T}_{k}}^p$, with $\star\in\{ \cavity,\visible,\invisible,\pml\}$, and all $u\in H_k^1(\Omega)$,
the Galerkin approximation, $u_h\in \mathcal{P}_{\mathcal{T}}^p$, to $u$ exists, is unique, and satisfies, for $0\leq \newell \leq p$,
 \renewcommand{\jot}{1pt}
 \begin{equation}
\begin{aligned}\label{e:simple}
&\left(
\begin{aligned}
&\|u-u_h\|_{H_k^{1-\newell}(\Omega_\cavity')}\\
&\|u-u_h\|_{H_k^{1-\newell}(\Omega_\visible')}\\
&\|u-u_h\|_{H_k^{1-\newell}(\Omega_\invisible')}\\
&\|u-u_h\|_{H_k^{1-\newell}(\Omega_\pml')}
\end{aligned}
\right)\\
&
\qquad\leq C
\Big[(\cH k)^{\newell}+
\transferIntro\mathcal{C} (\cH k)^p
+ k^{-N}(hk)^m \mathscr{F}
\Big]
\left(\begin{aligned}
&\|u-w_{h,\cavity}\|_{H_k^{1}(\Omega_\cavity)}\\
&\|u-w_{h,\visible}\|_{H_k^{1}(\Omega_\visible)}\\
&\|u-w_{h,\invisible}\|_{H_k^{1}(\Omega_\invisible)}\\
&\|u-w_{h,\pml}\|_{H_k^{1}(\Omega_\pml)}
\end{aligned}
\right),
\end{aligned}
\end{equation}
\renewcommand{\jot}{3pt}
where the inequality in \eqref{e:simple} is understood component-wise.
\end{theorem}
\bre From the estimate~\eqref{e:simple} and the interpretation of $\mathscr{T}$ as the propagation of Galerkin errors, the matrix $\mathcal{C}(\cH k)^p$ should be viewed as mapping best approximation errors to Galerkin errors -- 
we discuss this interpretation in \S\ref{s:interpret}. 
\ere

\begin{remark}
Numerical experiments illustrating the consequences of Theorem~\ref{t:R4} 
summarised in Table \ref{tab:regimes}
can be found in~\cite[\S2]{AGS2}.
\end{remark}

For a uniform mesh ($\,h_{\cavity}=h_{\visible}=h_{\invisible}=h_{\pml}=h\,$)~\eqref{e:simple} 
recovers the result of Theorem~\ref{t:R1}. Indeed, for a uniform mesh with $(hk)^{2p}\rho$ sufficiently small, all the elements of the matrix $\transferIntro$ are bounded by a constant, and all the elements of $\mathscr{C}(\cH k)^p$ are bounded by $\rho(k)(hk)^p $. 
However, Theorem~\ref{t:R4} provides much more information than Theorem~\ref{t:R1}:~it describes how the best approximation errors and local width in each region affect the Galerkin error in all other regions. 

Theorem~\ref{t:R4} is most interesting when $\rho(k)\gg k$, which is equivalent to the problem being trapping, i.e., $\cavity\neq \emptyset$ (see \cite{BoBuRa:10}, \cite[Theorem 7.1]{DyZw:19}). In particular, Theorem~\ref{t:R4} shows that in the trapping case there exist triangulations with $(hk)^{p}\rho(k)\gg 1$ whose finite-element solutions have guaranteed $k$-uniform quasioptimality. Even when $\cavity=\emptyset$, Theorem \ref{t:R4} gives new information including that one needs only a fixed number of points per wavelength in the PML.

To compare with the estimate in Theorem~\ref{t:R1} on relative error
(i.e., \eqref{e:rel_error}), we state the following corollary of Theorem~\ref{t:R4} which follows from standard piecewise-polynomial approximation estimates (see, e.g., \S\ref{app:A}).

\begin{corollary}
\label{c:relError}
Suppose $a$ is the sesqulinear form defining an $H^\infty$ Helmholtz problem approximated using a $C^\infty$ PML with $C^{1,1}$ truncation boundary. Let $k_0,N, \Upsilon$, $p\in \mathbb{N}\setminus\{0\}$, $\mathcal{J}\subset \mathbb{R}_+$ such that 
$$
\sup\big\{ k^{-N}\rho(k)\,:\, k\in (k_0,\infty)\setminus \mathcal{J}\big\}<\infty,
$$
and
let $\Oursubset{\star}$ be compactly contained in  $\Omega_{\star}$ with respect to the subspace topology of $\overline{\Omega}$, $\star\in\{ \cavity,\visible,\invisible,\pml\}$. 

There exist $c,C>0$ such that 
for all $k\in (k_0,\infty)\setminus \mathcal{J}$, all affine-conforming $C^{p+1}$ simplicial triangulations with constant $\Upsilon$, $\mathcal{T}$, that are quasiuniform at scale $(\Upsilon+k)^{-1}$ with constant $\Upsilon$ and satisfy~\eqref{e:meshConditions}
 all $w_{h,\star} \in V_{\mathcal{T}_{k}}^p$, with $\star\in\{ \cavity,\visible,\invisible,\pml\}$, and all $u\in H_k^1(\Omega)$,
the Galerkin approximation, $u_h\in \mathcal{P}_{\mathcal{T}}^p$, to $u$ exists, is unique, and satisfies, for $0\leq \newell \leq p$,

 \renewcommand{\jot}{1pt}
\begin{equation}
\begin{aligned}\label{e:relSimple}
&\left(\begin{aligned}
&\|u-u_h\|_{H_k^{1-\newell}(\Omega_\cavity')}\\
&\|u-u_h\|_{H_k^{1-\newell}(\Omega_\visible')}\\
&\|u-u_h\|_{H_k^{1-\newell}(\Omega_\invisible')}\\
&\|u-u_h\|_{H_k^{1-\newell}(\Omega_\pml')}
\end{aligned}
\right)\\
&\qquad
\leq C
\Big[(\cH k)^{\newell}+
\transferIntro \mathcal{C}(\cH k)^p
+ k^{-N}(hk)^m \mathscr{F}
\Big](\cH k)^{p}
\left(\begin{aligned}
&\|u\|_{H_k^{p+1}(\Omega_\cavity)}\\
&\|u\|_{H_k^{p+1}(\Omega_\visible)}\\
&\|u\|_{H_k^{p+1}(\Omega_\invisible)}\\
&\|u\|_{H_k^{p+1}(\Omega_\pml)}
\end{aligned}
\right).
\end{aligned}
\end{equation}
 \renewcommand{\jot}{3pt}
\end{corollary}
When the data, $f$, is $k$-oscillatory in the sense of Lemma~\ref{l:osci}, so is the solution $u$, and in this case, $\|u\|_{H_k^{p+1}(U')}\leq C\|u\|_{H_k^1(U)}$ for $U'\Subset U$. Hence, one can use~\eqref{e:relSimple} to find triangulations with $(hk)^{2 p}\rho(k) \gg 1$ that nevertheless have guaranteed control on the relative error.

\bre[Improvements in Theorem 3.11~\cite{AGS2}]
Theorem 3.11 in~\cite{AGS2} is stronger than Theorem \ref{t:R4} in that it considers arbitrary covers of $\Omega$, and bounds the high ($\gg k$) and low ($\lesssim k$) frequencies of the Galerkin error separately.
Two situations in which a more complicated cover is advantageous are the following. 1) There are two or more cavities that are dynamically separated, i.e., for which there is no billiard trajectory whose closure insects both cavities. 2) One has a priori information about the data and/or solution and hence can obtain good control on the right-hand side of~\eqref{e:simple}. Even when $\cavity=\emptyset$, such information combined with Theorem \ref{t:R4} allows one to define triangulations with a priori improved accuracy in some regions, without the need to choose a small meshwidth everywhere.
\ere

\subsection{Discussion of the ideas behind Theorem~\ref{t:R4} and a sketch of the proof}

\subsubsection{The ideas behind Theorem~\ref{t:R4}}

The following two important phenomena motivate Theorem \ref{t:R4}.
\begin{enumerate}
\item \emph{The solution operator $\mathcal{R}$ encodes the billiard dynamics in $\Omega$.} In particular, for $\chi_1,\chi_2\in C^\infty(\Omega)$ the operator $\chi_1\mathcal{R}\chi_2$ behaves differently depending on the locations of $\supp \chi_j$; e.g., $\|\mathcal{R}\|\gg k$ when $\Omega_{\cavity}\neq \emptyset$, but if both $\chi_1$ and $\chi_2$ are away from $\Omega_{\cavity}$ then $\|\chi_1\mathcal{R}\chi_2\|_{L^2\to L^2}\lesssim k$.  
\item \emph{The Galerkin error propagates.} The best possible situation would be local quasioptimality i.e., there exists $C>0$ such that the Galerkin solution $u_h$ satisfies, for every $U\subset\Omega$,
\begin{equation}
\label{e:localQO}
\|u-u_h\|_{H^1_k(U)}\leq C\inf_{w_h\in V_{\mathcal{T}}^p}\|u-w_h\|_{H^1_k(U)}.
\end{equation}
In this case, since approximation of oscillatory functions by piecewise polynomials is well understood (see \cite{Ga:22} for lower bounds and Appendix~\ref{a:poly} for upper bounds), 
the properties of the data and behaviour of $\mathcal{R}$ would dictate the meshwidth in each region. 
Unfortunately~\eqref{e:localQO} cannot hold for general triangulations. Indeed, suppose that \eqref{e:localQO} holds and let $\phi,\phi_1,\phi_2\in C^\infty(\Omega)$ be such that $\supp \phi\subset \{\phi_1\equiv 1\}$, $\phi\neq 0$, and $\phi_1+\phi_2\equiv 1$ on $\Omega$. Then,
\begin{equation}
\label{e:errorMoving}
\begin{aligned}
\phi(u-u_h)&=\sum_{j=1}^2\phi \mathcal{R}\phi_j P_\theta(u-u_h).
\end{aligned}
\end{equation}
We now consider a situation where $\mc{T}$ has arbitrarily small elements on $\supp \phi_1=:\Omega_1$ so that, by~\eqref{e:localQO}, 
$$
\|u-u_h\|_{H^1_k(\Omega_1)}\leq C\inf_{w_h\in V_{\mathcal{T}}^p}\|u-w_h\|_{H^1_k(\Omega_1)}\ll 1 .
$$
In particular,
$$
\|\phi(u-u_h)\|_{H^1_k}+\|\mathcal{R}\phi_1P_\theta(u-u_h)\|_{H^1_k}\ll 1
$$
(by continuity of $P_\theta$ and $\mathcal{R}$ and locality of $P_\theta$).
Then,~\eqref{e:errorMoving} implies that
$$
\|  \phi \mathcal{R}\phi_2P_\theta(u-u_h)\|_{H^1_k}\ll 1,
$$
which cannot be true unless the meshwidth is also sufficiently small on $\Omega_2$ or $\phi \mathcal{R}\phi_2\approx 0$. By Item 1, the latter is not the case whenever $\supp \phi$ and $\supp \phi_2$ are connected by a billiard trajectory. (For a striking illustration of this propagation of error, see~\cite[Figure 3]{AvGaSp:24}.)

This argument indicates, not only that the Galerkin error propagates, but that the norm of the operator $\phi \mathcal{R}\phi_2$ determines the strength of propagation from $\supp \phi_2$ to $\supp \phi$. 
\end{enumerate}
Item 1 motivates varying the meshwidth from one location to another, but Item 2 shows that, to be effective, this strategy must 
take into account the global behaviour of billiard trajectories. 
In particular, by Item 2, the error in the cavity is \emph{not} just dictated by the meshwidth in the cavity
-- the meshwidth also needs to be sufficiently small away from the cavity to control the propagating error.

\subsubsection{Sketch of the proof of Theorem~\ref{t:R4}}
\label{s:sketch}

For simplicity, we consider here the bound \eqref{e:simple} with $m=p$ and ignore improvements that are possible in the overlaps between subdomains, in the PML region, and by splitting the frequencies of the Galerkin error into those $\gg k$ and $\lesssim k$. 

The proofs of Theorem~\ref{t:R4} is,  at heart, a localised version of the elliptic projection-type argument used to prove Theorem~\ref{t:abstractHelmholtz}, and we first recap this argument. The key insight 
is the existence of a self-adjoint smoothing operator $S_k$ so that $P_{S_k}:= P_\theta+S_k$ is coercive 
(uniformly in $k$) and for all $N$ there exists $C>0$ such that for $k\geq k_0$ 
\beq\label{e:Sreg}
\|S_k\|_{H_k^{-N}\to H_k^N}\leq C
\eeq
(see Theorem~\ref{t:abstractHelmholtz} for the definition of the operator $S_k$).
Since $P_{S_k}$ is coercive, by Lemma~\ref{l:cea} there is an \emph{elliptic projection} $\Pi_{S_k}:H_k^1\to \mathcal{P}_{\mathcal{T}}^p$ such that 
\beq\label{e:introPiSharp}
\big\langle P_{S_k} w_h, (I- \Pi_{S_k})u\big\rangle =0 \quad\tfa w_h \in \mathcal{P}_{\mathcal{T}}^p
\eeq
and there exists $C>0$ such that for all $k>k_0$
\beq\label{e:introPiSharpCea}
\|(I - \Pi_{S_k})v\|_{H^1_k} \leq C \inf_{w_h \in \mathcal{P}_{\mathcal{T}}^p} \|v - w_h\|_{H^1_k}
\eeq
(i.e., $\Pi_{S_k}$ is the adjoint Galerkin projection associated to $P_{S_k}$). Moreover, by an Aubin--Nitsche-type duality argument (see~\eqref{e:aSLowNorm} )
\beq
\label{e:lowNormPiSharp}
\|(I - \Pi_{S_k})v\|_{(H^{p-1}_k)^*} \leq C(hk)^p\inf_{w_h \in \mathcal{P}_{\mathcal{T}}^p} \|v - w_h\|_{H^1_k}.
\eeq

By~\eqref{e:introPiSharp} and Galerkin orthogonality~\eqref{e:galerkinOrthogonality}, for all $w_h\in \mathcal{P}_{\mathcal{T}}^p$, $v\in H_k^{p-1}$,
\begin{equation}
\label{e:basicEllipticProjection}
\begin{aligned}
\langle u-u_h,v\rangle 
&= \big\langle P_\theta(u-u_h) , \mathcal{R}^*v \big\rangle\\
&= \big\langle P_\theta(u-u_h), (I-\Pi_{S_k})\mathcal{R}^* v \big\rangle \\
&= \big\langle P_{S_k}(u-u_h), (I-\Pi_{S_k}) \mathcal{R}^* v \big\rangle - 
\big\langle S_k(u-u_h), (I-\Pi_{S_k}) \mathcal{R}^* v \big\rangle \\
&= \big\langle P_{S_k}(u-w_h), (I-\Pi_{S_k}) \mathcal{R}^*v\big\rangle - 
\big\langle S_k(u-u_h), (I-\Pi_{S_k}) \mathcal{R}^* v\big\rangle.
\end{aligned}
\end{equation}
By~\eqref{e:introPiSharpCea},~\eqref{e:lowNormPiSharp}, and the mapping properties $S_k:(H_k^{p-1})^*\to H_k^{p-1}$ and $P_{S_k}:H_k^1\to (H_k^{1})^*$,
\begin{equation}
\label{e:preasymptoticSketch1}
\begin{aligned}
&|\langle u-u_h,v\rangle|\\
&\quad\leq C\eta_{_{H_k^{p-1}\to H_k^1}}(\mathcal{P}_{\mathcal{T}}^p)\Big( \inf_{w_h\in \mathcal{P}_{\mathcal{T}}^p}\|u-w_h\|_{H_k^1}+(hk)^p\|u-u_h\|_{H_k^{-p+1}}\Big)\|v\|_{H_k^{p-1}},
\end{aligned}
\end{equation}
where $\eta_{_{H_k^{p-1}\to H_k^1}}(\mathcal{P}_{\mathcal{T}}^p)$ is defined by \eqref{e:eta}.
 By duality,~\eqref{e:preasymptoticSketch1} implies
\beq \label{e:theFirstSystemSimple}
\begin{gathered}
 \|u-u_h\|_{(H_k^{p-1})^*}\leq C \Big( b\inf_{w_h\in \mathcal{P}_{\mathcal{T}}^p}\|u-w_h\|_{H_k^1}+\omega\|u-u_h\|_{(H_k^{p-1})^*}\Big),\\
 b:=\eta_{_{H_k^{p-1}\to H_k^1}}(\mathcal{P}_{\mathcal{T}}^p),\qquad \omega:=(hk)^{p}\eta_{_{H_k^{p-1}\to H_k^1}}(\mathcal{P}_{\mathcal{T}}^p).
 \end{gathered}
\eeq 
By Lemma~\ref{l:adjointApproximability} and the fact that $\rho(k)\geq ck$,
 $$
\eta_{_{H_k^{p-1}\to H_k^1}}(\mathcal{P}_{\mathcal{T}}^p) \leq C(hk)^p\big(1+\|\mathcal{R}\|_{L^2\to L^2}\big)= C(hk)^p\big(1+\rho(k)\big)\leq C(hk)^p\rho(k).
 $$
Thus, from \eqref{e:theFirstSystemSimple}, when $(hk)^{2p}\rho(k)$ is sufficiently small, $(1-C\omega)^{-1}$ exists and is positive, and then
 \begin{equation}
 \label{e:martinIsHappy}
 \|u-u_h\|_{(H_k^{p-1})^*}
 \leq C(1-C\omega)^{-1}b\inf_{w_h\in \mathcal{P}_{\mathcal{T}}^p}\|u-w_h\|_{H_k^1}\leq C(hk)^p\rho(k)\inf_{w_h\in \mathcal{P}_{\mathcal{T}}^p}\|u-w_h\|_{H_k^1},
 \end{equation}
which is the preasymptotic estimate 
\eqref{e:lowfreqbound} for $n=p$.

We now sketch the localised version of the above argument, which is used to prove Theorem~\ref{t:R4}.

\paragraph{Pseudolocality for $S$ and $\Pi_{S_k}$.}
By Lemma~\ref{l:spatialPseudolocalFinal}
for $\chi,\psi\in \overline{C^\infty}(\Omega)$ with $\supp \chi \cap \supp \psi=\emptyset$, 
\begin{equation}
\label{e:sPseudoloc}
\|\chi S_k\psi\|_{(H_k^{N})^*
\to H_k^N}+\|\chi \mc{R}_S\psi\|_{
(H_k^{N})^*
\to H_k^N}=O(k^{-\infty}).
\end{equation}

We also need pseudolocality of $\Pi_{S_k}$; i.e., for $\chi,\psi\in \overline{C^\infty}(\Omega)$ with $\supp \chi \cap \supp \psi=\emptyset$,
\begin{equation}
\label{e:piSPseudoloc}
\|\chi(I -\Pi_{S_k})\psi u\|_{H_k^1}\leq Ck^{-N}\|(I-\Pi_{S_k})\psi u\|_{(H_k^{p})^*}   .
\end{equation}
The proof of~\eqref{e:piSPseudoloc} is given in~\cite[\S7]{AGS2} and we only review the steps of the argument here.

\noindent{\bf Step 1:} Show that if $\chi_+\in \overline{C^\infty}(\Omega)$ with $\supp \chi\cap\supp (1-\chi_+)=\emptyset$ and $\supp \chi_+\cap\supp \psi=\emptyset$, then
$$
\|\chi(I-\Pi_{S_k})\psi u\|_{H_k^1}\leq C k^{-N}\big(\|\chi_+(I-\Pi_{S_k})\psi u\|_{L^2}+\|(I-\Pi_{S_k})\psi u\|_{H_k^{-N}}\big)
$$
\cite[Lemma 7.4]{AGS2}. 
The proof of this fact follows by using coercivity of $P_{S_k}$, the definition of $\Pi_{S_k}$, superapproximation and inverse estimates in $\mathcal{P}_{\mathcal{T}}^p$,  pseudolocality of $S$, and the boundedness of $S:(H_k^{N})^*\to H_k^N$. 

\noindent{\bf Step 2:} Show that for $\chi,\psi\in \overline{C^\infty}(\Omega)$ with $\supp \chi\cap \supp \psi=\emptyset$, 
$$
\|\chi(I-\Pi_{S_k})\psi u\|_{L^2}\leq C\|(I-\Pi_{S_k})\psi u\|_{H_k^{-N}}
$$
\cite[Lemma 7.5]{AGS2}. 
The proof of this estimate uses a Schatz duality type argument for $P_{S_k}$ and a more sophisticated version of~\eqref{e:piSPseudoloc} that controls how $\|\chi (I-\Pi_{S_k})\psi u\|_{H_k^1}$ behaves when $k^{-1}\leq d(\supp \chi,\supp\psi)\ll 1$.

\paragraph{Sketch Proof of Theorem~\ref{t:R4}}

To localise the elliptic-projection argument, we introduce an open cover of $\Omega$,  $\{\Omega_j\}_{j=1}^\domainnumber$ and $\{\phi_j\}_{j=1}^\domainnumber\subset \overline{C^\infty}(\Omega)$ a partition of unity subordinate to this cover.
(In Theorem~\ref{t:R4}, $\domainnumber=4$ and $(\Omega_1,\Omega_2,\Omega_3,\Omega_4):=(\Omega_\cavity,\Omega_\visible,\Omega_\invisible,\Omega_\pml)$.)
Next, let $\chi_j\in \overline{C^\infty}(\Omega)$, $j=1,\dots,\domainnumber$ such that 
$$
\supp \chi_j \subset \Omega_j\cup\partial\Omega,\qquad \chi_j\equiv 1\text{ in a neighbourhood of }\supp \phi_j.
$$
Arguing as in~\eqref{e:basicEllipticProjection}, for all $w_{h,j}\in V_k$, $j=1,\ldots,\domainnumber$, and $v\in H_k^{p-1}$, we obtain
\begin{equation}
\label{e:startDuality}
\begin{aligned}
&\langle \chi_i(u-u_h) ,v\rangle \\ 
&= \big\langle P_\theta(u-u_h) , \mathcal{R}^* \chi_i v \big\rangle,\\
&= \big\langle P_{S_k}(u-u_h), (I-\Pi_{S_k}) \mathcal{R}^* \chi_iv \big\rangle - 
\big\langle S_k(u-u_h), (I-\Pi_{S_k}) \mathcal{R}^* \chi_iv \big\rangle, \\ 
&=\sum_{j=1}^{\domainnumber} \bigg(\big\langle P_{S_k}(u-w_{h,j}), (I-\Pi_{S_k}) \phi_j \mathcal{R}^* \chi_iv\big\rangle - 
\big\langle S_k(u-u_h), (I-\Pi_{S_k}) \phi_j \mathcal{R}^* \chi_i v\big\rangle \bigg),
\\ 
&=\sum_{j=1}^{\domainnumber} \bigg(\big\langle P_{S_k}\chi_j(u-w_{h,j}), \chi_j(I-\Pi_{S_k}) \phi_j \mathcal{R}^* \chi_iv\big\rangle - 
\big\langle S_k\chi_j(u-u_h), \chi_j(I-\Pi_{S_k}) \phi_j \mathcal{R}^* \chi_i v\big\rangle \\
&\qquad +Ck^{-N}\big(\|u-u_h\|_{(H_k^{N})^*}+\|u-w_{h,j}\|_{H_k^1}\big)\|(I-\Pi_{S_k})\phi_j\mathcal{R}^*\chi_iv\|_{(H_k^p)^*})\bigg),
\end{aligned}
\end{equation}
Next, we need a local analogue  of~\eqref{e:aSLowNorm}.
We claim that
\begin{equation}
\label{e:localLowNormPiSharp}
\begin{gathered}
\|\chi_j (I - \Pi_{S_k})v\|_{(H^{p-1}_k)^*} \leq C\big((h_jk)^p+k^{-N}(hk)^p\big)\inf_{w_h \in \mathcal{P}_{\mathcal{T}}^p} \|v - w_h\|_{H^1_k},\\ \text{ where } \,h_j:=\max_{\substack{K\in\mc{T}\\K\cap \Omega_j\neq\emptyset}}h_K.
\end{gathered}
\end{equation}

To prove~\eqref{e:localLowNormPiSharp}, let $\chi_j^+\in C_c^\infty(\Omega_j\cup \partial\Omega)$ with $\supp (1-\chi_j^+)\cap \supp \chi_j=\emptyset$. Then, by \eqref{e:introPiSharp}, pseudolocality of $\cR_S$ \eqref{e:sPseudoloc}, the fixed-regularity polynomial-approximation results from, e.g., Theorem \ref{t:approxHighLowReg}, and \eqref{e:introPiSharpCea},
\begin{align*}
&|\langle w,  \chi_j (I - \Pi_{S_k})v\rangle|=|\langle P_{S_k}\mathcal{R}_{S_k}\chi_j w,  (I - \Pi_{S_k})v\rangle|\\
&\leq|\langle P_{S_k}\chi_j^+\mathcal{R}_{S_k}w,  (I - \Pi_{S_k})v\rangle|+|\langle P_{S_k}(1-\chi_j^+)\mathcal{R}_{S_k}\chi_jw,(I-\Pi_{S_k})v\rangle|,\\
&=|\langle P_{S_k}(\chi_j^+\mathcal{R}_{S_k}w-w_{h}),  (I - \Pi_{S_k})v\rangle|+|\langle P_{S_k}((1-\chi_j^+)\mathcal{R}_{S_k}\chi_jw-w_{h,1}),(I-\Pi_{S_k})v\rangle|,\\
&\leq C\Big((h_j k)^{p} + k^{-N}(hk)^p\Big)\|(I-\Pi_{S_k})v\|_{H_k^1}\|w\|_{H_k^{p-1}},\\
&\leq C\Big((h_j k)^{p} + k^{-N}(hk)^p\Big)\inf_{w_h\in\mathcal{P}_{\mathcal{T}}^p}\|v-w_h\|_{H_k^1}\|w\|_{H_k^{p-1}},
\end{align*}
and hence~\eqref{e:localLowNormPiSharp} follows by duality.

By 
\eqref{e:startDuality}, \eqref{e:localLowNormPiSharp} and the mapping properties $S_k:(H_k^{p-1})^*\to H_k^{p-1}$ \eqref{e:Sreg} and $P_{S_k}: H_k^1\to (H_k^{1})^*$,
\begin{equation}
\label{e:preasymptoticSketch2}
\begin{gathered}
\begin{aligned}
&|\langle \chi_i(u-u_h),v\rangle|\\
&\leq C\sum_j\eta_p(j\to i) \Bigg(\inf_{w_{h,j}\in \mathcal{P}_{\mathcal{T}}^p}\big(\|\chi_j(u-w_{h,j})\|_{H_k^1}+k^{-N}(hk)^p\|u-w_{h,j}\|_{H_k^1}\big)\\
&\quad+\big((h_jk)^p+k^{-N}(hk)^p\big)\|\chi_j(u-u_h)\|_{(H_k^{p-1})^*}+k^{-N}(hk)^p\|u-u_h\|_{(H_k^N)^*}\Bigg)\|v\|_{H_k^{p-1}},
\end{aligned}\\
\text{ where } \quad\eta_p(j\to i):=\sup_{0\neq v\in H_{k}^{m-1}}\inf_{w_h\in \mathcal{P}_{\mathcal{T}}^p}\frac{ \|\chi_j\mathcal{R}^*\chi_iv-w_h\|_{H_k^1}}{\|v\|_{H_k^{p-1}}}
\end{gathered}
\end{equation}
(compare to \eqref{e:preasymptoticSketch1}). 
 By duality,~\eqref{e:preasymptoticSketch2} implies that
 \begin{align*}
& \|\chi_i(u-u_h)\|_{(H_k^{p-1})^*}\\
&\leq C\sum_j\eta_p(j\to i) \Bigg(\inf_{w_{h,j}\in \mathcal{P}_{\mathcal{T}}^p}\big(\|\chi_j(u-w_{h,j})\|_{H_k^1}+k^{-N}(hk)^p\|u-w_{h,j}\|_{H_k^1}\big)\\
&\qquad+\big((h_jk)^p+k^{-N}(hk)^p\big)\|\chi_j(u-u_h)\|_{(H_k^{p-1})^*}+k^{-N}(hk)^p\|u-u_h\|_{(H_k^N)^*}\Bigg).
 \end{align*}
 We then use a 
local version of the frequency-splitting argument used to prove Lemma 
\ref{l:adjointApproximability} 
 (see~\cite[Lemma 8.5]{AGS2}) to obtain
 $$
 \eta_p(j\to i) \leq C(h_jk)^p\|\chi_j\mathcal{R}^*\chi_i\|_{L^2\to L^2}+C1_{\{\Omega_i\cap \Omega_j\neq \emptyset\}}(h_{ij}k)^p,
 $$
 where $h_{ij}:=\min (h_i,h_j)$.
 In particular, this yields the system of inequalities
 \begin{align} \nonumber
&\big(\| \chi_i(u-u_h)\|_{(H_k^{p-1})^*}\big)_{i=1}^\domainnumber
\\ \nonumber
&\quad \leq CB \Big(\inf_{w_{h,j}\in \mathcal{P}_{\mathcal{T}}^p}\big(\|\chi_j(u-w_{h,j})\|_{H_k^1}+k^{-N}(hk)^p\|u-w_{h,j}\|_{H_k^1}\big)\Big)_{j=1}^\domainnumber\\
&\qquad +C\oldT\big(\|\chi_j(u-u_h)\|_{(H_k^{p-1})^*}\big)_{j=1}^\domainnumber,
\label{e:theFirstSystem}
\end{align}
 where 
 \begin{align*}
 B_{ij}&:=\eta_p(j\to i)\leq C(h_jk)^p\|\chi_j\mathcal{R}^*\chi_i\|_{L^2\to L^2}+C1_{\{\Omega_i\cap \Omega_j\neq \emptyset\}}(h_{ij}k)^p,\\
 \oldT_{ij}&:=(h_jk)^p\eta_p(j\to i)+k^{-N}(hk)^p\sum_{\ell}\eta_p(\ell\to i),\\
 &\leq C(h_jk)^{2p}\|\chi_j\mathcal{R}^*\chi_i\|_{L^2\to L^2}+C(h_jk)^p1_{\{\Omega_i\cap \Omega_j\neq \emptyset\}}(h_{ij}k)^p +Ck^{-N}(hk)^{2p}
 \end{align*}
 (compare to \eqref{e:theFirstSystemSimple}). 
 If
\begin{equation}
\label{e:conditionW}
\sum_{n=0}^\infty (C W)^n<\infty,
 \end{equation}
then $(I-C\oldT)^{-1}$ exists and has non-negative entries. Hence~\eqref{e:theFirstSystem} implies that
 $$
\|  \underline{u-u_h}\|_{H_k^{-p+1}}\leq C(I-C\oldT)^{-1}B \|  \underline{u-w_h}\|_{H_k^1}.
 $$
 (compare to~\eqref{e:martinIsHappy}).
 
To understand when $\sum_n (CW)^n$ converges, consider $W$ as the weighted adjacency matrix of a directed graph with $\domainnumber$ nodes representing $\{\Omega\}_{j=1}^{\domainnumber}$. Observe that the entry in the $i^{\text{th}}$ row and $j^{\text{th}}$ column of $(CW)^{\ell}$ is given by $C^\ell$ times the sum of the weights over all paths of length $\ell$ from $j$ to $i$ in this graph. Hence, the sum converges if for any $i$ and $j$ the sum of the weights of all paths from $j$ to $i$ multiplied by $C^{\text{path length}}$ is finite. Using elementary graph analysis this condition can be reduced to the requirement that all the sum of such weights for all non-self intersecting loops is less than 1 (see~\cite[Appendix B]{AGS2}).

In the setting of Theorem~\ref{t:R4}, $\domainnumber=4$ and $(\Omega_1,\Omega_2,\Omega_3,\Omega_4):=(\Omega_\cavity,\Omega_\visible,\Omega_\invisible,\Omega_\pml)$. For $k\notin \mc{J}$,~\cite[\S4, Appendix C]{AGS2} obtains the bounds of Table~\ref{ta:resolve} on $\psi \mathcal{R}^*\chi$ according to the support of $\psi$ and $\chi$. 
\begin{table}[h]
\begin{center}
\renewcommand{\arraystretch}{1.5}
\begin{tabular}{|c|c|c|c|c|}
\hlineb
$\supp \psi\Big\backslash \supp \chi$&$\Omega_{\cavity}$&$\Omega_{\visible}$&$\Omega_{\invisible}$&$\Omega_\pml$\\
\hlineb
$\Omega_{\cavity}$&$\rho(k)$&$\sqrt{k\rho(k)}$&$O(k^{-\infty})$&$O(k^{-\infty})$\\
\hlineb
$\Omega_{\visible}$&$\sqrt{k\rho(k)}$&$k$&$k$&$1$\\
\hlineb
$\Omega_{\invisible}$&$O(k^{-\infty})$&$k$&$k$&$1$\\
\hlineb
$\Omega_{\pml}$&$O(k^{-\infty})$&$1$&$1$&$1$\\
\hlineb
\end{tabular}
\renewcommand{\arraystretch}{1}
\caption{\label{ta:resolve}Bounds on $\|\psi \mathcal{R}^*\chi\|_{L^2\to L^2}$ (up to $k$-independent constants) proved in~\cite[\S4]{AGS2} for $k\notin \mc{J}$.
For the scattering problem, i.e., without PML truncation, 
the upper bounds by $\sqrt{k\rho(k)}$ are given by Theorem \ref{l:measureAwayTrapping}, and the upper bounds by $k$ are given by Theorem \ref{l:everythingawaytrapping}.
}
\end{center}
\end{table}
As a result, the graph corresponding to $W$ is the one in Figure~\ref{f:graph1}, and the requirement that the sum of weights on all non-self intersecting loops be less than 1 reduces to~\eqref{e:meshConditions}. 

To obtain Theorem~\ref{t:R4}, two further improvements to the analysis described above are necessary.
\begin{enumerate}
\item One must treat the PML more carefully to see that, effectively, nothing propagates in or out of the PML (see the differences between Figures~\ref{f:graph1} and~\ref{f:graph2}).
\item The analysis given above obtains estimates in the lowest available norm, $(H_k^{p-1})^*$, whereas Theorem~\ref{t:R4} obtains estimates in all norms up to $H_k^1$. Unfortunately, changing the space in which one takes $v$ in ~\eqref{e:preasymptoticSketch2} is not quite sufficient. Instead, one separates out the high frequencies of $u-u_h$ on each domain and treats them as an additional part of the decomposition of $u$. 
\end{enumerate}

\begin{figure}[htbp]
\begin{center}
\begin{tikzpicture}[->,>=stealth,shorten >=1pt,auto,node distance=7cm,semithick]

\begin{scope}[xscale=1.9,yscale=1.5]
  \node[draw, circle] (A) at (0, 0) {${\Omega_\cavity}$};
  \node[draw, circle] (B) at (2, 0) {${\Omega_\visible}$};
  \node[draw, circle] (C) at (4, 0) {${\Omega_\invisible}$};
  \node[draw, circle] (D) at (6, 0) {${\Omega_{\pml}}$};

  \draw[<-] (A) to[bend left] node[midway, above] {$({h_\cavity }k )^{2p}\sqrt{k {\rho}}$} (B);
  \draw[<-] (B) to[bend left] node[midway, below] {$({h_\visible} k )^{2p}\sqrt{k {\rho}}$} (A);
  \draw[<-] (B) to[bend left] node[midway, above] {$({h_\visible} k )^{2p}k $} (C);
  \draw[<-] (C) to[bend left] node[midway, below] {$({h_\invisible}k )^{2p}k $} (B);
  \draw[<-] (D) to[in=0,out=-135] (4,-1.3)node[ below] {$({h_\pml}k )^{2p}$}to[in = -45,out=180] (B);
  \draw[<-] (B) to[in=180, out=45] (4,1.3)node[ above] {$({h_\visible}k )^{2p}$}to[in=135,out=0] (D);
    \draw[<-] (C) to[bend left] node[midway, below] {$({h_\invisible}k )^{2p}$} (D);
  \draw[<-] (D) to[bend left] node[midway, below] {$(h_{\pml}k )^{2p}$} (C);
%
  \draw[<-] (A) to[loop above] node[midway, above] {$({h_\cavity }k )^{2p}{\rho}$} (A);
  \draw[<-] (B) to[loop above] node[midway, above] {$({h_\visible} k )^{2p}k $} (B);
  \draw[<-] (C) to[loop above] node[midway, above] {$({h_{\invisible}}k )^{2p}k $} (C);
  \draw[<-] (D) to[loop above] node[midway, above] {$(h_{\pml}k )^{2p}$} (D);
  \end{scope}
\end{tikzpicture}
\end{center}
\caption{\label{f:graph1} The graph showing propagation of errors for the decomposition into $\Omega_\cavity$, $\Omega_\visible$, $\Omega_\invisible$, and $\Omega_{\pml}$ in the simplified setup of Section~\ref{s:sketch}. Note that this can be improved in several ways using the analysis in~\cite[\S8]{AGS2}. The graph corresponding to Theorem~\ref{t:R4} is shown in Figure~\ref{f:graph2}.}
\end{figure}

\begin{figure}[htbp]
\begin{center}
\begin{tikzpicture}[->,>=stealth,shorten >=1pt,auto,node distance=7cm,semithick]

\begin{scope}[xscale=1.9,yscale=1.5]
  \node[draw, circle] (A) at (0, 0) {${\Omega_\cavity}$};
  \node[draw, circle] (B) at (2, 0) {${\Omega_\visible}$};
  \node[draw, circle] (C) at (4, 0) {${\Omega_\invisible}$};
  \node[draw, circle] (D) at (6, 0) {${\Omega_{\pml}}$};

  \draw[<-] (A) to[bend left] node[midway, above] {$({h_\cavity}k )^{2p}\sqrt{k {\rho}}$} (B);
  \draw[<-] (B) to[bend left] node[midway, below] {$({h_\visible} k )^{2p}\sqrt{k {\rho}}$} (A);
  \draw[<-] (B) to[bend left] node[midway, above] {$({h_\visible} k )^{2p}k $} (C);
  \draw[<-] (C) to[bend left] node[midway, below] {$({h_\invisible}k )^{2p}k $} (B);
  \draw[<-] (D) to[in=0,out=-135] (4,-1.3)node[ below] {$(h_{\visible,\pml}k )^{N}$}to[in = -45,out=180] (B);
  \draw[<-] (B) to[in=180, out=45] (4,1.3)node[ above] {$(h_{\visible,\pml}k )^{N}$}to[in=135,out=0] (D);
    \draw[<-] (C) to[bend left] node[midway, above] {\scriptsize{$(h_{\invisible,\pml}k )^{N}$}} (D);
  \draw[<-] (D) to[bend left] node[midway, below] {\scriptsize{$(h_{\invisible,\pml}k )^{N}$}} (C);
%
  \draw[<-] (A) to[loop above] node[midway, above] {$({h_\cavity}k )^{2p}{\rho}$} (A);
  \draw[<-] (B) to[loop above] node[midway, above] {$({h_\visible} k )^{2p}k $} (B);
  \draw[<-] (C) to[loop above] node[midway, above] {$({h_{I}}k )^{2p}k $} (C);
  \draw[<-] (D) to[loop above] node[midway, above] {$({h_{\pml}}k )^{N}$} (D);
  \end{scope}
\end{tikzpicture}
\end{center}
\caption{The graph showing propagation of errors for the decomposition into $\Omega_\cavity$, $\Omega_\visible$, $\Omega_\invisible$. Recall that $h_{\visible,\pml} = \min (h_\visible,h_\pml)$ and $h_{\invisible,\pml} = \min (h_\invisible,h_\pml)$.}
\label{f:graph2} 
\end{figure}

\subsubsection{Interpretation as error propagation}
\label{s:interpret}

To properly interpret the matrices appearing in~\eqref{e:simple}, we return to~\eqref{e:startDuality}, which (modulo $O(k^{-\infty})$ errors) is equivalent to
\begin{equation}
\label{e:euanSoImportant}
\chi_i(u-u_h)=\sum_{j=1}^{\domainnumber}  \chi_i R_k\phi_j \big((I-\Pi_{S_k})^*\chi_jP_{S_k}\chi_j(u-w_{h,j}) - 
 (I-\Pi_{S_k})^*\chi_j S_k\chi_j(u-u_h)\big).
\end{equation}
We are interested in $\|\chi_i(u-u_h)\|_{(H_k^{p-1})^*}$, which we think of as the low frequencies of $\chi_i(u-u_h)$; these low frequencies are captured by $S_k\chi_i(u-u_h)$. For purposes of this discussion, we assume $S_k$ commutes with $\chi_i R_k\phi_j$. This is not quite true, but (away from the PML), since
$$
\|S_k\chi_i R_k\phi_j\|_{(H_{k}^{p-1})^*\to L^2}\leq C\|\chi_i R_k\phi_j\|_{L^2\to L^2},
$$ 
 $S_k\chi_iR_k\phi_j$ acts like $\chi_iR_k\phi_j \mc{L}$ where $\mc{L}$ is a lowpass filter. \cite[Theorem 4.2]{AGS2} uses semiclassical ellipticity of the operator in the PML to show that near the PML there is no propagation; we therefore ignore the PML here. 

With these caveats,~\eqref{e:euanSoImportant} implies
\begin{align*}
S_k\chi_i(u-u_h)&=\sum_{j=1}^{\domainnumber}  \chi_i R_k\phi_j\Big(S_k\chi_j (I-\Pi_{S_k})^*\chi_jP_{S_k}\chi_j(u-w_{h,j})\\
&\qquad \qquad\qquad\qquad- 
 S_k\chi_j (I-\Pi_{S_k})^*\chi_j \tilde{S}_kS_k\chi_j(u-u_h)\Big).
\end{align*}
The operator $\chi_i R_k\phi_j$ has the effect of propagating between domains. The operator $(I-\Pi_{S_k})^*$ essentially takes the best approximation in $H_k^1$ norm and $S_k$ then returns only the frequency $\lesssim k$ components. This process is represented in the graph in Figure~\ref{f:error}. 

To find $S_k\chi_i(u-u_h)$ in terms of the local best approximations to $u$, one inserts $\chi_jP_{S_k} \chi_j(u-w_{h,j})$ at node 1 in Figure~\ref{f:error} and follows the cycle to node 4, producing the first approximation to $S_k\chi_i(u-u_h)$. One then continues around the cycle arbitrarily many times, adding $\chi_jP_{S_k} \chi_j(u-w_{h,j})$ in each cycle. This process converges under the condition~\eqref{e:conditionW} and the final result at node $4$ is $(S_k\chi_i(u-u_h))_i$. The $W$ and $B$ matrices in~\eqref{e:theFirstSystem} are respectively one full cycle from node 4 to node 4 and a path from node 1 to node 4 in Figure~\ref{f:error}. 

\begin{figure}[htbp]
\centering
\begin{tikzpicture}[>=stealth, node distance=3cm, every path/.style={->, thick}]
\def\diameter{1cm}
  
  \draw[white] (-6,0)rectangle(6,1);
\node (AA) at (0,4.5){$\chi_j P_{S_k}\chi_j(u-w_{h,j})$};
\node (BB) at (-4.5,0)[left]{$S_k\chi_i(u-u_h)$};
\node[circle,draw,minimum size=\diameter]  (A) at (0, 2.5) {1};
\node[circle,draw,minimum size=\diameter]  (B) at (2.5, 0) {2};
\node[circle,draw,minimum size=\diameter]  (C) at (0, -2.5) {3};
\node[circle,draw,minimum size=\diameter]  (D) at (-2.5, 0) {4};

   \path 
   (AA) edge [dashed]node[above, rotate=90]{add}(A)
  (A)edge [bend left] node[above,rotate=-45]{$\overset{(I-\Pi_{S_k})^*}{\text{\small Take BAE}}$}(B)
    (B)edge [bend left]node[below,rotate=45]{$\underset{S_k}{\text{low pass filter}}$}(C)
      (C)edge [bend left]node[below, rotate =-45] {$\underset{\chi_i \mathcal{R}\chi_j}{\text{propagate}}$}(D)
        (D)edge  [bend left]node[above,rotate=45]{$\overset{-S_k}{\text{low pass filter}}$}(A)
        (D)edge[dashed](BB);
\end{tikzpicture}
\caption{The graph showing the process of error propagation when determining the low frequencies of $u-u_h$ from the local best approximation errors. The arrows are labelled first with the type of operation (low pass filter etc.) and then with the operator whose action gives this effect. These operators are applied multiplicatively.}
\label{f:error}
\end{figure}

\section{The parallel overlapping Schwarz method with PMLs on the subdomain boundaries (\ref{R5})}
\label{s:R5}

\subsection{Definition of the Helmholtz Cartesian PML problem and the parallel overlapping Schwarz method}

\paragraph{Definition of the subdomains.}
Let the $d$-dimensional hyperrectangular domain $\Omega_{\rm int}$ be given by the Cartesian product
$$
\Omega_{\rm int} := \prod_{1 \leq \ell \leq d} (0, L_\ell).
$$
Let $\{\Omega_{{\rm int}, j}\}_{j=1}^N$ be an overlapping decomposition of $\Omega_{\rm int}$ in hyperrectangular subdomains:~
\beq\label{eq:Omegaintj}
\Omega_{{\rm int}, j}  := \prod_{1 \leq \ell \leq d} (a^j_\ell, b^j_\ell).
\eeq

We extend $\Omega_{\rm int}$ and each $\Omega_{{\rm int}, j}$, $1\leq j \leq N$, by adding a PML layer to each, to form the domains $\Omega$ and $\Omega_{j}$.
Namely, let $\width > 0$ and $\width_0 > 0$ denote respectively the PML width on $\Omega$ and the interior PML width and let
$$
\Omega:= \prod_{1 \leq \ell \leq d} (-\width, L_\ell+\width),
\quad
\Omega_{j}  := \prod_{1 \leq \ell \leq d} (a^j_\ell - \width^{j}_{\ell}, b^j_\ell + \width^{j}_{\ell}),
$$
where $\width^{j}_{\ell} = \width_0$ if the corresponding edge of $\Omega_{{\rm int},j}$ belongs to the interior of $\Omega$, and $\width$ otherwise; i.e., edges of subdomains that touch the boundary have the same PML width as $\Omega$ (i.e., $\width$), and interior subdomain edges have PML width $\width_0$ which can be different than $\width$.

\paragraph{The scaled operators.}

We now define standard Cartesian PMLs of width $\width$ at the boundary of $\Omega$ and width $\width_0$ at the boundary of each 
$\Omega_j$. For simplicity, we assume that the same PML scaling function is used inside every PML; however,  
this assumption can be easily removed, with, say, one function used in the PMLs of width $\width$ and another used in the PMLs of width $\width_0$, at the cost of introducing more notation.

Let $\fPML \in C^\infty(\mathbb R)$ (with subscript $\newtheta$ standing for ``scaling") be such that
\begin{align*}
\{ x:\fPML(x) = 0 \} = \{ x:\fPML'(x) = 0 \} = \{x: x \leq 0\},
\\
\fPML'(x) > 0 \tfor x>0, 
\, \tand\, \fPML''(x)= 0 \tfor x\geq \width_{\rm lin}
\end{align*}
for some $\width_{\rm lin}< \width$ (observe that $\fPML$ is then linear for $x\geq \width_{\rm lin}$). 
(This assumption is to avoid technical issues about propagation of singularities -- see~\cite[Remark 2.5]{GGGLS2}.)

For any $1\leq \ell \leq d$, 
we define the following scaling functions in the $\ell$-direction $g_{\ell}\in C^\infty(\mathbb R^d)$  by
\beqs
g_{\ell} (x_\ell) :=
\begin{cases}
\fPML(x_\ell - L_{\ell}) &\text{if }x_\ell \geq L_\ell, \\
0&\text{if }x_\ell \in (0, L_\ell), \\
- \fPML( - x_\ell) &\text{if }x_\ell \leq 0,
\end{cases}
\eeqs
and, similarly, the subdomain scaling functions in the $\ell$-direction $g_{\ell, j}\in C^\infty(\mathbb R^d)$ for $1 \leq j \leq N$ by
\beqs
g_{\ell, j} (x_\ell) :=
\begin{cases}
\fPML(x_\ell - b_{\ell}^j) &\text{if }x_\ell \geq b_{\ell}^j, \\
0&\text{if }x_\ell \in (a_{\ell}^j, b_{\ell}^j), \\
- \fPML( a_{\ell}^j - x_\ell) &\text{if }x_\ell \leq a_{\ell}^j.
\end{cases}
\eeqs
We now define the scaled operators $\DeltaPML$ and $\DeltaPMLj$ by
\beqs
\DeltaPML := \sum_{\ell=1}^d \Big(\frac{1}{1+ig_{\ell}'(x_\ell)} \partial_{x_\ell} \Big)^2,
\eeqs
and
\beq\label{eq:DeltaPMLj}
\DeltaPMLj := \sum_{\ell=1}^d \Big(\frac{1}{1+ig'_{\ell,j}(x_\ell)} \partial_{x_\ell} \Big)^2 \hspace{0.3cm}\tfa 1 \leq j \leq N.
\eeq
Given $n\in C^\infty(\Rea^d)$ that is strictly positive and such that $\supp(1-n) \subset \Omega_{\rm int}$, let 
\beq\label{eq:PPj}
P_{\newtheta} := - k^{-2} \DeltaPML - n, \qquad P_{\newtheta}^j := - k^{-2} \DeltaPMLj - n \; \tfa 1 \leq j \leq N.
\eeq
These operators are defined on $H^1(\Rea^d)$, but we consider $P_\newtheta$ as a operator $H^1(\Omega)\to H^{-1}(\Omega)(= (H^1_0(\Omega))^*)$, 
and $P^j_{\newtheta}$ as an operator $H^1(\Omega_j)\to H^{-1}(\Omega_j)(= (H^1_0(\Omega_j))^*)$. If Dirichlet boundary data in $H^{1/2}$ is prescribed on the corresponding boundaries, these operators are then invertible.

In the proofs of our main results, a key region is the following 
subset of $\Omega_j$: 
\beqs
\supp(P^j_\newtheta-P_\newtheta):=\Omega_j \cap \bigg(\bigcup_{\ell=1}^d \overline{\big\{ x \in \Rea^d\,:\, g_{\ell,j}(x_\ell) \neq g_{\ell}(x_\ell)\big\}}\bigg);
\eeqs
$\supp(P^j_\newtheta-P_\newtheta)$ is the part of the PML region of $\Omega_j$ where the PML in $\Omega_j$ differs from the PML in $\Omega$ (i.e., where $P^j_\newtheta \neq P_\newtheta$). 
Note that when $\Omega_j$ is an interior subdomain, $\supp(P^j_\newtheta-P_\newtheta) = \Omega_j\setminus\Omega_{\rm int, j}$.

\paragraph{The Helmholtz problem.}
Given $f\in H^{-1}(\Omega)$, 
find $u\in H^1_0(\Omega)$ such that 
\beq\label{eq:PDE}
P_{\newtheta} u=f.
\eeq
Theorem \ref{t:blackBoxErr} above proved that the solution of the radial PML exists, is unique, and is exponentially accurate for sufficiently large $k$. An analogous result for $P$ nontrapping is proved in \cite[Lemma 4.5 and Theorem 4.6]{GGGLS2}, except with exponential accuracy replaced by superalgebraic accuracy (i.e., $O(k^{-\infty})$).

The existence and uniqueness results for Cartesian PML at fixed $k$ analogous to those for radial PMLs discussed in Remark \ref{r:PMLfixedk} are \cite[Theorem 5.5]{KiPa:10} and \cite[Theorem 5.7]{BrPa:13}.

\paragraph{The partition of unity.}

Let $\{\chi_j\}_{j=1}^{N}$ be a partition of unity subordinate to 
$\{ \Omega_{j}\setminus \supp(P_\newtheta^j - P_\newtheta)\}_{j=1}^{N}$, i.e.,
 $\{\chi_j\}_{j=1}^{N}$ is a family of non-negative elements of $C^\infty(\mathbb R^d)$ such that
\begin{align}
&\tfa\,  x \in \mathbb R^d, \quad \sum_{j=1}^{N} \chi_j(x) = 1,\quad
\operatorname{supp}\chi_j \cap \Omega \subset \big(\Omega_{ j}\setminus {\supp(P_\newtheta^j - P_\newtheta)}\big)  \text{ for all $j$},\label{eq:PoU}
\end{align}
and, additionally, in a neighbourhood of $\partial\Omega$, each $\chi_j$ does not vary in the normal direction to the boundary of $\Omega$ (this last assumption is for technical reasons, and can easily be achieved in practice).

\paragraph{The parallel overlapping Schwarz method.}

Given $\parallelu^n \in H^1_0(\Omega)$, for $n\geq 0$ and any $1\leq j \leq N$, let $u_j^{n+1}\in H^1(\Omega_j)$ be the solution to
\begin{equation}\label{eq:localprob}
\begin{cases}
P_{\newtheta}^j
u_j^{n+1} = 
P_{\newtheta}^j (\parallelu^n|_{\Omega_j})
-(P_{\newtheta}\parallelu^n )|_{\Omega_j}
+ f|_{\Omega_j}\in H^{-1}(\Omega_j)\\
u_j^{n+1} = \parallelu^n \in H^{1/2}(\partial \Omega_j),
\end{cases}
\end{equation}
and then set
\beq\label{eq:parallel}
\parallelu^{n+1} := \sum_{j=1}^N \chi_j u_j^{n+1};
\eeq
this method is 
 equivalent to \eqref{e:SchwarzIntro} and its discrete-level analogue can be understood as a natural PML-variant of the  well-known restricted additive Schwarz (RAS) method; see \cite[\S8.1]{GGGLS2}, \cite[Lemma 1.6]{DoJoNa:15}.

 \subsection{Definition of $\mathcal{N}$}
 \label{s:Ncurly}
Recall that $\varphi_t$ is the Hamiltonian flow (see~\eqref{e:flow}).
\begin{definition}
\label{d:trajectory}
Given an interval, $I$, and $(x,\xi)\in \mathbb{R}^{2d}$ such that $p(x,\xi)=0$, define
\beq\label{eq:gamma_I}
\gamma_I(x,\xi):= \Big\{ \varphi_t(x,\xi), \; t \in I\Big\}.
\eeq 
When $I=\mathbb{R}$, we say $\gamma_I$ is a trajectory for $P$ and usually denote it by $\gamma$ with no subscript.
\end{definition}

We now define a set of cutoffs that are bigger than the $\chi_j$s. Let $\{\TTchi_j\}_{j=1}^N$ be such that, for each $j$, $\TTchi_j\in C^\infty(\Rea^d)$, $\supp \TTchi_j\cap \Omega \subset \Omega_j$, $\supp(1-\TTchi_j)\cap \supp \chi_j=\emptyset$, and $\supp\TTchi_j\cap \supp(P^j_\newtheta-P_\newtheta)=\emptyset$; for an illustration of 
$\{\Tchi_j\}$ and $\{\chi_j\}$ in the case of two subdomains, see Figure \ref{fig:pou}.

\begin{figure}[h!]
\begin{center}
\includegraphics[scale=0.55]{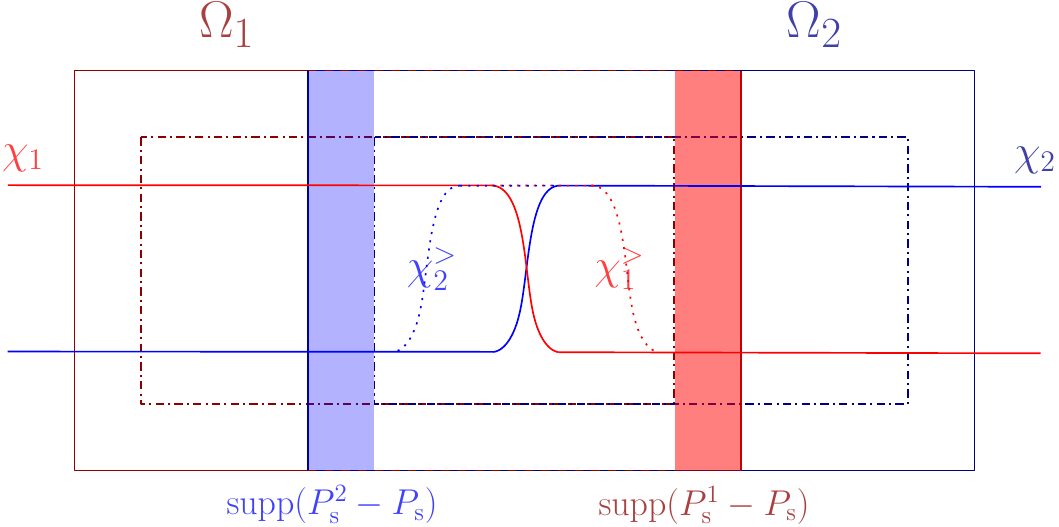}
\caption{The functions $\{\Tchi_j\}$ and $\{\chi_j\}$ for two subdomains}\label{fig:pou}
\end{center}
\end{figure}

Let $\mathcal W$ be the set of finite sequences of elements of $\{ 1, \ldots, N \}$ (i.e., the indices of subdomains) such that, for any $w = \{w_i\}_{1\leq i \leq n} \in \mathcal W$, $w_{i} \neq w_{i+1}$ for all $1\leq i \leq n$ (i.e., the $(i+1)$th component of $w$ corresponds to a different subdomain than the $i$th component). We call the elements of $\mathcal W$ \emph{words}.


\begin{definition}\label{def:follow}
A trajectory $\gamma$ for $P$ \emph{follows} a word $w \in \mathcal W$ of length $n\geq 2$ if 
\beqs
\gamma  = \prod_{2 \leq j \leq n} \gamma_j,
\eeqs
where the product stands for concatenation, 
$\gamma_j = \gamma_{(0, T_j]}(x_{j-1},\xi_{j-1})$ for some $T_j >0$ and $(x_j,\xi_j) \in \Rea^{2d}$ with
\ben
\item
for $1 \leq j \leq n-1$,
$x_j \in \operatorname{supp} \TTchi_{w_j} \cap \operatorname{supp} \big(P^{w_{j+1}}_{\rm s} - P_{\rm s} \big)$\label{i:toFollow},
\item For $2 \leq j \leq n$, define $(x_j(t),\xi_j(t)):=\gamma_j(t)$. Then $x_j(t)\in \overline{\Omega_{w_j}}$ for $t\in (0,T_j]$.
\item With
$(x_n, \xi_n) := \varphi_{ T_{n}}(x_{n-1},\xi_{n-1})$, $x_n\in  \operatorname{supp} \TTchi_{w_n}$.
\een
\end{definition}

Figure \ref{drw:word} illustrates a simple example of a trajectory following a word.

\begin{figure}[h!]
\begin{center}
\includegraphics[scale=0.7]{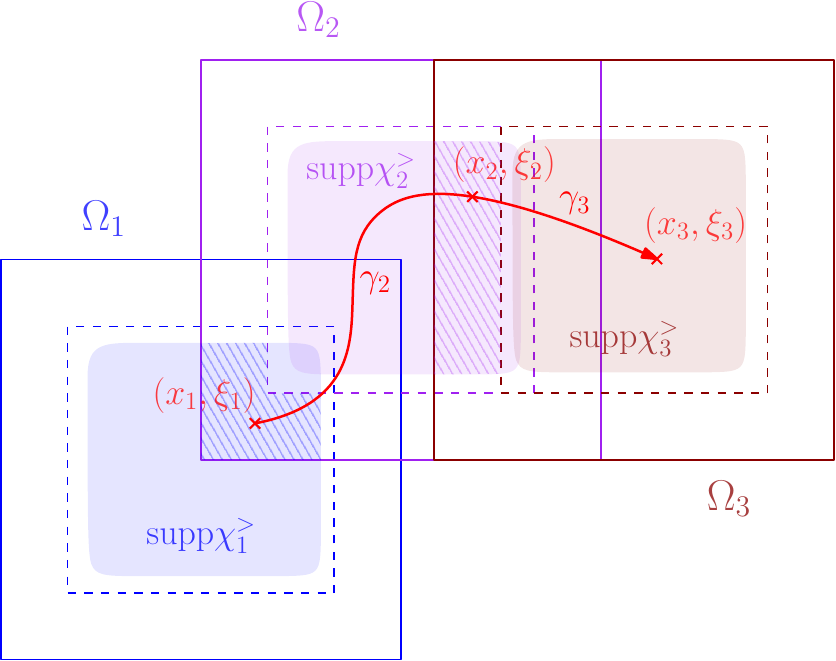}
\caption{A trajectory (in red) following the word $(1,2,3)$ (here $\Omega_1, \Omega_2, \Omega_3$ are interior subdomains of $\Omega$).
The blue hatched shading indicates the domain $\supp \TTchi_1 \cap \supp(P^2_\newtheta-P_\newtheta)$, and the purple hatched shading 
indicates the domain $\supp \TTchi_2 \cap \supp(P^3_\newtheta-P_\newtheta)$. 
(The points $(x_j,\xi_j)$, $j=1,2,3$, and sub-trajectories $\gamma_2,\gamma_3$ are used in the precise definition of  \emph{follow} in Definition \ref{def:follow} below.)
}\label{drw:word}
\end{center}
\end{figure}

The flow associated to $P$ only allows passage between certain sequences of subdomains; i.e., there are words that will not be followed by any trajectory. 
For example, if $n\equiv 1$, the trajectories of the flow are straight lines, and thus, with domains in Figure~\ref{drw:word} there is no trajectory following the word $(1,2,1)$. 
This motivates the following definition.

\begin{definition} \label{def:allowed}
A word $w\in \mathcal{W}$ is \emph{allowed} if there exists a trajectory for $P$ that follows $w$.
\end{definition}

\begin{definition} \label{def:Ncurly}
\beq\label{eq:mathcalN}
\mathcal N := \sup_{w \text{ allowed}} |w|.
\eeq
\end{definition}

Since the sets 
\beq\label{eq:follow_domains}
\supp \Tchi_{w_1} \cap \supp(P^{w_2}_s-P_s), 
\,\,
\ldots
\,\,, 
\,\,
\supp \Tchi_{w_{n-1}} \cap \supp(P^{w_{n}}_s-P_s),
\,\,
 \supp \Tchi_{w_n}.
\eeq
are compact and disjoint, the assumption that $P$ is nontrapping implies that the length of allowed words is bounded above, i.e., $\mathcal N < \infty$.

 \subsection{Error estimates for the parallel overlapping Schwarz method}

\begin{theorem}[Rigorous statement of \ref{R5}]
\label{t:R5}
Assume 
that $P$ is nontrapping.
Then there exists $k_0>0$ such that for all $M,s > 0$, there exists $C>0$ 
such that the following is true for any
 $f \in H^{-1}(\Omega)$ and $k \geq k_0$.
If $u$ is the solution to (\ref{eq:PDE}), and $\{\parallelu^n\}_{n \geq 0}$
is a sequence of iterates for the  parallel overlapping Schwarz method (defined by \eqref{eq:localprob} and \eqref{eq:parallel}), then 
$$
\Vert u - \parallelu^{\mathcal N} \Vert_{H^s_k(\Omega)} \leq C k^{-M} \Vert u - \parallelu^0 \Vert_{H^1_k(\Omega)}.
$$
\end{theorem}
\bre
The requirement that $k_0$ be sufficiently large in Theorem \ref{t:R5} is used to guarantee that the solution to all subdomain problems exists. Alternatively, one could allow $k_0>0$ arbitrary provided that we exclude an open neighborhood of the discrete set of $k$s where $P_s^j$ may not be invertible. 
\ere

 \subsection{Concrete values of $\mathcal{N}$ for strips and checkerboards when $n\equiv 1$}


For $N_\ell \in \mathbb{Z}^+$, $\ell=1,\ldots, d$, 
we say that $\{\Omega_{j}\}_{j=1}^N$ is a \emph{checkerboard} of size $N_1\times \ldots \times N_\ell$ if the following hold.

\bit
\item[(i)] 
For $\ell=1,\ldots, d$, there exist 
$$
0=y^\ell_0 < \ldots < y^\ell_{N_\ell}=L_\ell
$$
such that 
\beqs
\overline{\Omega_{\rm int}}= \bigcup_{j=1}^N \overline{ U_j}, \quad U_j \cap U_i =\emptyset \text{ if } j\neq i,
\eeqs
with each $U_j$ of the form $\prod_{\ell=1}^d (y_{m_\ell-1}^\ell, y_{m_{\ell}}^\ell)$ 
for some $m_\ell\in \{1,\ldots, N_\ell\}$, $\ell=1,\ldots, d$. 

\item[(ii)] The overlapping decomposition $\{\Omega_{{\rm int}, j}\}_{j=1}^N$ of $\Omega_{\rm int}$ comes from extending each $U_j$ 
in each coordinate direction (apart from at $\partial \Omega$).
\item[(iii)] The extensions in (ii) are such that, for all $j$, $\Omega_j$ overlaps only with $\Omega_{j'}$, where $\Omega_{j}$ and $\Omega_{j'}$ are extensions of adjacent nonoverlapping subdomains.
\eit

We say that  $\{\Omega_j\}$ is a $\mathfrak d$-checkerboard if $\mathfrak d := |\{ \ell, \; N_\ell \neq 1 \}|$ i.e., $\mathfrak d$ is the effective dimensionality of the checkerboard.

We say that $\{\Omega_{j}\}_{j=1}^N$ is a \emph{strip} if it is a $1$-checkerboard. To simplify notation, we then assume that the subdomains of a strip are ordered monotonically (i.e., so that $\Omega_i$ only overlaps $\Omega_{i-1}$ and $\Omega_{i+1}$).

Figure \ref{fig:AI} shows an example of the grid formed by the points $\{y_m^\ell\}$ that is used in the definition of a particular $2$-checkerboard of size $4\times 5$, along with a particular subdomain $\Omega_{{\rm int},j}$.

\begin{figure}[h!]
\begin{center}
\begin{tikzpicture}[scale=1]
\def \w{.2};
\def \j{2};
\def \i{1};
   \draw[fill=lightgray, lightgray] ({\i-\w},{(1+\j*(0.05))*(0.75)*\j-\w})  rectangle    ({\i+1+\w},{(1+(\j+1)*(0.05))*(0.75)*(\j+1)+\w});
   \node  at({\i+.5}, {((1+\j*(0.05))*(0.75)*\j+(1+(\j+1)*(0.05))*(0.75)*(\j+1) )/2}){\small{$\Omega_{{\rm int},j}$}};
    \foreach \y in {0,1,2,3,4,5} {
        \draw (0,{(1+\y*(0.05))*(0.75)*\y}) -- (4*1.07,{(1+\y*.05)*.75*\y});
    }

    \foreach \x in {0,1,2,3} {
        \draw (\x,0) -- (\x,{5*.75*(1+5*.05)});
    }

    \draw (4*1.07,0) -- (4*1.07,{5*.75*(1+5*.05)});

    \foreach \x [count=\xi] in {0,1,2,3,4,5} {
        \node[left] at (0,{.75*(\xi-1)*(1+(\xi-1)*.05)}) {$y^2_{{\x}}$};
    }
    \foreach \x [count=\xi] in {0,1,2,3} {
        \node[below] at (\xi-1,0) {$y^1_{{\x}}$};
    }
    \node[below] at (4*1.07,0) {$y^1_{{4}}$};
     
\end{tikzpicture}
\end{center}
\caption{The grid formed by the points $\{y_m^\ell\}$ that is used in the definition of a particular $2$-checkerboard of size $4\times 5$, and the subdomain $\Omega_{{\rm int}, j}$ formed by extending $(y_1^1,y_2^1)\times (y^2_2, y^2_3)$}
\label{fig:AI}
\end{figure}

\begin{lemma}[$\mathcal{N}$ for strips and checkerboards]\label{l:Ncurly}
Suppose that $n\equiv 1$.
If $\{\Omega_j\}$ is a $\mathfrak{d}$-checkerboard, then $\mathcal{N}= N_1+ \cdots + N_{\mathfrak{d}}- (\mathfrak{d}-1)$.
In particular, if $\{\Omega_j\}$ is a strip, then $\mathcal{N}=N$.
    \end{lemma}

Lemma \ref{l:Ncurly} is proved in \cite[\S7.3]{GGGLS2}.

\subsection{Sketch of the proof of Theorem~\ref{t:R5}}

The proof of Theorem~\ref{t:R5} boils down to understanding how the error $u-u^n$ at one step influences the error $u-u^{n+1}$ at the next step. The two main steps in the proof of Theorem \ref{t:R5} are the following.
\begin{enumerate}
\item Write down the Helmholtz problems satisfied by the errors on each subdomain,  and hence obtain \emph{the physical error propagation matrix} $\mathbf{T}$ which maps the vector of errors on each subdomain at step $n$ to the same at step $n+1$ (see \eqref{eq:Monday1} below).
\item Write $\mathbf{T}^n$ in terms of the operator products
\beq\label{eq:Tw}
\mathscr T_{w} := 
(\mathbf T)_{w_{n},w_{n-1}} (\mathbf T)_{w_{n-1}, w_{n-2}} \cdots (\mathbf T)_{w_2,w_1}
\eeq
(corresponding to the error propagating from $\Omega_{w_i}$ to $\Omega_{w_{i+1}}$ for $i=1,\dots, n-1$)
and show using SCA that 
$\|\mathscr T_w\|_{H^1_k(\Omega)\to H^s_k(\Omega)} = O(k^{-\infty})$
 if $w$ is not allowed.
\end{enumerate}
We highlight that $\mathcal{N}$ appears because no words of length $>\mathcal{N}$ are allowed.

\paragraph{Step 1:~The error propagation matrix $\mathbf{T}$ for the parallel method.}\label{sec:idea_error}

For any $n \geq 1$ and $1 \leq j \leq N$,  we define the global and local errors
\beq\label{eq:errors}
\parallele ^n := u -  \parallelu^n \in H^1_0 (\Omega)\qquad\text{and } \qquad \parallelejn := u|_{\Omega_j} - \parallelujn \in H^1(\Omega_j).
\eeq
These definitions, the definition of the iterate $\parallelu^n$ \eqref{eq:parallel}, and the fact that $\{\chi_j\}_{j=1}^N$ is a partition of unity \eqref{eq:PoU} imply that
\beq\label{eq:error_decomp}
\parallele^n = \sum_{j=1}^N   \chi_j  \parallelejn.
\eeq
We introduce the notation 
\beq\label{eq:local_phys_error}
\parallelepsilonjn := \chi_j  \parallelejn
\quad\text{ so that }\quad
 \parallele^n = \sum_{j=1}^N  \parallelepsilonjn.
\eeq
We can interpret $\parallelepsilonjn =\chi_j \parallelejn$, after extension by zero from $\Omega_j$, as an element of $H^1_0(\Omega)$, 
since $\parallelejn$ is zero on $\partial \Omega_j\cap \partial \Omega$ and $\chi_j = 0$ on $\partial \Omega_j\setminus \partial \Omega$. 
We call $\parallelepsilonjn $ the \emph{local physical error}. We highlight that 
the fact that $\parallele^n$ is expressed in \eqref{eq:local_phys_error} in terms of $\parallelepsilonjn=\chi_j  \parallelejn$ and not the whole local error $\parallelejn$ is expected, since $\parallelejn$ contains contributions from the PML layers of $\Omega_j$ that are not PML layers of $\Omega$, and hence not part of the original PDE problem \eqref{eq:PDE} (and hence ``unphysical'').

For $1 \leq i  \leq N$ and $v \in H^1_0(\Omega)$ let $\mathcal T_{i} v \in H^1(\Omega_j)$ be the solution to
\beq\label{eq:Tjl}
\begin{cases}
P^i_{\newtheta} \mathcal T_{i} v = P^i_{\newtheta}(v|_{\Omega_i}) - (P_{\newtheta}v)|_{\Omega_i}\qquad \text{in }{\Omega_i}, \\
\mathcal T_{i} v =v \qquad \text{on }\partial\Omega_i.
\end{cases}
\eeq
Recall that $\{\Tchi_j\}_{j=1}^N$ are a set of cut-off functions that have bigger support that the $\chi_j$s. 

We define the \emph{physical error propagation matrix} $\mathbf T$ as
\beq\label{eq:bold_T}
\mathbf T := (\chi_i \mathcal T_{i}\Tchi_j)_{1\leq i,j \leq N} : (H^1_0(\Omega))^N \mapsto  (H^1(\Omega))^N
\eeq
and, for any $n\geq 0$, the physical errors vector ${\bparallelepsilon^n}\in (H^1_0(\Omega))^N$ as

$$
(\bparallelepsilon^n)_j :=\parallelepsilonjn, \, 1\leq j\leq N.
$$

\begin{lemma}\mythmname{Physical error propagation for the parallel method}\label{lem:model_err_prop}
The operator $\mathbf{T}$ defined in~\eqref{eq:bold_T} satisfies $\mathbf T:(H^1_0(\Omega))^N \mapsto  (H^1_0(\Omega))^N$ and 
\beq\label{eq:Monday1}
\bparallelepsilon^{n+1} = \mathbf T \bparallelepsilon^{n} \quad\tfa n\geq 0.
\eeq
\end{lemma}
Lemma \ref{lem:model_err_prop} is proved in \cite[\S5.1]{GGGLS2}. 

\paragraph{Step 2:~Relating powers of $\mathbf{T}$ to trajectories of the flow defined by $P$.
}\label{sec:powers}

By \eqref{eq:Monday1}, $\bparallelepsilon^{n}= \mathbf{T}^n \bparallelepsilon^{0}$.
By the definition of $\mathbf T$ \eqref{eq:bold_T}, powers of $\mathbf T$ involve compositions of the maps $\chi_i \mathcal T_{i}\Tchi_j$, and these compositions 
applied to $\parallelepsilon^{0}$
can be interpreted as the error traveling from subdomain to subdomain through the iterations.

We make this more precise via the definition of the
composite map $\mathscr{T}_w$ \eqref{eq:Tw}, defined  for any $w \in \mathcal W$ of length $|w| :=n\geq 2$.
The presence of the cut offs in the definition of $(\mathbf T)_{i,j}$ \eqref{eq:bold_T} means that 
$\mathscr T_{w}$ is zero unless $\Omega_{w_j}$ and $\Omega_{w_{j+1}}$ overlap for all $j$. 



We make the following  remarks:
\bit
\item To see where the sets in \eqref{eq:follow_domains} (i.e., the sets in Item \ref{i:toFollow} of Definition \ref{def:follow}) 
come from, observe that, 
by the definition of $\mathcal T_i$ \eqref{eq:Tjl}, 
the functions $\mathcal T_{i}\Tchi_j v$ appearing in the action of $\mathbf T$ \eqref{eq:bold_T} satisfy 
a PDE on $\Omega_i$ with the operator $P_\newtheta^i$ 
and right-hand side supported in $\supp \Tchi_{j} \cap \supp(P^{i}_\newtheta-P_\newtheta)$. 
\item $\supp \chi_{w_i} \cap \supp(P^{w_i+1}_s-P_s)$ is the part of the PML of $\Omega_{w_{i+1}}$ that intersects $\supp \chi_{w_i}$; because $\chi_j$ is zero on the support of $P^j_\newtheta-P_\newtheta$, the sets in \eqref{eq:follow_domains} are disjoint.
\eit


The main goal of Step 2 is thus to prove the following lemma.
\ble
\label{lem:notallowed}
If $w\in \mathcal W$ is not allowed then, for any $s\geq 1$, 
\beqs
\|\mathscr T_w\|_{H^1_k(\Omega) \rightarrow H^s_k(\Omega)}= O(k^{-\infty}).
\eeqs
\ele

\bpf[Proof of Theorem \ref{t:R5} using Lemma \ref{lem:notallowed}]
By the definition of the error propagation matrix $\mathbf{T}$ \eqref{eq:bold_T} and the composite map $\mathscr{T}_w$ \eqref{eq:Tw}, for any $m\in \mathbb{Z}^+$, 
all entries of $\mathbf{T}^{m}$ are sums of operators $\mathscr{T}_w$ over words $w$ of size $m+1$. 
Since words of size $\mathcal{N}+1$ are not allowed by the definition of $\mathcal{N}$ \eqref{eq:mathcalN}, the result that $\|\mathbf{T}^{\mathcal N}\|_{(H^1_k(\Omega))^N \rightarrow (H^s_k(\Omega))^N} = O(k^{-\infty})$ follows from Lemma \ref{lem:notallowed}.
\epf

We now sketch the proof of Lemma \ref{lem:notallowed}.

\begin{informallemma}\label{lem:key_prop2}
For any $v\in H^1_0(\Omega)$, 
$$
\WF(\mathscr T_w v)\subset \big\{ \gamma_n(T_n)\,:\, \gamma \text{ follows }w\big\}.
$$ 
(where recall from Definition \ref{def:follow} that $\gamma_n(T_n)$ is the endpoint of the trajectory $\gamma_n$).

\end{informallemma}

Informal Lemma~\ref{lem:key_prop2} shows that if $w$ is not allowed then for any $v\in H_0^1(\Omega)$ $\WF(\mathscr{T}_wv)=\emptyset$. Combining this with the definition of wavefront set (Definition \ref{d:wavefront}) and using semiclassical ellipticity to handle very large frequencies (see \cite[Lemma B.8]{GGGLS2}) proves Lemma~\ref{lem:notallowed}.

Informal Lemma \ref{lem:key_prop2} is proved using the following propagation result.
\begin{informallemma}\label{lem:key_prop}
Let $v \in H( \Omega_j)$ be the solution to
$$
P^j_{\newtheta} v = f \hspace{0.3cm} \text{in }{\Omega_j}, \qquad
v = g1_{\partial\Omega_j\setminus \partial\Omega}\hspace{0.3cm} \text{on }\partial\Omega_j
$$
If $P$ is nontrapping, then, for $\chi \in \overline{C^\infty}(\Omega)$ with $\supp \chi \subset (\Omega_j\cup \partial\Omega)$,
\beq\label{e:informalWF1}
\WF(\chi v)\subset \supp \chi\cap\Big(\WF(f)\cup\bigcup_{t\geq 0} \varphi_t\big(\WF(f)\cap \{p_s^j=0\}\big)\Big).
\eeq
\end{informallemma}
Observe that \eqref{e:informalWF1} controls the wavefront set of $v$ away from the boundaries of $\Omega_j$ that are not also boundaries of $\Omega$.

Both Informal Lemmas \ref{lem:key_prop2} and \ref{lem:key_prop} are stated as informal lemmas because we have not dealt carefully with the boundaries $\partial\Omega$ and $\partial \Omega_j$, respectively. 
The rigorous statements of these results are \cite[Lemmas 4.1 and 6.2]{GGGLS2}, respectively. Our sketches below are similarly ignore these issues. 


\bpf[Sketch proof of Informal Lemma \ref{lem:key_prop2} using Informal Lemma \ref{lem:key_prop}]
Given $v\in H^1_0(\Omega)$, by the definitions of $(\mathbf T)_{i,j}$ \eqref{eq:bold_T} and $\mathcal{T}_{j}$ \eqref{eq:Tjl},
$(\mathbf T)_{w_{i+1}, w_i}v$ equals $\chi_{w_{i+1}}$ multiplied by a Helmholtz solution 
on $\Omega_{w_{i+1}}$ with data $(P_\newtheta^{w_{i+1}}-P_\newtheta)\Tchi_{w_i}v$. 
Therefore, by Lemma \ref{lem:key_prop2},
\begin{equation}
\begin{aligned}
&\WF\big((\mathbf{T})_{w_{i+1},w_i}v\big)\\
&\subset  \supp \chi_{w_{i+1}}\cap\\
&\Big(\WF\big((P_{\rm s}^{w_{i+1}}-P_{\rm s})\chi_{w_i}^>v\big)\cup\bigcup_{t\geq 0} \varphi_t\big(\WF\big((P_{\rm s}^{w_{i+1}}-P_{\rm s})\chi_{w_i}^>v\big)\cap \{p^{i+1}_{\rm s}=0\}\big)\Big);
\label{e:propagateDD}
\end{aligned}
\end{equation}
in applying 
Lemma~\ref{lem:key_prop2} we have used that $v=0$ on $\partial\Omega$ (so that the support condition in $g$ holds) and we need only estimate $\chi_{w_{i+1}}u$ for a Helmholtz solution $u$ because of the multiplication by $\chi_{i+1}$ in the definition of $\mathbf{T}$, \eqref{eq:bold_T}.
In other words, 
in the $k\to \infty$ limit, the mass of 
$(\mathbf T)_{w_{i+1}, w_{i}}v$ in phase space  comes only from $\supp \chi^>_{w_i} \cap \supp (P_\newtheta^{w_{i+1}}-P_\newtheta)$, i.e., the first domain in \eqref{eq:follow_domains}. 

Next, combining \eqref{e:propagateDD} with $i=1$ and $i=2$, the fact that 
$$
\bigcap_{i=1,2}\supp \chi_{w_{i+1}}
\cap
\supp (P_{\rm{s}}^{w_i}-P_{\rm s})\cap \supp \chi_{w_i}^>=\emptyset,
$$
 the definition of following a word (Definition \ref{def:follow}), and the fact that $\{p_s^j=0\}\subset \{p=0\}$
we see that 
$$
\WF\big(\mathbf{T}_{w_3,w_2}\mathbf{T}_{w_2,w_1}v\big)\subset \{p_{\rm s}=0\}\cap \big\{ \gamma_3(T_3)\,:\, \gamma_3 \text{ follows } w_1w_2w_3\big\}.
$$
In other words, 
in the $k\to \infty$ limit, the mass of 
$(\mathbf T)_{w_3, w_2}(\mathbf T)_{w_2, w_1}v$ in phase space comes only from the mass in 
$\supp \chi_{w_2} \cap \supp (P_\newtheta^{w_3}-P_\newtheta)$ that in turn came from
$\supp \chi_{w_1} \cap \supp (P_\newtheta^{w_2}-P_\newtheta)$.

The result then follows by repeating this argument for longer words.
\epf

\bpf[Sketch proof of Lemma \ref{lem:key_prop}]
By ellipticity (i.e., 
 Corollary \ref{c:ellipticWavefront} above),
\begin{equation}
\label{e:ellipticYAY}
\WF(v)\subset \{p_s^j=0\}\cup \WF(f).
\end{equation}
We now use propagation, almost in the form of Corollary~\ref{c:propagateWavefront}, except that since $P_s^j$ is not self-adjoint, but does have $\Im \sigma(P_s^j)\leq 0$ near $\sigma(P_s^j)=0$, one may only propagate forward in time
under the flow defined by 
$\Re \sigma(P_s^j)$ 
(see Remark \ref{r:direction}); i.e., Corollary \ref{c:propagateWavefront} holds with $T\geq 0$.

 The fact $P$ is nontrapping means that every trajectory of $p$ starting from $\supp \chi$ encounters $\{p_s^j\neq 0\}$ backward in time (see~\cite[Lemma 4.4]{GGGLS2}). It is enough to consider trajectories of the flow defined by $p$ rather than those of the flow defined by $\Re p_s^j$ since $H_{\Re p_s^j}=H_p$ on $\{p_s^j=0\}$ 
(see \cite[Lemma 4.2]{GGGLS2}) and~\eqref{e:ellipticYAY} allows to understand $\WF(\chi v)\cap\{p_s^j\neq 0\}$. 

 Therefore applying Corollary~\ref{c:propagateWavefront} 
with $T\geq 0$ (as discussed above), 
we see that if
$$(x_0,\xi_0)\in \{p_s^j=0\}\cap \Big(\WF(f)\cup\bigcup_{t\geq 0} \varphi_t\big(\WF(f)\cap \{p_s^j=0\}\big)\Big)^c,
$$
then there is $T>0$ such that $\varphi_{-T}(x_0,\xi_0)\in \{p_s^j\neq 0\}$ and
$$
\Big(\bigcup_{t=0}^T\varphi_{-t}(x_0,\xi_0)\Big)\cap \WF(f)=\emptyset.
$$
Hence, since $\varphi_{-T}(x_0,\xi_0)\notin \WF(v)$, we obtain $(x_0,\xi_0)\notin \WF( v)$ and hence also $(x_0,\xi_0)\notin \WF(\chi v)$. 
\epf

\appendix

\section{$h$- and $p$-explicit polynomial approximation}
\label{a:poly}

There is a large literature on using interpolants (i.e., polynomials defined via point values of a function) to produce both $H^1$-conforming and non-conforming approximations by piecewise polynomials; see, e.g., \cite[Chapter 4]{BrSc:08}, 
\cite[Chapter 4]{SaSc:11}, \cite{ApMe:17},
\cite[Part III]{ErGu:21}.
However, it is well known that designing interpolants that behave well in the large polynomial-degree limit can be subtle (due to, e.g., Runge's phenomenon) and thus
there has been much work
on obtaining $p$-explicit approximants that are not necessarily interpolants -- see the discussion and references in \S\ref{s:hpliterature}. 
However, the authors are unaware of a single source collecting all the ingredients for the proofs of these $p$-explicit approximation results.

The main goal of this section is therefore to give a self-contained treatment of $p$-and $h$-explicit polynomial approximation in $L^2$-based Sobolev spaces. 
The main results are given in Theorems~\ref{t:approxHighLowReg} and~\ref{t:approxHighLowRegL2} below, and cover 
approximation 
by arbitrary-degree piecewise polynomials 
of \emph{either} fixed-regularity functions 
\emph{or} high-regularity functions.
Furthermore, our results for the high-regularity functions, giving precise Sobolev space estimates, appear to be new (see the discussion in \S\ref{s:hpliterature}).






\subsection{Definition of shape-regular sequences of triangulations}

\begin{definition}[Triangulation]
A finite collection of closed sets $\mathcal{T}$ is a \emph{triangulation} of $\Omega\subset \mathbb{R}^d$ if
\begin{itemize}
\item[(i)] for each $T \in \mathcal{T}$, $T$ is closed, $\mathring{T}$ is nonempty and connected,
\item[(ii)] $\overline{\Omega}= \cup_{T\in \mathcal{T}}T$
\item[(iii)] if $T_1, T_2\in \mathcal{T}$ and $\mathring{T_1}\cap \mathring{T_2} \neq \emptyset$ then $T_1=T_2$.
\item[(iv)] for each $T \in \mathcal{T}$, $T$ has Lipschitz boundary and 
$$
\partial T=\bigcup_{j=1}^d \bigcup_{i=1}^{N_{j,d}}F_{j,i}(T)
$$
with $F_{j,i}$ a smooth, open, embedded submanifold of dimension $d-j$ such that $\partial F_{j,i}=\partial\overline{F_{j,i}}$,
\item[(v)] if $T_1,T_2\in\mathcal{T}$ and $F_{j_1,i_1}(T_1)\cap F_{j_2,i_2}(T_2)\neq \emptyset$, then $F_{j_1,i_1}(T_1)=F_{j_2,i_2}(T_2)$. 
\end{itemize}
We sometimes use the word mesh as a synonym for triangulation.
\end{definition}


\begin{definition}[$m$-simplex]
Let $m>0$. A set $K\Subset \mathbb{R}^d$ is an \emph{$m$-simplex} if there are $\{x_j\}_{j=0}^m\subset \mathbb{R}^d$ such that
$$
K=\text{convex hull}\big(\{ x_j\}_{j=0}^m\}\big).
$$
$K$ is an \emph{open $m$-simplex} if $K=(\widetilde{K})^\circ$ where $\widetilde{K}$ is an $m$ simplex  and $K\neq \emptyset$. 
\end{definition}

\begin{definition}[Reference element and element maps]
$\widehat{T}\Subset \mathbb{R}^d$ is a reference element for a triangulation $\mathcal{T}$ if $\operatorname{diam}(\widehat{T})=1$ and there exist a family of bi-Lipschitz maps 
$\{\mapF_{T}\}_{T\in \mathcal{T}}$ such that $T = \mapF_T( \widehat{T})$ for all $T\in \mathcal{T}$. $\mathcal{T}$ is \emph{simplicial} if $\widehat{T}$ is a $d$ simplex.
\end{definition}

We use the following measure of how star-shaped an open subset of $\mathbb{R}^d$ is. 
\begin{definition}[Chunkiness parameter]
Let $\Omega\subset \mathbb{R}^d$ and define
$$
\rho_{\max}(\Omega):=\sup \big\{ \rho>0\,:\, \Omega\text{ is star-shaped with respect to a ball of radius }\rho\big\}.
$$
The \emph{chunkiness parameter of $\Omega$} is defined by
$$
\gamma(\Omega):=\frac{\operatorname{diam}(\Omega)}{\rho_{\max}(\Omega)}.
$$
\end{definition}
(Observe that $\gamma(\Omega)=2$ when $\Omega$ is a ball.)

\begin{definition}[$C^r$ triangulation]
\label{d:Crtriang}
A triangulation is $C^r$ with constant $\Upsilon>0$ if $\gamma(\widehat{T})\leq \Upsilon$ and for each $T\in\mathcal{T}$ there is $h_T>0$ such that the  element map $\mapF_T$ can be written as $R_T \circ A_T$ where $A_T$ is an affine map and 
\begin{align}
&\| \partial A_T\|_{L^\infty} \le \Upsilon h_T, \quad 
\| (\partial A_T)^{-1}\|_{L^\infty}\leq \Upsilon h^{-1}_T,\quad \| (\partial R_T)^{-1}\|_{L^\infty}\leq \Upsilon,\nonumber\\
&
\| \partial^\alpha R_T\|_{L^\infty}\leq \Upsilon^{1+|\alpha|} \alpha! \quad\tfa\, |\alpha|\leq r.\label{e:mapsHighReg}
\end{align}
Let $h(\mathcal{T}):=\sup_{T\in\mathcal{T}}h_T$ be the \emph{width of the triangulation $\mathcal{T}$}. Furthermore, a triangulation is $C^\omega$ if~\eqref{e:mapsHighReg} holds for all $|\alpha|<\infty$.
\end{definition}

\bre 
Given a piecewise $C^{r}$ domain, a $C^r$ triangulation is constructed in \cite[\S6]{Be:89}, building on the 2-d results of \cite{Sc:73, Zl:73} and the isoparametric elements in general dimension of \cite{Le:86}.
\ere

\bre[Shape regularity]
A collection of triangulations, $\mathscr{T}$, is \emph{shape regular} in the sense of \cite[Definition 11.2]{ErGu:21}, \cite[Definition 4.4.13]{BrSc:08} if
\beqs
\sup_{\mathcal{T}\in\mathscr{T}}\sup_{T\in \mathcal{T}}\frac{h_T}{r_{\max}(T)}<\infty,
\eeqs
where
\beqs
r_{\max}(T):=\sup \big\{r>0\,:\, \text{there is }x_0\in T\text{ with }B(x_0,r)\subset T\big\}.
\eeqs
Given $\Upsilon>0$, the collection of $C^1$ triangulations with this $\Upsilon$ is then shape regular. Indeed, the bounds on $\partial \mapF_T$ and $\partial (\mapF_T)^{-1}$ from \eqref{e:mapsHighReg} 
and the fact that $\operatorname{diam}(\widehat{T})=1$
imply that
$$
\Upsilon^{-2}h_T\leq \operatorname{diam}(T)\leq \Upsilon^2 h_T,
$$
and 
$$
\Upsilon^{-2}h_T r_{\max}(\widehat{T})\leq r_{\max}(T)\leq \Upsilon^2 h_T r_{\max}(\widehat{T})
$$
and hence, since $\rho_{\max}(\widehat{T})\leq r_{\max}(\widehat{T})$,
$$
\frac{h_T}{r_{\max}(T)}\leq \frac{\Upsilon^2}{r_{\max}(\widehat{T})}\leq \frac{\Upsilon^2}{\rho_{\max}(\widehat{T})}=\Upsilon^2\gamma(\widehat{T}).
$$
\ere

\begin{definition}[Affine-conforming triangulation]
\label{d:AffineConforming}
A $C^r$ triangulation (in the sense of Definition \ref{d:Crtriang}) is \emph{affine conforming} if whenever $F_{j_1,i_1}(T_1)\cap F_{j_2,i_2}(T_2)\neq \emptyset$ there is an affine isomorphism $\kappa:\mapF_{T_1}^{-1}(F_{j_1,i_1}(T_1))\to \mapF_{T_2}^{-1}(F_{j_2,i_2}(T_2))$  such that 
\beq\label{e:isomorphism}
\mapF_{T_1}|_{
\overline{
\mapF_{T_1}^{-1}(F_{j_1,i_1}(T_1))
}}
=\mapF_{T_2}|_{
\overline{
\mapF_{T_2}^{-1}(F_{j_2,i_2}(T_2))
}}\circ \kappa.
\eeq
\end{definition}


We use the adjective \emph{conforming} in Definitions \ref{d:AffineConforming} 
because the additional property~\eqref{e:isomorphism} 
is used to generate an $H^1$-conforming polynomial approximant in
\ref{t:approxHighLowReg} below.

\bre
Without a condition on the element maps such as Definitions \ref{d:AffineConforming}, 
the space of mapped polynomials on two adjacent mesh elements 
may not even intersect non-trivially when restricted to a common face. If the intersection is trivial, there cannot be any non-trivial, globally-$H^1$ mapped polynomials. 

Conditions such as Definitions \ref{d:AffineConforming} 
are therefore common in the literature. 
For example, the Lagrange $\mathbb{P}_{k,d}$ elements in \cite{ErGu:21} are affine conforming (in the sense of Definition \ref{d:AffineConforming}) by 
 \cite[Definition 8.1 and Exercise 20.1]{ErGu:21}.
\ere


Given a triangulation $\mathcal{T}$ with element maps $\{\mapF_T\}_{T\in\mathcal{T}}$ and reference element $\widehat{T}$, the corresponding spaces of mapped polynomials are defined by
\beq\label{e:polySpace}
\mathcal{P}^p(\mathcal{T},\{\mapF_T\}):= \big\{ 
v\in H^1(\Omega) \, :\, \text{ for all }  T \in \cT : v|_T\circ \mapF_T \in \mathbb{P}^p(\widehat{T})
\big\},
\eeq
where $\mathbb{P}^p(\widehat{T})$ denotes the space of polynomials of degree $\leq p$ on the reference element. We often abbreviate $\mathcal{P}^p(\mathcal{T},\{\mapF_T\})$ to $\mathcal{P}^p_{\mathcal{T}}$.

\begin{definition}\label{d:triangulationPreserving}
An operator $\mathcal{I}:C^\infty(\Omega)\to L^2(\Omega)$ is \emph{triangulation preserving} if for each   $T\in\mathcal{T}$, 
$$
u|_{T}=0\quad  \Rightarrow\quad (\mathcal{I}u)|_{T}=0,
$$
and for any codimension $1$ boundary face $F_{1,i}$ of $T$, 
$$
u|_{\overline{F_{1,i}}}=0\quad  \Rightarrow\quad (\mathcal{I}u)|_{\overline{F_{1,i}}}=0.
$$
\end{definition}

\subsection{Statement of the polynomial-approximation results}\label{s:polystate}

To state results about approximation of high-regularity functions by arbitrary degree polynomials, we need to choose carefully the constants in the high regularity Sobolev spaces. 
For an open set $\Omega\subset \mathbb{R}^d$, we define
\begin{equation}
\label{e:highPSobolev}
\begin{aligned}
\|u\|_{H_k^p(\Omega)}^2&:= \sum_{\ell=0}^p \frac{1}{(\ell!)^2}|u|_{H_k^\ell(\Omega)}^2\\
|u|_{H_k^{p}(\Omega)}^2&:=\sum_{\substack{|\alpha|=p\\\alpha\in\mathbb{N}^d}}\frac{p!}{\alpha!}\|(k^{-1}D_{x_1})^{\alpha_1}(k^{-1}D_{x_2})^{\alpha_2}\dots (k^{-1}D_{x_d})^{\alpha_d}u\|_{L^2(\Omega)}^2.
\end{aligned}
\end{equation}
We also denote $\|\cdot\|_{H^p(\Omega)}:=\|\cdot \|_{H_1^p(\Omega)}$. The rationale for these choices is given in \S\ref{s:norms} below.


\begin{theorem}[Approximation in $H^1_k$]
\label{t:approxHighLowReg}
Given $\Upsilon>0$ and $t>d/2$ there exists $C>0$ such that for all $p\geq \max(t-1,1)$, all 
affine-conforming $C^{p+1}$ simplicial triangulations with constant $\Upsilon$, and all partitions of $\mathcal{T}$ into $\mathcal{T}_1\sqcup\mathcal{T}_2$, 
there exists a triangulation-preserving operator 
$\mathcal{C}^p_{\mathcal{T}}:H^1(\Omega)\cap (\oplus_{T\in\mathcal{T}} H^t(T))\to \mathcal{P}^p_{\mathcal{T}}\cap H^1(\Omega)$ such that for all $0\leq s\leq 1$ and $u\in H^1(\Omega)\cap (\oplus_{T\in\mathcal{T}} H^t(T))$  
\begin{align*}
\|(I-\mathcal{C}_{\mathcal{T}}^p)u\|_{ H_k^s(\Omega_{\mathcal{T}_1})}&\leq C\Big(\frac{h(\mathcal{T})k}{p}\Big)^{t-s}k^{-t}\Big(\sum_{T\in \mathcal{T}_1}\|u\|_{H^t(T)}^2\Big)^{\frac{1}{2}},\\
\|(I-\mathcal{C}^p_{\mathcal{T}})u\|_{H_k^s(\Omega_{\mathcal{T}_2})}&\leq (Ch(\mathcal{T})k)^{p+1-s}k^{-p-1}\Big(\sum_{T\in \mathcal{T}_2}\|u\|_{H^{p+1}(T)}^2\Big)^{\frac{1}{2}},
\end{align*}
where 
$$
\Omega_{\mathcal{T}_i}:=\bigcup_{T\in\mathcal{T}_i}T,
$$
and for $u\in\mathcal{P}_{\mathcal{T}}^p\cap H^1(\Omega)$, $\mathcal{C}_{\mathcal{T}}^pu=u$. 
In particular,
\begin{align*}
&\|(I-\Pi^{H^1}_{\mathcal{P}_{\mathcal{T}}^p})u\|_{H_k^s(\Omega)}\leq C\Big(\frac{h(\mathcal{T})k}{p}\Big)^{t-s}k^{-t}\Big(\sum_{T\in \mathcal{T}_1}\|u\|_{H^t(T)}^2\Big)^{\frac{1}{2}}\\
&\hspace{4.5cm}+(Ch(\mathcal{T})k)^{p+1-s}k^{-p-1}\Big(\sum_{T\in \mathcal{T}_2}\|u\|_{H^{p+1}(T)}^2\Big)^{\frac{1}{2}},
\end{align*}
where $\Pi^{H^1}_{\mathcal{P}_{\mathcal{T}}^p}$ is the $H_k^1$-orthogonal projector onto $\mathcal{P}_{\mathcal{T}}^p\cap H^1$. 
\end{theorem}

We emphasise that by taking \emph{either} $\mathcal{T}_2=\emptyset$ \emph{or} $\mathcal{T}_1=\emptyset$, one obtains from Theorem \ref{t:approxHighLowReg} separate results on approximating \emph{either} fixed-regularity functions \emph{or} high-regularity functions, respectively. However, Theorem \ref{t:approxHighLowReg} is the stronger result that, given any partition of the triangulation, 
there exists a single $H^1$-conforming and triangulation-preserving operator $\mathcal{C}^p_{\mathcal{T}}$ achieving the fixed-regularity approximation on the first part of the triangulation and the high-regularity approximation on the second part.

Next, we record the non-conforming analogue of Theorem~\ref{t:approxHighLowReg}.
\begin{theorem}[Approximation in $L^2$]
\label{t:approxHighLowRegL2}
Given $\Upsilon>0$ and $t\geq 0$, there exists $C>0$ such that for all $p\geq 0$ and all $C^{p+1}$-triangulations with constant $\Upsilon$, there exists a triangulation-preserving operator 
$\mathcal{C}^p_{\mathcal{T}}: (\oplus_{T\in\mathcal{T}} H^t(T))\to \mathcal{P}^p_{\mathcal{T}}$ such that for all $u\in (\oplus_{T\in\mathcal{T}} H^t(T))$,  
\begin{align*}
\|(I-\mathcal{C}_{\mathcal{T}}^p)u\|_{L^2(\Omega_{\mathcal{T}_1})}&\leq C\Big(\frac{h(\mathcal{T})k}{p}\Big)^{t}k^{-t}\Big(\sum_{T\in \mathcal{T}_1}\|u\|_{H^t(T)}^2\Big)^{\frac{1}{2}},\\
\|(I-\mathcal{C}^p_{\mathcal{T}})u\|_{L^2(\Omega_{\mathcal{T}_2})}&\leq (Ch(\mathcal{T})k)^{p+1}k^{-p-1}\Big(\sum_{T\in \mathcal{T}_2}\|u\|_{H^{p+1}(T)}^2\Big)^{\frac{1}{2}},
\end{align*}
where 
$$
\Omega_{\mathcal{T}_i}:=\bigcup_{T\in\mathcal{T}_i}T
$$
and for $u\in\mathcal{P}_{\mathcal{T}}^p\cap H^1(\Omega)$, $\mathcal{C}_{\mathcal{T}}^pu=u$. 
In particular,
\begin{align*}
&\|(I-\Pi^{L^2}_{\mathcal{P}_{\mathcal{T}}^p})u\|_{H_k^s(\Omega)}\\
&\quad\leq C\Big(\frac{h(\mathcal{T})k}{p}\Big)^{t}k^{-t}\Big(\sum_{T\in \mathcal{T}_1}\|u\|_{H^t(T)}^2\Big)^{\frac{1}{2}}+(Ch(\mathcal{T})k)^{p+1}k^{-p-1}\Big(\sum_{T\in \mathcal{T}_2}\|u\|_{H^{p+1}(T)}^2\Big)^{\frac{1}{2}},
\end{align*}
where $\Pi^{L^2}_{\mathcal{P}_{\mathcal{T}}^p}$ is the $L^2$-orthogonal projector onto $\mathcal{P}_{\mathcal{T}}^p$.
\end{theorem}

\subsection{Relationship to existing results in the literature}\label{s:hpliterature}






In dimensions 2 and 3, 
fixed-regularity results similar to those in 
Theorems
\ref{t:approxHighLowReg} and \ref{t:approxHighLowRegL2}
appear in  
\cite[Appendix B]{MeSa:10}; earlier related work includes \cite{BaSu:87, 
BaCrMaPi:91, Mu:97, GuZh:09}; see the discussion in \cite[\S3.4]{ApMe:17}.
Compared to~\cite{MeSa:10}, our presentation stresses the precise assumptions on the triangulation used
to convert approximants on reference elements into $H^1$-conforming approximants on the triangulation of the full domain, $\Omega$. Furthermore, we give a detailed proof of the construction of such an $H^1$-conforming approximant on all of $\Omega$ starting from 
an appropriate approximant on the reference element (having a modification of the ``element-by-element construction" in the sense of  \cite[Definition 5.3]{MeSa:10}; see Definition \ref{d:boundaryCompatible}, Proposition \ref{p:toBoundaryCompatible},  \eqref{e:whatIsQij}, and Remark \ref{r:elementByElement} below).

We are not aware of the high-regularity estimates in Theorems~\ref{t:approxHighLowReg} 
and \ref{t:approxHighLowRegL2} in the existing literature. While there are approximation results in spaces of analytic functions ~\cite{Me:02,MeSa:10,MeSa:20,BeChMe:25}, these do not lead to a norm bound on $I-\Pi_{\mathcal{P}_h^p}$ between Sobolev spaces (as in the high-regularity results in Theorems~\ref{t:approxHighLowReg} 
and \ref{t:approxHighLowRegL2}). 

This appendix presents an almost-self-contained proof of Theorems~\ref{t:approxHighLowReg} and~\ref{t:approxHighLowRegL2}. We 
use without proof only
basic properties of averaged Taylor polynomials
from~\cite[Chapter 4]{BrSc:08} (precisely~\cite[Corollary 4.2.18, and Propositions 4.1.17 and 4.2.8]{BrSc:08}). 

\subsection{Outline of the proofs}
In \S\ref{s:norms}, we motivate the choice of Sobolev norms in~\eqref{e:highPSobolev} and show that they behave well under pull-back by analytic maps. Then, in \S\ref{s:polyLipschitz}, we give two types of $p$-explicit results about approximation by polynomials in open subsets of $\mathbb{R}^d$ with Lipschitz boundaries. First, we estimate the error in approximating high-regularity functions by high-degree polynomials (coming from averaged Taylor expansions), following~\cite{BrSc:08} but keeping careful track of dependence on the polynomial degree and regularity. Second, we estimate the error in approximating fixed-regularity functions by high-degree Legendre polynomials (in 1-d, such results go back 
at least to \cite{CaQu:82}; see \cite{BaSu:94, BeMa:97, Sc:98} and the references therein). 
These estimates can then be used to construct good polynomial approximations of a function on a single element but they cannot be directly used to construct $H^1$-conforming approximations by piecewise polynomials on all of $\Omega$. In \S\ref{s:boundaryCompatible}, we consider polynomial approximation on simplices and modify the polynomial approximants constructed in \S\ref{s:polyLipschitz} so that they can be extended to $H^1$-conforming piecewise polynomials on the full triangulation. To do this, we use ideas from~\cite{MeSa:10}, which, in turn, uses ideas from~\cite{Mu:97} and~\cite{GuZh:09}. 

\subsection{$H_k^p$ norms}
\label{s:norms}
To determine the correct choice of $H_k^p$ norm, we return to $\mathbb{R}^d$, where the semiclassical Fourier transform (defined by \eqref{e:Fourier}) gives a natural choice of $H_k^p$ seminorm:
$$
|u|_{H_k^{p}(\mathbb{R}^d)}^2:=\frac{1}{(2\pi)^d}\big\||\xi|^{p}\mathcal{F}_k(u)(\xi)\big\|^2_{L^2(\mathbb{R}^d)}.
$$
We now write this in terms of $L^2$ norms of derivatives. 
\begin{align*}
|u|_{H_k^{p}(\mathbb{R}^d)}^2&=\frac{1}{(2\pi^d)}\int |\xi|^{2p}|\mathcal{F}_k(u)(\xi)|^2d\xi\\
&=\frac{1}{(2\pi^d)}\int \sum_{\substack{|\alpha|=p\\\alpha\in\mathbb{N}^d}}\frac{p!}{\alpha!}\xi_1^{2\alpha_1}\dots\xi_d^{2\alpha_d} |\mathcal{F}_k(u)(\xi)|^2d\xi\\
&= \sum_{\substack{|\alpha|=p\\\alpha\in\mathbb{N}^d}}\frac{p!}{\alpha!}\|(k^{-1}D_{x_1})^{\alpha_1}(k^{-1}D_{x_2})^{\alpha_2}\dots (k^{-1}D_{x_d})^{\alpha_d}u\|_{L^2(\mathbb{R}^d)}^2.
\end{align*}
Therefore, on an open set $\Omega\subset \mathbb{R}^d$, it is natural to define $|u|_{H_k^p(\Omega)}$ as in~\eqref{e:highPSobolev}.


Now, we need to determine the appropriate way to combine the $H_k^p$ seminorms into a norm. It is natural to consider a norm that contains a multiple of each $H_k^\ell$ seminorm, $\ell=0,\dots,p$ such that for analytic functions (whose $\alpha$th derivative goes like $C^{|\alpha|} |\alpha|!$) each term in the sum has roughly the same size and whose total size grows at most exponentially in $p$. 
The Sobolev norm in~\eqref{e:highPSobolev} then achieves these goals. 

To make sure our choice of norm is appropriate, we check two important properties. First, we estimate the norm of the derivative in $H_k^p$. 
\begin{lemma}
Let $\Omega \subset \mathbb{R}^d$. Then, for all $p \geq 0$,
$$
\|k^{-1}\nabla u\|_{H_k^{p}(\Omega)}\leq (p+1)\|u\|_{H_k^{p+1}(\Omega)}
$$
\end{lemma}
\begin{proof}
By the definition \eqref{e:highPSobolev},
\begin{align*}
&\|k^{-1}\nabla u\|_{H_k^{p}(\Omega)}^2\\
&=\sum_{\ell=0}^p\sum_{\substack{|\alpha|=\ell\\\alpha\in\mathbb{N}^d}}\sum_{j=1}^d\frac{1}{\ell!\alpha!}\|(k^{-1}D_{x_1})^{\alpha_1}(k^{-1}D_{x_2})^{\alpha_2}\dots (k^{-1}D_{x_d})^{\alpha_d}(k^{-1}D_{x_j})u\|_{L^2(\Omega)}^2\\
&\leq\sum_{\ell=0}^p\sum_{\substack{|\alpha|=\ell+1\\\alpha\in\mathbb{N}^d}}\frac{(\ell+1)^2 }{(\ell+1)!\alpha!}\|(k^{-1}D_{x_1})^{\alpha_1}(k^{-1}D_{x_2})^{\alpha_2}\dots (k^{-1}D_{x_d})^{\alpha_d}u\|_{L^2(\Omega)}^2\\
&\leq \sum_{\ell=0}^{p+1}\sum_{\substack{|\alpha|=\ell\\\alpha\in\mathbb{N}^d}}\frac{\ell^2}{(\ell!)^2}|u|_{H_k^{\ell}}^2\leq (p+1)^2\|u\|_{H_k^{p+1}(\Omega)}^2.
\end{align*}
\end{proof}

Next, we show that the $H^p$ norm defined by \eqref{e:highPSobolev} behaves well under pull-backs with analytic maps. 

\begin{lemma}\label{l:analyticpullback}
Let  $d\geq 1$, $C_0>0$. Then there is $C_1>0$ such that for all $t\geq 0$, $U,V\subset \mathbb{R}^d$ open, $\gamma:U\to V$ a bijection satisfying
$$
\|\partial^\alpha \gamma\|_{L^\infty}\leq C_0^{|\alpha|}\alpha!\, \tfor\alpha\in\mathbb{N}^d, |\alpha|\leq \max(p,1),\,\tand\, \|(\partial\gamma)^{-1}\|_{L^\infty}\leq C_0,
$$
and $u\in H^t(V)$,
$$
\frac{1}{t!}|\gamma^*u |_{H^t(U)}\leq C_1^t\|u\|_{H^t(V)},
$$
where $\gamma^*u:=u\circ \gamma$ denotes the pull-back operator.
\end{lemma}
\begin{proof}
We start from the Faa di Bruno formula~\cite[Theorem 2.1]{CoSa:96}. Let $\mu,\nu\in\mathbb{N}^d$. We say $\mu\prec \nu$ if one of the following holds
\begin{align*}
&(i) \,\,|\mu|<|\nu|,\qquad\text{or} \\
&(ii)\text{ there is }1\leq r\leq d,\text{ such that }\mu_i=\nu_i,\,i=1\dots,r-1\text{ and }\mu_r<\nu_r.
\end{align*}
Then, defining
$$
p_s(\nu,\lambda):=\left\{ \begin{aligned}(\mu_1,\dots,\mu_s;\ell_1,\dots,\ell_s)\in(\mathbb{N}^d)^{2s}\,:\, |&\mu_i|>0,\,0\prec \ell_1\prec\dots\prec \ell_s,\\
& \sum_{i=1}^s\mu_i=\lambda,\,\sum_{i=1}^s|\mu_i|\ell_i=\nu\end{aligned}\right\},
$$
we have
\begin{align*}
&\frac{1}{(t!)^2}|u\circ \gamma|_{H^t(U)}^2\\
&=\sum_{|\nu|=t}\int_U\Big(\sum_{m=1}^t \sum_{|\lambda|=m}(\partial^\lambda u)\circ \gamma\sum_{s=1}^t\sum_{p_s(\nu,\lambda)}\frac{\nu!}{\sqrt{\nu!t!}}\prod_{j=1}^s\frac{[\partial^{\ell_j}\gamma]^{\mu_j}}{\mu_j![\ell_j!]^{|\mu_j|}}\Big)^2\\
&\leq C_0^{2t}\sum_{|\nu|=t}\int_U\Big(\sum_{m=1}^t \sum_{|\lambda|=m}\frac{1}{m!}(\partial^\lambda u)\circ\gamma\sum_{s=1}^t\sum_{p_s(\nu,\lambda)}\frac{\nu!m!}{\sqrt{\nu! t!}}\prod_{j=1}^s\frac{1}{\mu_j!}\Big)^2\\
&\leq C_0^{2t}C_1\|u\|_{H^t(V)}^2 \sum_{|\nu|=t}\sum_{m=1}^t \sum_{|\lambda|=m}\frac{\lambda!}{m!}\Big(\sum_{s=1}^t\sum_{p_s(\nu,\lambda)}\frac{\nu!m!}{\sqrt{\nu! t!}}\prod_{j=1}^s\frac{1}{\mu_j!}\Big)^2.
\end{align*}
Therefore, we claim that there is $C>0$ such that
\begin{equation}
b_{t}:= \sum_{|\nu|=t}\sum_{m=1}^t \sum_{|\lambda|=m}\frac{\lambda!}{m!}\Big(\sum_{s=1}^t\sum_{p_s(\nu,\lambda)}\frac{\nu!m!}{\sqrt{\nu! t!}}\prod_{j=1}^s\frac{1}{\mu_j!}\Big)^2\leq C^t.
\label{e:bpClaim}
\end{equation}
To do this, we convert the sum over multiindices $\mu$ to a sum over only their lengths.
Let $r\in\mathbb{N}^s$ with $|r|=|\lambda|$ and consider the following. 
$$
a_{r,\lambda}:=\sum_{\substack{\mu_1+\dots+\mu_s=\lambda\\ |\mu_i|=r_i}}\prod_{j=1}^s\frac{1}{\mu_j!}.
$$
By the multinomial theorem,
$$
\sum_{|\mu|=r}\frac{x^\mu}{\mu!}=\frac{(x_1+\dots +x_d)^r}{r!}.
$$
Thus
\begin{align*}
\sum_{|\lambda|=|r|} a_{r,\lambda} x^\lambda 
=\sum_{
\substack{
\mu_1,\ldots,\mu_s
\\
|\mu_i|=r_i
}}\prod_{j=1}^s \frac{ x^{\mu_j}}{\mu_j!}
=
\prod_{j=1}^s \sum_{|\mu|=r_j}\frac{x^\mu}{\mu!}
&=\prod_{j=1}^s \frac{(x_1 + \dots+ x_d)^{r_j}}{r_j!} \\
&= \frac{(x_1+\dots+x_d)^{|r|}}{r!} \\
&= \frac{1}{ r!}\sum_{|\lambda|=|r|} \frac{|r|! x^\lambda}{\lambda!}.
\end{align*}
Therefore, $a_{r,\lambda}$ is given by the coefficient of $x^\lambda$ in the right hand side:
$$
a_{r,\lambda}=\frac{|\lambda|!}{\lambda!r!}.
$$
Hence, defining 
$$
q_s(\nu,\lambda):=\left\{ \begin{aligned}(r;\ell_1,\dots,\ell_s)\in\mathbb{N}^s\times (\mathbb{N}^d)^{s}\,:\, &r_i>0,\,0\prec \ell_1\prec\dots\prec \ell_s,\\
& |r|=|\lambda|,\,\sum_{i=1}^sr_i\ell_i=\nu\end{aligned}\right\},
$$
we have
\begin{equation}
\label{e:itIsB}
b_{t}= \sum_{|\nu|=t}\sum_{m=1}^t \sum_{|\lambda|=m}\frac{\nu!m!}{t!\lambda!}\Big(\sum_{s=1}^t\sum_{q_s(\nu,\lambda)}\frac{m!}{r!}\Big)^2.
\end{equation}
Observe that for fixed $\{\ell_i\}_{i=1}^s$ and $r$, $m!/r!$ is the number of ordered $m$-tuples $(w_1,\dots,w_m)\in(\mathbb{N}^d\setminus\{0\})^{m}$ in which exactly $r_i$ entries are equal to $\ell_i$.  Therefore
\begin{equation}
\label{e:toCombinatorics}
\sum_{s=1}^t\sum_{q_s(\nu,\lambda)}\frac{m!}{r!}\leq \#\Big\{ (w_1,\dots,w_m)\in(\mathbb{N}^d\setminus\{0\})^{m}\,:\, \sum_i w_i=\nu\Big\}.
\end{equation}
Now, 
\begin{equation}
\label{e:groupCount}
 \#\Big\{ (w_1,\dots,w_m)\in(\mathbb{N}^d\setminus\{0\})^{m}\,:\, \sum_i w_i=\nu\Big\}\leq \prod_{j=1}^d\frac{(\nu_j+m-1)!}{(m-1)!\nu_j!},
\end{equation}
since for the $j^{th}$ coordinate of $\nu$ we break $\nu_j$ into $m$ non-empty groups.
Finally, using the arithmetic mean, geometric mean inequality: 
$$
\Big(\prod_{i=1}^da_i\Big)^{\frac{1}{d}}\leq \frac{1}{d}\sum_{i=1}^da_i
\quad\tfor a_i\geq 0,
$$
we obtain
\begin{align*}
\prod_{j=1}^d\frac{(\nu_j+m-1)!}{(m-1)!\nu_j!}&\leq \prod_{j=1}^d\frac{(\nu_j+m-1)^{m-1}}{(m-1)!}\\
&\leq \frac{(\frac{|\nu|}{d}+m-1)^{d(m-1)}}{[(m-1)!]^d}\leq C^{d(m-1)}\Big(\frac{|\nu|}{d(m-1)}+1\Big)^{d(m-1)}.
\end{align*}
Now
$$
\sup_{x\in(0,\infty)}\big(1+x\big)^{\frac{1}{x}}<\infty
$$
so that
\begin{equation}
\label{e:upperEstimate}
\prod_{j=1}^d\frac{(\nu_j+m-1)!}{(m-1)!\nu_j!}\leq C^{d(m-1)}C^{|\nu|}.
\end{equation}
The combination of~\eqref{e:itIsB}, \eqref{e:toCombinatorics}, \eqref{e:groupCount} and~\eqref{e:upperEstimate} implies that
$$
b_t\leq  \sum_{|\nu|=t}\sum_{m=1}^t \sum_{|\lambda|=m}\frac{\nu!m!}{t!\lambda!}C^{t}.
$$
Now
\begin{equation*}
\sum_{|\lambda|=m}\frac{m!}{\lambda !} = d^m \quad\tand\quad
\sum_{|\nu|=t}\frac{\nu!}{t!} \leq 
\#\big\{ \nu : |\nu|=t\big\}
 \leq (t+1)^{d-1}
\end{equation*}
(by, respectively, the multinomial theorem and the fact that $\nu !\leq |\nu|!$).
Thus $b_t \leq C^t$,
from which~\eqref{e:bpClaim}, and hence the result, follows.
\end{proof}

\subsection{Polynomial approximation on subsets of $\mathbb{R}^d$}
\label{s:polyLipschitz}


The following result shows that one can approximate high-regularity functions very accurately by piecewise polynomials of high degree.

\begin{proposition}
\label{p:pExplicit1}
Given $d\in\mathbb{N}$ there is $C>0$ such that the following holds. Let $\Omega \Subset \mathbb{R}^d$ with $\gamma(\Omega)<\infty$. Then there are  $\mathcal{I}^{p-1}:H_k^{p}(\Omega)\to \mathbb{P}^{p-1}$, $p\geq 1$ such that for $|\alpha|\leq p$,  $D^{\alpha}\mathcal{I}^{p-1}=\mathcal{I}^{p-1-|\alpha|}D^\alpha$ and for  $\ell\in \{0,1,\dots,p\}$, $u\in H_k^{p}$,
\begin{align}\label{e:approxOmega}
\|(I-\mathcal{I}^{p-1})u\|^2_{H_k^\ell(\Omega)}\leq C\sum_{j=0}^\ell \frac{(k\operatorname{diam}(\Omega))^{2p-2j}}{[(p-j)!j!]^2}(\gamma+1)^{2d}d^{p-j}|u|^2_{H_k^p(\Omega)}.
\end{align}
Moreover, if $u\in \mathbb{P}^{p-1}$, then $\mathcal{I}^{p-1}u=u$.
\end{proposition}
To prove Proposition~\ref{p:pExplicit1}, we follow~\cite[Chapter 4]{BrSc:08}.

\begin{lemma}\mythmname{\cite[Corollary 4.2.18 and Propositions 4.1.17 and 4.2.8]{BrSc:08}}
\label{l:averagedTaylor}
Given $d\in\mathbb{N}$ there is $C>0$ such that the following holds. For all $\Omega \Subset \mathbb{R}^d$, with $\gamma(\Omega)<\infty$, there is a ball $B$ such that $\Omega$ is star-shaped with respect to $B$, and there are $\mathcal{I}^{p-1}:H_k^{p}(\Omega)\to \mathbb{P}^{p-1}$, $p\geq 0$ such that for $|\alpha|\leq p-1$,  $D^{\alpha}\mathcal{I}^{p-1}=\mathcal{I}^{p-1-|\alpha|}D^\alpha$ and for $u\in H_k^{p}(\Omega)$,
$$
(u-\mathcal{I}^{p-1}u)(x)=p\sum_{|\alpha|=p}\int_{C_x}k_{\alpha}(x,z)D^\alpha u(z)dz,
$$
where  $k_{\alpha}(x,z)=\frac{1}{\alpha!}(x-z)^\alpha k(x,z)$,
$$
|k(x,z)|\leq C(\gamma(\Omega)+1)^d|x-z|^{-d},
$$
and $C_x$ is the convex hull of $\{x\}\cup B$. 
\end{lemma}
\begin{proof}[Sketch of proof]
The result 
follows by averaging the Taylor polynomial over a ball and writing the integral form of the remainder. 
\end{proof}

\begin{lemma}
\label{l:riesz}
Let $\Omega\subset \mathbb{R}^d$ and for $m>0$ define the operator 
$$
T_mu(x):=\int_\Omega |x-z|^{m-d}u(z)dz.
$$
For all $m>0$,
$$
\|T_mu\|_{L^2(\Omega)}\leq \frac{(\operatorname{diam}(\Omega))^m}{m}\|u\|_{L^2(\Omega)}.
$$
\end{lemma}
\begin{proof}
The result follows from the Schur test for boundedness (see, e.g., \cite[Lemma 18.1.12]{Ho:85}) 
since 
\begin{align*}
\sup_{x\in\Omega}\int_{\Omega} |x-z|^{m-n}dz&=\sup_{z\in\Omega}\int_{\Omega} |x-z|^{m-d}dx\\
&\leq \int_0^{\operatorname{diam}(\Omega)}r^{m-d+d-1}dr\leq \frac{\big(\operatorname{diam}(\Omega)\big)^m}{m}.
\end{align*}
\end{proof}

\begin{proof}[Proof of Proposition~\ref{p:pExplicit1}]
We start with the case $m=0$. Then, by Lemma~\ref{l:averagedTaylor}, there is $\mathcal{I}^{p-1}$ with the required properties such that 
\begin{align*}
|u-\mathcal{I}^{p-1}u|(x)&\leq Cp(\gamma+1)^d\sum_{|\alpha|=p}\int_{\Omega}\frac{1}{\alpha!}|(x-z)^\alpha| |x-z|^{-d}|D^\alpha u(z)|dz\\
&\leq Cp\frac{1}{p!}(\gamma+1)^d\sum_{|\alpha|=p}\frac{p!}{\alpha!} \int_{\Omega} |x-z|^{p-d}|D^\alpha u(z)|dz.
\end{align*}
(The fact that for $u\in\mathbb{P}^{p-1}$, $\mathcal{I}^{p-1}u=u$ follows from the equality $D^\alpha \mathcal{I}^{p-1}=\mathcal{I}^{p-1-|\alpha|}D^{\alpha}$. )
By Lemma~\ref{l:riesz} and the triangle inequality, the Cauchy-Schwarz inequality, \eqref{e:consequencemulti}, and the definition 
of $|\cdot|_{H^p_k(\Omega)}$ \eqref{e:highPSobolev}, 
\begin{align}\nonumber
\|u-\mathcal{I}^{p-1}u\|_{L^2(\Omega)}&\leq C\frac{(\operatorname{diam}(\Omega))^p}{p!}(\gamma+1)^d\sum_{|\alpha|=p}\frac{p!}{\alpha!} \|D^\alpha u\|_{L^2(\Omega)}\\ \nonumber
&\leq C\frac{(\operatorname{diam}(\Omega))^p}{p!}(\gamma+1)^d\Big(\sum_{|\alpha|=p}\frac{p!}{\alpha!}\Big)^{\frac{1}{2}}\Big(\sum_{|\alpha|=p} \frac{p!}{\alpha!}\|D^\alpha u\|^2_{L^2(\Omega)}\Big)^{\frac{1}{2}}\\
&= C\frac{(k\operatorname{diam}(\Omega))^p}{p!}(\gamma+1)^dd^{\frac{p}{2}}|u|_{H_k^p(\Omega)}.
\label{e:clocksChange1}
\end{align}
By the definition 
of $|\cdot|_{H^\ell_k(\Omega)}$ \eqref{e:highPSobolev}, the property $D^{\alpha}\mathcal{I}^{p-1}=\mathcal{I}^{p-1-|\alpha|}D^\alpha$, \eqref{e:clocksChange1}, and 
\eqref{e:highPSobolev} again,
\begin{align*}
&|u-\mathcal{I}^{p-1}u|_{H_k^{\ell}(\Omega)}^2=\sum_{|\alpha|=\ell}\frac{\ell!}{\alpha!}\|k^{-|\alpha|}D^\alpha u-k^{-|\alpha|}D^\alpha \mathcal{I}^{p-1}u\|^2_{L^2(\Omega)}\\
&=\sum_{|\alpha|=\ell}\frac{\ell!}{\alpha!}\|k^{-|\alpha|}D^\alpha u-\mathcal{I}^{p-1-|\alpha|}(k^{-|\alpha|}D^\alpha u)\|^2_{L^2(\Omega)},\\
&\leq C\frac{(k\operatorname{diam}(\Omega))^{2p-2\ell}}{[(p-\ell)!]^2}(\gamma+1)^{2d}d^{p-\ell}\sum_{|\alpha|=\ell}\frac{\ell!}{\alpha!}|k^{-|\alpha|}D^\alpha u|^2_{H_k^{p-|\alpha|}(\Omega)},\\
&\leq C\frac{(k\operatorname{diam}(\Omega))^{2p-2\ell}}{[(p-\ell)!]^2}(\gamma+1)^{2d}d^{p-\ell}\\
&\qquad \sum_{|\alpha|=\ell}\frac{\ell!}{\alpha!}
\sum_{\substack{|\beta|=p-\ell\\\beta\in\mathbb{N}^d}}\frac{(p-\ell)!}{\beta!}\|(k^{-1}D_{x_1})^{\beta_1}\dots (k^{-1}D_{x_d})^{\beta_d}(k^{-1}D)^\alpha u\|_{L^2(\Omega)}^2,\\
&\leq C\frac{(k\operatorname{diam}(\Omega))^{2p-2\ell}}{[(p-\ell)!]^2}(\gamma+1)^{2d}d^{p-\ell}|u|_{H_k^p(\Omega)}^2,
\end{align*}
where in the last line we used that 
$$
\sup_{\substack{|\alpha|=\ell\\|\beta|=p-\ell}}\frac{(\alpha+\beta)!}{\alpha!\beta!}=\frac{p!}{\ell! (p-\ell)!}
$$
(interpreting the left-hand side combinatorially, we see that the number on the left is less than the number of ways of choosing a group of $\ell$ objects from a group of $p$ objects, and this number is obviously achieved when $\alpha$ and $\beta$ have only one non-zero entry).
Therefore, by \eqref{e:highPSobolev},
\begin{align*}
\|u-\mathcal{I}^{p-1}u\|_{H_k^{\ell}(\Omega)}^2&=\sum_{j=0}^\ell\frac{1}{(j!)^2}|u-\mathcal{I}^{p-1}u|_{H_k^j(\Omega)}^2\\
&\leq \sum_{j=0}^\ell C\frac{(k\operatorname{diam}(\Omega))^{2p-2j}}{[(p-j)!j!]^2}(\gamma+1)^{2d}d^{p-j}|u|_{H_k^p(\Omega)}^2,
\end{align*}
which is the result \eqref{e:approxOmega}.
\end{proof}

In Proposition \ref{p:pExplicit1}, the regularity of $u$ is tied to the degree of the polynomials; i.e., Proposition \ref{p:pExplicit1} applies to $u\in H^m$ with $m$ at least the polynomial degree $p$. We now
show that the error when
approximating a \emph{fixed-regularity} function by polynomials decreases with the polynomial degree. 
The main idea is to use the spectral theory of a self-adjoint elliptic operator whose eigenfunctions are polynomials to approximate a function with fixed regularity. In this case, we use the Legendre polynomials:~eigenfunctions of the operator $-\partial_x(1-x^2)\partial_x$ on $[-1,1]$.
\begin{proposition}
\label{p:lowregularityApproximation}
Let $\Omega\Subset \mathbb{R}^d$ open with Lipschitz boundary and $m \geq \ell\geq 0$. Then there is $C>0$ such that for all  $p\geq m$, there is $\mathcal{L}^{p-1}:H_k^m(\Omega) \to \mathbb{P}^{p-1}$ such that for all $u\in H_k^m(\Omega)$, 
\beq\label{e:Legendre1}
\| (I-\mathcal{L}^{p-1})u\|_{H_k^\ell(\Omega)}\leq C\frac{(\operatorname{diam}(\Omega)k)^{m-\ell}}{p^{m-\ell}}|u|_{H_k^m(\Omega)},
\eeq
and for $u\in\mathbb{P}^{p-1}$, $\mathcal{L}^{p-1}u=u$. 
\end{proposition}
\begin{proof}

We start by consider $\Omega$ with $\operatorname{diam}(\Omega)=1$. Then, without loss of generality, we can assume $\Omega\Subset [-1,1]^d$. 
By the now-classical result of~\cite{Ca:61}, there is $E:H^m(\Omega)\to H^m(\mathbb{R}^d)$ such that 
$$
\|Eu\|_{H^m(\mathbb{R}^d)}\leq C\|u\|_{H^m(\Omega)},\qquad (Eu)|_{\Omega}=u.
$$
Let $\chi \in C_c^\infty((-1,1)^d)$ with $\supp (1-\chi)\cap \Omega=\emptyset$. Define the sesquilinear form $Q:H^1([-1,1]^d)\times H^1([-1,1]^d)\to \mathbb{C}$ by 
$$
Q(u,v)=\sum_{i=1}^d \big\langle (1-x_i^2)\partial_{x_i}u,\partial_{x_i}v\big\rangle.
$$
Then, for all $u,v \in H^1$,
$$
0\leq Q(u,u),\qquad |Q(u,v)|\leq C\|u\|_{H^1}\|v\|_{H^1},
$$
and hence the Friedrichs extension of $Q$, which we denote by $P:L^2([-1,1]^d)\to L^2([-1,1]^d)$, is self adjoint with domain
$$
\mathcal{D}(P):=\big\{ u\in H^1\,:\, |Q(u,v)|\leq C_u\|v\|_{L^2}\, \text{for all }v\in H^1\big\}. 
$$
Observe that $H^2([-1,1]^d)\subset \mathcal{D}(P)$ and, for $u\in H^2$, 
$$
Pu= -\sum_{i=1}^d\partial_{x_i}(1-x_i^2)\partial_{x_i}u.
$$

Let $p_{\lambda_j}:[-1,1]\to \mathbb{R}$ be an orthonormal basis of eigenfunctions for $P$ when $d=1$; i.e., satisfying
$$
-\partial_x(1-x^2)\partial_xp_{\lambda_j}-\lambda_jp_{\lambda_j}=0.
$$ 
We claim that $\lambda_j= j(j+1)$ and $p_{\lambda_j}$ is a polynomial of degree $j$. To see this observe that $\mathbb{P}^\ell$ is invariant under $P$ for every $\ell$ and hence $P$ has $\ell+1$ linearly independent eigenfunctions in $\mathbb{P}^{\ell}$ for every $\ell$. This implies that for every $n$ there is $j$ such that $p_{\lambda_j}$ is a polynomial of degree exactly $n$. Since the set of polynomials is dense in $L^2([-1,1])$ we then obtain that for every $j$  there is $n_j$ such that $p_{\lambda_j}$ is a polynomial of degree $n_j$. To compute the eigenvalue observe that, since $p_{\lambda_j}(x)=\sum_{\ell=0}^{n_j}b_\ell x^\ell$ with $b_{n_j}\neq 0$
$$
0=(-\partial_{x}(1-x^2)\partial_xp_{\lambda_j}-\lambda_j p_{\lambda_j})= \sum_{\ell=0}^{n_j} a_\ell x^\ell,
$$
where
$$
a_{n_j}x^{n_j}=\partial_xx^2\partial_x b_{n_j}x^{n_j} -\lambda_jb_{n_j}x^{n_j}=(n_j(n_j+1)-\lambda_j) b_{n_j}x^{n_j}.
$$
Hence $\lambda_j=n_j(n_j+1)$. We now now relabel these eigenfunctions by their degree, and recall that, in fact, 
$p_n:[-1,1]\to\mathbb{R}$, $n=0,1,\dots,$ are the Legendre polynomials.

For $\vec{n}\in\mathbb{N}^d$, define $p_{\vec{n}}:[-1,1]^d\to \mathbb{R}$ by
$$
p_{\vec{n}}(x):=\prod_{j=1}^dp_{\vec{n}_j}(x_j).
$$
By the definition of the Legendre polynomials,
$$
\Big(P-\sum_{j=1}^d \vec{n}_j(\vec{n}_j+1)\Big)p_{\vec{n}}= 0,
$$
and hence $\{p_{\vec{n}}\}_{\vec{n}\in\mathbb{N}^d}$ forms an orthonormal basis of eigenfunctions of $P$ for $L^2([-1,1]^d)$. 

In particular, there is $c>0$ such that, with $\Pi_{p-1}:L^2([-1,1]^d)\to \mathbb{P}^{p-1}$  the orthogonal projector,  $[\Pi_{p-1},P]=0$ and, since
\beqs
c_d \Big( \sum_j \vec{n}_j\Big)^2
\leq \sum_j \vec{n}_j(\vec{n}_j+1)
\eeqs
for some $c_d>0$,
$$
1_{(-\infty, cp^2]}(P)\Pi_{p-1}=1_{(-\infty, cp^2]}(P). 
$$
Now set $\widetilde{\mathcal{L}}^{p-1}:= 1_{\Omega}\Pi_{p-1}\chi E$. 
Observe that, since $1_\Omega E= I$ and $1_\Omega\chi= 1_\Omega$,
\begin{align*}
(I-\widetilde{\mathcal{L}}^{p-1})&= 1_{\Omega} (I-\Pi_{p-1})\chi E\\
&=1_{\Omega} (P+1)^{-m/2}1_{(cp^2,\infty)}(P)(I-\Pi_{p-1})(P+1)^{m/2} \chi E.
\end{align*}
Now, by local elliptic regularity, 
$$
\|1_{\Omega}(P+1)^{-t/2}\|_{L^2([-1,1]^d)\to H^{t}(\Omega)}\leq C_t,
$$
so that 
\begin{align*}
&\|(I-\widetilde{\mathcal{L}}^{p-1})\|_{H^\ell(\Omega)\to H^m(\Omega)}\\
&\leq C_{m-\ell}\|(P+1)^{(m-\ell)/2}1_{(cp^2,\infty)}(P)\|_{L^2\to L^2}\|(P+1)^{\ell/2} \chi E\|_{H^\ell(\Omega)\to L^2}\\
&\leq  C\langle p\rangle^{-\ell+m}.
\end{align*}

To define $\mathcal{L}^{p-1}$, we need to construct an operator so that $(I-\mathcal{L}^{p-1})u$ is controlled by the $H^m$ seminorm of $u$ (as in  \eqref{e:Legendre1}). For this, let $\mathcal{Z}:H^m(\Omega)\to \mathbb{P}_{m-1}$ be the unique operator such that 
\begin{equation}
\label{e:average0}
\int_{\Omega} \partial^\alpha (u-\mathcal{Z}u)=0,\qquad |\alpha|\leq m-1,
\end{equation}
and set 
$$
\mathcal{L}^{p-1}u= \widetilde{\mathcal{L}}^{p-1}(u-\mathcal{Z}u)+\mathcal{Z}u
$$
Then,
$$
\|u-\mathcal{L}^{p-1}u\|_{H^\ell(\Omega)}=\|u-\mathcal{Z}u- \widetilde{\mathcal{L}}^{p-1}(u-\mathcal{Z}u)\|_{H^\ell}\leq C\langle p\rangle^{-\ell+m}\|u-\mathcal{Z}u\|_{H^m}.
$$
Using the Poincar\'e inequality together with~\eqref{e:average0} repeatedly, we obtain
$$
\|u-\mathcal{L}^{p-1}u\|_{H^\ell(\Omega)}\leq C\langle p\rangle^{-\ell+m} |u-\mathcal{Z}u|_{H^m(\Omega)}= C\langle p\rangle^{-\ell+m}|u|_{H^m(\Omega)},
$$
and the result \eqref{e:Legendre1} then follows by scaling.

The fact that $\mathcal{L}^{p-1}u=u$ for $u\in\mathbb{P}^{p-1}$ follows from the fact that $u=\tilde{\mathcal{L}}u$ if $u\in\mathbb{P}^{p-1}$.
\end{proof}

\subsection{Boundary compatible approximations}	

\label{s:boundaryCompatible}
We now show that given a polynomial approximant on a simplex with control on a sufficiently-high Sobolev norm of the error (more than $d/2$ derivatives in $L^2$) one can modify the polynomial approximant to one that is \emph{boundary compatible} in the sense of Definition~\ref{d:boundaryCompatible} below. Boundary-compatible polynomial approximants can then be extended to conforming piecewise-polynomial approximants on $C^1$ triangulations (in the sense of Definition \ref{d:Crtriang}); this is done in \S\ref{s:proofBoundaryCompatible}.

\subsubsection{Simplices and boundary compatibility} 

\begin{definition}
Let $K\subset \mathbb{R}^d$ be an open simplex. The boundary of $K$ can be decomposed as 
$$
\bigcup_{j=1}^d \bigcup_{i=1}^{N_{j,d}}F_{j,i}
$$
where $F_{j,i}$ is an open simplex in a $(d-j)$-dimensional hyperplane for $j<d$,$F_{d,i}=\{x_i\}$, for some $x_i\in\mathbb{R}^d$, and $F_{j,i}\cap F_{\ell,m}=\emptyset$ for $(j,i)\neq (\ell,m)$. We call $F_{j,i}$ \emph{a face of codimension $j$ for $K$. } For convenience, we set $F_{0,1}=K$, and $N_{0,1}=1$.
\end{definition}

\begin{definition}
\label{d:boundaryCompatible}
 Let $K\subset \mathbb{R}^d$ be a simplex. An operator $T:C^\infty(K)\to C^\infty(K)$ is \emph{boundary compatible} if for all faces $F$
\begin{equation}
\label{e:boundaryCompatible1}
u|_{\overline{F}}=0\qquad \Rightarrow\qquad (Tu)|_{\overline{F}}=0
\end{equation}
and for any pair of faces of $K$, $F_1$ and $F_2$, and any affine isomorphism $\kappa:\overline{F_1}\to \overline{F_2}$, 
\begin{equation}
\label{e:boundaryCompatible2}
\big[T(u|_{\overline{F_2}}\circ \kappa)\big]\circ \kappa^{-1}=Tu|_{\overline{F_2}}.
\end{equation}
(Note that~\eqref{e:boundaryCompatible2} makes sense since, by~\eqref{e:boundaryCompatible1}, $T$ defines a map from $C^\infty(F)$ to itself.)
\end{definition}

The goal of this subsection (\S\ref{s:boundaryCompatible}) is to establish the following proposition which shows that one can modify polynomial approximants to boundary-compatible polynomial approximants. 

\begin{proposition}
\label{p:toBoundaryCompatible}
Let $K\subset \mathbb{R}^d$ be an open $d$-simplex, then for each face of $K$, $F_{j,i}$, there are operators $Q^{p-1}_{j,i}:H^{1/2}(F_{j,i})\to \mathbb{P}^{p-1}(F_{j,i})$, $p=2,3,\dots$ such that  for all $m> \frac{d}{2}$, $m\geq 1$ there is $C>0$ such that if $\ell\geq m$, $\mathcal{J}^{p-1}:H_p^\ell(K)\to \mathbb{P}^{p-1}$, and $(I-\mathcal{J}^{p-1}):H_p^\ell(K)\to H_p^m(K)$, then there is a boundary compatible operator $\mathcal{C}^{p-1}:H^\ell(K)\to \mathbb{P}^{p-1}$ satisfying 
$$
\|(I-\mathcal{C}^{p-1})u\|_{H_p^1(K)}\leq C\|(I-\mathcal{J}^{p-1})u\|_{H_p^m(K)}
$$
and
\begin{equation}
\label{e:faceOperator}
(\mathcal{C}^{p-1}u)|_{F_{j,i}}=Q^{p-1}_{j,i}(u|_{F_{j,i}}).
\end{equation}
In addition, if for all $u\in\mathbb{P}^{p-1}$, $\mathcal{J}^{p-1}u=u$, then for all $u\in\mathbb{P}^{p-1}$, $\mathcal{C}^{p-1}u=u$. 
\end{proposition}

\begin{remark}\label{r:elementByElement}
Proposition \ref{p:toBoundaryCompatible} is one important place where our construction of approximants on reference elements differs from that in \cite{MeSa:10}. Indeed, we construct approximants satisfying~\eqref{e:boundaryCompatible2}, with this property used in our construction of globally $H^1$-conforming piecewise polynomial approximants (in \S\ref{s:theEndPoly}). In contrast, \cite{MeSa:10} construct approximants having an ``element-by-element construction''
in the sense of~\cite[Definition 5.3]{MeSa:10}. Such approximants do not have the property~\eqref{e:boundaryCompatible2} because they involve, for example,
\begin{equation}
\label{e:minimizeToday}
\operatorname{argmin}\big\{p\|u-q\|_{L^2(f)}+\|u-q\|_{H^1(f)}\big\}
\end{equation}
(see \cite[Equation 5.5]{MeSa:10}).
In general,
\begin{multline*}
\operatorname{argmin}\big\{ p\|u-q\|_{L^2(f)}+\|u-q\|_{H^1(f)}\,:\, q\in \mathbb{P}^{p-1}\big\}\\
\neq (A^{-1})^*\operatorname{argmin}\big\{p\|A^*u-q\|_{L^2(A^{-1}(f))}+\|A^*u-q\|_{H^1(A^{-1}(f))}\,:\, q\in \mathbb{P}^{p-1}\big\}
\end{multline*}
for an affine isomorphism $A$. (One can see this, e.g., in one dimension from the fact that $\|u\|_{L^2}$ and $|u|_{H^1}$ scale differently.)
To deal with this issue, we average minimisations similar to \eqref{e:minimizeToday} over all possible boundary faces (see~\eqref{e:optimalPolyExtend} below).
\end{remark}

We prove Proposition~\ref{p:toBoundaryCompatible} in \S\ref{s:proofBoundaryCompatible}, but we discuss the ideas used in the proof here.

\paragraph{Idea of the  proof of Proposition~\ref{p:toBoundaryCompatible}.}

To modify $\mathcal{J}^{p-1}$ into a boundary-compatible operator, we follow the steps from~\cite[Appendix B]{MeSa:10}. 
Let $e_d=(I-\mathcal{J}^{p-1})u$.
For $j=d,d-1,\dots, 1$, 
let $g_j$ be a polynomial of degree $p$ such that on each codimension $j$ face of $K$, $F_{j,i}$, $g_j|_{F_{j,i}}=:\tilde{T}_{j,i}(e_j|_{F_{j,i}})$, for some operator $T_{j,i}$, is a good polynomial approximant of $e_j|_{F_{j,i}}$ with zero boundary conditions on $\partial F_{j,i}$ (in the case $j=d$, there are no boundary conditions and the polynomial $g_d$ matches $e_d$ exactly at each vertex), and such that 
\beq\label{e:theEnd1}
\tilde{T}_{j,i}u=(\kappa^{-1})^*\big(\tilde{T}_{j,i'}(\kappa^* u)\big)
\eeq
for any affine isomorphism, $\kappa$ that sends $F_{j,i'}$ to $F_{j,i}$ (compare to \eqref{e:boundaryCompatible2}). 
Set $e_{j-1}:=e_j-g_j$. Then, the desired boundary compatible operator $\mathcal{C}^{p-1}=u-e_0$.

The key ingredient in this procedure is the construction of $g_j$. To achieve this, we take the average of the best $H^{1/2}$ approximations of $\kappa^*(e_j|_{F_{j,i}})$ (among polynomials with zero boundary conditions) where $\kappa$ ranges over all affine isomorphisms mapping $F_{j,i'}$ to $F_{j,i}$ for some $i'$ (strictly speaking, pulled back to $F_{j,i}$ using $(\kappa^{-1})^*$) and extend it to all of $K$ with good enough $H^1$ estimates. This is done in \S\ref{s:extension} following ideas from~\cite{Mu:97,BeDaMa:95} to find a $p$-independent polynomial extension operator that is bounded from $H^{1/2}$ of a codimension 1 boundary face to $H^1$ of a simplex (see Lemma~\ref{l:extendA}). Using an induction procedure together with some ideas from~\cite[Appendix B]{MeSa:10}, we use this extension to construct a $p$-dependent extension operator from any boundary face into $H^1$ of all lower codimension boundary faces having  improved estimates as a function of $p$ (see Lemma~\ref{l:extend}). One of the key elements of this last step is a polynomial approximation lemma that respects zero boundary conditions (see Lemma~\ref{l:polynomialsWithBoundaries}).

\subsubsection{Polynomial approximation in simplices respecting boundary conditions}
\label{s:boundaryConditionPoly}
Before constructing boundary compatible polynomial approximation operators (via Proposition~\ref{p:toBoundaryCompatible}), we need one further lemma on polynomial approximation specifically in simplices. This lemma shows that, one can both accurately approximate $H^2$ functions (in $H^1$) with high degree polynomials \emph{and} respect boundary conditions. The proof of the lemma is similar to that of Proposition~\ref{p:lowregularityApproximation}, but we use the associate Legendre polynomials of order two; eigenfunctions of $-\partial_x(1-x^2)\partial_x +4(1-x^2)^{-1}$ on $[-1,1]$ when we need to preserve boundary data.
\begin{lemma}
\label{l:polynomialsWithBoundaries}
Suppose that $K\subset \mathbb{R}^{d}$ is an open $d$ simplex and $0\leq m\leq d$. Then there is $C>0$ such that for all $p\geq 1$ there is $\mathcal{A}^{p-1}:\tilde{H}^1_{m,0}(K)\to \mathbb{P}^{p-1}\cap \tilde{H}^1_{m,0}(K)$ satisfying
$$
\|(I-\mathcal{A}^{p-1})u\|_{H^1(K)}\leq Cp^{-1}\|u\|_{H^2(K)},
$$
where 
$$
\tilde{H}_{m,0}^1(K):=\big\{ u\in H^1(K)\,:\, u|_{F_{1,j}}=0,\quad j=0,\dots,m\big\}.
$$
\end{lemma}
\begin{proof}
The proof will be similar to that of Proposition~\ref{p:lowregularityApproximation} except that we change the operator $P$ to impose boundary conditions. To do this, we use the reflection symmetry of the operator
$$
-\partial_x(1-x^2)\partial_x +\frac{4}{1-x^2}
$$
(which has polynomial eigenfunctions on $[-1,1]$) to find an operator with certain Dirichlet boundary conditions on the faces of a simplex that still has polynomial eigenfunctions.

Without loss of generality, we can assume that 
$$
K=\big\{ x\in\mathbb{R}^d\,:\, -1<x_1<x_2<\dots<x_d<1\}
$$
and 
$$
F_{1,0}\subset \{ x_1=-1\},\quad F_{1,j}\subset \{x_{j-1}<x_j\},\, j=1,\dots, d-1,\,\, F_{1,d}\subset \{ x_d=1\}.
$$
Define also
$$
\tilde{K}:=\big\{ x\in\mathbb{R}^{m}\,:\, -1<x_1<x_2<\dots<x_m<1\big\}\times (-1,1)^{d-m}.
$$
Let $\tilde{E}:C^\infty(K)\to C^\infty(\tilde{K})$ such that $1_{K}\tilde{E}=I$,
and $\tilde{E}:H^s(K)\to H^s(\tilde{K})$ is bounded for all $s\geq 0$. 

For $\sigma:\{1,\dots,m\}\to \{1,\dots,m\}$ a permutation, define 
$$
\tilde{K}^\sigma:=\{ x\in\mathbb{R}^m\,:\, -1<x_{\sigma(1)}<x_{\sigma(2)}<\dots<x_{\sigma(m)}<1\}\times (-1,1)^{d-m},
$$
and 
$$
\iota_\sigma(x_1,\dots, x_d)=(x_{\sigma(1)},\dots, x_{\sigma(m)},x_{m+1},\dots,x_d).
$$
Let $\operatorname{sgn}(\sigma)$ denote the signature of a permutation.
Let
\begin{align*}
H_m^1(\tilde{K}):=\big\{ u\in H^1(\tilde{K})\,:\, u|_{x_1=-1}=0,\, u|_{x_{j-1}=x_j}=0,\, j=1,\dots,m\big\}\\
H_m^1([-1,1]^d):=\big\{ u\in H^1([-1,1]^d)\,:\, u|_{x_j=\pm 1}=0,\,j=1,\dots,m\big\}.
\end{align*}
We then define an extension operator
$E:H_{m,0}^1(K)\to H_{m,0}^1([-1,1]^d)$ by
$$
(Eu)(x)=(-1)^{\operatorname{sgn}(\sigma)}\tilde{E}u(\iota_{\sigma}^{-1}(x)),\qquad x\in \overline{\tilde{K}^{\sigma}}.
$$
Note also that $E:H_{m,0}^1(K)\cap H^2(K)\to H_{m,0}^1([-1,1]^d)\cap H^2([-1,1]^d)$. 

We now define two self-adjoint operators, one on $L^2(\tilde{K})$ and one on $L^2([-1,1]^d)$. Define the sesquilinear forms $Q:H^1_{m,0}([-1,1]^d)\times H^1_{m,0}([-1,1]^d)\to \mathbb{C}$ and $Q^{\tilde{K}}:H^1_{m,0}(\tilde{K})\times H^1_{m,0}(\tilde{K})\to \mathbb{C}$ by
\begin{align*}
Q(u,v)&:=\sum_{i=1}^d \langle (1-x_i^2)\partial_{x_i}u,\partial_{x_i}v\rangle_{L^2([-1,1]^d)} +4\sum_{i=1}^m\langle\tfrac{1}{(1-x_i^2)}u,v\rangle_{L^2([-1,1]^d)}\\
Q^{\tilde{K}}(u,v)&:=\sum_{i=1}^d \langle (1-x_i^2)\partial_{x_i}u,\partial_{x_i}v\rangle_{L^2(\tilde{K})} +4\sum_{i=1}^m\langle\tfrac{1}{(1-x_i^2)}u,v\rangle_{L^2(\tilde{K})}.
\end{align*}
Then, using the Hardy inequality,
\beq\label{e:Hardy}
\int_0^a \Big|\frac{1}{t}\int_{x}^{x+t}f(s)ds\Big|^2dt\leq 4\int_0^a|f(x+t)|^2dt,
\eeq
for the upper bounds,
\begin{gather*}
0\leq Q(u,u),\qquad |Q(u,v)|\leq C\|u\|_{H^1([-1,1]^d)}\|v\|_{H^1([-1,1]^d)},\\ 0\leq Q^{\tilde{K}}(u,u),\qquad |Q^{\tilde{K}}(u,v)|\leq C\|u\|_{H^1(\tilde{K})}\|v\|_{H^1(\tilde{K})}
\end{gather*}
and hence, the Friedrichs extensions of $Q$ and $Q^{\tilde{K}}$, respectively $P:L^2([-1,1]^d)\to L^2([-1,1]^d)$ and $P^{\tilde{K}}:L^2(\tilde{K})\to L^2(\tilde{K})$, are self adjoint with domains
\begin{align*}
\mathcal{D}(P)&:=
\big\{ u\in H^1_{m,0}([-1,1]^d)\,:\, |Q(u,v)|\leq C_u\|v\|_{L^2},\, \text{for all }v\in H_{m,0}^1([-1,1]^d)\big\}\\
\mathcal{D}(P^{\tilde{K}})&:=\big\{ u\in H^1_{m,0}(K)\,:\, |Q^{\tilde{K}}(u,v)|\leq C_u\|v\|_{L^2},\, \text{for all }v\in H_{m,0}^1(\tilde{K})\big\}.
\end{align*}

Notice that $H^2([-1,1]^d)\cap H_{m,0}^1([-1,1]^d)\subset \mathcal{D}(P)$ and, for $u\in H^2\cap H_{m,0}^1([-1,1]^d$, 
$$
Pu= -\sum_{i=1}^d\partial_{x_i}(1-x_i^2)\partial_{x_i}u+\sum_{i=1}^m\tfrac{4}{(1-x_i^2)}u.
$$
Let $p_{n+2}^{(2)}:[-1,1]\to \mathbb{R}$, $n=0,1,\dots$ be the associated Legendre polynomials of order $2$ (see, e.g., \cite[\S14.2(ii)]{Di:14}) normalized such that 
$$
\langle p_{n+2}^{(2)}, p_{l+2}^{(2)}\rangle_{L^2([-1,1])}=\delta_{nl},
$$
and $p_{n}:[-1,1]\to \mathbb{R}$, $n=0,1,\dots$ be the Legendre polynomials normalized such that 
$$
\langle p_{n}, p_{l}\rangle_{L^2([-1,1])}=\delta_{nl}.
$$
Then for $\vec{n}\in\mathbb{N}^d$, define $p_{\vec{n}}:[-1,1]^d\to \mathbb{R}$ by
$$
p_{\vec{n}}(x):=\prod_{i=m+1}^d p_{\vec{n}_i}(x_i)\prod_{j=1}^mp_{\vec{n}_j+2}^{(2)}(x_j),
$$
so that, by the definition of $p_{n+2}^{(2)}$ and $p_n$,
$$
\Big(P-\sum_{j=1}^m (\vec{n}_j+2)(\vec{n}_j+3)-\sum_{i=1}^{m+1}(\vec{n}_i)(\vec{n}_i+1)\Big)p_{\vec{n}}= 0.
$$
We claim that the associated Legendre polynomials of order 2 are complete. Indeed, observe that by~\cite[(14.7.8)]{Di:14}, 
$$
p_{n+2}^{(2)}=(1-x^2)\partial_x^2p_{n+2},
$$
and hence, if $u\in L^2$ is orthogonal to $p_{n+2}^{(2)}$ for all $n$, then,
$$
\langle u, (1-x^2)\partial_x^2p_{n+2}\rangle_{[-1,1]}=0 \quad\tfa n.
$$
But, since $p_n$ is a polynomial of degree $n$, $\operatorname{span}\{\partial_x^2p_{n+2}\}_{n=0}^\infty$ contains all polynomials and hence is dense in $L^2$. In particular, this implies that $(1-x^2)u=0$ and hence $u=0$.

Since both the Legendre polynomials and the associated Legendre polynomials of order 2 are complete, $\{p_{\vec{n}}\}_{\vec{n}\in\mathbb{N}^d}$ forms an orthonormal basis of eigenfunctions of $P$ for $L^2([-1,1]^d)$.

Now, suppose that $P^{\tilde{K}}u=\lambda u$. By the definitions of $E, Q,$ and $Q^{\tilde K}$, $PEu=\lambda Eu$ and hence $Eu$ is an eigenfunction of $P$ with eigenvalue $\lambda$.  This implies that
(very similar to in the proof of 
Proposition 
\ref{p:lowregularityApproximation}) there is $c>0$ such that, with $\Pi_{p-1}:L^2(\tilde{K})\to \mathbb{P}^{p-1}\cap H_0^1(\tilde{K})$ be the orthogonal projector,  $[\Pi_{p-1},P^{\tilde{K}}]=0$ and
$$
1_{(-\infty, cp^2]}(P^{\tilde{K}})\Pi_{p-1}=1_{(-\infty, cp^2]}(P^{\tilde{K}}). 
$$
Let $\mathcal{A}^{p-1}:=1_{K}\Pi_{p-1}\tilde{E}$. Then
\begin{align*}
(I-\mathcal{A}^{p-1})&= 1_K(I-\Pi_{p-1})\tilde{E}\\
&= 1_K(P^{\tilde{K}}+1)^{-1}1_{(cp^2,\infty)}(P^K)(I-\Pi_{p-1})(P^K+1)\tilde{E}.
\end{align*}
 Now, $(P^{\tilde{K}}+1)^{-1/2}:L^2(\tilde{K})\to H_{m,0}^1(\tilde{K})$, $(P^{\tilde{K}}+1):H_{m,0}^1(\tilde{K})\cap H^2(\tilde{K})\to L^2$ and 
 $$
  \|(P^{\tilde{K}}+1)^{-1/2}1_{(cp^2,\infty)}\|_{L^2\to L^2}\leq c\langle p\rangle^{-1}. 
 $$
Therefore, 
\begin{align*}
&\|(I-\mathcal{A}^{p-1})\|_{H^2(K)\cap H_{m,0}^1(K)\to H_{m,0}^1(K)}\\
&\leq \|(P^K+1)^{-1/2}1_{(cp^2,\infty)}(P)\|_{L^2\to L^2}\|(P^K+1) \chi E\|_{H^2(K)\cap H_{m,0}^1(K)\to L^2}\\
&\leq  C\langle p\rangle^{-1}.
\end{align*}
\end{proof}

\subsubsection{Extension operators}
\label{s:extension}
To extend polynomials effectively from faces into a simplex $K$, we need to construct extension operators that are bounded from $H^{1/2}$ to $H^1$, respect traces, and preserve the class of polynomials of a fixed degree. 

We start by adapting the arguments from~\cite[Section 2]{Mu:97} to construct a $p$-independent extension operator.
Let $Q=(0,1)^d$,
$Q_I:=(0,1)^d\times (0,\frac{1}{2})$, $\chi \in L^\infty(Q)$ with $\int \chi =1$ and define the operator 
$E_\chi:C^\infty(Q)\to L^\infty(Q_I)$ by
$$
(E_{\chi} u)(x,t):=\int u\big( (1-t)x+ty\big)\chi(y)dy.
$$

\begin{lemma}
\label{l:firstBounds}
Let $d\geq 1$. Then there is $C>0$ such that for all $\chi\in L^\infty$ with $\int \chi=1$ and $u\in C^\infty(\overline{Q})$, 
\begin{gather}
\|E_{\chi} u\|_{L^2(Q_I)}\leq C\|\chi\|_{L^\infty}\min_i \|x_i^{\frac{1}{2}}u\|_{L^2(Q)}.\label{e:goodL2}
\end{gather}
\end{lemma}
\begin{proof}
To prove~\eqref{e:goodL2}, we change variables so that \begin{align*}
|E_{\chi}u(x,t)|&= \Big|t^{-d}\int u(w)\chi\big(t^{-1}(w-(1-t)x)\big)dw\Big|\\
&\leq \|\chi \|_{L^\infty}t^{-d}\int |u(w)|1_{tQ}(w-(1-t)x)dw.
\end{align*}
We prove~\eqref{e:goodL2} with $i=1$; the other estimates come from a permutation of coordinates.
Then, by the Hardy inequality \eqref{e:Hardy}
in the seventh line and the $L^2$ boundedness of the Hardy-Littlewood maximal function 
$$
M_{d-1}f(z):= \sup_{t>0}\Big|\frac{1}{t^{d-1}}\int_{tQ_{d-1}}f(z-w)dw\Big|
$$
in the eighth line, 
\begin{align*}
&\|F_{\Omega}u\|^2_{L^2(Q_I)}\\
&\leq \int_0^{\frac{1}{2}}\int_{Q}\Big|t^{-d}\|\chi \|_{L^\infty}\int |u(w)|1_{tQ}\big(w-(1-t)x\big)dw\Big|^2dxdt\\
&=\int_0^{\frac{1}{2}}\int_{(1-t)Q}(1-t)^{-d}\Big|t^{-d}\|\chi \|_{L^\infty}\int |u(w)|1_{tQ}(w-z)dw\Big|^2dzdt\\
&=2^d\|\chi \|_{L^\infty}^2\int_{Q}\int_0^{\min(\min_i(1-z_i),\frac{1}{2})}\Big|t^{-1}\int_{z_1}^{z_1+t}t^{-d+1}\int |u(w)|1_{tQ_{d-1}}(w'-z')dw'dw_1\Big|^2dtdz\\
&=2^d\|\chi \|_{L^\infty}^2\int_{Q}\int_0^{\min(\min_i(1-z_i),\frac{1}{2})}\Big|t^{-1}\int_{z_1}^{z_1+t}\int M_{d-1}|u(w_1,\bullet)|(z')dw_1\Big|^2dtdz\\
&=2^{d+2}\|\chi \|_{L^\infty}^2\int_{Q}\int_0^{\min(\min_i(1-z_i),\frac{1}{2})}\Big|\int M_{d-1}|u(z_1+t,\bullet)|(z')\Big|^2dtdz\\\
&= 2^{d+2}\|\chi \|_{L^\infty}^2\int_0^{1}\int_{0}^{\min(1-z_1,\frac{1}{2})}\int_{Q_{d-1}}\Big|\int M_{d-1}|u(z_1+t,\bullet)|(z')\Big|^2dz'dtdz_1\\
&\leq C2^{d+2}\|\chi \|_{L^\infty}^2\int_0^{1}\int_{0}^{\min(1-z_1,\frac{1}{2})}\|u(z_1+t,\bullet)\|_{L^2_{z'}}^2dtdz_1\\
&=C2^{d+2}\|\chi \|_{L^\infty}^2\int_0^{1/2}\int_{z_1}^{\min(1,\frac{1}{2}+z_1)}\|u(s,\bullet)\|_{L^2_{z'}}^2dsdz_1\\
&\leq C2^{d+2}\|\chi \|_{L^\infty}^2\int_0^1\int_{0}^s\|u(s,\bullet)\|_{L^2_{z'}}^2dz_1ds\leq C2^{d+2}\|\chi \|_{L^\infty}^2\|x_1^{\frac{1}{2}}u\|_{L^2}^2.
\end{align*}
\end{proof}

\begin{lemma}
Let $\Omega \subset Q$ be convex and $\chi \in C^1_{\comp}(\Omega)$. Then there is $C>0$ such that
\begin{equation}
\label{e:basicH1}
\big\|E_{\chi} u\big\|_{H^1(\Omega\times (0,1/2)}\leq C\|u\|_{H^{1/2}(\Omega)}
\end{equation}
and, for any $1\leq j\leq {d-1}$,
\begin{equation}
\label{e:conjugatedH1}
\Big\|\prod_{i=1}^{j}x_i\Big[E_{\chi}u(\bullet) \prod_{i=1}^{j}(\bullet)_{i}^{-1}\Big](x,t)\Big\|_{H^1(\Omega \times (0,1/2))}\leq C\Big(\sum_{i=1}^j\|x_i^{-1/2}u\|_{L^2(\Omega)}+\|u\|_{H^{1/2}(\Omega)}\Big). 
\end{equation}
\end{lemma}
\begin{proof}
We begin by proving~\eqref{e:basicH1} and first change variables to $z=(1-t)x$. (Since $0\leq t\leq\frac{1}{2}$, this is a diffeomorphism.)
By~\eqref{e:goodL2}, it remains to estimate $\partial_{z_i}E_{\chi} u$ and $\partial_t E_{\chi} u$ in $L^2$. Changing variables as in Lemma~\ref{l:firstBounds}, and setting $z=(1-t)x$,  we have 
\begin{align*}
\partial_{z_i}E_{\chi}&=t^{-d-1}\int u(w)\partial_{z_i}\chi(t^{-1}(w-z))dw\\
&=t^{-d-1}\int (u(w)-u(z))(\partial_{z_i}\chi(t^{-1}(w-z))dw\\
&=t^{-\frac{d+1}{2}}\int \frac{u(w)-u(z)}{|z-w|^{\frac{d+1}{2}}}t^{-\frac{d+1}{2}}|w-z|^{\frac{d+1}{2}}\partial_{z_i}\chi(t^{-1}(w-z))dw.
\end{align*}
Hence, letting $\tilde{\chi}\in C_c^\infty(\Omega)$ with $\supp (1-\tilde{\chi})\cap \supp \chi=\emptyset$ and using the Cauchy-Schwarz inequality,
\begin{equation}
\label{e:itsALongRoad}
\begin{aligned}
&\|\partial_{z_i}E_{\chi} u\|_{L^2}^2\\
&\leq \int_0^{\frac{1}{2}}\int t^{-d-1}\Big|\int \frac{u(w)-u(z)}{|z-w|^{\frac{d+1}{2}}}t^{-\frac{d+1}{2}}|w-z|^{\frac{d+1}{2}}\partial_{z_i}\chi(t^{-1}(w-z))dw\Big|^2dz dt\\
&\leq \int_0^{\frac{1}{2}}\int t^{-d-1}
\bigg(
\int \frac{|u(w)-u(z)|^2}{|z-w|^{d+1}} \tilde{\chi}(t^{-1}(w-z))dw\\
&
\hspace{4cm}\cdot\int t^{-d+1}|w-z|^{d+1}|\partial_{z_i}\chi(t^{-1}(w-z))|^2dw\bigg) dz dt\\
&\leq C\int_0^{\frac{1}{2}}\int t^{-1}\int \frac{|u(w)-u(z)|^2}{|z-w|^{d+1}} \tilde{\chi}(t^{-1}(w-z))dw dz dt\\
&\leq C\int \int \frac{|u(w)-u(z)|^2}{|z-w|^{d+1}} \int_{c|z-w|}^{C\min(|z-w|,\frac{1}{2})}t^{-1}dt dw dz \\
&\leq C\int \int \frac{|u(w)-u(z)|^2}{|z-w|^{d+1}} dw dz \leq C\|u\|_{H^{1/2}}^2
\end{aligned}
\end{equation}
by the definition of the Slobodeckij seminorm (see, e.g., \cite[Equation 3.18]{Mc:00}).
Next, consider 
\begin{align*}
&\partial_t E_{\chi} u(z,t)\\
&=  \partial_t(E_{\chi} u(z,t)-u(z))\\
&=-t^{-d-1}\int (u(w)-u(z))\Big(d \chi(t^{-1}(w-z)) -\sum_{i}t^{-1}(w_i-z_i)\partial_{z_i}\chi(t^{-1}(w-z)\Big)dw.
\end{align*}
Now, since $\|s_i \partial_{z_i}\chi(s)\|_{L^\infty}<\infty$ and $\supp \chi\subset \Omega$, the proof 
that $\|\partial_t E_\chi u\|\leq C \| u\|_{H^{1/2}}$ follows from the same argument leading to~\eqref{e:itsALongRoad}.

We now prove~\eqref{e:conjugatedH1}; let
$$
v(x,t):=\prod_{i=1}^{j}x_i\Big[E_{\chi} u(\bullet)\prod_{i=1}^{j}(\bullet)_{i}^{-1}\Big](x,t)
$$
and first observe that for $x\in Q$, 
$$
|v(x,t)|\leq E_{\chi}(|u|)(x,t).
$$
Hence, by~\eqref{e:goodL2},
$$
\|v\|_{L^2(Q_I)}\leq \|E_{\chi}(|u|)\|_{L^2(Q_I)}\leq C\min_i\|x_i^{1/2}u\|_{L^2(Q)}.
$$
Thus, it remains to estimate the derivatives of $v$.

First, we estimate $\partial_{z_\ell}v$ for $1\leq \ell \leq j$. In fact, up to a permutation of indices, it is enough to estimate $\partial_{z_1}v$. 
\begin{align*}
\partial_{z_1}v&=\frac{1}{[(1-t)t]^d}\int u(w)\prod_{i=2}^j\frac{ z_i}{w_i}\Big(w_1^{-1}\chi(t^{-1}(w-z))+ \frac{z_1}{w_1}t^{-1}\partial_{z_1}\chi(t^{-1}(w-z)))\Big)dw\\
&=:I+II.
\end{align*}
For $I$, on $\supp \chi(t^{-1}(w-z))$, $w_i\geq z_i$, and hence
$$
|I|\leq \big|E_{\chi} |u(\bullet_1)^{-1}|(z,t)\big|.
$$
Hence,~\eqref{e:goodL2} implies
$$
\|I\|_{L^2}\leq C\|x_1^{-1/2}u\|_{L^2}.
$$
For $II$, 
\begin{align*}
II&=\frac{1}{[(1-t)t]^d}\int \Big[(u(w)-u(z))\prod_{i=1}^j\frac{z_i}{w_i}\\
&\hspace{4cm}+u(z)\Big(\prod_{i=1}^j\frac{z_i}{w_i}-1\Big)\Big]t^{-1}\partial_{z_1}\chi(t^{-1}(w-z))dw\\
&=:II_a+II_b.
\end{align*}
Since $|\frac{z_i}{w_i}|\leq 1$ on the support of the integrand, $II_a$ can be treated as in~\eqref{e:itsALongRoad} to obtain
$$
\|II_a\|_{L^2}\leq C\|u\|_{H^{1/2}}.
$$

For $II_b$, observe that the integrand is supported on $z_\ell\leq w_\ell\leq z_\ell+t$ for $1\leq \ell \leq d$ and hence
\begin{align*}
\int \Big|\frac{z_i}{w_i}-1\Big|t^{-1}|\partial_{z_1}\chi(t^{-1}(w-z))|&\leq  Ct^{d-2}\int_{z_i}^{z_i+t}1-\frac{z_i}{w_i}dw_i\\
&= Ct^{d-2}\big(t-z_i\log (1+\frac{t}{z_i})\big)\\
&\leq Ct^{d-2}\min (t^2z_i^{-1},t).
\end{align*}
Therefore, since 
\begin{align*}
\Big|\prod_{i=1}^j\frac{z_i}{w_i}-1\Big|&\leq  \sum_{\ell=1}^{j-1}\Big|\Big(\frac{z_\ell}{w_\ell}-1\Big)\prod_{i=\ell+1}^j\frac{z_i}{w_i}\Big| +\Big|\frac{z_j}{w_j}-1\Big|\\
&\leq \sum_{i=1}^j\Big|\frac{z_j}{w_j}-1\Big|,
\end{align*}
we have
$$
\|II_b\|_{L^2}^2\leq C\sum_{i=1}^j\int |u(z)|^2\int_0^{\frac{1}{2}}\min (z_i^{-2},t^{-2})dtdz\leq C\sum_{i=1}^j\|z_i^{-1/2}u\|_{L^2}^2.
$$
The combination of the bounds on $I$, $II_a$, and $II_b$ implies that
$$
\|\partial_{z_1}v\|_{L^2}\leq C\sum_{i=1}^j\|z_i^{-1/2}u\|_{L^2}+\|u\|_{H^{1/2}}.
$$
Now, for $\ell>j$, we have only terms of the form $II$ and hence the bounds there imply 
$$
\|\partial_{z_\ell}v\|_{L^2}\leq C\sum_{i=1}^j\|z_i^{-1/2}u\|_{L^2}+\|u\|_{H^{1/2}}.
$$
Finally, 
\begin{align*}
&\partial_t[(1-t)^{-d}v]\\
&=\partial_t \Big(t^{-d}\int \Big(u(w)\prod_{i=1}^j\frac{z_i}{w_i}-u(z)\Big)\chi(t^{-1}(w-z))dw\\
&= \int \Bigg((u(w)-u(z))\prod_{i=1}^j\frac{z_i}{w_i}+u(z)\Big(\prod_{i=1}^j\frac{z_i}{w_i}-1\Big)\Bigg)\partial_t (t^{-d}\chi(t^{-1}(w-z)))dw,
\end{align*}
and each term can be estimated as above (as in the estimate for $\partial_t E_{\chi}$, we use that the term arising from differentiating $\chi(t^{-1}(w-z))$ in $t$ has an extra power of $|z-w|$.)

\end{proof}

\begin{lemma}[Extension operators: version 1]
\label{l:extendA}
 Let $K\subset \mathbb{R}^d$ be an open $d$-simplex. Then there is an extension operator $E:H^{1/2}(\partial K)\to H^1(K)$ and $C>0$ such that for all $p\geq 0$, $E:\mathbb{P}^{p}(\partial K)\to \mathbb{P}^{p}(K)$ and
 \begin{gather*}
 E|_{\partial K}=I,\qquad
 \|E u\|_{H^1(K)}\leq C\|u\|_{H^{1/2}(\partial K)},
 \end{gather*}
 where 
 \begin{align*}
 \|u\|_{H^{s}(\partial K)}:=\inf\Big\{ \sum_{i}\|u|_{F_{1,i}}\|_{H^{s}(F_{1,i})}\,:\, U\in H^{s+1/2}(K)\,:\,U|_{\partial K}=u\Big\}.
 \end{align*}
\end{lemma}
\begin{proof}
Without loss of generality, we can assume that 
$$
K=\{x\in\mathbb{R}^d\,:\,0<x_i,\, \textstyle{\sum_{i=1}^{d}} x_i <1\}
$$
and 
$$
F_0:=F_{1,1}=\{ x\in \mathbb{R}^{d-1}\,:\, x_d=0,\, 0<x_i,\, \textstyle{\sum_{i=1}^{d-1}}x_i<1\}
$$
There is an invertible affine map taking $\tilde{\Phi}:\overline{K}\to \overline{K}$ so that 
$$
\tilde{\Phi}(F_0)=
\{ x\in \mathbb{R}^{d-1}\,:\, \, 0<x_i,\, \textstyle{\sum_{i=1}^{d}x_i}=1\}=:F.$$

Fix $\chi \in C_c^1(F)$ and for $u\in C^\infty(\overline{F})$ set $\tilde{E}u(x',x_d):= [E_{\chi}u]((1- \frac{1}{2}x_d)^{-1}x',  \frac{1}{2}x_d)$ and observe that 
$$
K\subset \big\{ (x', x_d)\,: \, (1-\tfrac{1}{2} x_d)^{-1}x'\in F,\quad 0< x_d<1\big\}.
$$
Then, define 
$$
E_0:= \tilde{\Phi}^*\tilde{E}(\tilde{\Phi}^{-1})^*,
$$
where, for a map $g:\overline{K}\to \overline{K}$, we denote by $g^*$ the pull back i.e., $g^*u(x):=u(g(x))$.  
Then, for all $p\geq 0$, $E_0:H^{1/2}(F_{1,1})\to H^1(K)$, $E_0:\mathbb{P}^p(F_{1,1})\to \mathbb{P}^p(K)$, and
\begin{gather*}
\|E_0u\|_{H^1(K)}\leq C\|u\|_{H^{1/2}(F_{1,1})},\qquad
E_0|_{F_{1,1}}=I.
\end{gather*}

Suppose by induction that for some $0\leq j<d$ we have constructed $E_j$ such that 
\begin{gather}
E_j:H^{1}(\partial K)\cap C^\infty(\partial K)\to C^\infty(\overline{K})\label{e:smoothExtend}\\
E_j:\mathbb{P}^{p}(\partial K)\to \mathbb{P}^p(K),\,\text{ for all }p\geq 0\label{e:polynomialExtend}\\
\|E_{j}u\|_{H^1(K)}\leq C\|u\|_{H^{\frac{1}{2}}(\partial K)},\label{e:mappingExtend}\\
E_{j}|_{\cup_{1\leq i\leq j+1}F_{1,i}}=I.\label{e:restrictionExtend}
\end{gather}
Here, we write $C^\infty(\partial K)$ for the restriction to $\partial K$ of $C^\infty (\overline{K})$. 

We now construct $E_{j+1}$ satisfying~\eqref{e:smoothExtend} to~\eqref{e:restrictionExtend} with $j$ replaced by $j+1$.

Let $\Phi_j:\overline{K}\to \overline{K}$ be an invertible affine map such that 
$$
\Phi_j(F_{1,j})=F_0,\qquad \Phi_j(F_{1,i})\subset \{x_i=0\},\, i=1,\dots,j+1.
$$
Then, define
$$
E_{j+1}u:=E_j u-  \Phi_j^*\Big(x_1\dots x_{j+1}E_0\Big[ \big(((\Phi_j^{-1})^*(E_ju-u))/x_1\dots x_{j+1}\big)|_{F_0}\Big]\Big)
$$

We first show~\eqref{e:smoothExtend} for $E_{j+1}$ i.e., that $E_{j+1}: C^\infty(\partial K)\to C^\infty(\overline{K})$. Indeed, observe that for $u\in C^\infty(\partial K)$, $E_ju\in C^\infty(\overline{K})$ and $(\Phi_j^{-1})^*E_ju|_{x_i=0}=(\Phi_j^{-1})^*u|_{x_i=0}$ for $i=1,\dots,j+1$. Hence $x_1^{-1}\dots x_{j+1}^{-1}(\Phi_j^{-1})^*(E_ju-u)\in C^\infty(\partial K)$ which implies $E_{j+1}u\in C^\infty(\overline{K})$.

Next, we show~\eqref{e:polynomialExtend} for $E_{j+1}$ i.e., that $E_{j+1}:\mathbb{P}^p(\partial K)\to \mathbb{P}^p(K)$. Indeed,  observe that for $u\in \mathbb{P}^p(\partial K)$, $E_ju\in \mathbb{P}^p(K)$ and $(\Phi_j^{-1})^*(E_ju-u)|_{x_i=0}=0$ for $i=1,\dots,j+1$. Hence $x_1^{-1}\dots x_{j+1}^{-1}(\Phi_j^{-1})^*(E_ju-u)|_{\partial K}\in \mathbb{P}^{p-j-1}$ which implies $E_{j+1}u\in \mathbb{P}^{p}(K)$ since $E_0:\mathbb{P}^{p-j-1}(F_{0})\to \mathbb{P}^{p-j-1}(K)$.

Next, we claim that~\eqref{e:mappingExtend} and~\eqref{e:restrictionExtend} hold; i.e.
\begin{gather}
\|E_{j+1}u\|_{H^1(K)}\leq C\|u\|_{H^{\frac{1}{2}}(\partial K)},\label{e:purpleEstimates}\\
E_{j+1}|_{F_{1,i}}=I,\quad 1\leq i\leq j+2\label{e:purpleEquality}.
\end{gather}

To see~\eqref{e:purpleEstimates}, first observe that by~\eqref{e:conjugatedH1},
\begin{align*}
\Big\|\Phi_j^*\big(x_1\dots x_{j+1}E_0\big[ \big((w)/x_1\dots x_{j+1}\big)&|_{F_0}\big]\big)\Big\|_{H^1(K)}\\
&\leq C\Big(\sum_{i=0}^{j+1}\|x_i^{-1/2}w\|_{L^2(F_0)}+\|w\|_{H^{1/2}(F_0)}\Big).
\end{align*}
Therefore, to prove~\eqref{e:purpleEstimates}, we need to check that 
\begin{equation}
\label{e:indigoInequality}
\sum_{i=0}^{j+1}\|x_i^{-1/2}(\Phi_j^{-1})^*(E_ju-u)\|_{L^2(F_0)}+\|(\Phi_j^{-1})^*(E_ju-u)\|_{H^{1/2}(F_0)}\leq C\|u\|_{H^{1/2}(\partial K)}
.
\end{equation}

For this, let $u\in H^{1/2}(\partial K)$ and let $u_n\in  C^\infty(\partial K)$ with $u_n\overset{H^{1/2}}{\to} u$. The boundedness of $E_j$ implies that $(\Phi_j^{-1})^*E_ju_n\to (\Phi_j^{-1})^*u$ in $H^{1}(K)$. 

Then, $(\Phi_j^{-1})^*E_ju_n\in C^\infty(K)$ and $(\Phi_j^{-1})^*(E_ju_n-u_n)|_{\{x_i=0\}\cap \partial K}=0$ for $1\leq i\leq j+1$. Moreover,
$$
\|(\Phi_j^{-1})^*E_j(u_n-u)\|_{H^{1/2}(\partial K)}\leq C\|(\Phi_j^{-1})^*E_j(u_n-u)\|_{H^1(K)}\leq C\|u_n-u\|_{H^{1/2}(\partial K)}.
$$
In particular, $[(\Phi_j^{-1})^*(E_ju_n-u_n)]|_{F_0}$ converges to $[(\Phi_j^{-1})^*(E_ju-u)]|_{F_0}$ in $H^{1/2}(\partial K)$. 

Next, observe that for $w\in H^1(\partial K)$ with $w\in H_0^1(F_0)$, ,
\begin{align*}
|w|^2_{H^{1/2}(\partial K)}&\geq c\int_{F_0}\int_{\Phi_j(F_{1,i})}\frac{|w(x)|^2}{|x_i^2+|x'-y'|^2+y_d^2|^{(d+1)/2}}dy_ddy'dx'dx_i\\
&\geq c\|x_i^{-1/2}w\|^2_{L^2}.
\end{align*}
Since $(\Phi_j^{-1})^*(E_ju_n-u_n)|_{\{x_i=0\}}=0$,  $i=1,\dots,j+1$, and $u_n\in H^1(\partial K)$, this implies
\begin{align*}
&\sum_{i=0}^{j+1}\|x_i^{-1/2}(\Phi_j^{-1})^*[E_j(u_n-u_m)-u_n+u_m\|_{L^2(F_0)}\\
&\qquad\qquad\leq C\|(\Phi_j^{-1})^*[E_j(u_n-u_m)-u_n+u_m]\|_{H^{1/2}(F_0)},
\end{align*}

This implies that for $i=1,\dots,j+1$, $x_i^{-1/2}(\Phi_j^{-1})^*E_ju\in L^2$ and 
$$
\sum_{i=0}^{j+1}\|x_i^{-1/2}(\Phi_j^{-1})^*E_ju\|_{L^2(F_0)}+\|(\Phi_j^{-1})^*E_ju\|_{H^{1/2}(F_0)}\leq C\|u\|_{H^{1/2}(\partial K)},
$$
which is~\eqref{e:indigoInequality} and hence we have proved~\eqref{e:purpleEstimates}.

The equalities~\eqref{e:purpleEquality} follow easily for $u\in C^\infty(\partial K)$ and follow for $u\in H^{1/2}(\partial K)$ by density.
The proof is now complete by induction.
\end{proof}

\begin{lemma}[Extension operators: version 2]
\label{l:extend}
 Let $K\subset \mathbb{R}^d$ be an open $d$-simplex. Then there is $C>0$ such that for all $p\geq 0$, there are extension operators $E_{j,i}:\mathbb{P}^{p}(F_{j,i})\cap H_0^1(F_{j,i})\to \mathbb{P}^{p}(K)$, $1\leq j<d$, and $E_{d,i}:\mathbb{C}\to \mathbb{P}^p(K)$ such that 
 \begin{gather*}
 E_{j,i}|_{\overline{F_{j,i}}}=I,\qquad E_{j,i}|_{\overline{F_{j,m}}}=0,\,\, m\neq i,\\
 E_{j,i}|_{\overline{F_{r,m}}}=0,\,\,\overline{F_{r,m}}\cap \overline{F_{j,i}}=\emptyset
 \\
\|E_{j,i}|_{F_{r,l}}q\|_{H_p^{1}(F_{r,l})}\leq C p^{\frac{r-j}{2}}\| q\|_{H_{p}^{1/2}(F_{j,i})},\quad 0\leq r<j.
 \end{gather*}
 where 
 \begin{align*}
 \|u\|_{H_{p}^{1}(F)}:=\|u\|_{L^2}+\langle p\rangle ^{-1}|u|_{H^{1}}.
 \end{align*}
\end{lemma}
\begin{proof}
We claim that it is enough to show that for an open $m+1$ simplex, $K_m\subset \mathbb{R}^m$, there there is $E_m:\mathbb{P}^{p}(\partial K_m)\to \mathbb{P}^p(K_m)$ such that 
\begin{equation}
\label{e:pExtendClaim1}
E^{K_m}_m|_{\partial K_m}=I,
\end{equation}
and 
\begin{equation}
\label{e:pExtendClaim2}
\|E^{K_m}_m u\|_{H_p^1(K_m)}\leq C \langle p\rangle^{-1/2}\|u\|_{H_p^{1/2}(\partial K)}.
\end{equation}
Indeed, suppose such an operator exists for any $K_m$. 

When $d=1$, $E^{K_m}_m$ is the required extension. Therefore, we can assume by induction that the statement of the lemma holds for some $d\geq 1$.

Let $K\subset \mathbb{R}^{d+1}$ be an open $d+1$ simplex. Then, by the inductive hypothesis, for any $j>1$, and any $i$, and any $F_{1,l}$ such that $F_{j,i}\subset \overline{F}_{1,l}$, there is an extension
$$
E_{j,i}^{l}:\mathbb{P}^p(F_{j,i})\cap H_0^1(F_{j,i})\to \mathbb{P}^p(F_{1,l})
$$
such that 
\begin{gather*}
E_{j,i}^l|_{F_{j,i}}=I,\qquad E_{j,i}^l|_{F_{j,m}}=0,\quad m\neq i, F_{j,m}\subset \overline{F}_{1,l},\,\\
E_{j,i}^l|_{F_{r,m}}=0,\,\, F_{r,m}\subset \overline{F}_{1,l}\,\text{ and }\,\overline{F_{r,m}}\cap \overline{F_{j,i}}=\emptyset,
\end{gather*}
and, for $F_{r,m}\subset \overline{F_{1,l}}$, 
$$
\|E_{j,i}|_{F_{r,m}}u\|_{H_p^1(F_{r,m})}\leq Cp^{\frac{r-j}{2}}\|u\|_{H_{p}^{1/2}(F_{j,i})}.
$$
Notice that for $u\in H_{0}^{1}(F_{j,i})$, defining $v$ on $\partial K$ by
$$
v:=\begin{cases} E_{j,i}^lu,&F_{j,i}\subset \overline{F_{1,l}}\\
0,&F_{j,i}\cap \overline{F_{1,l}}=\emptyset,
\end{cases}
$$
we have $v\in H^1(\partial K)$ since it is in $H^1$ on each boundary face and is continuous across codimension-two boundary faces.

Fix $l$ such that $F_{j,i}\subset \overline{F_{1,l}}$. Then, by our claims~\eqref{e:pExtendClaim1} and~\eqref{e:pExtendClaim2}, there is $E_{d+1}^{K}:\mathbb{P}^{p}(\partial K)\to \mathbb{P}^p(K)$ such that 
$$
E_{d+1}^{K}|_{\overline{F_{1,l}}}=I,\quad (E_{d+1}^Kv)|_{F_{r,m}}=0\,\text{ if }r=j,\,m\neq i\text{ or }\overline{F_{r,m}}\cap \overline{F_{j,i}}=\emptyset,
$$
and
$$
\|E_{d+1}^Kv\|_{H_p^1(K)}\leq C\langle p\rangle ^{-\frac{1}{2}}\|v\|_{H_p^{1/2}(\partial K)}\leq C\langle p\rangle ^{-\frac{1}{2}}\|v\|_{H_p^1(\partial K)}\leq \langle p\rangle^{-\frac{j}{2}}\|u\|_{H_{p}^{1/2}},
$$
which provides the required extension for $j>1$. 

For $j=1$, and $u\in H_{0}^{1}(F_{1,i})$, we define $\tilde{u}\in H^{1}(\partial K)$ by
$$
\tilde{u}|_{\overline{F_{1,l}}}=\begin{cases} u&l=i\\0&\text{else}\end{cases}.
$$
Then, define 
$$
E_{1,i}u=E_d^{K}\tilde{u}
$$
and the required estimates follow from~\eqref{e:pExtendClaim1} and~\eqref{e:pExtendClaim2}. 

It therefore remains to prove that $E_m^{K_m}$ exists. Without loss of generality, we can assume that
$$
K_m=\Big\{ x\in \mathbb{R}^m\,:\, 0<x_i,\, \sum_{i=1}^mx_i<1\Big\}.
$$
Let $\psi\in C_c^\infty(\mathbb{R};[0,1])$ with $0\notin \supp (1-\psi)$, and define $\psi_{p,K_m}(x):=\psi(\langle p\rangle d(x,\partial K_m))$. Let $E$ be the extension operator constructed in Lemma~\ref{l:extendA} and define
$$
E_m^{K_m}:= E+\Pi_{H_0^1(K_m)}^{\mathbb{P}^p}(\psi_{p,K_m}-1)E,
$$
where $\Pi^{\mathbb{P}^p}_{H_{0}^1(K_m)}$ is the orthogonal projector with respect to the $H^1$ norm onto $\mathbb{P}^p(K_m)\cap H_{0}^1(K_m)$.

It is now enough to show that
\begin{equation}
\label{e:H1BoundEm}
\|E_mu\|_{H_p^1}\leq C\langle p\rangle^{-1/2}\|u\|_{H_p^{1/2}},
\end{equation}

To see this, we first estimate the $H^1$ norm 
\begin{align*}
\langle p\rangle^{-1}\|E_m^{K_m}u\|_{H^1(K_m)}&\leq 2\langle p\rangle^{-1}\|Eu\|_{H^1(K_m)}+\|\psi_{p,K_m}Eu\|_{H_p^1}\\
&\leq C\langle p\rangle^{-1}\|Eu\|_{H^{1/2}}+\langle p\rangle^{-1}\|\psi_{p,K_m}Eu\|_{H^1}.
\end{align*}
Thus, we need to show that 
\begin{equation}
\label{e:H1Ema}
\langle p\rangle^{-1}\|\psi_{p,K_m}Eu\|_{H^1}\leq C\langle p\rangle^{-1/2}\|u\|_{H_p^{1/2}}
\end{equation}
to obtain
\begin{equation}
\label{e:derivativeBoundEm}
\langle p\rangle^{-1}\|E_m^{K_m}u\|_{H^1(K_m)}\leq C\langle p\rangle^{-1/2}\|u\|_{H_{p}^{1/2}}.
\end{equation}

To see~\eqref{e:H1Ema}, we make a Lipschitz change of coordinates so that $d(x,\partial K_m)$ becomes $x_d$ (note that if $\gamma$ is a Lipschitz bijection, then $\|u\circ \gamma \|_{H^1}\leq C\|u\|_{H^1}$, $\|u\circ \gamma \|_{L^2}\leq C\|u\|_{L^2}$ and hence it is enough to estimate in the new coordinates) and hence 
$$
\psi_{p,K_m}Eu(x_d,x')=\chi(\langle p\rangle x_d )\frac{1}{x_d}\int_0^{x_d} \partial_{x_d}(Eu)(x',x_d)ds+ \chi(\langle p\rangle x_d)Eu(x',0).
$$
Therefore, by the Hardy inequality \eqref{e:Hardy},
\begin{equation}
\label{e:L2BoundsWeighted}
\begin{aligned}
\|\psi(\langle p\rangle x_d)Eu\|_{L^2}&\leq \|\psi(\langle p\rangle x_d)x_d\|_{L^\infty} \|Eu\|_{H^1}+C\langle p\rangle^{-1/2}\|u\|_{L^2}\\
&\leq C \langle p\rangle ^{-1}\|Eu\|_{H^1}+C\langle p\rangle ^{-1/2}\|u\|_{L^2}\\
&\leq C \langle p\rangle ^{-1/2}\big(\langle p\rangle ^{-1/2}\|u\|_{H^{1/2}}+\|u\|_{L^2}\big)\\
&\leq C\langle p\rangle ^{-1/2}\|u\|_{H_p^{1/2}}.
\end{aligned}
\end{equation}
Next, 
\begin{align*}
\langle p\rangle^{-1}\|\nabla \psi(\langle p\rangle x_d)Eu\|_{L^2}&\leq \|\psi'(\langle p\rangle x_d)Eu\|_{L^2}+\langle p\rangle^{-1}\|Eu\|_{H^1}\\
&\leq C\langle p-1\rangle ^{-\frac{1}{2}}\|u\|_{H_p^{1/2}}+C\langle p\rangle^{-1}\|u\|_{H^{1/2}}\\
&\leq C\langle p \rangle^{-1/2}\|u\|_{H_p^{1/2}},
\end{align*}
and hence we have proved~\eqref{e:H1Ema} and, in particular,~\eqref{e:derivativeBoundEm}. 

For the $L^2$ bound, by \eqref{e:derivativeBoundEm}, \eqref{e:L2BoundsWeighted}, and the self-adjointness of $(I-\Pi_{\tilde{H}_{0}^1}^{\mathbb{P}^p})$ on $H^1$, 
\begin{align*}
&\|E_m^{K_m}u\|_{L^2}\\
&\leq \|\psi_{p,K_m}Eu\|_{L^2}+\|I-\Pi_{H_0^1}^{\mathbb{P}^p})(\psi_{p,K_m}-1)Eu\|_{L^2}\\
&\leq C\langle p\rangle^{-1/2}\|u\|_{H_p^{1/2}}+ \|(I-\Pi_{\tilde{H}_0^1}^{\mathbb{P}^p})\|_{H_{1,0}^1\to L^2}\|(\psi_{p,K_m}-1)Eu\|_{H^1}\\
&\leq C\langle p\rangle^{-1/2}\|u\|_{H_p^{1/2}}+C\|(I-\Pi_{H_0^1}^{\mathbb{P}^p})\|_{H^2\cap H^1_{0}\to H_{0}^1}\Big(\langle p\rangle^{1/2}\|u\|_{H_p^{1/2}}+\|u\|_{H^{1/2}}\Big).
\end{align*}
Now, Lemma~\ref{l:polynomialsWithBoundaries} implies that 
$$
\|(I-\Pi_{\tilde{H}_0^1}^{\mathbb{P}^p})\|_{H^2\cap H^1_{1,0}\to H_{1,0}^1}\leq C\langle p\rangle^{-1},
$$
and hence 
\begin{align*}
\|E_m^{K_m}u\|_{L^2}&\leq C\langle p\rangle^{-1/2}\|u\|_{H_p^{1/2}}+C\langle p\rangle^{-1/2}\Big(\|u\|_{H_p^{1/2}}+\langle p\rangle^{-1/2}\|u\|_{H^{1/2}}\Big)\\
&\leq C\|u\|_{H_{p}^{1/2}}.
\end{align*}
Combining this with~\eqref{e:derivativeBoundEm} implies~\eqref{e:H1BoundEm} and hence completes the proof.
\end{proof}

\subsection{Proof of Proposition~\ref{p:toBoundaryCompatible}}
\label{s:proofBoundaryCompatible}
\begin{proof}[Proof of Proposition~\ref{p:toBoundaryCompatible}]
We follow the steps outlined in the paragraph ``Idea of the  proof of Proposition~\ref{p:toBoundaryCompatible}" (just after the statement of Proposition \ref{p:toBoundaryCompatible}).

First observe that it is enough to consider
$$
K= \Big\{ x\in\mathbb{R}^d\,: 0< \sum_{i=1}^d x_i< 1,\, x_i>0\Big\}.
$$
For $0<r< d$, define 
$$
\mathcal{A}_{r,i}:= \big\{A:F_{r,i}\to F_{r,i'}\,:\, A\text{ is an affine isomorphism},\, 1\leq i'\leq N_{r,d}\big\},
$$
and let
\begin{equation}
\label{e:optimalPolyExtend}
\begin{gathered}
(\tilde{T}_{r,i}u)(x):=q(x)\, \text{ for }x\in F_{r,i}, \\
q:=\frac{1}{|\mathcal{A}_{r,i}|}\sum_{A\in\mathcal{A}_{r,i}}(A^{-1})^* \mathcal{M}_{A^{-1}(F_{r,i})}\Big(A^*u, H_0^1( A^{-1}(F_{r,i}))\Big),\\
\text{ where } \,\,\mathcal{M}_F(u,V):=\operatorname{argmin}\big\{ \|u-q\|_{H_{p}^{1/2}(F)}\,:\, q\in V\cap \mathbb{P}^{p-1}\big\},
\end{gathered}
\end{equation}
and set $(\tilde{T}_{d,i}u):=u|_{F_{d,i}}\in\mathbb{C}$. Then, define
$$
T_{r,i}:=E_{r,i}\tilde{T}_{r,i},
$$
where $E_{r,i}$ is the extension from Lemma~\ref{l:extend}.
Let
\beq\label{e:star}
e_d:= (I-\mathcal{J}^{p-1})u,\qquad e_j:= e_{j+1}-\sum_{i=1}^{N_{j+1,d}}T_{j+1,i}e_{j+1},\quad j=0,\dots, d-1
\eeq
(i.e., $\sum_{i=1}^{N_{j+1,d}}T_{j+1,i}e_{j+1}$ is the function $g_{j+1}$ in the discussion of the ideas in the proof above).
We then let
$$
\mathcal{C}^{p-1}u:=u-e_0.
$$
Note first that $\mathcal{C}^{p-1}u\in\mathbb{P}^{p-1}$ since the image of all the $T_{r,i}$ maps is contained in $\mathbb{P}^{p-1}$.
We now define recursively $Q_{d,i}(u|_{F_{d,i}}):=u|_{F_{d,i}}$ and for $1\leq j<d$, 
\begin{multline}
    \label{e:whatIsQij}
{\tempIT}Q_{j,i}v:=\frac{1}{|\mathcal{A}_{j,i}|}\sum_{A\in\mathcal{A}_{j,i}}\\
(A^{-1})^*\mathcal{M}_{A^{-1}(F_{j,i})}\Big(A^*v, \{q\in C^\infty \,:\,q|_{A^{-1}(F_{j+1,\ell})}=A^*Q_{j+1,\ell}v\text{ if }F_{j+1,\ell}\subset\partial F_{j,i}\big\}\Big)\Big).
\end{multline}
We claim that~\eqref{e:faceOperator} holds with this $Q_{j,i}$, so that, in particular,~\eqref{e:boundaryCompatible1} holds. 

To prove this, we claim that all $1\leq j\leq d$ and all $i$
\begin{equation}
\label{e:polyInductionClaim}
e_{j-1}|_{F_{j,i}}=u|_{F_{j,i}}-Q_{j,i}(u|_{\overline{F_{j,i}}}).
\end{equation}
Once this is proved, 
$$
\mathcal{C}^{p-1}u|_{F_{j,i}}=(u-e_0)|_{F_{j,i}}=Q_{j,i}(u|_{\overline{F_{j,i}}}),
$$
and hence~\eqref{e:faceOperator} holds.

To prove~\eqref{e:polyInductionClaim}, first observe that, by the definition of $e_{d-1}$ and the fact $T_{d,i}e_{d,i}|_{F_{d,j}}=e_{d,i}\delta_{ij}$
for all $i$ and $j$,
$$
e_{d-1}|_{F_{d,i}}=0=u|_{F_{d,i}}-Q_{d,i}(u|_{F_{d,i}}). 
$$
Suppose by induction that for some $0<j\leq d-1$, all $i$, and all $r =d,d-1,\dots, j+1$, $e_{j}|_{F_{r,i}}=u-Q_{r,i}u$. 
Then, by \eqref{e:star},
\beq\label{e:rainy1}
e_{j-1}=e_{j}-\sum_{i=1}^{N_{j,d}}T_{j,i}e_{j}
\eeq
and $T_{j,i}e_j|_{F_{r,i}}=0$ for $j<r\leq d$ 
(since $T_{j,i}e_j|_{F_{j,i}}$ has zero boundary data by the definition of $\tilde{T}_{j,i}$ and $T_{j,i}|_{\overline{F_{j,i'}}}=0$ for $i\neq i'$
by the definition of $E_{j,i}$)
and hence
$$
e_{j-1}|_{F_{r,i}}=e_j|_{F_{r,i}},\qquad j<r\leq d.
$$
Therefore, to prove~\eqref{e:polyInductionClaim} it remains to check that $e_{j-1}|_{F_{j,i}}=u-Q_{j,i}u$. By \eqref{e:rainy1} and \eqref{e:optimalPolyExtend},
\begin{align*}
&e_{j-1}|_{F_{j,i}}\\
&=e_{j}|_{F_{j,i}}-\frac{1}{|\mathcal{A}_{j,i}|}\sum_{A\in\mathcal{A}_{j,i}}(A^{-1})^*\mathcal{M}_{A^{-1}(F_{j,i})}\Big(A^*e_j,H_0^1(A^{-1}(F_{j,i}))\Big)\\
%
&=e_{j}|_{F_{j,i}}\\
&-\frac{1}{|\mathcal{A}_{j,i}|}\sum_{A\in\mathcal{A}_{j,i}}(A^{-1})^*\operatorname{argmin}\Big\{\\
&\qquad\qquad \|A^*(e_{j}-u)+A^*u-q\|_{H_p^{1/2}(A^{-1}(F_{j,i}))}\,:\,q\in\mathbb{P}^{p-1}\cap H_0^1(A^{-1}(F_{j,i}))\Big\}.
\end{align*}
Since $e_j-u\in\mathbb{P}^{p-1}$, we have 
\begin{align*}
&e_{j-1}|_{F_{j,i}}\\
&=e_{j}|_{F_{j,i}}+u-e_j|_{F_{j,i}}\\
&-\frac{1}{|\mathcal{A}_{j,i}|}\sum_{A\in\mathcal{A}_{j,i}}(A^{-1})^*\mathcal{M}_{A^{-1}(F_{j,i})}\Big(A^*u,\{ q\in C^\infty\,:\, (q+A^*(e_j-u))|_{\partial A^{-1}(F_{j,i})}=0\}\Big)\\
&=u-Q_{j,i}(u|_{\overline{F_{j,i}}}),
\end{align*}
where we have used \eqref{e:polyInductionClaim} in the last step.

Now, to check~\eqref{e:boundaryCompatible2}, observe that for $\kappa:F_{r,i_2}\to F_{r,i_1}$, an affine isomorphism, 
$$
\mathcal{A}_{r,i_2}=\big\{ A\circ\kappa\,:\, A\in\mathcal{A}_{r,i_1}\big\}.
$$
Hence,~\eqref{e:boundaryCompatible2} with $T$ replaced by $Q_{j,i}$ follows directly from the definition of $Q_{j,i}$. 

Notice that if $\mathcal{J}^{p-1}u=u$, then $e_d=0$ and hence $e_0=0$. Thus $\mathcal{C}^{p-1}u=u$. Therefore, if $\mathcal{J}^{p-1}u=u$ for all $u\in\mathbb{P}^{p-1}$ then $\mathcal{C}^{p-1}u=u$ for all $u\in\mathbb{P}^{p-1}$.

Finally, to prove estimates on $\|e_0\|_{H_p^j(k)}=\|(I-\mathcal{C}^{p-1})u\|_{H_p^j(K)}$, we start by observing that
\beq\label{e:lastDay1}
\|e_d\|_{H_p^j}
=\|(I-\mathcal{J}^{p-1})u\|_{H_p^j(K)}
\eeq
by the definition of $e_d$ \eqref{e:star}. First note that, since $j>\frac{d}{2}$, for $0\leq l\leq d$, and all $i$,
$$
p^{-\frac{l}{2}}\|e_d|_{F_{l,i}}\|_{H_p^{j-\frac{l}{2}}}\leq C\|e_d\|^{1-\frac{l}{2j}}_{L^2}
\|e_d\|^{\frac{l}{2j}}_{H^{j}_p}.
 $$


We now proceed by induction.
Suppose for some $0<n\leq d$ 
\begin{equation}
\label{e:induction}
\|e_n\|_{H_p^1(K)}+\sum_{\substack{0<l\leq n\\i}}p^{\frac{-l}{2}}\|e_{n}|_{F_{l,i}}\|_{H_{p}^{1/2}}\leq C\|e_d\|_{L^2}^{1-\frac{d}{2j}}\|e_d\|_{H_p^j}^{\frac{d}{2j}}.
\end{equation}
By Lemma \ref{l:extend}, for each $i$,
\begin{align*}
\sum_{l<n}p^{\frac{-l}{2}}\|(T_{n,i}e_n)|_{F_{l,i}}\|_{H_p^{1}(F_{l,i})}&\leq C p^{-\frac{n}{2}}\|\tilde{T}_{n,i}e_n\|_{H_{00,p}^{1/2}(F_{n,i})}\\
 &\leq Cp^{-\frac{n}{2}}\|e_n|_{F_{n,i}}\|_{H_{p}^{1/2}}
\leq C\|e_d\|_{L^2}^{1-d/(2j)}\|e_d\|_{H_p^j}^{\frac{d}{2j}}.
\end{align*}
Therefore, by \eqref{e:star} and \eqref{e:induction},
$$
\|e_{n-1}\|_{H^1_p(K)}\leq \|e_{n}\|_{H^1_p(K)}+C\|e_d\|_{L^2}^{1-\frac{d}{2j}}\|e_d\|_{H_p^j}^{\frac{d}{2j}}\leq C\|e_d\|_{L^2}^{1-\frac{d}{2j}}\|e_d\|_{H_p^j}^{\frac{d}{2j}},
$$
and for $1<l\leq n-1$,
$$
p^{-\frac{l}{2}}\|e_{n-1}\|_{H^{1/2}_{p}(F_{l,i})}\leq p^{-\frac{l}{2}}\|e_{n}\|_{H^{1/2}_{p}(F_{l,i})}+C\|e_d\|_{L^2}^{1-\frac{d}{2j}}\|e_d\|_{H_p^j}^{\frac{d}{2j}}\leq C\|e_d\|_{L^2}^{1-\frac{d}{2j}}\|e_d\|_{H_p^j}^{\frac{d}{2j}}.
$$
Hence,~\eqref{e:induction} holds with $n$ replaced by $n-1$ and the induction is complete. Since 
$\|e_d\|_{L^2}^{1-\frac{d}{2j}}\|e_d\|_{H_p^j}^{\frac{d}{2j}}\leq 
\|e_d\|_{H_p^j}$,
the bound on
$\|e_0\|_{H_p^1(K)}=\|(I-\mathcal{C}^{p-1})u\|_{H_p^1(K)}$ follows from \eqref{e:lastDay1}.
\end{proof}

\subsection{Proof of Theorems~\ref{t:approxHighLowReg} and \ref{t:approxHighLowRegL2}}\label{s:theEndPoly}

We now combine Proposition~\ref{p:toBoundaryCompatible} with Propositions~\ref{p:lowregularityApproximation} and~\ref{p:pExplicit1} to prove Theorem~\ref{t:approxHighLowReg} 
(with Proposition \ref{p:lowregularityApproximation} used for the fixed-regularity approximation, and Proposition \ref{p:pExplicit1} used for the high-regularity approximation).

\begin{proof}[Proof of Theorem~\ref{t:approxHighLowReg}]
Let $\widehat{T}$ be the reference element for $\mathcal{T}$, $p\geq \max(t-1,1)$, $\mathcal{L}^{p}:H^t(\widehat{T})\to \mathbb{P}^p$ be the operator constructed in Proposition~\ref{p:lowregularityApproximation}, and $\mathcal{I}^p$ be the operator constructed in Proposition \ref{p:pExplicit1} with $\Omega=\widehat{T}$.

Then, by Proposition~\ref{p:toBoundaryCompatible}, for each $T\in\mathcal{T}$,  there are boundary compatible operators ${\tempIT}\mathcal{C}_i^{p}:H^\ell(\widehat{T})\to \mathbb{P}^{p}$, $i=1,2$ such that, by \eqref{e:Legendre1} (with $k=1$ and $\ell=0,1,\ldots,j$),
\begin{align}\label{e:run1}
\|(I-{\tempIT}\mathcal{C}_1^p)v\|_{H_p^1(\widehat{T})}\leq C\|(I-\mathcal{L}^p)v\|_{H_p^j(\widehat{T})}\leq \frac{C}{p^{t}}|v|_{H^t(\widehat{T})},
\end{align}
by \eqref{e:approxOmega} (with $k=1$ and $\ell=0,1,\ldots,j$),
\begin{align}
\|(I-{\tempIT}\mathcal{C}_2^p)v\|_{H_p^1(\widehat{T})}\leq C\|(I-\mathcal{I}^p)v\|_{H_p^j(\widehat{T})}\leq \frac{C^{p+1}}{p^{p+1}}|v|_{H^{p+1}(\widehat{T})},
\label{e:run2}
\end{align}
and 
\beq\label{e:lastDay77}
{\tempIT}\mathcal{C}_1^p v|_{\partial\widehat{T}}={\tempIT}\mathcal{C}_2^p v|_{\partial\widehat{T}}.
\eeq

Then, define $\mathcal{C}^p_{\mathcal{T}}:C^0(\Omega)\cap \oplus_{T\in\mathcal{T}}H_k^t(T)\to \mathcal{P}_{\mathcal{T}}^p$ by
\beq\label{e:CTp}
\mapF_T^*(\mathcal{C}^p_{\mathcal{T}}u)|_{T}:= \begin{cases}{\tempIT}\mathcal{C}_1^p(\mapF_T^*u|_{T}),&T\in\mathcal{T}_1\\
{\tempIT}\mathcal{C}_2^p(\mapF_T^*u|_{T}),&T\in\mathcal{T}_2.
\end{cases}
\eeq
Observe that $\mathcal{C}_{\mathcal{T}}^p$ is triangulation preserving (in the sense of Definition \ref{d:triangulationPreserving}) since $\mathcal{C}_j^p$, $j=1,2,$ are boundary compatible. Furthermore, since $\mathcal{T}$ is a $C^{p+1}$ simplicial triangulation, by \eqref{e:run1}
and Lemma \ref{l:analyticpullback}, for $\ell=0,1$
\begin{align*}
\sum_{T\in\mathcal{T}_1}k^{-2\ell}|(I-\mathcal{C}_{\mathcal{T}}^p)u|^2_{H_{p}^\ell(T)}&\leq C\sum_{T\in\mathcal{T}_1} h_{T}^{d-2\ell}k^{-2\ell}\|(I-{\tempIT}\mathcal{C}^{p}_1)\mapF_T^*u\|^2_{H_{p}^\ell(\widehat{T})}\\
&\leq \sum_{T\in\mathcal{T}_1}\frac{C}{p^{t-\ell}}k^{-2\ell}h_{T}^{d-2\ell}|A_T^*R_T^*u|^2_{H^t(\widehat{T})}\\
&\leq \frac{C}{p^{2(t-\ell)}}\sum_{T\in\mathcal{T}_1}h_{T}^{2(t-\ell)}k^{-2\ell}|R_T^*u|^2_{H^t(A_{T}^{-1}(\widehat{T}))}\\
&\leq \frac{C}{p^{2(t-\ell)}}\sum_{T\in\mathcal{T}_1}h_{T}^{2(t-\ell)}k^{-2\ell}\|u\|^2_{H^t(T)}\\
&\leq \frac{C}{p^{2(t-\ell)}}\sum_{T\in\mathcal{T}_1}(h_{T}k)^{2(t-\ell)}k^{-t}\|u\|^2_{H^t(T)}.
\end{align*}
and, by \eqref{e:run2}
and Lemma \ref{l:analyticpullback} (applied with $t=p+1$), for $\ell=0,1$
\begin{align*}
\sum_{T\in\mathcal{T}_2}k^{-2\ell}|(I-\mathcal{C}_{\mathcal{T}}^p)u|^2_{H_{p}^\ell(T)}&\leq C\sum_{T\in\mathcal{T}_2} h_{T}^{d-2\ell}k^{-2\ell}|(I-{\tempIT}\mathcal{C}_{2}^{p})\mapF_T^*u|^2_{H_{p}^{p+1}(\widehat{T})}\\
&\leq \sum_{T\in\mathcal{T}_2}\frac{C^{2(p+1)}}{p^{2(p+1)}}k^{-2\ell}h_{T}^{d-2\ell}|A_T^*R_T^*u|^2_{H^{p+1}(\widehat{T})}\\
&\leq \frac{C^{2(p+1)}}{p^{2(p+1)}}\sum_{T\in\mathcal{T}_2}h_{T}^{2(p+1-\ell)}k^{-2\ell}|R_T^*u|^2_{H^{p+1}(A_{T}^{-1}(\widehat{T}))}\\
&\leq C^{2(p+1)}\sum_{T\in\mathcal{T}_2}(h_{T}k)^{2(p+1-\ell)}k^{-2(p+1)}\|u\|^2_{H^{p+1}(T)}.
\end{align*}
Therefore, to establish the result we only need 
to show that $\mathcal{C}_{\mathcal{T}}^pu\in H^1(\Omega)$. 

It is enough to check that $\mathcal{C}_{\mathcal{T}}^pu\in C^0(\Omega)$. It is clear that $\mathcal{C}_{\mathcal{T}}^p\in \oplus_{T\in \mathcal{T}}\overline{C^0}(T)$. Therefore, we need only check that if $x\in T_1\cap T_2$, then 
\beq\label{e:lastDay2}
[\mathcal{C}^p_{\ell_1}(\mapF_{T_1}^*u)](\mapF_{T_1}^{-1}(x))=[\mathcal{C}^p_{\ell_2}(\mapF_{T_2}^*u)](\mapF_{T_2}^{-1}(x)),
\eeq
where $T_1\in\mathcal{T}_{\ell_1}, T_2\in \mathcal{T}_{\ell_2}$. 
By the definition of a triangulation, if $x\in T_1\cap T_2$, then $x\in \partial T_1\cap \partial T_2$. In particular, since $\mathcal{T}$ is affine-conforming (in the sense of Definition \ref{d:AffineConforming}) there are $(j_1,i_1)$ and $(j_2,i_2)$ such that $x\in F_{j_1,i_1}(T_1)\cap F_{j_2,i_2}(T_2)$ and hence, $F_{j_1,i_1}(T_1)=F_{j_2,i_2}(T_2)$ and there is an affine isomorphism $\kappa:
\overline{\mapF_{T_1}^{-1}(F_{j_1,i_1}(T_1))}\to \overline{\mapF_{T_2}^{-1}(F_{j_2,i_2}(T_2))}$ such that  
\beq\label{e:lastLastDay1}
\mapF_{T_1}|_{\overline{\mapF_{T_1}^{-1}(F_{j_1,i_1}(T_1))}}=\mapF_{T_2}|_{\overline{\mapF_{T_2^{-1}(F_{j_2,i_2}(T_2)}}}\circ \kappa.
\eeq
Since $u\in C^0$,
\beq\label{e:lastDay3}
(\mapF_{T_1}^*u)|_{\overline{\mapF_{T_1}^{-1}(F_{j_1,i_1})}}
=
\Big((\mapF_{T_2}^*u)|_{\overline{\mapF_{T_2}^{-1}(F_{j_2,i_2})}}\Big)\circ \kappa.
\eeq
By \eqref{e:lastDay3}, \eqref{e:lastLastDay1}, 
\eqref{e:boundaryCompatible2} with $T$ replaced by $\mathcal{C}^p_{\ell_1}$ (which holds since 
$\mathcal{C}^p_{\ell}, \ell=1,2$, are boundary compatible), 
and \eqref{e:lastDay77}, for $T_1\in\mathcal{T}_{\ell_1}$ and $T_2\in\mathcal{T}_{\ell_2}$,
\begin{align*}
\mathcal{C}^p_{\ell_1}(\mapF_{T_1}^*u)(\mapF_{T_1}^{-1}(x))&=\mathcal{C}^p_{\ell_1}(\kappa^*\mapF_{T_2}^*u)(\mapF_{T_1}^{-1}(x))\\
&=\mathcal{C}^p_{\ell_1}(\kappa^*\mapF_{T_2}^*u)(
\kappa^{-1}\circ \mapF_{T_2}^{-1}(x))\\
&=\mathcal{C}^p_{\ell_1}(\mapF_{T_2}^*u)(\mapF_{T_2}^{-1}(x))\\
&=\mathcal{C}^p_{\ell_2}(\mapF_{T_2}^*u)(\mapF_{T_2}^{-1}(x)),
\end{align*}
as claimed in \eqref{e:lastDay2}.
\end{proof}

The proof of Theorem~\ref{t:approxHighLowRegL2} is simpler than that of Theorem~\ref{t:approxHighLowReg}; indeed, one can directly use $\mathcal{L}^p$ or $\mathcal{I}^p$ on each element $\mathcal{T}$ and therefore we do not require the triangulation to be affine conforming.

\section{Definition of the 
standard boundary-integral operators}\label{app:A}

In this section we show how scattering of a plane-wave by an obstacle with zero Dirichlet or Neumann boundary conditions can be reformulated as a BIE involving the operators $A_k, A'_k$ \eqref{e:DBIEs} (for the Dirichlet problem) and $\Breg, \Breg'$ \eqref{e:NBIEs} (for the Neumann problem).
The proof that the general Dirichlet and Neumann problems (i.e., with arbitrary boundary data) can be reformulated via these BIEs is very similar; see, e.g., \cite[\S2.6]{ChGrLaSp:12}.


Let $\Omega_- \subset \Rea^d$, $d\geq 2$ be a bounded open set such that its open complement $\Omega_+ :=\Rea^d \setminus \overline{\Omega_- }$ is connected.
Let $\Gamma_-:= \partial \Omega_- $. The results in \S\ref{s:R3} require that $\Gamma$ is $C^\infty$, but the results in this appendix hold when $\Gamma$ is Lipschitz.
Let $\nu$ be the outward-pointing unit normal vector to $\Omega_- $, and let $\gamma^\pm$ and $\partial^{\pm}_\nu$ denote the Dirichlet and Neumann traces, respectively, on $\Gamma_-$ from $\Omega^{\pm}$.

\begin{definition}[Plane-wave sound-soft/-hard scattering problems]\label{def:scat}
Given $k>0$ and the incident plane wave $u^I(x):= \exp(\ri k x\cdot a)$ for $a\in\Rea^d$ with $|a|_2=1$
find the total field $u\in H^1_{\rm loc}(\Omega_+)$ satisfying
$\Delta u + k^2 u =0$ in $\Omega_+$, 
\beqs
\text{ either }\gamma^+ u = 0 \text{ (sound-soft) } \,\,\text{ or }\,\, \partial_\nu^+u =0  \text{ (sound-hard) }\,\,\ton\Gamma,
\eeqs
and, where \(u^S := u - u^I\) is the scattered field,
\beq\label{e:src}
\dfrac{\partial u^S }{\partial r} -\ri ku^S = o \left(\frac{1}{r^{(d-1)/2}}\right)  \text{ as }r:=|x|\rightarrow \infty, \text{ uniformly in $x/r$}.
\eeq
\end{definition}

The results proving existence and uniqueness of the Helmholtz scattering problem \eqref{e:edp} also show that 
the solutions of the sound-soft and sound-hard plane-wave scattering problems exist and are unique; see, e.g., \cite[Theorem 2.12 and Corollary 2.13]{ChGrLaSp:12}.

Let $\Phi_k(x,y)$ be the fundamental solution of the Helmholtz equation defined by
\beq\label{e:fund}
\Phi_k(x,y):= \frac{\ri}{4}\left(\frac{k}{2\pi |x-y|}\right)^{(d-2)/2}H_{(d-2)/2}^{(1)}\big(k|x-y|\big)= \left\{\begin{array}{cc}
  \displaystyle{\frac{\ri}{4}H_0^{(1)}\big(k|x-y|\big)}, & d=2, \\
  \displaystyle{\frac{\re^{\ri k |x-y|}}{4\pi |x-y|}}, & d=3,
     \end{array}\right.
\eeq
where $H^{(1)}_\nu$ denotes the Hankel function of the first kind of order $\nu$.
The single- and double-layer potentials, $\mc{S}\ell$ and $\mc{D}\ell$ respectively, are defined for $k\in \mathbb{C}$, $\phi\in L^2(\Gamma)$, and $x \in \Rea^d\setminus \Gamma$ by
\begin{align}\label{e:SLP}
    \mc{S}\ell \varphi (x) = \int_{\Gamma} \Phi_k (x,y) \varphi (y) d s (y) \quad \tand\quad
    \mc{D}\ell \varphi (x) = \int_{\Gamma} \dfrac{\partial \Phi_k (x,y)}{\partial \nu(y)} \varphi (y) d s (y).
\end{align}
The standard single-layer, adjoint-double-layer, double-layer, and hypersingular operators are defined for $k\in \mathbb{C}$, $\phi\in L^2(\Gamma_-)$, $\psi\in H^1(\Gamma_-)$, and $x\in \Gamma_-$ by
\begin{align}\label{e:SD'}
&V_k \phi(x) := \int_{\Gamma_-} \Phi_k(x,y) \phi(y)\,d s(y), \qquad
\DL_k' \phi(x) := \int_{\Gamma_-} \frac{\partial \Phi_k(x,y)}{\partial \nu(x)}  \phi(y)\,d s(y),\\
&\DL_k \phi(x) := \int_{\Gamma_-} \frac{\partial \Phi_k(x,y)}{\partial \nu(y)}  \phi(y)\,d s(y),
\quad
H_k \psi(x) := \frac{\partial}{\partial_{\nu(x)}} \int_{\Gamma_-} \frac{\partial \Phi_k(x,y)}{\partial \nu(y)}  \psi(y)\,d s(y).
\label{e:DH}
\end{align}
We use the notation $V_k$ for the single-layer, instead of $S_k$ (used, in particular, in \cite{GaRaSp:25}), to avoid a notational clash with the smoothing operator in \S\ref{s:R1}. Similarly, we use the notation $\DL_k$, $\DL_k'$ for the double-layer and its adjoint, instead of $D_k$, $D_k'$, to avoid a notational clash with the operator $D:= -\ri \partial$ used in the rest of the paper.

\begin{theorem}\label{thm:BIEs}

(i) If $u$ is solution of the sound-soft scattering problem of Definition \ref{def:scat}, then
\beq\label{e:Ddirect}
A_k' \partial_\nu^+ u = \partial_\nu^+ u^I - \ri u^I \quad\text{ and } \quad u=u^I-\mc{S}\ell(\partial_\nu^+ u).
\eeq

(ii) If $v\in L^2(\Gamma_-)$ is the solution to
\beq\label{e:Dindirect}
A_k v = -\gamma^+ u^I,
\quad\text{ then } \quad
u:= u^I + (\mc{D}\ell -\ri \mc{S}\ell)v
\eeq
is the solution of the sound-soft scattering problem of Definition \ref{def:scat}.

(iii) If $u$ is solution of the sound-hard scattering problem of Definition \ref{def:scat}, then
\beq\label{e:Ndirect}
\Breg \gamma^+ u = \ri \gamma^+ u^I - S_{\ri k} \partial^+_\nu u^I  \quad\text{ and } \quad u=u^I+\mc{D}\ell(\gamma^+ u).
\eeq

(iv) If $v\in L^2(\Gamma_-)$ is the solution to
\beqs
\Breg' v = -\partial_\nu^+ u^I,
\quad\text{ then } \quad
u:= u^I + (\mc{D}\ell S_{\ri k} -\ri  \mc{S}\ell)v
\eeqs
is the solution of the sound-hard scattering problem of Definition \ref{def:scat}.
\end{theorem}

\bpf[References for the proof]
Part (i) is proved in, e.g., \cite[Theorem 2.46]{ChGrLaSp:12}. Part (ii) is proved in, e.g., \cite[Equations 2.70-2.72]{ChGrLaSp:12}. Part (iii) is proved in, e.g., \cite[Equation 1.6]{GaMaSp:21N}. Part (iv) is proved in, e.g., \cite[Equation 1.8]{GaMaSp:21N}.
\epf

\ble\label{lem:inversebound}
If $\Gamma$ is $C^1$ then
\begin{align*}
\N{A_k'}_{L^2(\Gamma_-)\to L^2(\Gamma_-)} \geq 1/2,\qquad &\N{(A_k')^{-1}}_{L^2(\Gamma_-)\to L^2(\Gamma_-)} \geq 2,\\
\N{\Breg}_{L^2(\Gamma_-)\to L^2(\Gamma_-)} \geq \left(\frac{1}{2}+\frac{1}{4}\right),
\qquad&\N{(\Breg)^{-1}}_{L^2(\Gamma_-)\to L^2(\Gamma_-)} \geq \left(\frac{1}{2}+\frac{1}{4}\right)^{-1}.
\end{align*}
\ele

\bpf
The results for $A_k'$ are proved in \cite[Lemma 4.1]{ChGrLaLi:09} using that $A_k' - (1/2)I$ is compact when $\Gamma$ is $C^1$. The results for $\Breg$ are proved in an analogous way, using that $\Breg - ( \ri/2 -1/4)I$ is compact by, e.g., \cite[Proof of Theorem 2.2]{GaMaSp:21N}.
\epf

\section{Precise statement of Informal Theorem \ref{t:rules}}

\begin{theorem}[Properties of semiclassical pseudodifferential operators]\label{t:rules2}

\ 

\begin{enumerate}
\item (Quantisation.) 
The  \emph{quantisation} map $\Op:S^m(\mathbb{R}^{2d})\to \Psi^{m}_\hbar(\Rea^d)$ 
defined by \eqref{e:quant2} 
is such that $\Op(a)$ 
is properly supported, $\Op(1)=I$, and $[\Op(a)u](x)=a(x)u(x)$ when $a$ is independent of $\xi$.
\item (Principal symbol.)
There is 
a \emph{principal symbol map} $\sigma_\hbar:\Psi^{m}_\hbar(\Rea^d)\to S^m(\mathbb{R}^d)
/ \hbar S^{m-1}(\mathbb{R}^d)
$ such that 
$$\sigma_\hbar\circ \Op = I 
\quad\text{ mod }\quad \hbar S^{m-1}(\mathbb{R}^d).$$
If $A\in \Psi^m_\hbar(\mathbb{R}^d)$ and $\hbar^{-1}\sigma_{\hbar}(A)\in S^{m-1}$, then $\hbar^{-1} A\in \Psi^{m-1}_\hbar(\mathbb{R}^d)$. 
\item (Boundedness in Sobolev norms.)
Given $a\in S^m(\mathbb{R}^{2d})$,
for all $s$, there is $C_s>0$ such that \beq\label{e:appbound1}
\|\Op(a)\|_{H_k^s(\Rea^d)\to H_k^{s-m}(\Rea^d)}\leq C_s
\eeq
(i.e., the quantisation of a symbol of order $m$ decreases Sobolev regularity by $m$). 
Furthermore, for $a\in S^0(\mathbb{R}^{2d})$, there is $C>0$ such that 
\beq\label{e:appbound2}
\|\Op(a)\|_{L^2(\Rea^d)\to L^2(\Rea^d)}\leq \sup_{(x,\xi)\in \Rea^{2d}}|a(x,\xi)|+C\hbar.
\eeq
\item (Composition.) For $a\in S^{m_1}(\mathbb{R}^{2d})$, $b\in S^{m_2}(\mathbb{R}^{2d})$, $\Op(a)\Op(b)\in \Psi^{m_1+m_2}_\hbar(\mathbb{R}^d)$ and 
$$\sigma_\hbar(\Op(a)\Op(b))-ab\in \hbar S^{m_1+m_2-1}(\mathbb{R}^{2d}).$$
Furthermore, there is $c \in S^{m_1+m_2}(\Rea^{2d})$ 
with 
$\supp c \subset \supp a \cap \supp b$ 
such that $\Op(a)\Op(b)=\Op(c)+O(\hbar^\infty)_{\Psi_{\hbar}^{-\infty}}$.

\item (Adjoints.) For $a\in S^{m}(\mathbb{R}^d)$, $(\Op(a))^*\in \Psi^m_\hbar(\mathbb{R}^d)$ and $$\sigma((\Op(a))^*)-\bar{a}\in \hbar S^{m-1}(\mathbb{R}^{2d}).$$
\item (Commutators.) For $a\in S^{m_1}(\mathbb{R}^{2d})$ and $b\in S^{m_2}(\mathbb{R}^{2d})$, $i\hbar^{-1} [\Op(a),\Op(b)]\in \Psi^{m_1+m_2-1}_\hbar(\mathbb{R}^d)$ and 
$$
\sigma_{\hbar}\big(i\hbar^{-1} [\Op(a),\Op(b)]\big)-\{ a,b\}\in \hbar S^{m_1+m_2-2},
$$
where $\{a,b\}=\langle \partial_{\xi}a,\partial_x b\rangle -\langle \partial_{\xi}b,\partial_x a\rangle.$
\item (Sharp G\aa rding inequality.) If $a\in S^m(\mathbb{R}^{2d})$, $w\in S^{m/2}(\mathbb{R}^{2d};\mathbb{R})$ and $\Re a\geq w^2$, then
\begin{equation*}
\Re \langle \Op(a)u,u\rangle\geq \|\Op(w)u\|_{L^2(\Rea^d)}^2-C\hbar\|u\|_{H_k^{(m-1)/2}(\mathbb{R}^d)}^2. 
\end{equation*}
\end{enumerate}
\end{theorem}

\bpf[References for the proof]
Part 1 follows from the definition \eqref{e:quant2}.

Part 2 is proved in \cite[Proposition E.16]{DyZw:19}, \cite[\S9.3.3]{Zw:12}. (Note that \cite[Appendix E]{DyZw:19} consider a restricted class of symbols -- so called \emph{polyhomogeneous symbols}; this ensures that the principal symbol is a function, rather than an equivalence class, but this distinction is not important for the proof of Theorem \ref{t:rules2}.)

For Part 3, 
the bound \eqref{e:appbound1} is reduced to the proof of $L^2 \to L^2$ boundedness for elements of $\Psi^0_\hbar$ in \cite[Proposition E.19]{DyZw:19}. 
The $L^2\to L^2$ bound
\eqref{e:appbound2}
is proved in \cite[Theorem 13.13]{Zw:12}, with a weaker bound (still sufficient to imply \eqref{e:appbound1}) proved in \cite[Theorem 18.1.11]{Ho:85}.

Parts 4 and 5 are proved in \cite[Theorems 9.5 and 4.14]{Zw:12} (see also \cite[Proposition E.8]{DyZw:19}).

Part 6 is \cite[E.1.24]{DyZw:19}, and follows as 
a consequence of the composition result in 
\cite[Theorem 9.5]{Zw:12}/\cite[Proposition E.8]{DyZw:19}.

Part 7 is proved in \cite[Theorem 9.11]{Zw:12}.
\epf

\section*{Acknowledgements}

This article collects research done in collaboration with 
Martin Averseng (CNRS, Angers),
Dean Baskin (Texas A\&M), 
Th\'eophile Chaumont-Frelet (INRIA, Lille),
Shihua Gong (CUHK, Shenzhen), 
Ivan Graham (University of Bath),
David Lafontaine (CNRS, Toulouse),
Manas Rachh (ITT Bombay), and 
Jared Wunsch (Northwestern University) over the last 10 years, and it is a pleasure to acknowledge these collaborations here.

Furthermore, the authors thank Martin Averseng, Th\'eophile Chaumont-Frelet, 
 Simon Chandler-Wilde (University of Reading), 
   David Lafontaine, 
   Markus Melenk (TU Wien), 
   Mostafa Meliani (University of Bath), 
 Mar{\'i}a Ignacia Fierro-Piccardo (University of Bath),  Siavash Sadeghi (University of Reading),
  Alastair Spence (University of Bath), and Maciej Zworski (UC Berkeley) for discussions about drafts of this article.

During the writing of this article, JG was 
supported by EPSRC grants EP/V001760/1 and EP/V051636/1 and Leverhulme Research Project Grant RPG-2023-325, and JG and ES were both supported by the ERC synergy grant ``PSINumScat" 101167139.

\clearpage 
\addcontentsline{toc}{section}{References}
\bibliography{biblio_combined_sncwadditions}
\label{lastpage}
\end{document}